\documentclass[reqno,11pt]{amsart}
\usepackage{amssymb,amsmath,amsthm,}
  \usepackage{a4wide} 
\usepackage{amscd}
\usepackage{amsfonts}
\usepackage{amssymb}
\usepackage{latexsym}
\usepackage{color}
\usepackage{esint}

\usepackage{graphicx}
\graphicspath{ {images/} }

\usepackage[makeroom]{cancel}

\newtheorem{theorem}{Theorem}

\newtheorem{definition}{Definition}
\newtheorem{lemma}{Lemma}
\newtheorem{proposition}[theorem]{Proposition}
\newtheorem{remark}{Remark}
 \newtheorem*{theorem*}{Rough version of the Main theorem}

\let\e=\varepsilon

\let\p=\partial

\let\O=\Omega

\let\o=\omega

\numberwithin{equation}{section}

\let\hide\iffalse
\let\unhide\fi

\DeclareMathAlphabet{\mathpzc}{OT1}{pzc}{m}{it}

\newcommand{\R}{\mathbb{R}}

\renewcommand{\S}{\mathbb{S}}

\newcommand{\be}{\begin{equation}}
\newcommand{\bm}{\begin{multline}}
\newcommand{\ee}{\end{equation}}
\newcommand{\dd}{\mathrm{d}}

\newcommand{\xb}{x_{\mathbf{b}}}
\newcommand{\tb}{t_{\mathbf{b}}}

\newcommand{\f}{\frac}

\newcommand{\Bes}{\begin{eqnarray*}}
\newcommand{\Ees}{\end{eqnarray*}}
\newcommand{\Be}{\begin{equation} }
\newcommand{\Ee}{\end{equation}}
\newcommand{\Bs}{\begin{split}}   

\newcommand{\vertiii}[1]{{\left\vert\kern-0.25ex\left\vert\kern-0.25ex\left\vert #1 
    \right\vert\kern-0.25ex\right\vert\kern-0.25ex\right\vert}}
    
    \newcommand{\vertip}[1]{{\left\{\kern-0.7ex\left\{\kern-0.25ex
    #1 
     \right\}\kern
    -0.7ex\right\}}}

 \makeatletter
\def\munderbar#1{\underline{\sbox\tw@{$#1$}\dp\tw@\z@\box\tw@}}
\makeatother

\def\p{\partial}

\def\O{\Omega}
\def\R{\mathbb{R}}

\def\B{\begin{equation}}
\def\E{\end{equation}}
\def\BN{\begin{eqnarray*}}
\def\EN{\end{eqnarray*}}

\def\bcb{\begin{color}{blue}}
\def\ec{\end{color}}

\def\bcr{\begin{color}{red}}
\def\ec{\end{color}}

\begin{document}
 \date{ \today
 }

 \title{Passage from the Boltzmann equation with Diffuse Boundary %Condition 
 to the Incompressible Euler equation with Heat Convection %transfer
 }

 \author{Yunbai Cao}
  \address{Department of Mathematics, Rutgers University, Piscataway, NJ, 08854, USA email: yc1157@math.rutgers.edu 
 }
 \author{Juhi Jang}
  \address{Department of Mathematics, University of Southern California, Los Angeles, CA, 90089 USA, email: juhijang@usc.edu}
  
 \author{Chanwoo Kim}
  \address{Department of Mathematics, University of Wisconsin-Madison, Madison, WI, 53706, USA,  
  email: chanwoo.kim@wisc.edu
 }

 \maketitle
 
 \begin{abstract} 
We derive the incompressible Euler equations  with heat convection  with the no-penetration boundary condition from the Boltzmann equation with the diffuse boundary  in the hydrodynamic limit for the scale of large Reynold number. 
% We prove the convergence of the Boltzmann equation with the diffuse boundary condition to the incompressible Euler equations with the no-penetration boundary condition with heat convection %transfer 
% in the hydrodynamic limit for the scale of large Reynold number. 
Inspired by the recent framework in \cite{JK}, we consider the Navier-Stokes-Fourier system with no-slip boundary conditions as an intermediary approximation and develop a Hilbert-type expansion of the Boltzmann equation around the global Maxwellian that  allows the nontrivial heat transfer by convection in the limit. 
   % We employ  %follow 
% the recent %new 
% framework in \cite{JK} by considering the Navier-Stokes-Fourier system with no-slip boundary condition as an intermediary approximation and using a new Hilbert-type expansion of the Boltzmann equation around the global Maxwellian. 
To justify our expansion and the limit, a new direct estimate of the heat flux and its derivatives in the Navier-Stokes-Fourier system is established adopting a recent Green's function approach in the study of the inviscid limit. 
% A new direct estimate of the heat flux and its derivatives in the Navier-Stokes-Fourier system is established adopting a recent Green's function approach to justify the limit. %This generalizes the recent result of the incompressible Euler limit in \cite{JK}.
\end{abstract}   

\tableofcontents

%\noindent \textbf{Acknowledgements}: 

\section{Introduction}

Building the connection between kinetic theory and macroscopic fluid dynamics has been an important subject over the past decades. Kinetic equations study the time evolution of the distribution function $F(t,x,v) \ge 0$ representing the density of particles of some rarified gas with position $x$ and velocity $v$ in the phase space $\O \times \mathbb R^3$ at time $t$. The interaction of particles through collisions is often modeled by some binary collision operator $Q(F,F)$. 
 
When the gas is dense enough such that particles go through many collisions, 
the hydrodynamic limits are obtained. 
A small parameter $\mathpzc{Kn}$ called the \textit{Knudsen number}, which represents the ratio of the mean free path of particles between collisions to the characteristic length, is a key dimensionless number 
%can be introduced to describe 
in describing such phenomena. On the other hand, the velocity scale that some macroscopic portion of the gas is transported, described by the kinetic \textit{Strouhal number} $\mathpzc{St}$, also affect the limits. %Hence 
 The dimensionless Boltzmann equation takes the form of 
\Be\label{Boltzmann}
 \mathpzc{St} \p_t F + v\cdot \nabla_x F  = \f{1}{\mathpzc{Kn}} Q(F ,F ), \text{ on } [0, \infty) \times \O \times \mathbb R^3, 
\Ee 
%Here the distribution function of the gas is denoted by $F(t,x,v) \geq 0$ with the time variable $t \in \R_+: = \{ t \geq 0\}$, the space variable $x = (x_1,x_2,x_3)\in \O \subset \R^3$, and the velocity variable $v  = (v_1,v_2,v_3) \in \R^3$.  
where $\O$ is an open subset of $\mathbb R^3$. Throughout this paper, we assume the hard sphere Boltzmann collision operator:
% $Q(\cdot, \cdot)$ of the hard sphere takes the form of 
\Be\label{Q}
\begin{split}
Q(F,G)  = \frac{1}{2} \int_{\R^3} \int_{\S^2}  |(v-v_*) \cdot  \mathfrak{u} | \{ &
F(v^\prime) G(v^\prime_*) + G(v^\prime) F(v^\prime_*)\\
 &- F(v ) G(v_* ) - G(v ) F(v_* ) 
\}\dd \mathfrak{u} \dd v_*, \end{split}\Ee
where $v^\prime := v- ((v-v_*) \cdot  \mathfrak{u} )  \mathfrak{u} $ and $v_*^\prime := v_*+ ((v-v_*) \cdot  \mathfrak{u} )  \mathfrak{u} $. %In this paper we assume a hard sphere cross section $B(v-v_* ,\mathfrak{u} ) =$. %^\gamma q_0 (\frac{v-v_*}{|v-v_*|} \cdot  \mathfrak{u} )$ with $0 \leq \gamma \leq 1$ and $q_0 (\frac{v-v_*}{|v-v_*|} \cdot \o) \in L^1(\{\o \in\S^2\})$.
This collision operator enjoys the collision invariance property: for any $F(v)\text{ and } G(v) $,
%decaying sufficiently fast as $|v|\rightarrow 0$, 
\Be\label{collision_inv}
\int_{\R^3} Q(F,G) (v)  \Big( 1 , v  , \frac{| {v } |^2-3}{\sqrt{6}}
 \Big) \dd v=(0,0,0). %\ \ \ \text{for any } F(v)\text{ and } G(v) 
\Ee
%which represents the local conservation laws of mass, momentum and energy. 
%It is well-known that the global Maxwellian $\mu$ satisfies $Q(\cdot, \cdot) = 0 $, where 
We denote the global Maxwellian, which satisfies  $Q(\cdot, \cdot) = 0 $, by $\mu$
\Be\label{mu_e}
\mu(v) :  = \frac{1}{(2\pi)^{\frac{3}{2}}} e^{-\frac{|v|^2}{2}}. 
\Ee

Another parameter %that can be 
 considered  in hydrodynamic limits is the \textit{Mach number} $\mathpzc{Ma}$, which can be viewed as the scale of fluctuations around some reference flow. In this paper, we take the scale that the \textit{Mach number} is equal to the  \textit{Strouhal number}: 
%\textit{Knudsen number}:
\Be \label{StMae}
 \mathpzc{St}= \mathpzc{Ma} =\e.
\Ee
The reciprocal viscosity of a fluid can be measured in terms of the Reynold number $\mathpzc{Re} = \mathpzc{Ma} / \mathpzc{Kn}$. In this paper, we consider the scale of large Reynold number with
\Be \label{Knke}
\mathpzc{Kn} = \kappa \e,
\Ee
where $\kappa = \kappa (\e)  \to 0 $ as $\e \to 0$. Under such scaling, we will derive the incompressible Euler equations with heat convection with the no-penetration boundary condition in the limit
\begin{align}
\p_t u_E+ u_E \cdot \nabla_x u_E  + \nabla_x p_E  &=0 \ \ \text{in} \ \O, 
\label{Euler}\\
\nabla_x \cdot u_E&=0  \ \ \text{in} \ \O, 
 \label{incomp_E}
  \\
 \ \ u_E \cdot n  &=0  \ \ \text{on} \ \p\O, \label{no-pen}
\\
\p_t \theta_E + u_E \cdot \nabla_x \theta_E    &=0 \ \ \text{in} \ \O.
\label{heat_E}
\end{align} 
Here $n = n(x)$ is the unit outward normal vector at $x \in \p \O$. Note that the boundary value of the heat convection $\theta_E$ is completely determined by the transport equation \eqref{heat_E} and the initial data, therefore no boundary condition on $\theta_E$ is imposed.
 
In many important physical applications, e.g. turbulence theory, boundary effect plays an important role in global dynamics, and it is of both physical and mathematical interests to take the boundary into consideration in the hydrodynamic limit.  %could be taken into consideration in the hydrodynamic limit. 
In this paper we consider one of the physical boundary conditions, the  so-called \textit{diffuse boundary condition}, which models the ideal 
%takes into account the idea 
situation that gas particles reflected from the boundary reach an instantaneous thermal equilibration (see \cite{EGKM}): for $ (x,v) \in 
\{ \p\O \times \R^3: n(x) \cdot v<0\},$
\Be\label{diffuse_BC}
F (t,x,v)  = c_\mu \mu (v) \int_{n(x) \cdot \mathfrak{v}>0} F  (t,x,\mathfrak{v}) (n(x) \cdot \mathfrak{v}) \dd\mathfrak{v}.%\ \text{ for }  (x,v) \in 
%\{ \p\O \times \R^3: n(x) \cdot v<0\},
\Ee
Here, the constant $c_\mu $ is chosen to be $\sqrt{2 \pi }$ so that $\int_{n(x) \cdot \mathfrak{v}>0} \mu  (\mathfrak{v}) (n(x) \cdot \mathfrak{v}) \dd\mathfrak{v} =1$, which ensures the null flux condition $ \int_{\R^3} F(t,x,v) (n(x) \cdot v) \dd v=0 \text{ for }x \in \p\O$.

The first mathematical studies of hydrodynamic limits of the Boltzmann equation may date back to Hilbert in \cite{Hilbert1}, where he introduced the method of Hilbert expansion to obtain the derivations at the formal level. Since then there have been many results on the rigorous justifications of hydrodynamic limits based on the truncated asymptotic expansions method. For instance, the compressible Euler limits with heat transfer are derived in \cite{Caflisch, Lach, UA}; incompressible Navier-Stokes limits are achieved in \cite{Mouhot, DEL, Guo2006}; and diffusive limits from the Vlasov-Maxwell-Boltzmann system has been obtained in \cite{Jang}. On the other hand, since Bardos, Golse, and Levermore \cite{BGL1,BGL2} have proposed the derivation of weak solutions of fluid equations from the renormalized solutions of the Boltzmann equation \cite{DP} in the early nineties, extensive studies have been made in this direction and a lot of important results have been obtained  %success has been achieved. 
(see \cite{Golse,JM,Laure} for the references in this direction). Notably, the Leray solutions of the incompressible Navier-Stokes equations have been successfully derived %achieved 
in \cite{GSR, LM1,LM2}; and dissipative solutions of the incompressible Euler have been obtained in \cite{BGPDB, LM1,LM2,SR}.

It should be noted that all the results mentioned above deal with domains that do not contain boundaries. In general, however, solutions of the Boltzmann equation with physical boundary conditions %can 
behave differently; in particular, high regularity may not be expected (see \cite{GKTT1,GKTT2,Kim}). To overcome this an $L^2-L^\infty$ framework was developed in \cite{Guo10} to study global solutions of the Boltzmann equation with various boundary conditions, which  prompted %leads to 
substantial development in various directions including  \cite{CaoK, CKL, CK1, CKLi, EGKM, EGKM2,EGKM3,GJJ, GJ,GJ1,GW, JK2, JinK, Kim2, KL,KL1,KYun, Wu}. Among them in \cite{EGKM, EGKM2}, the hydrodynamic limit of the incompressible Navier-Stokes-Fourier system was derived from the Boltzmann equation under the diffuse boundary condition.

In terms of the incompressible Euler limit, based on the relative entropy method, the convergence of renormalized solutions of the Boltzmann equation to dissipative solutions of the incompressible Euler equations was first obtained in \cite{BGPDB} assuming the local conservation of momentum which is not guaranteed for renormalized solutions of the Boltzmann equation, some nonlinear estimate, and the initial data to be well-prepared, 
in particular, the initial temperature fluctuation $\theta_{in} = 0$, so there is no heat transfer. Later, the local conservation of momentum assumption was removed in \cite{LM1} by using the local momentum conservation with matrix-valued defect measure satisfied by renormalized solutions of the Boltzmann equation. 
The nonlinear estimate assumption in \cite{BGPDB} was further removed in \cite{SR} using refined dissipation estimates which were developed first in the framework of the BGK equation in \cite{SRBGK}. When considering physical boundaries, the incompressible Euler limit with no heat transfer has been derived in \cite{BGP} from the Boltzmann equation under the Maxwell boundary condition. Finally, in \cite{SR09} the well-prepared initial data assumption in \cite{BGPDB} was removed by the construction of refined approximate solutions with converging modulated entropy, and the incompressible Euler limit with heat transfer has been justified under specular reflection boundary condition. We refer to \cite{Laure} for detailed results and discussions in this direction. 

To the best of our knowledge, the incompressible Euler limit with heat transfer from the Boltzmann equation with diffuse boundary has not been established in any framework yet. The main goal of the paper is to rigorously justify the incompressible Euler limit with heat convection under the no-penetration boundary condition \eqref{Euler}-\eqref{heat_E} from the Boltzmann equation \eqref{Boltzmann} in the scale of \eqref{StMae}, \eqref{Knke} with the diffuse boundary condition \eqref{diffuse_BC}. 

%On the other hand, the dissipative solutions of the incompressible Euler limit with no heat transfer has been derived in \cite{BGP} from the renormalized solutions of the Boltzmann equation under the Maxwell boundary condition. The limits with heat transfer have also been justified in \cite{SR09} under specular reflection boundary condition.
%
%
% 
% \vspace{2cm}
%

Under the scaling \eqref{StMae}, \eqref{Knke}, it is well-known that a mismatch exists between the diffuse boundary condition \eqref{diffuse_BC} of the Boltzmann equation and the no-penetration boundary condition \eqref{no-pen} of the Euler equation. To overcome such difficulty, adopting the recent framework %the same spirit 
in \cite{JK}, we study the Euler limit with heat transfer from the Boltzmann equation through the Navier-Stokes-Fourier system with the no-slip boundary condition 
\begin{align}
\p_t u+ u \cdot \nabla_x u  - \kappa \eta_v \Delta u + \nabla_x p  &=0 \ \ \text{in} \ \O, 
\label{NS}\\
\nabla_x \cdot u&=0  \ \ \text{in} \ \O, 
 \label{incomp_NS}
  \\  \ \ u  &=0  \ \ \text{on} \ \p\O, \label{NSB}
  \end{align}
  \begin{align}
%\\ 
\p_t \theta + u \cdot \nabla_x \theta  - \kappa \eta_c \Delta \theta   &=0 \ \ \text{in} \ \O,
\label{heat_NS1}
\\  \ \ \theta &= 0 \ \ \text{on} \ \p\O, \label{thetaB1}
 \\
\text{Boussinesq relation} \ \ \  \rho+ \theta &= 0 \ \ \text{in} \ \O,
 \label{Bouss} 
\end{align} 
where $\eta_v$ and $\eta_c$ are physical constants that can be computed explicitly by the Boltzmann theory as in \eqref{etavcomp}, \eqref{etavcomp}.
%Choose 
%\Be\label{scale}
% \mathpzc{St}= \mathpzc{Ma} =\e , \ \ \mathpzc{Kn}= \e \kappa.
%\Ee

We expand the Boltzmann solution $F$ around a global Maxwellian $\mu$ plus the first and second correction associated with a Navier-Stokes-Fourier flow \eqref{NS}-\eqref{Bouss}: 
 \Be\label{F_e}
F = \mu+ \e f_1\sqrt{\mu}+ \e^2 f_2 \sqrt{\mu }  + \e^{3/2} f_R\sqrt{\mu}   ,
 %+ 
 %\delta \e^{2  } F_{ 2}  
% + 
%  \delta  \e^{3  } F_{ 3}
%   + 
%  \delta  \e^{4  } F_{ 4} 
%   +  \delta ^{\f{3}{2}} \e F_{ R} .
\Ee 
where $f_1 = (\rho + u\cdot v + \theta \frac{|v|^2 - 3}{2} ) \sqrt{\mu}$, and $f_2$, which also determined by $u$, $\rho$, $\theta$, will be specified in \eqref{IPf_2}-\eqref{Pf2choice}. Here $F=F^\e$, $f_R=f_R^\e$ depend on $\e$, but we drop the superscript $\e$ for the sake of simplicity. Then the equation for the Boltzmann remainder $f_R$ is
%(Alert: there might be sign errors) check \cite{Guo}
 \begin{align}
& 
 \p_t f_R 
 + 
  \frac{1}{\e} v\cdot \nabla_x f_R+    \frac{1}{ \e^2\kappa} Lf_R
  \notag
\\ &= - \frac{1}{\e^{5/2 }\kappa}L f_1\label{eqtn_f_1}
\\ & \ \ - \frac{1}{ \e^{3/2}}  \Big\{  v \cdot \nabla_x  f_1   - \frac{1}{\kappa} \Gamma(f_1,f_1) + \frac{ 1}{    \kappa} L f_2 \Big\}
\label{eqtn_f_2}
\\ & \ \  -\frac{1}{\e^{1/2}}  \Big\{ \p_t f_1 +  v \cdot \nabla_x f_2-  \frac{2}{\kappa} \Gamma(f_1,f_2)   \Big\} \label{eqtn_f_3}
\\ &  \ \ - \e^{1/2} \p_t f_2 +  \frac{2}{\e \kappa} \Gamma(f_1,f_R) +  \frac{2}{\kappa}  \Gamma({f_2}, f_R) +   \frac{ \e^{1/2} }{ \kappa } \Gamma({f_2}, f_2) +    \frac{  1 }{\e^{1/2}\kappa}\Gamma(f_R,f_R)   \label{eqtn_f_5}
\end{align}  
where 
\Be\label{L_Gamma}
L  f = \frac{-2}{\sqrt{\mu }} Q(\mu , \sqrt{\mu }f)   ,\ \
\Gamma (f,g)= \frac{1}{ \sqrt{\mu}}  Q(\sqrt{\mu}f, \sqrt{\mu}g) .
\Ee
The operators $L$ and $\Gamma$ can be written as 
%$Lf (v+\e u)= \nu(v)\tilde{f} (v) - \int_{\R^3} \mathbf{k}(v,v_*) \tilde{f} (v_*) \dd v_*$, and hence
 \begin{align}
Lf(v)&= \nu  f(v) -Kf(v)= \nu (v) f(v)- \int_{\R^3} \mathbf{k} (v , v_* ) f(v_*) \dd v_*,\label{nu_K}\\
\Gamma(f,g) (t,v)  &=  \Gamma_+ (f,g) (t,v) - \Gamma_- (f,g) (t,v)
\notag
\\
&= \iint_{\R^3 \times \S^2} 
|(v-v_*) \cdot \mathfrak{u}| \sqrt{ \mu(v_* )} 
\big(
f(t,v ^\prime) g(t,v_*^\prime)
+ g(t,v ^\prime) f(t,v_*^\prime)\big)\dd \mathfrak{u} \dd v_*\label{Gamma}
%+ \sqrt{\mu(v^\prime)\mu(v_*^\prime)} f(t,v ^\prime)
\\
& \ \ -   \iint_{\R^3 \times \S^2} 
|(v-v_*) \cdot \mathfrak{u}| \sqrt{ \mu(v_* )} 
\big(f(t,v  ) g(t,v_* )+g(t,v  ) f(t,v_* )\big)
%-\sqrt{\mu(v )\mu(v_* )} f(t,v  ) 
\dd \mathfrak{u} \dd v_*\notag
\end{align} 
where the collision frequency $\nu$ is defined as
\Be
\nu(v) : = \iint_{\R^3 \times \mathbb{S}^2} |(v-v_*) \cdot \mathfrak{u}| \mu_0(v_*) \dd \mathfrak{u} \dd v_* \ \sim \  \langle v \rangle : =  \sqrt{1 + |v|^2 },
\Ee
and
\Be
\mathbf{k}(v,v_*) = C_1 |v-v_*| e^{- \frac{|v|^2 + |v_*|^2}{4}}-  \frac{C_{2}}{|v-v_*|}  e^{- \frac{|v-v_*|^2}{8}- \frac{1}{8} \frac{(|v|^2 - |v_*|^2)^2}{|v-v_*|^2}} 
\Ee
for some constant $C_1, C_2 > 0 $.
The null space of $L$, denoted by $\mathcal{N}$, is 
% $\text{Ker} L= \langle \{\varphi_i \sqrt{\mu_\delta}\}_{i=1}^5
 %\varphi_1 \sqrt{\mu_\delta},  \varphi_2 \sqrt{\mu_\delta},  \varphi_3 \sqrt{\mu_\delta},  \varphi_4 \sqrt{\mu_\delta}
%\rangle_{L^2_v(\R^3)}$, 
a subspace of $L^2(\R^3)$ spanned by orthogonal bases $\{\varphi_i \sqrt{\mu }\}_{i=0}^4$ with%and $\{ \tilde{\varphi}_i \sqrt{\mu_0 }\}_{i=0}^4$ with
 \Be\label{basis}
\begin{split}
% \ \ \text{with} \ \ 
 &\varphi_0 := 1
,   \  \ \ \varphi_i: =   {v_i  } 
 \ \ \text{for} \ i=1,2,3 
,   \  \ \ \varphi_4: =   
\frac{| v  |^2-3 }{2}.
%
%& \tilde{\varphi}_0 := 1
%,   \  \ \  \tilde{\varphi}_i: =   {v_i   } 
% \ \ \text{for} \ i=1,2,3 
%,   \  \ \  \tilde{\varphi}_4: =   
%( | {v } |^2-3 )/{\sqrt{6}}
%
%\frac{2/\sqrt{15} }{\sqrt{(\f{2}{5}+ \delta \theta )^2 + \f{6}{25} }} \frac{|v -\e u |^2-3}{2}.
\end{split}\Ee

We define a hydrodynamic projection $\mathbf{P}$ as an $L_v^2$-projection on $\mathcal{N}$ such as 
\Be\begin{split}
\label{P}
\mathbf{P} g:= \sum%_{j=0}^4
  ( {P}_j g) \varphi_j \sqrt{\mu }, \ \ 
 {P}_j g:= \langle g ,\varphi_j \sqrt{\mu } \rangle  , \  \text{and}  \ \  
P g:= (P_0 g, P_1 g, P_2 g, P_3 g, P_4 g ),%\\
%
%\mathbf{\tilde{P}} g:= \sum%_{j=0}^4 
%( \tilde{P}_j g) \tilde{\varphi}_j \sqrt{\mu_0}, \ \ 
% \tilde{P}_j g:= \langle g ,\tilde{\varphi}_j \sqrt{\mu } \rangle  , \  \text{and}  \ \  
%\tilde{P} g:= (\tilde{P_0} g, \tilde{P_1} g, \tilde{P_2} g, \tilde{P_3} g, \tilde{P_4} g ),
%
\end{split}
\Ee
where $\langle \cdot, \cdot \rangle $ stands for an $L^2_v$-inner product. It is well-known that the operators enjoy $\mathbf{P}L=L \mathbf{P}=\mathbf{P} \Gamma=0$. Importantly the linear operator $L$ enjoys a coercivity away from the kernel $\mathcal{N}$: %non-negativity
for $\nu(v)\geq0$ defined in (\ref{nu_K}),
\Be\label{s_gap}
\langle Lf, f\rangle \geq \sigma_0 \|  \sqrt{\nu}(\mathbf{I} - \mathbf{P}) f \|_{L^2 (\R^3)}^2 \ \ \text{for some }   \sigma_0>0.
\Ee

From the no-slip boundary condition \eqref{NSB}, \eqref{thetaB1}, $f_1=0$ on $\p\O$ and hence $\mu + \e f_1 \sqrt{\mu}$ satisfies the diffuse reflection boundary condition (\ref{diffuse_BC}). By plugging (\ref{F_e}) into the boundary condition, we arrive at %derive that  
\Be\notag
 (\e^2 f_2  +  \e^{3/2} f_R )|_{\gamma_-}=
 c_\mu  \sqrt{\mu(v) } \int_{n(x) \cdot \mathfrak{v}>0}  (  \e^2 f_2 +  \e^{3/2} f_R )  \sqrt{\mu(\mathfrak{v})}(n(x) \cdot \mathfrak{v}) \dd \mathfrak{v} .
\Ee 
Letting $P_{\gamma_+}$ be an $L^2 (\{ v : n(x) \cdot v>0 \})$-projection of $\sqrt{c_\mu \mu}$, we derive that 
\Be \label{bdry_fR}
\begin{split}
&f_R(t,x,v)|_{\gamma_-}=    P_{\gamma_+}   f_R(t,x,v)- \e^{1/2} (1- P_{\gamma_+}) f_2(t,x,v)\\
&:=      \sqrt{c_\mu \mu(v)} \int_{n(x) \cdot \mathfrak{v}>0} f_R  (t,x,\mathfrak{v})\sqrt{c_\mu \mu(\mathfrak{v})} (n(x) \cdot \mathfrak{v}) \dd \mathfrak{v} - \e^{1/2} (1- P_{\gamma_+}) f_2(t,x,v).
 \end{split} \Ee
 Note that $\int_{n(x) \cdot v>0} c_\mu\mu(v) (n(x) \cdot v) \dd v=1$. And $\p_t f_R$ satisfies
  \Be\label{bdry_fR_t}
 \begin{split}
 \p_t f_R |_{\gamma_-} = P_{\gamma_+} \p_t f_R- \e^{1/2} (1-P_{\gamma_+})  \p_t (\mathbf{I} - \mathbf{P})f_2.
 \end{split}
 \Ee
 
%  \subsection{Main Theorem}\label{sec:MT}

%The main goal of the paper is to rigorously justify the incompressible Euler limit with heat transfer under the no-penetration boundary condition \eqref{Euler}-\eqref{heat_E} from the Boltzmann equation \eqref{Boltzmann} in the scale of \eqref{StMae}, \eqref{Knke} with the diffuse boundary condition \eqref{diffuse_BC}. 
%Due to many technical terms and assumptions, we present an informal statement here for simplicity. The precise statement can be found in Theorem \ref{main_theorem} in Section 4.

We are now ready to state the main result of this paper (informal version); the precise statement can be found in Theorem \ref{main_theorem} in Section 4.

\begin{theorem}[Informal statement]\label{Informal statement}
Let $\O$ be a upper half space in 3D:
 \Be\label{domain}
 \begin{split}
\O := \mathbb{T}^2 \times \R_+ \ni (x_1, x_2, x_3),
 \ \ & \text{where } \mathbb R_+ :=\{ x_3 \in \mathbb R: x_3 > 0 \},
 \\ & \text{ and } \mathbb{T} \text{ is the periodic interval } (-\pi, \pi).
 \end{split}
\Ee
Suppose an initial velocity field $u_{in}$ is divergence-free, the corresponding initial vorticity $\o_{in}= \nabla_x \times u_{in}$ and the initial heat $\theta_{in}$ live in some real analytic space and satisfy certain compatibility conditions.
%Then exists a unique real analytic solution $u$ and $\theta$ to the Navier-Stokes-Fourier flow (\ref{NS})-(\ref{Bouss}) in $[0,T] \times \O$ for some $T>0$. 
Choosing the proper correction terms $f_1$ and $f_2$, then there exists a large set of initial data $f_{R,in }$ such that for some $T>0$, there exists a unique solution $F(t,x,v)$ of the form \eqref{F_e} to the Boltzmann equation (\ref{Boltzmann}) with the diffuse reflection boundary condition (\ref{diffuse_BC}) under the scale of (\ref{StMae}), (\ref{Knke}) on $[0,T]$ such that for some choice of $\e$ and $\kappa(\e)$,

   \Be\notag
\sup_{0 \leq t \leq T}\left\|
\frac{F (t,x,v)- \mu (v)}{\e  \sqrt {\mu (v) } } - \left( - \theta_E(t,x) + u_E(t,x) \cdot v + \theta_E(t,x) \frac{|v|^2-3}{2} \right) \sqrt {\mu(v)}  \right\|_{L^2(\O \times \R^3 )}
\\    \longrightarrow 0
% \ \ \text{as} \ \ \e \downarrow 0.
\Ee
as $\e \to 0$, where $u_E$ and $\theta_E$ are the solutions of the incompressible Euler equations with no-penetration boundary condition under heat convection (\ref{Euler})-(\ref{heat_E}).

  \end{theorem}

For the rest of this section, we present the strategy and key ideas to the theorems for rigorously justifying the limit. The core of Section $2$ is to establish an $L^2$ estimate for the Boltzmann remainder $f_R$ uniformly in $\e$. From our Hilbert expansion around the global Maxwellian $\mu$ in \eqref{F_e}, the equation for $f_R$ is \eqref{eqtn_f_1}-\eqref{eqtn_f_5}. In order to have a desired bound on $f_R$ it is necessary to dispose of the highly singular terms in \eqref{eqtn_f_1}-\eqref{eqtn_f_3}. We show in Section \ref{Hilbertexpansion} that by the choice of the first and second order correction $f_1, f_2$ associated with a Navier-Stokes-Fourier flow as in \eqref{f_1}-\eqref{Pf2choice}, we obtain $\eqref{eqtn_f_1} = \eqref{eqtn_f_2} = 0$, and $\mathbf P \eqref{eqtn_f_3} = 0 $. The remaining terms in \eqref{eqtn_f_1}-\eqref{eqtn_f_5} are of low enough singularity that  can be %is allowed to be 
controlled eventually. 

The $L^2$ estimate for $f_R$ shares the same framework as in \cite{JK}. A trilinear forcing term comes up as $ \frac{1 }{\kappa \e^{1/2} } \int^t_0 \iint_{\O \times \R^3}  \Gamma (f_R,f_R) (\mathbf{I} - \mathbf{P}) f_R$ in the energy estimate of $f_R$. Utilizing the dissipation from the linearized Boltzmann operator \eqref{s_gap} we can bound this term as
\[
 \frac{ \e^{1/2} }{\kappa^{1/2}} \|     {P} f_R \|_{L^\infty_tL^6_{x }} \|    {P} f_R \|_{L^2_tL^3_{x }} \| \kappa^{-\frac{1}{2}} \e^{-1} \sqrt{\nu} (\mathbf{I} - \mathbf{P}) f_R \|_{L^2_{t,x,v}}.
\]
The $L^2_tL^3_x$-estimate for $Pf_R$ is achieved by extending $f_R$ into a particularly designed domain first and then into the whole space using special cutoff functions, followed by Duhamel's formula and employing the $TT^*$-method developed in \cite{EGKM3,GV,JV}. Detailed estimates can be found in section \ref{L2L3section}. An $L^6$ integrability gain for $\mathbf P f_R$ is obtained by employing the micro-macro decomposition method as in \cite{EGKM2}. We use the test function method as in \cite{EGKM} to invert the operator $v \cdot \nabla_x \mathbf P$ and bound the $L^6$ norm of $P f_R $ by the dissipation and the $L^2$ bound of $\p_t f_R$. Now the temporal derivative of $f_R$ comes into play and we need to include the estimates of $\p_t f_R$ in a parallel way to those of $f_R$. Also, the temporal derivatives of the fluid part will need to be considered in order to close the estimate. The $ L_{t,x}^\infty$ of $f_R$ and the $L_t^2L_x^\infty$ of $\p_t f_R$ will be controlled using Duhamel's formula and analyzing the particle's interaction with the diffuse boundary condition. Note that the space $L_t^2L_x^\infty$ for $\p_t f_R$ is needed as the forcing term $\nabla_x \p_t^2 u$ processes an initial-boundary layer.

We point out one major difference between our estimate for $f_R$ and %comparing to 
the corresponding one in \cite{JK}. Here in \eqref{F_e} we expand around the \textit{global Maxwellian} $\mu$ and have two correction terms $f_1, f_2$; while in \cite{JK} the expansion was performed around the local Maxwellian $ \tilde \mu = \frac{1}{ (2\pi)^{3/2} } e^{ - \frac{ |v -\e u(t,x) |^2  }{2}} $ and only have one correction term and $u(t,x)$ is the solution of the Navier-Stokes equation. In both cases, effective cancellations on the singular terms are %can be 
obtained by making the correct choice of the correction terms. One benefit in our global Maxwellian setting is that whenever the derivatives $\p_t, \nabla_x$ hit $\mu$, it vanishes. This reduces the amount of computation and makes the estimates more straightforward %cleaner 
compared to the local Maxwellian setting. However, such a setting also brings an extra singular term $ \frac{2}{\e \kappa} \Gamma(f_1,f_R)$ in the equation for $f_R$ in \eqref{eqtn_f_5}, which can be compared to the term $\frac{(\p_t + \e^{-1} v\cdot \nabla_x) \sqrt{\tilde \mu}}{\sqrt{\tilde \mu}} f_R$ in \cite{JK}.  This singular term will result in a bound  
\[
 \frac{1}{\kappa}   \int^t_0 \| P f_R (s) \|_{L^2_x}^2 \dd s 
\]
in the $L^2$ estimate of $f_R$, which would essentially give rise to the growth of $e^{\frac{1}{\kappa}}$ by the Gronwall's inequality. Fortunately, the expansion leaves enough room so that we can find a range of $\e$ in terms of $\kappa$ in a scale of large Reynolds number to absorb the growth. 

In Section \ref{fluidestimate} we present the inviscid limit of the Navier-Stokes-Fourier system \eqref{NS}-\eqref{Bouss}. It is well known that due to the mismatch of the boundary conditions \eqref{no-pen} and \eqref{NSB}, boundary layers will form. In \cite{SC1,SC2} the famous Prandtl expansion was introduced to justify the limit with analytic initial data. In a recent work \cite{NN2018}, the authors proved the inviscid limit using a new Green's function approach based on the boundary vorticity formulation established in \cite{Mae14}. In this paper, we employ the same Green's function approach using the vorticity formulation \eqref{NS_o}-\eqref{thetaB} for the Navier-Stokes-Fourier system. As mentioned above, temporal derivatives of the solution to the fluid equations also need to be controlled.  To this end, %To achieve the desired bound on the temporal derivatives 
we follow the same strategy in \cite{JK} by setting the compatibility as in \eqref{CC} and deriving a similar integral representation formula for $\p_t \o$ without any initial layer.

%The estimate for the heat flow $\theta$ is similar to that of the vertical component of the vorticity $\o_3$. 
It turns out the equation for the heat flow $\theta$ resembles the one for the vertical component of the vorticity $\o_3$.
  Using Green's formula \eqref{thetaphi} we can establish the analytic bounds \eqref{thetanorm_bound} for $\theta$ in the same way for $\o_3$. However, this does not fulfill our need since the conormal derivative in the analytic norm \eqref{normtheta} allows a $\frac{1}{x_3}$ singularity of the normal derivative of $\theta$ in the boundary layer, and thus prevents us from getting %to get 
the $L^\infty$ bound. In order to tackle the issue we design a similar, yet different ${\vertiii \cdot }_z  $ norm in \eqref{px3thetanorm} to directly target the normal derivative of $\theta$ and its temporal derivatives. Taking $\p_{x_3}$ derivative directly to \eqref{thetaphi} and using integration by parts we derive the integral formula for $\p_{x_3}  \theta$. From there we carefully bound the bilinear integrand and applying the analytic recovery lemma to obtain the desired bound in \eqref{thetapx3_bound}. Note that this estimate leads to the highest singularity as
\[
\| \nabla_x^2 \theta \|_{L^\infty_{t,x} } \sim \frac{1}{\kappa},
\]
which is of the same order as that of $\nabla_x^2 u $. But because of our choice of $\e$ with respect to $\kappa$, such growth %can be eventually absorbed 
 is affordable in the Hilbert expansion, which concludes the rigorous justification of the incompressible Euler limit with heat flow.

 \section{Boltzmann estimate}
 \subsection{Hilbert expansion} \label{Hilbertexpansion}
 
 In this section we prove the following proposition:
 
 \begin{proposition}\label{prop:Hilbert}
 Suppose that $F$ of (\ref{F_e}), solves (\ref{Boltzmann}) and (\ref{diffuse_BC}) with (\ref{StMae}), (\ref{Knke}), and that $(\rho,u,\theta)$ solves (\ref{NS})-(\ref{Bouss}).  We choose %a hydrodynamic part 
  $f_1$ and $f_2$ as 
 \Be \label{f_1}
 f_1 = (\rho + u\cdot v + \theta \frac{|v|^2 - 3}{2} ) \sqrt{\mu},
 \Ee
 \Be \label{IPf_2}
 (\mathbf{I-P}) f_2 = L^{-1} ( - \kappa v \cdot \nabla_x f_1 +  \Gamma (f_1, f_1) ),
 \Ee
  \Be\label{Pf_2}
  \mathbf P f_2 = (\rho_2 + u_2 \cdot v + \theta_2 \frac{|v|^2 - 3}{2} ) \sqrt{\mu},
\Ee
where $\rho_2, u_2$, and $\theta_2$ satisfy 
\Be \label{Pf2choice} \begin{split}
\nabla \cdot u_2 &= -\p_t \rho,  
\\ \nabla_x \left( \rho_2 +  \theta_2  -  \frac{1}{3} |u|^2 \right) & =  \nabla_x  p.
%\left( \frac{\kappa}{3} \left \langle L^{-1} ( - v \cdot \nabla_x f_1 ), |v|^2 \sqrt \mu \right \rangle + \frac{5}{4} \theta^2+  \frac{1}{2} |u|^2  \right)
\end{split} \Ee
Then $f_{R}$ in (\ref{F_e}) satisfies
\Be      \label{eqtn_fR}  \begin{split}
 \Big[    \p_t %f_R 
  + \frac{1}{\e} v\cdot \nabla_x %f_R
    &+    \frac{1}{ \e^2\kappa} L%f_R 
  \Big] f_R
 = - \frac{1}{\e^{1/2}} \left \{ (\mathbf{I-P}) (v \cdot \nabla_x f_2) - \frac{2}{\kappa} \Gamma(f_1,f_2) \right \} 
 \\ & + \e^{1/2} \left(  - \p_t f_2 + \frac{1}{\kappa}  \Gamma(f_2,f_2) \right) +  \frac{2}{\e \kappa} \Gamma(f_1 + \e f_2,f_R)  
+    \frac{   1 }{\e^{1/2}\kappa}\Gamma(f_R, f_R),
\end{split}\Ee
and $\p_t f_R$ satisfies
\Be \label{eqtn_fR_t} \begin{split}
&  \Big[    \p_t %f_R 
  + \frac{1}{\e} v\cdot \nabla_x %f_R
    +    \frac{1}{ \e^2\kappa} L%f_R 
  \Big] \p_t f_R 
  \\ =&   - \frac{1}{\e^{1/2}} \left \{ (\mathbf{I-P}) (v \cdot \nabla_x \p_t f_2) - \frac{2}{\kappa} \Gamma(\p_t f_1,f_2) - \frac{2}{\kappa} \Gamma( f_1, \p_t f_2) \right \} 
 \\ & + \e^{1/2}  \left( - \p_t^2 f_2  + \frac{2 }{\kappa} \Gamma(\p_t f_2,f_2)  \right)
 \\ & +  \frac{2}{\e \kappa} \left(  \Gamma(\p_t f_1 + \e \p_t f_2,f_R)  + \Gamma(f_1 + \e f_2,\p_t f_R)  \right)
\\ &+    \frac{   2 }{\e^{1/2}\kappa} \Gamma(\p_t f_R, f_R).
 \end{split}\Ee
 
 \begin{remark}
 There is some freedom to chose $u_2, \rho_2, \theta_2$ satisfying \eqref{Pf2choice}. One possible choice could be
\Be
u_2(x_h,x_3) =
\begin{bmatrix}
0 \\ 0 \\  \int_0^{x_3} \p_t \theta(x_h, y ) \, dy
\end{bmatrix}, \, \rho_2 = \frac{1}{3} |u|^2 +p, \, \theta_2 = 0.
\Ee

 \end{remark}
 
\end{proposition}

\begin{proof} To show \eqref{eqtn_fR}, it suffices to prove $\eqref{eqtn_f_1} = 0 $, $\eqref{eqtn_f_2} = 0 $ and $\mathbf P  \eqref{eqtn_f_3}  = 0 $ from the remainder equation \eqref{eqtn_f_1}-\eqref{eqtn_f_5}, and \eqref{eqtn_fR_t} follows from a direct differentiation of  \eqref{eqtn_fR}. 

From our choice \eqref{f_1}, we have $ \mathbf{ (I-P)} f_1 = 0$, therefore $\eqref{eqtn_f_1} = 0 $.
Next, let's show \eqref{IPf_2} is well-defined, which implies $\eqref{eqtn_f_2} = 0 $. By the Fredholm alternative, the inverse operator of $L$ maps
\[
L^{-1}: \mathcal N^\perp \to \mathcal N^\perp, \text{ where } \mathcal N^\perp \text{ stands an } L_v^2-\text{orthogonal complement of } \mathcal N.
\]
Hence in order to show \eqref{IPf_2} is well-defined, all we need is to show
\[
 \mathbf P ( - \kappa v \cdot \nabla_x f_1 +  \Gamma (f_1, f_1) ) =- \kappa  \mathbf P ( v \cdot \nabla_x f_1 ) =0,
 \]
or equivalently
\Be \label{Pvcdotf10}
\left \langle v \cdot \nabla_x f_1 , \left[ 1, v , \frac{|v|^2 }{2} \right] \sqrt \mu \right \rangle  =0.
\Ee
 By direction computation, from the oddness in the $v$-integral and the incompressible condition \eqref{incomp_NS} we have
 \Be \label{Pvcdotf11}
 \begin{split}
\langle   v \cdot \nabla_x f_1, \sqrt \mu \rangle = & \langle   v \cdot \nabla_x \rho ,  \mu \rangle + \langle   v \cdot \nabla_x u \cdot v ,  \mu \rangle + \langle   v \cdot \nabla_x \theta , \frac{|v|^2 - 3}{2}  \mu \rangle
\\ = & \sum_{i =1}^3  \p_i u^i \left( \int_{\mathbb R^3} v_i^2 \mu \,  dv \right)
  =   \nabla \cdot u = 0.
 \end{split}
 \Ee
Next, for fixed $i = 1,2,3$, again from the oddness in the $v$-integral and \eqref{Bouss} we have
 \Be \label{Pvcdotf12}
 \begin{split}
 \langle   v \cdot \nabla_x f_1, v_i \sqrt \mu \rangle = & \langle   v \cdot \nabla_x \rho ,  v_i \mu \rangle + \langle   v \cdot \nabla_x u \cdot v , v_i \mu \rangle + \langle   v \cdot \nabla_x \theta , \frac{|v|^2 - 3}{2} v_i  \mu \rangle
 \\ = &   \p_i \rho  \left( \int_{\mathbb R^3} v_i^2 \mu \, dv \right) +   \p_i \theta \left( \int_{\mathbb R^3}  \frac{|v|^2 - 3}{2} v_i^2 \mu \, dv \right)
  =   \p_i (\rho + \theta) = 0.
 \end{split}
 \Ee
Finally similar to \eqref{Pvcdotf11}, we have
\Be \label{Pvcdotf13}
 \begin{split}
\left \langle   v \cdot \nabla_x f_1, \frac{|v|^2 }{2} \sqrt \mu \right \rangle = &\left  \langle   v \cdot \nabla_x \rho ,  \frac{|v|^2 }{2} \mu  \right \rangle + \left \langle   v \cdot \nabla_x u \cdot v , \frac{|v|^2 }{2} \mu \right \rangle + \left \langle   v \cdot \nabla_x \theta , \frac{|v|^2 -3}{2}\frac{|v|^2}{2}  \mu \right \rangle
\\ = & \sum_{i=1}^3 \p_i u^i \left( \int_{\mathbb R^3} v_i^2 \frac{|v|^2}{2} \mu dv \right)
 =   \frac{5}{2} \nabla \cdot u = 0.
 \end{split}
\Ee
Therefore from \eqref{Pvcdotf11}, \eqref{Pvcdotf12}, and \eqref{Pvcdotf13} we prove \eqref{Pvcdotf10}.

%Therefore $(\mathbf{I-P}) f_2 = L^{-1} ( - \kappa v \cdot \nabla_x f_1 +  \Gamma (f_1, f_1) )$ is well-defined.

The only thing left to get \eqref{eqtn_fR} is then to prove $\mathbf P  \eqref{eqtn_f_3}  = 0 $. 
%We claim that $\mathbf P (\p_t f_1+ v \cdot \nabla_x f_2-  \frac{2}{\kappa} \Gamma(f_1,f_2) ) = 0 $. 
It suffices to prove
\Be
\left \langle \p_t f_1 + v \cdot \nabla_x f_2 , \left[ 1, v , \frac{|v|^2 }{2} \right] \sqrt \mu \right \rangle  =0.
\Ee
This is equivalent to
\Be \label{IPf2eq2} \begin{split}
\p_t \rho + \nabla \cdot u_2 &= 0,
\\ \p_t u + \langle v \cdot \nabla_x ( \mathbf{I-P}) f_2 , v \sqrt \mu \rangle +  \nabla_x (\rho_2 + \theta_2 ) & = 0,
\\  \frac{3}{2} \p_t (  \rho +  \theta) + \frac{1}{2} \langle v \cdot \nabla_x ( \mathbf{I-P} ) f_2 , |v|^2 \sqrt \mu \rangle + \frac{5}{2} \nabla \cdot u_2 & = 0.
\end{split} \Ee

Let's consider the second equation of \eqref{IPf2eq2}. Using \eqref{IPf_2} we compute
\Be \label{IPf2vsqrtmu} 
\begin{split}
\langle v \cdot \nabla_x ( \mathbf{I-P}) f_2 , v \sqrt \mu \rangle = & \langle v \cdot \nabla_x L^{-1} ( - \kappa v \cdot \nabla_x f_1 +  \Gamma (f_1, f_1) ) , v \sqrt \mu \rangle
\\ = & \underbrace{ \langle v \cdot \nabla_x L^{-1} ( - \kappa v \cdot \nabla_x f_1  ) , v \sqrt \mu \rangle}_{\eqref{IPf2vsqrtmu}_1} +  \underbrace{ \langle v \cdot \nabla_x L^{-1} (   \Gamma (f_1, f_1) ) , v \sqrt \mu \rangle}_{\eqref{IPf2vsqrtmu}_2}.
\end{split} 
\Ee
We first consider $\eqref{IPf2vsqrtmu}_1$. It suffices to compute for fixed $i$,
\Be \label{vivj}
\sum_j \left \langle v_i v_j \p_j L^{-1}  ( - \kappa v \cdot \nabla_x f_1 )  , \sqrt \mu \right \rangle.
\Ee
Since $(v_i v_j   - \frac{|v|^2}{3} \delta_{ij} )\sqrt{\mu} \in \mathcal N^\perp $. Define
 \Be\label{A0_ij}
 {A}_{ij}(v):= L^{-1}\Big((v_i v_j   - \frac{|v|^2}{3} \delta_{ij} )\sqrt{\mu}\Big)(v).
 \Ee
By splitting $v_i v_j = (v_i v_j - \frac{1}{3} |v|^2 \delta_{ij}  ) + \frac{1}{3} |v|^2 \delta_{ij} $ in \eqref{vivj}, and using the self-adjointness of $L^{-1}$ we have
\Be %\label{vivj2}
\notag
\begin{split}
\eqref{vivj} & =  - \kappa \sum_j \left \langle  L^{-1}  ( \sum_k v_k \p_{jk} f_1 ) ,   (v_i v_j - \frac{1}{3} |v|^2 \delta_{ij}  )\sqrt \mu \right \rangle + \p_i \left( \frac{\kappa}{3} \left \langle L^{-1} ( - v \cdot \nabla_x f_1 ), |v|^2 \sqrt \mu \right \rangle \right)
\\ & =  - \kappa \sum_j \left \langle  L^{-1}  ( \sum_k v_k \p_{jk} f_1 ) ,   (v_i v_j - \frac{1}{3} |v|^2 \delta_{ij}  )\sqrt \mu \right \rangle
\\ & = - \kappa  \sum_{j, k,l} \p_{jk} u^\ell  \left \langle  v_k v_l \sqrt \mu  , A_{ij} \right \rangle
  = - \kappa  \sum_{j, k,l} \p_{jk} u^\ell  \left \langle  L A_{kl}  , A_{ij} \right \rangle,
\end{split}
\Ee
where the $\theta$ and $\rho$ terms vanish because of the oddness in the $v$-integration. Next by the computation (Lemma 4.4 in \cite{BGL2}), the above term is equal to 
\Be \label{etavcomp} \begin{split}
- \kappa \sum_{j,k,l } \langle A_{ij}, L A_{kl } \rangle \p_{jk} u^l & =- \kappa \eta_v \sum_{j,k,l} \left( (\delta_{ik } \delta_{jl} + \delta_{il}\delta_{jk } ) - \frac{2}{3} \delta_{ij} \delta_{kl} \right)  \p_{jk} u^l
\\ & = - \kappa\eta_v \{\Delta u^{i}  -   \p_i    \nabla \cdot  u  -  \frac{2}{3}   \p_i  \nabla \cdot  u   \}= - \kappa\eta_v \Delta u^{i}
 \ \ \text{for} \ i =1,2,3.
\end{split} \Ee
Here we have used the incompressible condition \eqref{NS}. Therefore we get
\Be \label{IPf2vsqrtmu1}
\eqref{IPf2vsqrtmu}_1 = - \kappa\eta_v \Delta u.
\Ee

Next we consider $\eqref{IPf2vsqrtmu}_2$. First by direct computation we have
\[
\begin{split}
 &  \Gamma (f_1, f_1) - L\left(\frac{f_1^2}{ 2 \sqrt \mu } \right) 
  \\ = & \frac{1}{\sqrt \mu } \left( Q( \sqrt \mu f_1, \sqrt \mu f_1 ) +  Q(\mu, f_1^2 ) \right)
  \\  = & \frac{1}{2 \sqrt \mu } \iint |(v - v_*) \cdot \mathfrak{u} | \left( 2 \sqrt \mu (v') f_1(v') \sqrt \mu (v_*' ) f_1(v_*')- 2 \sqrt \mu (v) f_1(v) \sqrt \mu (v_*) f_1(v_*)  \right) d\mathfrak{u} d v_*
  \\  & +\frac{1}{2 \sqrt \mu } \iint |(v - v_*) \cdot \mathfrak{u} | \left(  \mu(v' ) f_1^2(v_*') + \mu(v_*' ) f_1^2(v')  - \mu(v ) f_1^2(v_*) - \mu(v_* ) f_1^2(v)  \right) d\mathfrak{u} d v_* 
  \\ = & \frac{1}{2 \sqrt \mu } \iint   |(v - v_*) \cdot \mathfrak{u} | \left(  \sqrt  \mu(v' ) f_1(v_*') + \sqrt \mu(v_*' ) f_1(v')  - \sqrt \mu(v ) f_1(v_*) - \sqrt \mu(v_* ) f_1(v)  \right)
  \\ & \quad \quad \quad \quad  \times  \left(  \sqrt  \mu(v' ) f_1(v_*') + \sqrt \mu(v_*' ) f_1(v')  + \sqrt \mu(v ) f_1(v_*) + \sqrt \mu(v_* ) f_1(v)  \right)  d\mathfrak{u} d v_* 
  \\ = & \frac{1}{2 } \iint   \sqrt \mu(v_*)   |(v - v_*) \cdot \mathfrak{u} | \left( \frac{ f_1(v_*')}{\sqrt \mu( v_*')} + \frac{f_1(v')}{\sqrt \mu(v')}  - \frac{ f_1(v_*)}{\sqrt \mu(v_*)} - \frac{ f_1(v)}{\sqrt \mu(v)}  \right)
  \\ & \quad \quad \quad \quad  \times  \left(  \sqrt  \mu(v' ) f_1(v_*') + \sqrt \mu(v_*' ) f_1(v')  + \sqrt \mu(v ) f_1(v_*) + \sqrt \mu(v_* ) f_1(v)  \right)  d\mathfrak{u} d v_*   .
\end{split}
\]
And since $f_1 \in \mathcal N^{}$, $\frac{ f_1(v_*')}{\sqrt \mu( v_*')} + \frac{f_1(v')}{\sqrt \mu(v')}  - \frac{ f_1(v_*)}{\sqrt \mu(v_*)} - \frac{ f_1(v)}{\sqrt \mu(v)} = 0 $. Thus
\Be \label{LinvL}
 \Gamma (f_1, f_1) = L\left(\frac{f_1^2}{ 2 \sqrt \mu } \right), \text{ therefore } L^{-1} \left(  \Gamma (f_1, f_1) \right) = (\mathbf{I-P})(\frac{f_1^2}{ 2 \sqrt \mu }).
\Ee
From \eqref{f_1},
\Be \label{IPf_1^2}
(\mathbf{I-P})(\frac{f_1^2}{ 2 \sqrt \mu }) = \frac{1}{2} (\mathbf{I-P}) (\rho + u\cdot v + \theta \frac{|v|^2 - 3}{2} )^2 \sqrt{\mu}  =  \frac{1}{2} (\mathbf{I-P}) ( u\cdot v + \theta \frac{|v|^2 - 3}{2} )^2 \sqrt{\mu}.
\Ee

Now for fixed $i$, from \eqref{LinvL}, \eqref{IPf_1^2}, and using $|v|^2 \sqrt \mu \in \mathcal N$, 
\Be
\begin{split}
& \langle   v \cdot \nabla_x L^{-1} (   \Gamma (f_1, f_1) ) , v_i \sqrt \mu \rangle 
%\\
% = & \frac{1}{2} \left \langle v \cdot \nabla_x \left( (\mathbf{I-P}) \left( ( u\cdot v + \theta \frac{|v|^2 - 3}{2} )^2 \sqrt{\mu} \right) \right)   , v_i \sqrt \mu \right \rangle
\\ 
= & \frac{1}{2}  \sum_j \left  \langle  \p_j  \left( (\mathbf{I-P}) \left( ( u\cdot v + \theta \frac{|v|^2 - 3}{2} )^2 \right)  \sqrt \mu \right), v_j v_i \sqrt \mu \right \rangle
\\ = & \frac{1}{2}  \sum_j \left  \langle  \p_j  \left( (\mathbf{I-P}) \left( ( u\cdot v + \theta \frac{|v|^2 - 3}{2} )^2 \right)  \sqrt \mu \right), ( v_j v_i  - \frac{1}{3} \delta_{ij} |v|^2 ) \sqrt \mu \right \rangle.
\end{split}
\Ee
Now using $( v_j v_i - \frac{1}{3} \delta_{ij} |v|^2 ) \sqrt \mu  \in \mathcal N^\perp$, we have
\Be \label{Linvga} \begin{split}
 & \langle   v \cdot \nabla_x L^{-1} (   \Gamma (f_1, f_1) ) , v_i \sqrt \mu \rangle 
\\  = & \frac{1}{2}  \sum_j \left  \langle  \p_j  \left(   ( u\cdot v + \theta \frac{|v|^2 - 3}{2} )^2  \sqrt \mu \right), ( v_j v_i  - \frac{1}{3} \delta_{ij} |v|^2 ) \sqrt \mu \right \rangle
\\  = & \underbrace{ \frac{1}{2}  \sum_j \left  \langle  \p_j  \left(   ( u\cdot v)^2\sqrt \mu \right), ( v_j v_i  - \frac{1}{3} \delta_{ij} |v|^2 ) \sqrt \mu \right \rangle}_{\eqref{Linvga}_1}
 + \underbrace{ \frac{1}{2} \sum_j \left \langle \p_j \left( \theta^2 ( \frac{|v|^2 - 3}{2} )^2  \sqrt \mu \right), ( v_j v_i  - \frac{1}{3} \delta_{ij} |v|^2 ) \sqrt \mu \right \rangle }_{\eqref{Linvga}_2}
 \\ & + \underbrace{  \sum_j  \left \langle (u \cdot v ) \theta ( |v|^2 -3 ) \sqrt \mu, ( v_j v_i  - \frac{1}{3} \delta_{ij} |v|^2 ) \sqrt \mu \right \rangle }_{\eqref{Linvga}_3}.
\end{split} \Ee
First we have $ \eqref{Linvga}_3 = 0$ from oddness of the integration. By direct computation
\[
\begin{split}
\eqref{Linvga}_2  & =  \frac{1}{2} \sum_j \left \langle  \left(  \frac{|v|^2 - 3}{2}  \right)^2 \p_j (\theta ^2) \sqrt \mu,  \left( v_j v_i - \frac{1}{3} \delta_{ij} |v|^2\right )  \sqrt \mu \right \rangle 
 \\ &  =\frac{1}{2} \p_i (\theta^2) \int \left(  \frac{|v|^2 - 3}{2}  \right)^2 \left(  v_i^2 - \frac{1}{3} |v|^2\right ) \mu \, dv = 0 .
 \end{split}
\]
%\[
% \sum_j \left \langle v_j \left(  \frac{|v|^2 - 5}{2}  \right) \p_j ( (\rho + \theta)\theta ) \sqrt \mu, v_i \sqrt \mu \right \rangle =  \frac{1}{3} \p_i (\theta (\rho + \theta ) ) \int |v|^2 \left(  \frac{|v|^2 - 5}{2}  \right) \mu dv = 0 .
%\]
Lastly for $\eqref{Linvga}_1$, we have from computation 
%(can put more details here)
\[
\begin{split}
%& \frac{1}{2}  \sum_j \left  \langle  \p_j  \left(   ( u\cdot v)^2 \sqrt \mu \right), ( v_j v_i  - \frac{1}{3} \delta_{ij} |v|^2 ) \sqrt \mu \right \rangle
 \eqref{Linvga}_1% & =  \frac{1}{2} \sum_j \p_j \langle v_i v_j (v \cdot u)^2 \sqrt \mu, \sqrt \mu \rangle  - \frac{1}{6}  \p_i \left \langle \left(   ( u\cdot v)^2 \sqrt \mu \right),  |v|^2 \sqrt \mu \right \rangle
 %\\ 
 & = \frac{1}{2} \int \left( \sum_{j,k,\ell} v_i v_j v_k v_\ell  \p_j (u^k u^\ell) \right) \mu dv -  \frac{1}{6} \int \left( \sum_{k,\ell } \p_i (u^k u^\ell) v_k v_\ell |v|^2 \right)  \mu dv
 \\ & = \frac{1}{2} \int \left( \sum_{j \neq i } 2 v_i^2 v_j^2   \p_j (u^i u^j) \right) \mu dv + \frac{1}{2} \int \left( \sum_{k}  v_i^2 v_k^2   \p_i (u^k)^2 \right) \mu dv - \frac{1}{6} \int \left( \sum_k v_k^2 |v|^2 \p_i (u^k)^2 \right) \mu dv
\\ & = \sum_{j \neq i }  \p_j (u^i u^j) + \frac{3}{2} \p_i (u^i)^2 + \frac{1}{2} \sum_{k \neq i } \p_i (u^k)^2 - \frac{5}{6} \p_i ( |u|^2 ) 
\\ & =  \sum_{j  }  \p_j (u^i u^j) +   \frac{1}{2} \p_i (|u|^2) - \frac{5}{6} \p_i ( |u|^2 ) =   \sum_{j  }  [ (\p_j u^i )u^j ] -   \frac{1}{3} \p_i (|u|^2),
\end{split}
\]
where we've used $\nabla \cdot u = 0$. Therefore we get 
\[
 \eqref{IPf2vsqrtmu}_2 = u \cdot \nabla_x u -  \frac{1}{3} \nabla_x ( |u|^2 ).
 \]
Combing with \eqref{IPf2vsqrtmu1} we get the second equation of \eqref{IPf2eq2} is equivalent to
\[
\p_t u + u \cdot \nabla_x u - \kappa \nu_v \Delta u + \nabla_x \left( \rho_2 +  \theta_2  - \frac{1}{3} |u|^2 \right) = 0,
\]
which is guaranteed by \eqref{Pf2choice}.

We now consider the third equation of \eqref{IPf2eq2} subtracting $\frac{5}{2}$ of the first equation of  \eqref{IPf2eq2}:
\Be \label{thetaf2}
 \frac{5}{2} \p_t  \theta + \frac{1}{2} \langle v \cdot \nabla_x ( \mathbf{I-P} ) f_2 , |v|^2 \sqrt \mu \rangle  = 0.
\Ee
Using \eqref{IPf_2} we compute
\Be \label{IPf2v2sqrtmu}
 \begin{split}
  &  \frac{1}{2} \langle v \cdot \nabla_x ( \mathbf{I-P} ) f_2 , |v|^2 \sqrt \mu \rangle 
    = \frac{1}{2} \langle v \cdot \nabla_x L^{-1} ( - \kappa v \cdot \nabla_x f_1 +  \Gamma (f_1, f_1) ) , |v|^2 \sqrt \mu \rangle
\\ & = \underbrace{ \frac{1}{2} \langle v \cdot \nabla_x L^{-1} ( - \kappa v \cdot \nabla_x f_1  ) , |v|^2 \sqrt \mu \rangle }_{\eqref{IPf2v2sqrtmu}_1} + \underbrace{ \frac{1}{2} \langle v \cdot \nabla_x L^{-1} (   \Gamma (f_1, f_1) ) , |v|^2 \sqrt \mu \rangle}_{\eqref{IPf2v2sqrtmu}_2}.
\end{split}
\Ee
Since $( \frac{|v|^2 -5}{2} ) v_i \sqrt \mu \in \mathcal N^\perp$. Define
\[
B_i = L^{-1} \left( ( \frac{|v|^2 -5 }{2} ) v_i \sqrt \mu \right).
\]
By the self-adjointness of $L^{-1}$, we compute
\Be \label{deltathetacal}
\begin{split}
\eqref{IPf2v2sqrtmu}_1 = & - \frac{\kappa}{2} \sum_{j,k} \left \langle \p_j L^{-1} ( v_k \p_k f_1), \frac{v_j |v|^2}{2} \sqrt \mu \right \rangle
 \\ %= &- \frac{\kappa}{2} \sum_{j,k} \left \langle \p_j L^{-1} ( v_k \p_k f_1), v_j  \left( \frac{ |v|^2}{2} - \frac{5}{2} \right) \sqrt \mu \right \rangle
% \\ 
 = & - \frac{\kappa}{2} \sum_{j,k} \left \langle v_k \p_{jk} \left( \rho + v \cdot u + \frac{|v|^2 -3 }{ 2 } \theta \right), B_j \right \rangle
 \\ = & \underbrace{  - \frac{\kappa}{2} \sum_{j,k} \langle v_k \p_{jk}(v \cdot u ) \sqrt \mu ,B_j \rangle}_{\eqref{deltathetacal}_1} - \underbrace{  \frac{\kappa}{2} \sum_{j,k} \left \langle v_k \left(  \frac{|v|^2 -5 }{ 2 } \p_{jk} \theta \right) \sqrt \mu, B_j \right \rangle }_{\eqref{deltathetacal}_2}
 \\ &  - \underbrace{ \frac{\kappa}{2} \sum_{j,k} \langle v_k \p_{jk} (\rho + \theta) \sqrt \mu ,B_j \rangle }_{\eqref{deltathetacal}_3}.
\end{split} 
\Ee
From \cite{Guo2006}, $B_j$ is odd in $v_j$, thus $\eqref{deltathetacal}_1 = 0$. Also $\eqref{deltathetacal}_3 = 0$ as $\rho + \theta = constant$. Now by lemma $4.4$ in \cite{BGL2}, we have for some constant $\eta_c$, 
\Be \label{etaccomp}
\langle L B_k, B_j \rangle = 5 \eta_c \delta_{kj}.
\Ee
Thus
\[
\eqref{deltathetacal}_2 =  \frac{\kappa}{2} \sum_{j,k} \left \langle v_k \left(  \frac{|v|^2 -5 }{2 } \p_{jk} \theta \right) \sqrt \mu, B_j \right \rangle =   \frac{\kappa}{2} \sum_{j,k}   \langle L B_k, B_j \rangle \p_{jk } \theta = \frac{5}{2} \kappa    \eta_c \Delta \theta.
\]
Finally we compute $ \frac{1}{2} \langle v \cdot \nabla_x L^{-1} (   \Gamma (f_1, f_1) ) , |v|^2 \sqrt \mu \rangle$. From \eqref{LinvL}, \eqref{IPf_1^2} and using $ v_j  (|v|^2-5) \sqrt \mu  \in \mathcal N^\perp$,
\[ \begin{split}
 & \frac{1}{2}  \langle v \cdot \nabla_x L^{-1} (   \Gamma (f_1, f_1) ) , |v|^2 \sqrt \mu \rangle 
% = &  \frac{1}{4}  \left \langle v \cdot \nabla_x  (\mathbf{I-P} )\left( ( u\cdot v + \theta \frac{|v|^2 - 3}{2} )^2 \sqrt{\mu} ) \right)  , |v|^2 \sqrt \mu \right \rangle
%\\
 % = &  \frac{1}{4} \sum_j  \left \langle   \p_j  (\mathbf{I-P} )\left( ( u\cdot v + \theta \frac{|v|^2 - 3}{2} )^2 \sqrt{\mu} ) \right)  , v_j |v|^2 \sqrt \mu \right \rangle
%\\
  =    \frac{1}{4} \sum_j  \left \langle   \p_j  \left( ( u\cdot v + \theta \frac{|v|^2 - 3}{2} )^2 \sqrt{\mu} ) \right)  , v_j (|v|^2-5) \sqrt \mu \right \rangle
\\ = & \frac{1}{4} \sum_j  \left \langle   \p_j \left( ( u\cdot v ) \theta (|v|^2 - 3 ) \sqrt{\mu} \right)   , v_j (|v|^2-5) \sqrt \mu \right \rangle
  =   \frac{1}{4} \sum_j  \p_j  ( u^j  \theta )  \int  \left( (|v|^2 - 3 )  v_j^2 (|v|^2-5)  \mu \right) dv \\
   = &    \frac{5}{2} u \cdot \nabla \theta.
\end{split}\]
Therefore \eqref{thetaf2} gives
\[
\frac{5}{2} \left(  \p_t  \theta +   u \cdot \nabla \theta -\kappa \eta_c \Delta \theta \right) = 0,
\]
which is guaranteed by \eqref{heat_NS1}.
\end{proof}

 \subsection{$L^2$-Energy estimate}
 
 We start with an $L^2$-energy estimate for $f_R$ and $\p_t f_R$. From \eqref{eqtn_fR}, we define
\Be      \label{initial_ft}  \begin{split}
   \p_t  f_{R,0}
 = & - \frac{1}{\e} v\cdot \nabla_x f_{R,in}  -     \frac{1}{ \e^2\kappa } L f_{R,in} -  \frac{1}{\e^{1/2}} \left \{ (\mathbf{I-P}) (v \cdot \nabla_x f_{2,in}) - \frac{2}{\kappa} \Gamma(f_{1,in},f_{2,in}) \right \} 
 \\ & - \e^{1/2} \p_t f_{2,in} +  \frac{2}{\e \kappa} \Gamma(f_{1,in} + \e f_{2,in},f_{R,in})  + \frac{\e^{1/2} }{\kappa} \Gamma(f_{2,in},f_{2,in}) +    \frac{ 1 }{\e^{1/2}\kappa}\Gamma(f_{R,in}, f_{R,in}).
\end{split}\Ee
Here $ \p_t f_{2,in}$ is defined through solving the fluid equations for $\p_t$ part and evaluate at $t=0$. 

For the sake of notational  simplicity, let's denote $\mathcal I, \mathcal I_2 : [0,T ] \times \O \to \mathbb R^5$, 
\Be \label{HH2}
\mathcal I = \left[ \rho, u^1, u^2, u^3, \theta \right], \text{ and } \mathcal I_2 = \left[ \rho_2, u_2^1, u_2^2, u_2^3, \theta_2 \right].
\Ee
The following trace theorem is useful to control the boundary terms. 
\begin{lemma}[Trace theorem]
 
\Be\label{trace}
\frac{1}{\e}\int^t_0 \int_{\gamma_+^N} |h| \dd \gamma \dd s \lesssim_N  \iint_{\O \times \R^3} |h(0)|  +   \int^t_0 \iint_{\O \times \R^3} |h|  + \int^t_0 \iint_{\O \times \R^3 } | \p_t h +\frac{1}{\e} v\cdot \nabla_x h|,
\Ee
where $\gamma_+^N:= \{(x,v) \in \gamma_+: |n(x) \cdot v|> 1/N \ \text{and} \ 1/N<|v|<N \}$.\end{lemma}
\begin{proof}
The proof is standard (for example see Lemma 3.2 in \cite{EGKM2} or Lemma 7 in \cite{CKL}).
\end{proof}

\begin{proposition}\label{prop:energy} Under the same assumptions in Proposition \ref{prop:Hilbert}, we have 
\Be
\begin{split}\label{est:Energy}
&   
 \| f_R(t) \|_2^2 + d_2 \| \kappa^{-\frac{1}{2}} \e^{-1} \sqrt{\nu} (\mathbf{I} - \mathbf{P})f_R \|_{L^2_{t,x,v}}^2
  + \int^t_0 | \e^{-\frac{1}{2}}
%(1- P_{\gamma_+}) 
f_R
|_{  L^2_\gamma}^2
  \\  \lesssim  & \ \| f_R (0)\|_{L^2_{x,v}}^2  + \frac{1}{\kappa}    (1 +   \|  \mathcal I \|_{L^\infty_{t,x}}^2 )\int^t_0 \| P f_R (s) \|_{L^2_x}^2 \dd s 
\\ &    +   \frac{ \e}{\kappa^3 } \| \kappa^{1/2}   {P} f_R(s) \|_{L^\infty_tL^6_{x }}^2  \| \kappa^{1/2}P  f_R \|^2_{L^2_tL^3_{x}} 
\\
&
 + \e \kappa \| (\ref{estR21})\|_{L^2_{t,x,v}}^2 + \frac{\e}{\kappa} \|\kappa (\ref{estR22})\|_{L^2_{t,x,v}}^2 +  \| \eqref{IPf2est} \|_{L^2_t L^2 (\p\O)},
  \end{split}
\Ee
where 
\Be \begin{split}\label{d2}
{d}_{2}
:= & \frac{\sigma_0}{2}-  \e^{3/2} \| \mathfrak{w} f_R \|_{L^\infty_{t,x,v}}
- \left(\e^2 \| (\ref{est:f2})  \|_{L^\infty_{t,x}} + \e  \|  \nabla_x \mathcal I \|_{L^\infty_{t,x}}
+ \frac{\e^2}{\kappa}  \| (\ref{est:f2})  \|_{L^\infty_{t,x}}^2 \right)
\\ &  - \kappa^{ 1/2} \e^{1/2}   \| (\ref{estR11})  \|_{L^2_{t,x}}  - \frac{\e^{1/2}}{\kappa^{1/2}} \| \kappa \eqref{estR12} \|_{L^2_{t,x}}   -   \kappa^{ 1/2} \e^{3/2} \| (\ref{estR21})\|_{L^2_{t,x,v}}  - \frac{\e^{3/2}}{\kappa^{1/2}}  \| \kappa (\ref{estR22})\|_{L^2_{t,x,v}} .
\end{split} \Ee

 We also have 
 %\bcb
 \Be\begin{split}\label{est:Energy_t}
&\| \p_t f_R (t) \|_{L^2_{x,v}}^2 + {d}_{2,t}
 \| \kappa^{-\frac{1}{2}} \e^{-1} \sqrt{\nu}(\mathbf{I} - \mathbf{P}) \p_t f_R\|_{L^2_{t,x,v}}^2
 +  |\e^{- \frac{1}{2}} \p_t f_R|_{L^2_tL^2_\gamma}^2
 \\ \lesssim& \ \| \p_t f_R (0) \|_{L^2_{x,v}}^2  + \e \kappa \| (\ref{estR21t})\|_{L^2_{t,x,v}}^2 + \frac{\e}{\kappa} \|\kappa (\ref{estR22t})\|_{L^2_{t,x,v}}^2
\\ & +   \frac{ \e}{\kappa^3 } \| \kappa^{1/2}   {P}  f_R(s) \|_{L^\infty_tL^6_{x }}^2   \| \kappa^{1/2}P \p_t f_R \|^2_{L^2_tL^3_{x}} +  \frac{1}{\kappa}  \int^t_0  \|     P \p_tf_R(s) \|_{L^{2}_x }^2 \dd s 
\\ & +  \Big\{ 1+ \frac{1}{\kappa} \left( \|  \mathcal I    \|_{L^\infty_{t,x}}^2 + \| \p_t \mathcal I    \|_{L^\infty_{t,x}}^2 \right)    \Big\} \times \int^t_0  \|     P  f_R(s) \|_{L^{2}_x }^2 \dd s  
\\  &+ \Big\{ 1 + \e^2 \| \sqrt{\nu} \p_t f_1 \|_{L^\infty_{t,x,v}}^2 \Big\} \times  \| \e^{-1}\kappa^{-1/2} \sqrt{\nu} (\mathbf{I} -\mathbf{P}) f_R \|_{L^2_{t,x,v}}^2,
 \end{split}\Ee
where $\p_t f_R (0,x,v) := f_{R,t}(0,x,v)$ is defined in \eqref{initial_ft}. %(\ref{f:initial_t}). 
Here
\Be\label{d2t}
\begin{split}
{d}_{2,t}&:= \frac{\sigma_0}{2}-  \frac{\e^2}{\kappa} \|  \eqref{estptf2} \|_{L^\infty_{t,x,v}}^2 -  \frac{ \e^2}{\kappa} \|  \mathcal I_2  \|_{L_{t,x}^\infty} - \frac{\e^2}{\kappa} \| \mathbf{(I-P) } f_2 \|_{L_{t,x}^\infty}  
\\ & \ \ \  - \e  \|   \mathcal I \|_{L_{t,x}^\infty} -  \e^{3/2} \|   \mathfrak{w}f_R\|_{L^\infty_{x,v}}
\\ & \ \ \ - \kappa^{ 1/2} \e^{1/2}   \| (\ref{estR11t})  \|_{L^2_{t,x}}  - \frac{\e^{1/2}}{\kappa^{1/2}} \| \kappa \eqref{estR12t} \|_{L^2_{t,x}} -  \kappa^{ 1/2} \e^{3/2} \| (\ref{estR21t})\|_{L^2_{t,x,v}}  - \frac{\e^{3/2}}{\kappa^{1/2}}  \| \kappa (\ref{estR22t})\|_{L^2_{t,x,v}}
\end{split}
\Ee

\end{proposition}

\begin{proof}

We fix a $0< \varrho < \frac{1}{4}$. Define
\Be \label{weight}
\mathfrak{w}_{\varrho}(x,v)=\mathfrak{w} := e^{\varrho |v|^2}.
\Ee
%We often abuse the notation of $\mathfrak{w}_{\varrho}$ and $\mathfrak{w}$. 
We will use both $\mathfrak{w}_{\varrho}$ and $\mathfrak{w}$ equivalently. %and likewise $\mathfrak{w}_{\varrho'}=\mathfrak{w'}$. 

From \eqref{Pf_2} we have $|\mathbf P f_2 | \lesssim e^{ - \varrho |v|^2} | \mathcal I_2 | $. And from \eqref{IPf_2}, by \cite{BGL1} we have
\Be \label{IPf2est} \begin{split}
| \mathbf{(I-P)} f_2  | & = | L^{-1} ( - \kappa v \cdot \nabla_x f_1 +  \Gamma (f_1, f_1) ) | 
\\ & \lesssim \kappa e^{ - \varrho |v|^2} | \nabla_x \mathcal I |  + e^{- \varrho |v|^2} | \mathcal I |^2.
\end{split} 
\Ee
Thus
\Be \label{est:f2} 
| f_2 | \lesssim  e^{ - \varrho |v|^2} \left(| \mathcal I |^2 + | \mathcal I_2 | + \kappa  | \nabla_x \mathcal I | \right).
\Ee

An energy estimate to (\ref{eqtn_fR}) and (\ref{bdry_fR}) reads as 
\begin{align}
&\frac{1}{2}\|f _R(t)\|_{L^2_{x,v}}^2 - \frac{1}{2} \|f _R (0)\|_{L^2_{x,v}}^2
+ \frac{1}{ \kappa \e^{2 }} \int^t_0 \iint _{\O \times \R^3}f _R L  f_R\label{Energy_LHS} 
\\ & +  \frac{1}{2\e} \int^t_0 \int_{\gamma_+} |f_R|^2 
-  \frac{1}{2\e} \int^t_0 \int_{\gamma_-} | 
P_{\gamma_+}   f_R-  \e^{1/2} (1- P_{\gamma_+})
(\mathbf{I} -\mathbf{P})
 f_2
|^2 \label{Energy_bdry}
 \\
=  
&   \frac{1 }{\kappa \e^{1/2} } \int^t_0 \iint_{\O \times \R^3} 
\Gamma (f_R,f_R) (\mathbf{I} - \mathbf{P}) f_R 
\label{Energy_Gamma}
\\
&+ \frac{2}{  \kappa \e} \int^t_0 \iint_{\O \times \R^3} \Gamma( f_1 + \e f_2, f_R) (\mathbf{I} - \mathbf{P}) f_R \label{Energy_Gf_2}
\\
&+ \frac{1}{\e^{1/2}} \int^t_0 \iint_{\O \times \R^3} -  \left \{ (\mathbf{I-P}) (v \cdot \nabla_x f_2) - \frac{2}{\kappa} \Gamma(f_1,f_2) \right \}  (\mathbf{I} - \mathbf{P})  f_R \label{Energy_I-PR_1}\\
&+ \e^{1/2}  \int^t_0 \iint_{\O \times \R^3}  \left( -  \p_t f_2  + \frac{1}{\kappa} \Gamma(f_2,f_2) \right)  f_R \label{Energy_R2}
%\\
%&+ \int^t_0 \iint_{\O \times \R^3}
%\frac{- (\p_t + \e^{-1} v\cdot \nabla_x ) \sqrt{\mu}}{\sqrt{\mu}} |f_R|^2.\label{Energy_v^3}
\end{align}

We start with (\ref{Energy_LHS}). 
%\textit{Estimate of (\ref{Energy_LHS}):} 
From the spectral gap estimate in (\ref{s_gap}), we have 
\Be\label{est:EL}
(\ref{Energy_LHS})\geq \frac{1}{2}\|f _R(t)\|_{L^2_{x,v}}^2 - \frac{1}{2} \|f _R (0)\|_{L^2_{x,v}}^2
+  \sigma_0
\| \kappa^{-\frac{1}{2}} \e^{-1} \sqrt{\nu}(\mathbf{I} - \mathbf{P}) f_R\|_{L^2_{t,x,v}}^2. 
\Ee

Now we consider (\ref{Energy_Gamma}), in which we need integrability gain of $\mathbf{P}f_R$ in $L^6_x$ %and $L^2_t L^3_x$ 
of the next sections. From decomposition $f_R= \mathbf{P} f_R + (\mathbf{I} - \mathbf{P}) f_R$ and $\Gamma=\Gamma_+- \Gamma_-$ in (\ref{Gamma}),  we derive 
\Be
\begin{split}
\label{est:EG}
|(\ref{Energy_Gamma})| \lesssim &  \ 
\frac{1}{\kappa \e^{1/2}}  \sum_{i=\pm }\int^t_0 \iint_{\O \times \R^3} | \nu^{-\frac{1}{2}} \Gamma_{i}( |f_R|, (\mathbf{I} -\mathbf{P}) f_R)| | \sqrt{\nu}(\mathbf{I} - \mathbf{P}) f_R|\\
&
+\frac{\delta}{\kappa \e^{1/2}}  \sum_{i=\pm } \int^t_0 \iint_{\O \times \R^3}  | \nu^{-\frac{1}{2}} \Gamma_{i}(|\mathbf{P}f_R|, |\mathbf{P} f_R|)| | \sqrt{\nu}(\mathbf{I} - \mathbf{P}) f_R|
\\
 \lesssim &   \  \e^{3/2} \|   \mathfrak{w}f_R\|_{L^\infty_{x,v}}  \| \kappa^{-\frac{1}{2}} \e^{-1} \sqrt{\nu}
(\mathbf{I} - \mathbf{P}) f_R \|^2_{L^2_{t,x,v}} 
\\
&+  \frac{ \e^{1/2}% t^{1/2}
}{\kappa^{3/2}}
 \|\kappa^{1/2}     {P} f_R \|_{L^\infty_tL^6_{x }}
 \| \kappa^{1/2}   {P} f_R \|_{L^2_tL^3_{x }}
\| \kappa^{-\frac{1}{2}} \e^{-1} \sqrt{\nu}
(\mathbf{I} - \mathbf{P}) f_R \|_{L^2_{t,x,v}}.%\\
\end{split}\Ee

Next we consider (\ref{Energy_Gf_2}), we have
\Be
\begin{split}
&|(\ref{Energy_Gf_2})|\\ &\leq  
\frac{1}{\kappa^{1/2}}
 \left( \e \|  (\ref{est:f2}) \|_{L^\infty_{t,x}} + \|  \mathcal I  \|_{L^\infty_{t,x}}  \right)
  \{
   \|  P f_R \|_{L^2_{t,x}} 
+  \kappa^{ \frac{1}{2}} \e  \| \kappa^{-\frac{1}{2}} \e^{-1}  \sqrt{\nu}(\mathbf{I} - \mathbf{P}) f_R \|_{L^2_{t,x,v}} 
  \}\\
  & \ \ \ \ \ \ \ \times
 \| \kappa^{-\frac{1}{2}} \e^{-1}  \sqrt{\nu}(\mathbf{I} - \mathbf{P}) f_R \|_{L^2_{t,x,v}} \\
 &\lesssim  
\{\e^2 \| (\ref{est:f2})  \|_{L^\infty_{t,x}} + \e  \|   \mathcal I \|_{L^\infty_{t,x}}
+ \frac{\e^2}{\kappa}  \| (\ref{est:f2})  \|_{L^\infty_{t,x}}^2
\}
  \| \kappa^{-\frac{1}{2}} \e^{-1}  \sqrt{\nu} (\mathbf{I} - \mathbf{P}) f_R \|_{L^2_{t,x,v}}^2
\\ & \ \ \ +  \frac{\sigma_0}{10} \| \kappa^{-\frac{1}{2}} \e^{-1}  \sqrt{\nu} (\mathbf{I} - \mathbf{P}) f_R \|_{L^2_{t,x,v}}^2
+ \frac{10}{\sigma_0}  \frac{1}{\kappa}   \|   \mathcal I \|_{L^\infty_{t,x}}^2 \|  P  f_R \|_{L^2_tL^2_{x}}^2.
  \label{est:Gf_2}
\end{split}
\Ee

From \eqref{IPf2est} and \eqref{est:f2}
\begin{align}
& \left| (\mathbf{I-P}) (v \cdot \nabla_x f_2) - \frac{2}{\kappa} \Gamma(f_1,f_2) \right|  \notag
\\ & \lesssim \left|   (\mathbf{I-P}) (v \cdot  \nabla_x \mathbf P f_2) \right| +  \left|   (\mathbf{I-P}) (v \cdot  \nabla_x L^{-1} ( - \kappa v \cdot \nabla_x f_1 + \Gamma(f_1, f_1) )) \right|  + \frac{1}{\kappa} |   \Gamma(f_1,f_2) | \notag
\\ & \lesssim e^{ - \varrho |v|^2}  | \nabla_x \mathcal I_2 |+ \kappa e^{ - \varrho |v|^2} | \nabla_x^2 \mathcal I | + e^{- \varrho |v|^2}  |\mathcal I |^2  \label{estR11} 
\\ & \ \ \ +  \frac{1}{\kappa} e^{ - \varrho |v|^2} |\mathcal I |  \left(  |\mathcal I |^2  + |\mathcal I_2 |+ \kappa  | \nabla_x \mathcal I | \right) \label{estR12},
\end{align}
we derive that 
\Be\begin{split}\label{est:I-PR_1R2}
|(\ref{Energy_I-PR_1})| &\lesssim 
\left( \kappa^{ 1/2} \e^{1/2}   \| (\ref{estR11})  \|_{L^2_{t,x}}  +\frac{\e^{1/2}}{\kappa^{1/2}} \| \kappa \eqref{estR12} \|_{L^2_{t,x}} \right) \| \kappa^{-1/2} \e^{-1} (\mathbf{I} - \mathbf{P} ) f_R \|_{L^2_{t,x,v}}.
 \end{split}
\Ee

Next we have from \eqref{LinvL}
\[ \begin{split}
\p_t f_2 & = \p_t \mathbf P f_2 + \p_t \mathbf{(I-P)} f_2 
\\ &=   (  \p_t   \rho_2 +   \p_t   u_2 \cdot v + \p_t    \theta_2 \frac{|v|^2 - 3}{2} ) \sqrt{\mu} +  L^{-1} ( - \kappa v \cdot \nabla_x   \p_t   f_1 +  \Gamma (  \p_t   f_1, f_1) + \Gamma(f_1, \p_t f_1)) 
\\ & = (  \p_t   \rho_2 +   \p_t   u_2 \cdot v + \p_t    \theta_2 \frac{|v|^2 - 3}{2} ) \sqrt{\mu} +  L^{-1} ( - \kappa v \cdot \nabla_x   \p_t   f_1 ) + ( \mathbf{I-P}) ( \frac{ f_1 \cdot \p_t f_1}{\sqrt \mu } ).
\end{split} \]
So we can bound
\Be \label{estptf2} \begin{split}
| \p_t f_2 | \lesssim  & e^{ - \varrho |v|^2} \left(  |\p_t \mathcal I_2 | +  \kappa   | \nabla_x \p_t \mathcal I | \right) +   e^{ - \varrho |v|^2}   | \mathcal I | | \p_t \mathcal I | .
\end{split} \Ee
Together with \eqref{est:f2} we have
\begin{align}
  \left| -  \p_t f_2  + \frac{1}{\kappa} \Gamma(f_2,f_2) \right|& \lesssim  
    e^{ - \varrho |v|^2} \left( | \p_t \mathcal I_2|  +  \kappa  | \nabla_x \p_t \mathcal I |  + | \mathcal I | | \p_t \mathcal I | \right)  \label{estR21} 
\\ & \ \ \ + \frac{1}{\kappa} \left(  e^{ - \varrho |v|^2} \left( | \mathcal I |^2  + | \mathcal I_2|  + \kappa | \nabla_x \mathcal I | \right) \right)^2  \label{estR22}
\end{align}
Thus
\Be \label{est:Energy_v^3} \begin{split}
|(\ref{Energy_R2})| \lesssim &  
\e \kappa \| (\ref{estR21})\|_{L^2_{t,x,v}}^2 + \frac{\e}{\kappa} \|\kappa (\ref{estR22})\|_{L^2_{t,x,v}}^2  + \frac{1}{\kappa}   \| \mathbf{P}f_R\|_{L^2_{t,x,v}}^2  
\\ & + \left(  \kappa^{ 1/2} \e^{3/2} \| (\ref{estR21})\|_{L^2_{t,x,v}}  + \frac{\e^{3/2}}{\kappa^{1/2}}  \| \kappa (\ref{estR22})\|_{L^2_{t,x,v}}  \right)
\|
  \kappa^{-1/2} \e^{-1}
  (\mathbf{I} -\mathbf{P})f_R\|_{L^2_{t,x,v}} 
\end{split}\Ee

Finally we control the boundary term (\ref{Energy_bdry}) using a trace theorem (\ref{trace}).
%%%%%%%%%%%%%%%%%%%
\hide
 (Lemma 3.2 in \cite{EGKM2})
\Be\label{trace1}
\frac{1}{\e}\int^t_0 \int_{\gamma_+^N} |h| \dd \gamma \dd s \lesssim_N  \iint_{\O \times \R^3} |h(0)|  +   \int^t_0 \iint_{\O \times \R^3} |h|  + \int^t_0 \iint_{\O \times \R^3 } | \p_t h +\frac{1}{\e} v\cdot \nabla_x h|,
\Ee
where $\gamma_+^N:= \{(x,v) \in \gamma_+: |n(x) \cdot v|> 1/N \ \text{and} \ 1/N<|v|<N \}$. 
\unhide
%%%%%%%%%%%%%%%%%%%%
First we have, from (\ref{bdry_fR}),
\Be
\begin{split}
(\ref{Energy_bdry}) &= \frac{1}{2\e} \int_0^t \int_{\gamma_+} \{
|f_R|^2 - |P_{\gamma_+} f_R|^2
\} - \frac{1}{2} \int^t_0 \int_{\gamma_-} |(1- P_{\gamma_+})
 (\mathbf{I}-\mathbf{P}) 
 f_2|^2\\
 &
 \ \ -  \int^t_0 %\cancel{ 
\int_{\gamma_-}  \frac{1}{\e^{1/2}}P_{\gamma_+ } f_R (1-  {P}_{\gamma_+}) (\mathbf{I}-\mathbf{P})  f_2 
%}
\\
&\geq  \frac{1}{2 } 
| \e^{-\frac{1}{2}}
(1- P_{\gamma_+}) f_R
|_{L^2_t L^2_{\gamma_+}}^2  
-\frac{1}{8C} | \e^{-\frac{1}{2}}
  P_{\gamma_+}   f_R
|_{L^2_t L^2_{\gamma_+}}^2\\
& \ \ \ 
 -( \frac{1}{2}
  + 2C
 )
  \int^t_0 \int_{\gamma_-} |(1- P_{\gamma_+}) 
   (\mathbf{I}-\mathbf{P}) 
   f_2|^2
  \ \ \text{for} \ C\gg 1,
\label{est:EB}
\end{split}
\Ee
where we have used the fact $|P_{\gamma_+}f_R|_{L^2_{\gamma_+}}= |P_{\gamma_+}f_R|_{L^2_{\gamma_-}}$ from $P_{\gamma_+}f_R(t,x,v)$ being a function of $(t,x,|v|)$ due to $u|_{\p\O}=0$.

Now we estimate $P_{\gamma_+}f_R$. % which is missed in (\ref{est:EB}). 
Since $P_{\gamma_+}$ in (\ref{bdry_fR}) is a projection of $c_\mu\sqrt{\mu}$ on $\gamma_+$, it follows $\int_{\gamma_+} |P_{\gamma_+}f|^2 \leq 2 \int_{\gamma_+^N} |P_{\gamma_+}f|^2$ for large enough $N>0$, where $\gamma_+^N:= \{(x,v) \in \gamma_+: |n(x) \cdot v|> 1/N \ \text{and} \ 1/N<|v|<N \}$. Setting $h=|f_R|^2$ in (\ref{trace}) and using (\ref{eqtn_fR}) we derive 
\Be \label{est:bdry1}
\begin{split}
&\frac{1}{\e}\int^t_0 \int_{\gamma_+^N } |f_R|^2 \dd \gamma \dd s 
 \leq C_N  \iint_{\O \times \R^3} |f_R(0)|^2  +   \int^t_0 \iint_{\O \times \R^3} |f_R|^2  \\
& \ \ \ + \int^t_0 \iint_{\O \times \R^3 }\Big|\Big[
 - \frac{1}{\e^2 \kappa} Lf_R  + \frac{2}{  \kappa \e}  \Gamma( f_1 + \e f_2, f_R) 
+    \frac{   1 }{ \e^{1/2}\kappa}\Gamma(f_R, f_R) \\
&  \ \ \ \ \ \ \ \ \ \ \ \  \ \ \ \ \ \ \ \ \ \ \ \ 
 -  \frac{1}{\e^{1/2}}  \left \{ (\mathbf{I-P}) (v \cdot \nabla_x f_2) - \frac{2}{\kappa} \Gamma(f_1,f_2) \right \} + \e^{1/2}   \left( -  \p_t f_2  + \frac{1}{\kappa} \Gamma(f_2,f_2) \right)  \Big] f_R\Big|\\
 &\leq C_N \big\{  \|f_R(0)\|^2_{L^2_{x,v}}  + \|    \mathbf{P}  f_R \|_{L^{2}_{t,x,v}}^2
 + \| \e^{-1} \kappa^{-1/2} \sqrt{\nu} (\mathbf{I} - \mathbf{P}) f_R \|_{L^{2}_{t,x,v}}^2\\
 &  \ \ \ \ \ \ \ \ \ \ \ \  \ \ \ 
 + (\ref{est:EG}) + (\ref{est:Gf_2}) +(\ref{est:I-PR_1R2})
 + (\ref{est:Energy_v^3})\big\}.
 \end{split}
\Ee
Furthermore from (\ref{bdry_fR}) and (\ref{est:bdry1}) 
\Be\label{f:+-}
\begin{split}
| f_R|_{L^2_t L^2_{\gamma_-}
 }^2
  \lesssim 
| f_R|_{L^2_t L^2_{\gamma_+}
 }^2 + 
\e | (1- P_{\gamma_+})  (\mathbf{I}-\mathbf{P}) f_2| _{L^2_t L^2_{\gamma_-}
 }
 ^2
 = | f_R|_{L^2_t L^2_{\gamma_+}
 }^2 + 
{\e } \| \eqref{IPf2est} \|_{L^2_t L^2 (\p\O)
 }
 ^2 .
 \end{split}
\Ee

%From the fact that $P_{\gamma_+}$ is a projection as in (\ref{bdry_fR}), we derive that 

Finally we collect the terms as 
\Be\notag
\begin{split}
&\text{r.h.s of} \ (\ref{est:EL})+ (\ref{est:EB})
+ \frac{1}{4C}  | \e^{-\frac{1}{2}}
  P_{\gamma_+}   f_R
|_{L^2_t L^2_{\gamma_+}}^2 + \frac{\e^{-1}}{16C} | f_R|_{ L^2_t L^2_{\gamma_-}}^2\\
&\leq \text{r.h.s of} \  (\ref{est:EG}) + (\ref{est:Gf_2}) +  (\ref{est:I-PR_1R2}) + (\ref{est:Energy_v^3}) + \frac{1}{4C} \times  \text{r.h.s of} \   (\ref{est:bdry1})
+ \frac{\e^{-1}}{16C}  \times  \text{r.h.s of} \    (\ref{f:+-}). 
\end{split}
\Ee
We choose large $N$ and then large $C$ so that $\frac{C_N}{4C}\ll \sigma_0$. Using Young's inequality for products, and then moving contributions of $ \| \kappa^{-\frac{1}{2}} \e^{-1} \sqrt{\nu} (\mathbf{I} - \mathbf{P}) f_R \|_{L^2_{t,x,v}}^2$ to l.h.s., we derive (\ref{est:Energy}). 

Next we prove (\ref{est:Energy_t}). An energy estimate to (\ref{eqtn_fR_t}) and (\ref{bdry_fR_t}) lead to (\ref{est:Energy_t})
\begin{align}
&\frac{1}{2}\| \p_t f _R(t)\|_2^2 - \frac{1}{2} \| \p_t  f _R (0)\|_2^2
+ \frac{1}{ \kappa \e^{2 }} \int^t_0 \iint _{\O \times \R^3}\p_t  f _R L  \p_t f_R\label{Energy_LHS_t} \\
& +  \frac{1}{2\e} \int^t_0 \int_{\gamma_+} |\p_t f_R|^2 \notag\\
&-  \frac{1}{2\e} \int^t_0 \int_{\gamma_-} | 
  P_{\gamma_+} \p_t f_R-  \e^{1/2}  (1-P_{\gamma_+})  \p_t   (\mathbf{I} -\mathbf{P}) f_2 |^2 \label{Energy_bdry_t}
 \end{align}
 \begin{align}
=  &  \frac{2 }{\kappa \e^{1/2} } \int^t_0 \iint_{\O \times \R^3}  \Gamma(\p_t f_R, f_R) \mathbf{(I-P)}\p_t f_R \label{Energy_Gamma_t}
\\ & +    \frac{2}{\e \kappa}  \int_0^t \iint_{\O \times \R^3} \left(  \Gamma(\p_t f_1 + \e \p_t f_2,f_R)  + \Gamma(f_1 + \e f_2,\p_t f_R)  \right) \p_t  f_R 
 \label{Energy_Gf_2_t}
\\
&+   \frac{1}{\e^{1/2}} \int^t_0 \iint_{\O \times \R^3}   \left \{  - (\mathbf{I-P}) (v \cdot \nabla_x \p_t f_2) + \frac{2}{\kappa} \Gamma(\p_t f_1,f_2) + \frac{2}{\kappa} \Gamma( f_1, \p_t f_2) \right \}  \mathbf{(I-P)} \p_t f_R \label{Energy_I-PR_1_t}
\\ &  + \e^{1/2} \int^t_0 \iint_{\O \times \R^3}   \left( - \p_t^2 f_2  + \frac{2 }{\kappa} \Gamma(\p_t f_2,f_2) \right) \p_t f_R \label{Energy_tlast}
\end{align}

We control the terms similarly as in the proof of (\ref{est:Energy}):
\begin{align}
(\ref{Energy_LHS_t})\geq &  \  \frac{1}{2}\| \p_t f _R(t)\|_{L^2_{x,v}}^2 - \frac{1}{2} \|\p_t f _R (0)\|_{L^2_{x,v}}^2
+  \sigma_0
\| \kappa^{-\frac{1}{2}} \e^{-1} \sqrt{\nu}(\mathbf{I} - \mathbf{P}) \p_t f_R\|_{L^2_{t,x,v}}^2,\label{est:EL_t} 
\\
|(\ref{Energy_Gamma_t})|  \lesssim &   \ 
(   \e^{3/2} \|   \mathfrak{w}f_R\|_{L^\infty_{t,x,v}}   ) \| \kappa^{-\frac{1}{2}} \e^{-1} \sqrt{\nu}
(\mathbf{I} - \mathbf{P}) \p_t f_R \|^2_{L^2_{t,x,v}} \notag
\\
&+  \frac{ \e^{1/2}}{\kappa^{3/2}}
  \|\kappa^{1/2}     {P} f_R \|_{L^\infty_tL^6_{x }} 
 \| \kappa^{1/2}   {P} \p_t f_R \|_{L^2_tL^3_{x }}
\| \kappa^{-\frac{1}{2}} \e^{-1} \sqrt{\nu}
(\mathbf{I} - \mathbf{P}) \p_t f_R \|_{L^2_{t,x,v}}\label{est:EG_t}
\end{align}

\Be \label{est:EG_t2} \begin{split}   
 & |(\ref{Energy_Gf_2_t})| 
\\  \lesssim   & \left( \kappa^{-\frac{1}{2}}\e \| \sqrt{\nu} \p_t f_2 \|_{L^\infty_{t,x,v}} + \kappa^{-\frac{1}{2}} \| \sqrt{\nu} \p_t f_1 \|_{L^\infty_{t,x,v}} \right) \{ \|P f_R \|_{L^2_{t,x}} +  \kappa^{\frac{1}{2}} \e  \|\kappa^{-\frac{1}{2}} \e^{-1} (\mathbf{I} -\mathbf{P}) f_R \|_{L^2_{t,x,v}} \}
 \\   &  \ \ \   \times \| \kappa^{-\frac{1}{2}} \e^{-1} (\mathbf{I} - \mathbf{P})  \p_t f_R \|_{L^2_{t,x,v}}
\\   &+ \left(  \frac{\e}{\kappa^{1/2}} \| P f_2 \|_{L_{t,x}^\infty} + \frac{\e}{\kappa^{1/2}} \| \mathbf{(I-P) } f_2 \|_{L_{t,x}^\infty}  +\frac{1}{\kappa^{1/2}}  \| \nu f_1 \|_{L_{t,x}^\infty} \right)
\\ & \ \ \ \times \left(  \| P \p_t f_R \|_{L_{t,x}^2} + \kappa^{\frac{1}{2}} \e \| \kappa^{-\frac{1}{2}} \e^{-1}  (\mathbf{I} - \mathbf{P})  \p_t f_R \|_{L^2_{t,x,v}}  \right) \| \kappa^{-\frac{1}{2}} \e^{-1} (\mathbf{I} - \mathbf{P})  \p_t f_R \|_{L^2_{t,x,v}}
\\  \lesssim & \left(  \frac{\sigma_0}{10} +  \frac{\e^2}{\kappa} \| \sqrt{\nu} \p_t f_2 \|_{L^\infty_{t,x,v}}^2 +  \frac{ \e^2}{\kappa} \| P f_2 \|_{L_{t,x}^\infty} + \frac{\e^2}{\kappa} \| \mathbf{(I-P) } f_2 \|_{L_{t,x}^\infty}  + \e  \| \nu f_1 \|_{L_{t,x}^\infty} \right)  \| \kappa^{-\frac{1}{2}} \e^{-1} (\mathbf{I} - \mathbf{P})  \p_t f_R \|_{L^2_{t,x,v}}^2 
\\ & + \left( 1+ \frac{1}{\kappa} \| \sqrt{\nu} \p_t f_1 \|_{L^\infty_{t,x,v}}^2 + \frac{1}{\kappa} \| \sqrt{\nu}  f_1 \|_{L^\infty_{t,x,v}}^2 \right)  \|P f_R \|_{L^2_{t,x}}^2  + \left( 1 + \e^2 \| \sqrt{\nu} \p_t f_1 \|_{L^\infty_{t,x,v}}^2 \right) \|\kappa^{-\frac{1}{2}} \e^{-1} (\mathbf{I} -\mathbf{P}) f_R \|_{L^2_{t,x,v}}^2
\end{split} \Ee

We have
\begin{align}
& \left| - (\mathbf{I-P}) (v \cdot \nabla_x \p_t f_2) + \frac{2}{\kappa} \Gamma(\p_t f_1,f_2) + \frac{2}{\kappa} \Gamma(f_1,\p_ tf_2 ) \right|  \notag
\\ & \lesssim \left|   (\mathbf{I-P}) (v \cdot  \nabla_x \mathbf P \p_t f_2) \right| +  \left|   (\mathbf{I-P}) (v \cdot  \nabla_x L^{-1} ( - \kappa v \cdot \nabla_x \p_t f_1 + 2 \Gamma(\p_t f_1, f_1) )) \right|   \notag
\\ & \ \ \ + \frac{1}{\kappa} |   \Gamma(\p_t f_1,f_2) | + \frac{1}{\kappa} |   \Gamma(f_1, \p_t f_2) | \notag
\\ & \lesssim e^{ - \varrho |v|^2} | \nabla_x \p_t \mathcal I_2 | + \kappa e^{ - \varrho |v|^2}  | \nabla_x^2 \p_t \mathcal I |  + e^{ - \varrho |v|^2} | \mathcal I | | \p_t \mathcal I |  \label{estR11t} 
\\ & \ \ \ +  \frac{1}{\kappa} e^{ - \varrho |v|^2} | \p_t \mathcal I |  \left( | \mathcal I |^2  + | \mathcal I_2|  + \kappa | \nabla_x \mathcal I | \right)  +  \frac{1}{\kappa} e^{ - \varrho |v|^2}| \mathcal I |  \left( | \p_t \mathcal I_2 | +  \kappa  | \nabla_x \p_t \mathcal I |   +  | \mathcal I | | \p_t \mathcal I | \right). \label{estR12t}
\end{align}
So we derive that 
\Be\begin{split}\label{est:I-PR_1R2_t}
|(\ref{Energy_I-PR_1_t})| &\lesssim 
\left( \kappa^{ 1/2} \e^{1/2}   \| (\ref{estR11t})  \|_{L^2_{t,x,v}}  +\frac{\e^{1/2}}{\kappa^{1/2}} \| \kappa \eqref{estR12t} \|_{L^2_{t,x,v}} \right) \| \kappa^{-1/2} \e^{-1} (\mathbf{I} - \mathbf{P} ) \p_t f_R \|_{L^2_{t,x,v}}.
 \end{split}
\Ee

And we have
\begin{align}
&  \left| -  \p_t^2 f_2  + \frac{2}{\kappa} \Gamma(\p_tf_2,f_2) \right| \lesssim  \notag
\\  &   e^{ - \varrho |v|^2} \left(  |\p_t^2 \mathcal I_2|  +  \kappa  | \nabla_x \p_t^2 \mathcal I | \right)  + e^{ - \varrho |v|^2} \left(  |\mathcal I |  | \p_t^2 \mathcal I |   +  | \p_t \mathcal I |^2 \right)   \label{estR21t} 
\\ &  + \frac{1}{\kappa}  e^{ - \varrho |v|^2} \left( | \mathcal I |^2 + |\mathcal I_2|  + \kappa | \nabla_x \mathcal I | \right)  \left( | \p_t \mathcal I_2| +  \kappa  |\nabla_x \p_t \mathcal I|  + |\mathcal I | |\p_t \mathcal I |    \right) \label{estR22t}
\end{align}
Thus
\Be \label{estfR2last}  \begin{split}
|(\ref{Energy_tlast})| \lesssim &  
\e \kappa \| (\ref{estR21t})\|_{L^2_{t,x,v}}^2 + \frac{\e}{\kappa} \|\kappa (\ref{estR22t})\|_{L^2_{t,x,v}}^2  + \frac{1}{\kappa}   \| \mathbf{P} \p_t f_R\|_{L^2_{t,x,v}}^2  
\\ & + \left(  \kappa^{ 1/2} \e^{3/2} \| (\ref{estR21t})\|_{L^2_{t,x,v}}  + \frac{\e^{3/2}}{\kappa^{1/2}}  \| \kappa (\ref{estR22t})\|_{L^2_{t,x,v}}  \right) \|  \kappa^{-1/2} \e^{-1} (\mathbf{I} -\mathbf{P})\p_t f_R\|_{L^2_{t,x,v}} 
\end{split}\Ee

Lastly we estimate (\ref{Energy_bdry_t}). As in (\ref{est:EB}) we derive that (\ref{Energy_bdry_t}) is bounded from below by
 \Be
 \begin{split}\label{est:EB_t}
 \frac{1}{2 } 
&| \e^{-\frac{1}{2}}
(1- P_{\gamma_+}) \p_t  f_R
|_{L^2((0,T); L^2_{\gamma_+})}^2  
-\frac{1}{8C} | \e^{-\frac{1}{2}}
  P_{\gamma_+} \p_t   f_R
|_{L^2 ((0,T); L^2_{\gamma_+})}^2
\\& - C    |(1- P_{\gamma_+})  (\mathbf{I} -\mathbf{P})  \p_t f_2|_{L^2 ((0,T); L^2_{\gamma_-})}^2  
 \\ \geq  & \frac{1}{2 } 
 | \e^{-\frac{1}{2}}
(1- P_{\gamma_+}) \p_t  f_R
|_{L^2((0,T); L^2_{\gamma_+})}^2  
-\frac{1}{8C} | \e^{-\frac{1}{2}}
  P_{\gamma_+} \p_t   f_R
|_{L^2 ((0,T); L^2_{\gamma_+})}^2\\
&
 -
 C
 \Big\{
    \| e^{ - \varrho |v|^2} ( \kappa |\nabla_x \p_t \mathcal I |+ | \mathcal I | | \p_t \mathcal I | ) \|_{L^2_t L^2(\p\O)} 
  \Big\}
   \end{split}
 \Ee
 for some large $C$.
 
Now we bound $P_{\gamma_+} \p_t f_R$ using (\ref{trace}). Following the argument arriving at (\ref{est:bdry1}) and setting $h=|\p_t f_R|^2$ we derive 
 \Be
 \begin{split}
& \frac{1}{\e} \int^t_0 \int_{\gamma_+} |   \p_t f_R  |^2 \dd \gamma \dd s \\
 &\lesssim_N
 \| \p_t f_R(0) \|_{L^2_{x,v}} +  \| \p_t f_R  \|_{L^2_{t,x,v}} 
 + \int_0^t \iint_{\O \times \R^3} \Big|\Big( - \frac{1}{\e^2 \kappa} L \p_t f_R  + \text{r.h.s of } (\ref{eqtn_fR_t})\Big)\p_t f_R  \Big|\\
 &\lesssim_N \| \p_t f_R(0) \|_{L^2_{x,v}} ^2+  \| P \p_t f_R  \|_{L^2_{t,x}} ^2
 + \| \e^{-1} \kappa^{-1/2} \sqrt{\nu} (\mathbf{I}-\mathbf{P}) \p_t f_R  \|_{L^2_{t,x,v}} ^2
 \\ & \ \ \  \ \ \   + (\ref{est:EG_t}) +  (\ref{est:EG_t2})+
 (\ref{est:I-PR_1R2_t})+ (\ref{estfR2last}).
 \end{split}\Ee

We conclude (\ref{est:Energy_t}) by collecting the terms.
\end{proof}

\subsection{$L^2_tL^3_x$-integrability gain for $\mathbf{P}f_R$}  \label{L2L3section}
The goal for this section is the following proposition:

\begin{proposition}\label{prop:average}
 Assume the same assumptions in Proposition \ref{prop:Hilbert}. Then we have, for $2<p<3$,  %\bcb
\Be\begin{split}\label{average_3D}
&d_3\big\|    {P} f_R
 \big\|_{L^2_t L^p_x }
 \\
\lesssim & \  \| f_R \|_{L^\infty_t L^2 _{x,v}} + \| f_R (0) \|_{L^2_\gamma}    + \frac{1}{\kappa} \| \mathcal I \|_{L^\infty_{t,x}}  \| Pf_R \|_{L^2_{t,x}}
\\
& +\Big\{
\frac{1}{\e \kappa }  + \frac{1}{\kappa} \left( \| \mathcal I \|_{L^\infty_{t,x}} + \e \|  (\ref{est:f2})  \|_{L^\infty_{t,x} } \right)
+ \frac{\e^{1/2}}{\kappa} \| \mathfrak{w}_{\varrho} f_R \|_{L^\infty_{t,x,v} } 
%+   \| \mathfrak{w}_{\varrho, \ss } f_R   \|_{ L^\infty _{t,x,v}}^{\frac{p-2}{p}}
\Big\}
 \| \sqrt{\nu} (\mathbf{I} - \mathbf{P})f_R \|_{L^2_{t,x,v} }
 \\ & + \e^{1/2}  \| (\ref{estR11})  \|_{L^2_{t,x}}  +\frac{\e^{1/2}}{\kappa} \| \kappa \eqref{estR12} \|_{L^2_{t,x}} + \e^{3/2} \| (\ref{estR21})\|_{L^2_{t,x,v}}  + \frac{\e^{3/2}}{\kappa}  \| \kappa (\ref{estR22})\|_{L^2_{t,x,v}},
\end{split}\Ee
with 
\Be\label{d3}
d_{3}:= 1 -  \frac{\e}{\kappa} \|(\ref{est:f2}) \|_{L^\infty_t   L_x^{\frac{2p}{p-2}} } - \frac{\e^{1/2}}{\kappa} \| Pf_R \|_{L^\infty_tL_x^{6} }^{\frac{3(p-2)}{p}} 
\|  \mathfrak{w}_{\varrho}  f_R \|_{L_{t,x,v}^{\infty} }^{\frac{6-2p}{p}},
\Ee
and for $\varrho'<\varrho$ 
\Be\begin{split}\label{average_3Dt}
&d_{3,t} \big\|    {P}  \p_t f_R
 \big\|_{L^2_t L^p_x }
 \\
\lesssim & 
\Big\{ \frac{1}{\kappa} \left( \| \p_t \mathcal I \|_{L^\infty_{t,x,v}}  + \e \| \eqref{estptf2} \|_{L^\infty_{t,x,v } }  \right)  \Big\}  \Big\{ \| Pf_R \|_{L^2_{t,x}}+ \| \sqrt{\nu} (\mathbf{I} -\mathbf{P}) f_R \|_{L^2_{t,x,v}} \Big\}
\\ &+ \Big\{ \frac{1}{\kappa \e} + \frac{\e^{1/2}}{\kappa} \|   \mathfrak{w}f_R\|_{L^\infty_{t,x,v}} +  \frac{1}{\kappa} \left( \| \mathcal I \|_{L^\infty_{t,x}} + \e \|(\ref{est:f2})  \|_{L^\infty_{t,x}}  \right)  \Big\} \| \sqrt{\nu} (\mathbf{I} - \mathbf{P}) \p_t f_R \|_{L^2_{t,x,v}}
\\ & + (\kappa \e ) ^{\frac{2}{p-2}} \|  \mathfrak{w}_{\varrho^\prime} \p_t  f_R \|_{L^2_t  L^\infty_{x,v}  }+ \| \p_t f_R \|_{L^\infty_t L^2_{x,v}}  + \|\p_t  f_R (0) \|_{L^2_\gamma} 
 + \frac{1}{\kappa} \|\mathcal I \|_{L^\infty_{t,x}} \| P \p_t f_R \|_{L^2_tL^2_{x }}
\\  &+  \e^{1/2}   \| (\ref{estR11t})  \|_{L^2_{t,x,v}}  +\frac{\e^{1/2}}{\kappa} \| \kappa \eqref{estR12t} \|_{L^2_{t,x,v}} + \e^{3/2} \| (\ref{estR21t})\|_{L^2_{t,x,v}}^2 + \frac{\e^{3/2}}{\kappa} \|\kappa (\ref{estR22t})\|_{L^2_{t,x,v}}^2,
\end{split}\Ee
with 
\Be\label{d3t}
d_{3,t}:= 1 -  \frac{\e}{\kappa} \|  (\ref{est:f2}) \|_{L^\infty_t   L_x^{\frac{2p}{p-2}}  }-  \frac{\e^{1/2}}{\kappa} \| P f_R \|_{L^\infty_t  L_x^{6} }^{\frac{3(p-2)}{p}} \|  \mathfrak{w}_{\varrho} f_R \|_{L^{\infty} _{t,x,v}}^{\frac{6-2p}{p}},
\Ee
where both bounds are uniform-in-$p$ for $2<p<3$. 
 
\end{proposition}

We prove the proposition by several steps. 

\
 
%We prove the proposition by Lemma \ref{lemma:average} and an extension in (\ref{ext:f+}). 

%\begin{proof}[\textbf{Proof of Proposition \ref{prop:average}}]
\textbf{Step 1: Extension.} We define a subset 
\Be
\tilde{\O} := (0,2\pi)\times (0,2\pi) \times (0,\infty) \subset \R^3.\label{def:tilde_O}
\Ee 
We regard $\tilde{\O}$ as an open subset but not a periodic domain as $\O$. Without loss of generality we may assume that $f_R(0,x,v)$ is defined in $\R^3 \times \R^3$ and $\|f_R(0)
\|_{L^p(\R^3 \times \R^3)} \lesssim \|f_R(0)
\|_{L^p(\tilde{\O} \times \R^3)}$ for all $1 \leq p \leq \infty$.  %We denote $\O^c= (\mathbb{T
%}^2 \times \R)\backslash \O$. 
Then we extend a solution for whole time $t \in \R$ as 
%an initial datum $f_R(0,x,v)$ in $(x,v) \in \R^3 \times \R^3$ as 
\Be\label{ext:f1}
f_I (t,x,v) : = \mathbf{1}_{t \geq 0 } f_R (t,x,v) + \mathbf{1}_{t \leq 0} \chi_1 (t) f_R (0, x%- \frac{t}{\e} v
,v ), 
\Ee
where a smooth non-negative function $\chi_1$ satisfies $\chi_1 (t)\equiv  1$ for $t \in [-1,0]$, $\chi_1 (t) \equiv 0$ for $t<-2$, and $0 \leq \frac{d}{dt}\chi_1 \leq 4$.

 A closure of $\tilde{\O}$ is given as $cl(\tilde{\O})= [0,2\pi]\times [0,2\pi] \times [0,\infty)$. Let us define $\tilde{t}_B(x,v) \in \R$ for $(x,v)  \in  (\R^3 \backslash \tilde{\O}) \times \R^3$. %If $x \in \p \tilde{\O}$ then we set $\tilde{t}_B(x,v)=0$. For $x \in \R^3 \backslash cl(\tilde{\O})= (\R^3 \backslash \tilde{O}) \backslash \p \tilde{\O}$ 
We consider $\tilde{B}(x,v):=\{s \in \R:x+ s v \in \R^3 \backslash cl(\tilde{\O})  \}$. Clearly if $\tilde{B}(x,v) \neq \emptyset$ then $ \{s>0\} \subset \tilde{B} (x,v)$ or $  \{s<0\} \subset \tilde{B} (x,v)$ exclusively. 

If $ \{s>0\} \subset \tilde{B} (x,v)$, let $ I_+$ be the largest interval such that $ \{s > 0 \} \subset I_+ \subset \tilde B(x,v) $. And if $ \{s < 0\} \subset \tilde{B} (x,v)$, let $ I_-$ be the largest interval  such that $ \{s > 0 \} \subset I_-  \subset \tilde B(x,v)$. We define
\Be\label{tB}
\tilde{t}_B(x,v)  = \begin{cases} 0 & \text{if} \ \  x \in \p\tilde{\O}, 
\\
\inf I_+
& \text{if} \ \ x \in  %(\R^3 \backslash \tilde{\O}) \backslash \p \tilde{\O}=
\R^3 \backslash cl(\tilde{\O})
\ \text{and} \ \tilde{B}(x,v) \neq \emptyset \ \text{and} \ 
 \{s>0\}\subset I_+ \subset \tilde{B}(x,v) 
,\\
\sup I_-
&  \text{if} \ \ x \in  %(\R^3 \backslash \tilde{\O}) \backslash \p \tilde{\O}=
\R^3 \backslash cl(\tilde{\O})
\ \text{and} \ \tilde{B}(x,v) \neq \emptyset \ \text{and} \ 
\{s<0\} \subset I_- \subset \tilde{B}(x,v) ,\\
  - \infty& \text{if} \  \ \tilde{B}(x,v) = \emptyset \ \text{and} \  x \notin \p\tilde{\O}. 
\end{cases}
\Ee
Using (\ref{tB}) we define 
\Be\label{ext:f2}
f_E (t,x,v): =  \mathbf{1}_{(x,v) \in (\R^3 \backslash \tilde{\O}) \times \R^3 } f_I (t+ \e \tilde{t}_B(x,v), \tilde{x}_B(x,v),v) \ \ \text{with} \ \ \tilde{x}_B(x,v):= x+\tilde{t}_B(x,v)v.
\Ee
 It is easy to see that $\e \p_t f_E + v\cdot \nabla_x f_E=0$ in the sense of distributions.

%\Be\label{ext_f0}
%\begin{split} 
%f_R(0,x,v): =  \mathbf{1}_{x_3\geq 0}   f_R(0,x_1- 2 \pi n_1,x_2- 2 \pi n_2,x_3,v) +  \mathbf{1}_{x_3 < 0}   f_R(0,x_1- 2 \pi n_1,x_2- 2 \pi n_2,-x_3,v) \\
%\text{for} \ (x_1,x_2) \in [2\pi n_1 - \pi, 2\pi n_1 + \pi] \times [2\pi n_2 - \pi, 2\pi n_2 + \pi]. 
%\end{split}
%\Ee

Next we define two cutoff functions. For any $N>0$ we define smooth non-negative functions as
\Be
\begin{split}
 &\chi_2 (x) \equiv 1 \ \text{for} \  x \in [-\pi, 3\pi] \times [-\pi, 3\pi] \times [-\pi, \infty), \\ 
& \chi_2 (x ) \equiv 0 \ \text{for} \  x \notin [-2\pi, 4\pi] \times [-2\pi, 4\pi] \times [-2\pi, \infty)   , \ \  
 %\chi_1(t) =e^{t}  \ \text{for} \ t \leq 0, \ \ 
|\nabla_{x}\chi_2  |  \leq 10   %\ \text{for} \ t \in [-\e ,0]
,\label{chi1}\\
\end{split}\Ee
\Be
\begin{split}
&\chi _3 (v) \equiv 1 \ \text{for}  \ 
|v| \leq N-1,  \ \text{and} \ |v_i|\geq 2/N  \ \text{for all } i=1,2,3,\\
&\chi_3  (v) \equiv 0 \ \text{for} \ |v| \geq  N  \ \text{or} \ |v_i| \leq 1/N \ \text{for any } i=1,2,3, \ \ |\nabla_v \chi _3  | \leq 10  %\ \text{for} \ v \in \R^3
.\label{chi_v}
\end{split}\Ee
We denote 
\Be\label{def:UV}
U : = [-2\pi, 4\pi] \times [-2\pi, 4\pi] \times [-2\pi, \infty), \ \   V: = 
\{ v\in \R^3: |v| \leq N \} \cap \bigcap_{i=1,2,3} \{ v\in \R^3: |v_i| \geq 1/N  \} 
\Ee

We define an extension of cut-offed solutions   %for $(t,x,v) \in [0,T] \times (\mathbb{T}^2 \times \R) \times \R^3$  %$\mathbf{1}_{|v_3|\geq 1/N}f_R(t,x,v)$ as 
\Be\label{ext:f+}
%\begin{split}
 \bar{f}_{R}(t,x,v) :=%&  
 \chi_2 (x) \chi_3 (v)  \big\{
 % \chi_1(t) 
 %\Big\{
  \mathbf{1}_{\tilde{\O}}(x)% \mathbf{1}_{t\geq 0} 
 f_I(t,x,v)
 +  f_E (t,x,v)\big\}
 \  \ \text{for} \ (t,x,v) \in (-\infty,T] \times \R^3 \times \R^3
 .
 % +  \mathbf{1}_{\O}(x) \mathbf{1}_{t<0}
  %e^{- \frac{\nu t}{\kappa \e^2}}
 % f_R(0,x- \frac{t}{\e} v,v)
%  \\
% &  
% + \mathbf{1}_{\O^c}(x)   \mathbf{1}_{t- \e \frac{x_3}{v_3}
% \geq 0}
%\mathbf{1}_{v_3<0}
% e^{-(t-\tb(x,v)% \e \frac{x_3}{v_3}
% )}
 %e^{- \frac{\nu }{\kappa \e^2 } \frac{ \e x_3}{v_3}}
% f_R(t-\e\frac{x_3}{v_3}%\tb(x,v) 
% ,  x-  \frac{x_3}{v_3} v
% , v)
%\\
 %& \ \ \ \ \ \ \ \ \ 
%  +\mathbf{1}_{\O^c}(x)   \big[ \mathbf{1}_{t-  \e\frac{x_3}{v_3}
% < 0}  \mathbf{1}_{v_3<0}
 %e^{ t-  \e \frac{x_3}{v_3}
 % }
% f_R(0, x- \frac{t}{\e} v %x-    \frac{x_3}{v_3} v
% , v) 
% + 
  %\mathbf{1}_{t- \tb(x,v) %\e\frac{x_3}{v_3}
%< 0} 
 %\mathbf{1}_{\O^c}(x) 
%   \mathbf{1}_{v_3>0}
 %e^{ t- \tb(x,v) %\e \frac{x_3}{v_3}
 % }
%\big] e^{- \frac{\nu t}{\kappa \e^2}} f_R(0, x- \frac{t}{\e} v %x-    \frac{x_3}{v_3} v
 %, v)  \Big\}
 % \end{split}
\Ee
%%%%%%%%%%%%
%%%%%%%%%%%%
%%%%%%%%%%%%
\hide
\Be\label{ext:f+}
 \bar{f}_R(t,x,v) = 
 \begin{cases}
 \mathbf{1}_{|v_3|\geq 1/N}
 \big[\mathbf{1}_{t\geq 0} e^{-t}f_R(t,x,v) 
+\mathbf{1}_{t<0} e^tf_R(0,x- \frac{t}{\e} v,v)
 \big]& \text{if } x_3\geq 0,\\ 
 \mathbf{1}_{|v_3|\geq 1/N}
 \big[ 
 \mathbf{1}_{t- \tb(x,v)%\e \frac{x_3}{v_3}
 \geq 0}
 e^{-(t-\tb(x,v)% \e \frac{x_3}{v_3}
 )}
 f_R(t-\tb(x,v)% \e\frac{x_3}{v_3}
 , \xb(x,v)%x-  \frac{x_3}{v_3} v
 , v)
 +  \mathbf{1}_{t- \tb(x,v) %\e\frac{x_3}{v_3}
 < 0}
 e^{ t- \tb(x,v) %\e \frac{x_3}{v_3}
  }
 f_R(0, 
 x- \frac{t}{\e} v
 %x-    \frac{x_3}{v_3} v
 , v)
 \big]
 & \text{if } x_3<0.%, |v_3|\geq 1/N\\
 %  & \text{if } x_3<0, v_3\leq -1/N\\
% 0  & \text{if } x_3<0, |v_3|<1/N
 \end{cases}
\Ee\unhide
%Then we extend (\ref{ext:f+}) in a time interval $[-T,T]$ by the even extension in $t$ at $t=0$ as 
%\Be\label{ext:f}
%\bar{f}_R (t,x,v):=  \bar{f}_{R,+} (t,x,v) +  \bar{f}_{R,+} (-t,x,v) \ \ \text{for} \  (t,x,v)\in [-T,T] \times (\mathbb{T}^2 \times \R) \times \R^3. 
%\Ee 
\hide

\Be
\begin{split}
f_1(t,x,v): =  \mathbf{1}_{(x_1,x_2) \in [-\pi, \pi]^2}
\Big\{& % \mathbf{1}_{t\geq 0} 
\mathbf{1}_{x_3\geq 0 } f_R(t,x,v)
 % +  \mathbf{1}_{\O}(x) \mathbf{1}_{t<0}
  %e^{- \frac{\nu t}{\kappa \e^2}}
 % f_R(0,x- \frac{t}{\e} v,v)
%  \\
% &  
 + \mathbf{1} _{x_3<0}   \mathbf{1}_{t- \e \frac{x_3}{v_3}
 \geq 0}
\mathbf{1}_{v_3<0}
% e^{-(t-\tb(x,v)% \e \frac{x_3}{v_3}
% )}
 %e^{- \frac{\nu }{\kappa \e^2 } \frac{ \e x_3}{v_3}}
 f_R(t-\e\frac{x_3}{v_3}%\tb(x,v) 
 ,  x-  \frac{x_3}{v_3} v
 , v)
\\
 & 
  +\mathbf{1} _{x_3<0}   \big[ \mathbf{1}_{t-  \e\frac{x_3}{v_3}
 < 0}  \mathbf{1}_{v_3<0}
 %e^{ t-  \e \frac{x_3}{v_3}
 % }
% f_R(0, x- \frac{t}{\e} v %x-    \frac{x_3}{v_3} v
% , v) 
 + 
  %\mathbf{1}_{t- \tb(x,v) %\e\frac{x_3}{v_3}
%< 0} 
 %\mathbf{1}_{\O^c}(x) 
   \mathbf{1}_{v_3>0}
 %e^{ t- \tb(x,v) %\e \frac{x_3}{v_3}
 % }
\big] %e^{- \frac{\nu t}{\kappa \e^2}} 
f_R(0, x- \frac{t}{\e} v %x-    \frac{x_3}{v_3} v
 , v) \Big\}
  \end{split}
\Ee
\Be
\begin{split}
 f_2 (t,x,v ): =    \mathbf{1}_{(x_1,x_2) \notin [-\pi, \pi]^2}\Big\{&
 \end{split}
\Ee
\Be
\begin{split}
\end{split}
\Ee\unhide
%%%%%%%%%%%%
%%%%%%%%%%%%
%%%%%%%%%%%% 
%$\{x_3=0\}$ %, %$\{t=0\}$, 
%and $\{t-\e \frac{x_3}{v_3} =0\}$, and hence continuous in $(0,T) \times (\mathbb{T}^2 \times \R) \times \R^3$. 
We note that in the sense of distributions $\bar{f}_R$ solves%is a weak solution in $\R \times (\mathbb{T}^2 \times \R) \times \R^3$ to an equation 
\Be\label{eqtn:bar_f}
\begin{split}
& \e \p_t \bar{f}_R + v\cdot \nabla_x \bar{f}_R % + \frac{\nu }{\kappa \e} \bar{f}_R
%=& \  + \mathbf{1}_{\O}(x) \chi_1(t)\chi_2(v) \big[
%\mathbf{1}_{t\geq 0}  (\e \p_t   + v\cdot \nabla_x  )f_R(t,x,v) 
% + \mathbf{1}_{t<0}  (\e \p_t   + v\cdot \nabla_x  )f_R(0,x- \frac{t}{\e} v,v)
% \big]\notag
%\\
%&+ \mathbf{1}_{\O^c}(x)\chi_1(t)\chi_2(v)\big[  
% \mathbf{1}_{t- \tb(x,v)%\e \frac{x_3}{v_3}
% \geq 0}
% e^{-(t-\tb(x,v)% \e \frac{x_3}{v_3}
% )}
%(\e \p_t   + v\cdot \nabla_x  ) f_R(t-\tb(x,v)% \e\frac{x_3}{v_3}
 %, \xb(x,v)%x-  \frac{x_3}{v_3} v
% , v)
% + \mathbf{1}_{t- \tb(x,v) %\e\frac{x_3}{v_3}
 %< 0}
 %e^{ t- \tb(x,v) %\e \frac{x_3}{v_3}
 % }
%(\e \p_t   + v\cdot \nabla_x  ) f_R(0,  x- \frac{t}{\e} v %x-    \frac{x_3}{v_3} v
 %, v)
%\big] 
 %
 %\mathbf{1}_{x_3\geq 0} \e \chi_1^\prime (t) \chi_2(v) 
  %\big[\mathbf{1}_{t\geq 0}  
  %f_R(t,x,v) 
 %+\mathbf{1}_{t<0}  f_R(0,x- \frac{t}{\e} v,v)
  %\big]\notag\\
  %&+ 
 % \mathbf{1}_{x_3< 0} \e \chi_1^\prime (t) \chi_2(v)  \big[ 
 % \mathbf{1}_{t- \tb(x,v)%\e \frac{x_3}{v_3}
  %\geq 0}
  %e^{-(t-\tb(x,v)% \e \frac{x_3}{v_3}
 % )}
 % f_R(t-\tb(x,v)% \e\frac{x_3}{v_3}
 % , \xb(x,v)%x-  \frac{x_3}{v_3} v
  %, v)
 % +  \mathbf{1}_{t- \tb(x,v) %\e\frac{x_3}{v_3}
  %< 0}
  %e^{ t- \tb(x,v) %\e \frac{x_3}{v_3}
 %  }
  %f_R(0, \xb(x,v)%x-    \frac{x_3}{v_3} v
  %, v)
 % \big]
% \notag
%\\
=      \bar{g}%:=    \frac{v\cdot \nabla_x \chi_2  }{\chi_2  } \bar{f}_R + \mathbf{1}_{t \geq 0} \mathbf{1}_{\tilde{\O}} (x) \chi_2(x)  \chi_3(v) 
%[\e \p_t + v\cdot \nabla_x] f_R
 \ \ \text{in} \ (-\infty,T] \times \R^3 \times \R^3,\\
 & \bar{g}:=     \frac{v\cdot \nabla_x \chi_2  }{\chi_2  } \bar{f}_R + \mathbf{1}_{t \geq 0} \mathbf{1}_{\tilde{\O}} (x) \chi_2(x)  \chi_3(v) 
[\e \p_t + v\cdot \nabla_x] f_R\\
&  \ \ \ \ \ \ +\mathbf{1}_{t \leq 0} \{
\e\p_t \chi_1(t) f_R (0,x,v) + \chi_1 (t) v\cdot \nabla_x f_R (0,x,v) 
\}
  %,
%\\
%\bar{g}(t,x,v):=&      \mathbf{1}_{t\geq 0}   \mathbf{1}_{\O}(x)\chi_1 (x) \chi_2 (v) %\big[
%\big[\e \p_t   + v\cdot \nabla_x+ \frac{\nu }{\kappa \e}  \big]f_R(t,x,v)
% + \frac{ v\cdot \nabla_x \chi_1 (x) }{\chi_1 (x)} \bar{f}_R (t,x,v )
% :=     \mathbf{1}_{t\geq 0} \chi_1(t)  \mathbf{1}_{\O}(x) \chi_2(v) %\big[
% (\e \p_t   + v\cdot \nabla_x  )f_R(t,x,v) +  [\e {\chi_1^\prime(t)} / {\chi_1 (t)}] \bar{f}_R (t,x,v)
 %+ \mathbf{1}_{t<0}  (\e \p_t   + v\cdot \nabla_x  )f_R(0,x- \frac{t}{\e} v,v)
 %\big]
%\mathbf{1}_{t \geq 0} \mathbf{1}_{x_3\geq 0}\chi_1(t)\chi_2(v) 
% (\e \p_t   + v\cdot \nabla_x   )f_R(t,x,v)
 % 
  \end{split}
\Ee\hide
with
\begin{align}
\bar{g} 
:=& \   \mathbf{1}_{t\geq 0}   \mathbf{1}_{\O}(x) \chi_1 (x) \chi_2 (v) %\big[
 (\e \p_t   + v\cdot \nabla_x+ \frac{\nu }{\kappa \e}  )f_R(t,x,v).
% +   \e  \chi_1^\prime(t)  \chi_2 (v)  \mathbf{1}_{\O}(x) \mathbf{1}_{t\geq 0}  f_R(t,x,v)
\label{etqn_barf1} 
\end{align}\unhide
\hide
  \\
  &+ %\{
  %   \e  \chi_1^\prime(t)  
     \mathbf{1}_{\O}(x)
     \chi%_2
      (v)
       % + \frac{\nu}{\kappa \e} \chi_1(t) \}
%   \chi_2 (v)   
\big[
  \mathbf{1}_{t\geq 0} 
  f_R(t,x,v)
 + 
 \mathbf{1}_{t<0}
e^{- \frac{\nu t}{\kappa \e^2}} 
 f_R(0,x- \frac{t}{\e} v,v) \big]
  \label{etqn_barf2}
    \\ 
  &+% \{  
    \e  \chi_1^\prime(t)   %+ \frac{\nu}{\kappa \e} \chi_1(t)\}  
   \chi_2 (v)   \mathbf{1}_{t-  \e \frac{x_3}{v_3}
 \geq 0}\mathbf{1}_{\O^c}(x)   
\mathbf{1}_{v_3<0}
% e^{-(t-\tb(x,v)% \e \frac{x_3}{v_3}
% )}
 f_R(t-  \e\frac{x_3}{v_3}
 ,  x-  \frac{x_3}{v_3} v
 , v)\label{etqn_barf3}
   \\
    &+  % \{ 
       \e  \chi_1^\prime(t) %  + \frac{\nu}{\kappa \e} \chi_1(t) \}
         \chi_2 (v)   \mathbf{1}_{\O^c}(x)  \big[ \mathbf{1}_{t-  \e\frac{x_3}{v_3}
 < 0}  \mathbf{1}_{v_3<0}
 %e^{ t- \tb(x,v) %\e \frac{x_3}{v_3}
 % }
% f_R(0, x- \frac{t}{\e} v %x-    \frac{x_3}{v_3} v
% , v) 
 + 
  %\mathbf{1}_{t- \tb(x,v) %\e\frac{x_3}{v_3}
%< 0} 
 %\mathbf{1}_{\O^c}(x) 
   \mathbf{1}_{v_3>0}
 %e^{ t- \tb(x,v) %\e \frac{x_3}{v_3}
 % }
\big] f_R(0, x- \frac{t}{\e} v %x-    \frac{x_3}{v_3} v
 , v).  \label{etqn_barf4} 
 %  \\
%\\
%=& \Big(
% -\frac{1}{\e  \kappa}Lf_R+  \frac{\e}{\kappa} \Gamma({f_2}, f_R)
%+    \frac{   \delta }{ \kappa}\Gamma(f_R, f_R)
%-  \frac{( \e  \p_t + 
%  v\cdot \nabla_x) \sqrt{\mu}}{\sqrt{\mu}} f_{R}
%  + 
%\e(\mathbf{I}- \mathbf{P})\mathfrak{R}_1 +\e \mathfrak{R}_2
%\%Big)
\unhide
Here we have used the fact that $\bar{f}_R$ in (\ref{eqtn:bar_f}) is continuous along the characteristics across $\p \tilde{\O}$ and $\{t=0\}$. %in the support of $\chi_2 (x) \chi_3 (v)$. We also have used the fact that $[\e \p_t   + v\cdot \nabla_x   ]f_R(0,x- \frac{t}{\e} v,v)=0$ and $[\e \p_t   + v\cdot \nabla_x   ]f_1 (t+ \e \tilde{t}_B(x,v), x+\tilde{t}_B(x,v)v,v)=0$ in the sense of distribution.  
We derive that, using (\ref{eqtn:bar_f}),
\Be\label{f_R:duhamel}
\bar{f}_R(t,x,v) =
%e^{-   \frac{C_\nu }{\kappa \e^2 }   t} 
%|f_R (0,x- \frac{t}{\e} v , v)|
%+
 \frac{1}{\e} \int^t_{-\infty} 
%e^{-  \frac{C_\nu }{\kappa \e^2 } (t-s) }
 \bar{g} (s, x- \frac{t-s}{\e} v, v)  \dd s \ \ \text{for} \ (t,x,v) \in 
(-\infty,T] \times \R^3%(\mathbb{T}^2 \times \R) 
 \times \R^3.
\Ee

%%%%%%%%%%
%%%%%%%%%%
%%%%%%%%%%
\hide
$\bar{f}_R$ is continuous across $\{x_3=0\}$, $\{t=0\}$, and $\{t-\tb(x,v)=t- \e  \frac{x_3}{v_3}=0\}$, and the fact of $(\e \p_t   + v\cdot \nabla_x   )f_R(t-\e \frac{x_3}{v_3}, x- \frac{x_3}{v_3} v, v)=0$,  %(\e \p_t   + v\cdot \nabla_x   ) f_R(0, x-    \frac{x_3}{v_3} v, v)=0$, and
$(\e \p_t   + v\cdot \nabla_x  )f_R(0,x- \frac{t}{\e} v,v)=0$, and $(\e \p_t + v\cdot \nabla_x) (t-\e \frac{x_3}{v_3})=0$ in the sense of distributions.
\unhide
%%%%%%%%%%
%%%%%%%%%%
%%%%%%%%%%

Recall ${\varphi}_i \in \{{\varphi}_0, \cdots {\varphi}_4\}$ in (\ref{basis}). From (\ref{ext:f+}) we note that   
\Be\begin{split}
 &
\left\|\int_{\R^3} \bar{f}_R(t,x,v) {\varphi}_i (v) \sqrt{\mu(v)}    \dd v \right\|_{L^2_t ((0,T); L^p_x ( \tilde{\O}))} \\ %\notag\\
&= %& \ 
 \left\|\int_{\R^3}
\chi_2 (x) \chi_3 (v)
 f_R(t,x,v)  \tilde{\varphi}_i (v) \sqrt{\mu(v)}    \dd v   \right\|_{L^2_t ((0,T); L^p_x (\tilde{\O}))}\label{decomp1:L2L3}
 \end{split}\Ee
 %%%%%%%%%%%
 %%%%%%%%%%%
 %%%%%%%%%%%
 \hide
 \begin{align}
 \\
&+
 \left\|\int_{\R^3} 
\chi_2 (x) \chi_3 (v)  \mathbf{1}_{t-\e \frac{x_3}{v_3}\geq 0}
 \mathbf{1}_{v_3<0}
 e^{- \frac{\nu }{\kappa \e^2 } \frac{ \e x_3}{v_3}}
 f_R(t-\e\frac{x_3}{v_3}%\tb(x,v) 
 ,  x-  \frac{x_3}{v_3} v
 , v)
%  \chi_1(t) 
%  \chi_2(v) 
%  \mathbf{1}_{t- \tb(x,v)%\e \frac{x_3}{v_3}
 % \geq 0}
% e^{-(t-\tb(x,v)% \e \frac{x_3}{v_3}
% )}
%  f_R(t-\tb(x,v)% \e\frac{x_3}{v_3}
%  , \xb(x,v)%x-  \frac{x_3}{v_3} v
%  , v)  
 \tilde{\varphi}_i (v) \sqrt{\mu_0(v)}    \dd v
 \right\|_{L^2_t ((0,T); L^p_x (\O^c))}\label{decomp2:L2L3}\\  
 &+
 \left\|\int_{\R^3}  
\chi_2 (x) \chi_3 (v) \big[ \mathbf{1}_{t-  \e\frac{x_3}{v_3}
 < 0}  \mathbf{1}_{v_3<0}
 %e^{ t-  \e \frac{x_3}{v_3}
 % }
% f_R(0, x- \frac{t}{\e} v %x-    \frac{x_3}{v_3} v
% , v) 
 + 
  %\mathbf{1}_{t- \tb(x,v) %\e\frac{x_3}{v_3}
%< 0} 
 %\mathbf{1}_{\O^c}(x) 
   \mathbf{1}_{v_3>0}
 %e^{ t- \tb(x,v) %\e \frac{x_3}{v_3}
 % }
\big] e^{- \frac{\nu t}{\kappa \e^2}} f_R(0, x- \frac{t}{\e} v %x-    \frac{x_3}{v_3} v
 , v)  
  \tilde{\varphi}_i (v) \sqrt{\mu(v)}    \dd v
 \right\|_{L^2_t ((0,T); L^p_x (\O^c))}\label{decomp3:L2L3}.%\\ 
 %&+
% \left\|\int_{\R^3}  \chi_1(t) 
% \chi_2(v)  
 %\mathbf{1}_{t-\tb(x,v)<0}
% e^{-(t-\tb(x,v)% \e \frac{x_3}{v_3}
% )}
% f_R(0% \e\frac{x_3}{v_3}
% , \xb(x,v)
 %x-  \frac{x_3}{v_3} v
% , v)  \tilde{\varphi}_i (v) \sqrt{\mu_0(v)}    \dd v
% \right\|_{L^2_t(\R ) L^3_x (\O^c)}. \label{decomp4:L2L3}
\end{align} 
\unhide
%%%%%%%%%%%%%%
%%%%%%%%%%%%%%
%%%%%%%%%%%%%%
From (\ref{P}), we decompose 
%$\Big\{ a + b \cdot v + c \frac{|v|^2-3}{\sqrt{6}} \Big\} \sqrt{\mu_0}$
\Be \label{est:decomp1:L2L3}
\begin{split}
(\ref{decomp1:L2L3})\geq &   \  \Big\| 
\sum_j     \chi_2 (x)  {P}_jf_R(t,x)
\int_{\R^3}  \chi_3 (v)    {\varphi}_j (v)    {\varphi} _i (v)  \mu  (v)    \dd v \Big\|_{L^2_t ((0,T); L^p_x (\tilde{\O}))}
\\& -  \left\|\int_{\R^3}
 \chi_3 (v) 
(\mathbf{I} - \mathbf{{P}}) f_R(t,x,v)  {\varphi}_i (v) \sqrt{\mu(v)}    \dd v   \right\|_{L^2_t ((0,T); L^p_x (\tilde{\O}))}
\\ \ge &   \big\|    {P} f_R
 \big\|_{L^2_t((0,T);L^p_x(\tilde{\O}))} -C_{T,N}  \| (\mathbf{I} - \mathbf{P}) f_R \|_{L^p((0, T) \times\tilde{\O} \times \R^3)}
\\ \ge & \big\|    {P} f_R
 \big\|_{L^2_t((0,T);L^p_x(\tilde{\O}))} -C_{T,N}   \| \mathfrak{w}_{\varrho, \ss } f_R (t) \|_{ L^\infty((0,T) \times \tilde{\O} \times \R^3)}^{\frac{p-2}{p}} \| (\mathbf{I} - \mathbf{P}) f_R \|_{L^2((0, T) \times \tilde{\O} \times \R^3)}^{\frac{2}{p}} ,
\end{split}\Ee
%We consider the right hand side of above terms. From (\ref{diff:P-tP}), $\int \tilde{\varphi}_i\tilde{\varphi}_j \mu_0 = \delta_{ij}$, and (\ref{chi_v}), the first term can be bounded below by $\big(1-O( \e) \|u\|_\infty- O(\frac{1}{N})\big)
% \big\|   \chi_2   {P} f_R
% \big\|_{L^2_t ((0,T); L^p_x (\tilde{\O}))}$. 
 where for the second inequality we use $L^2_t (0,T)  \subset L^p_t( 0,T )$, and $L^1(\{|v|\leq N\}) \subset L^p(\{|v|\leq N\})$.
% $C_{T,N} \| (\mathbf{I} - \mathbf{P}) f_R \|_{L^p((0, T) \times \tilde{\O} \times \R^3)}+ \big( O( \e) \|u\|_\infty+ O(\frac{1}{N})\big)
% \big\|   {P} f_R
% \big\|_{L^2_t( (0,T)  ;L^p_x(\tilde{\O}))}$. Hence we derive 
%\Be
%\begin{split} 
%&(\ref{decomp1:L2L3}) \\
%\geq  & \  \big(1- O(\e )\|u\|_\infty- O(\frac{1}{N})\big)
% \big\|   {P} f_R
% \big\|_{L^2_t( (0,T)  ;L^p_x(\tilde{\O}))} - C_{T,N} \| (\mathbf{I} - \mathbf{P}) f_R \|_{L^p((0, T) \times\tilde{\O} \times \R^3)}
%  \\
% \geq  &  \ \big(1- O(\e) \|u\|_\infty- O(\frac{1}{N})\big)  \big\|    {P} f_R
% \big\|_{L^2_t((0,T);L^p_x(\tilde{\O}))}\\
% &
% -C_{T,N}   \| \mathfrak{w}_{\varrho, \ss } f_R (t) \|_{ L^\infty((0,T) \times \tilde{\O} \times \R^3)}^{\frac{p-2}{p}} \| (\mathbf{I} - \mathbf{P}) f_R \|_{L^2((0, T) \times \tilde{\O} \times \R^3)}^{\frac{2}{p}}
%.\label{est:decomp1:L2L3}
%\end{split}
%\Ee 

\

\textbf{Step 2: Average lemma}. Recall ${\varphi}_i \in \{{\varphi}_0, \cdots {\varphi}_4\}$ in (\ref{basis}). We choose ${\varphi}(v)$ such that 
\Be\label{tilde_varphi}
\begin{split}
 \chi_3 (v) 
 |{\varphi}_i(v)|  \sqrt{\mu_0 (v)}  \leq {\varphi}(v) , \ \ {\varphi}(v) \in  C^\infty_c (\R^3) \\
 \ \ \text{and} \ \ {\varphi}(v)\equiv 0 \ \  \text{for} \ \ |v|\geq N \ \ \text{or} \ \ |v_i| \leq 1/N \ \text{for any } i=1,2,3. 
\end{split}\Ee
%%%%%%%%%%%%%%%%%%%%%
%%%%%%%%%%%%%%%%%%%%%
%%%%%%%%%%%%%%%%%%%%%
\hide
 $
 \bar{f}_R(t,x,v)=
%e^{-   \frac{C_\nu }{\kappa \e^2 }   t} 
%|f_R (0,x- \frac{t}{\e} v , v)|
%+
 \frac{1}{\e} \int^t_{-\infty} 
%e^{-  \frac{C_\nu }{\kappa \e^2 } (t-s) }
 \bar{g} (s, x- \frac{t-s}{\e} v, v)  \dd s$ for $(t,x,v) \in 
 [0,T] \times \tilde{\O}%(\mathbb{T}^2 \times \R) 
 \times \R^3.$ If $|v_i| \leq 1/N$ for $i=1$ or $i=2$ then $\bar{g} (s, x- \frac{t-s}{\e} v, v)=0$ from (\ref{chi_v}) and (\ref{eqtn:bar_f}). Now we consider the case of $|v_i| > 1/N$ for $i=1,2$. For $(t,x) \in [0,T] \times \tilde{\O}$ then $x_i- \frac{t-s}{\e} v_i\notin [-2\pi, 4\pi]$ if $s<t- 10N \e $. On the other hand from (\ref{chi1}) and (\ref{eqtn:bar_f}) we derive that $\bar{g} (s, x- \frac{t-s}{\e} v, v)=0$ if $x_i- \frac{t-s}{\e} v_i \notin [-2\pi, 4\pi]$ for either $i=1$ or $i=2$. Therefore 
\Be\label{f_R:duhamel}
 \bar{f}_R(t,x,v) =
%e^{-   \frac{C_\nu }{\kappa \e^2 }   t} 
%|f_R (0,x- \frac{t}{\e} v , v)|
%+
 \frac{1}{\e} \int^t_{ t-10N \e } 
%e^{-  \frac{C_\nu }{\kappa \e^2 } (t-s) }
  \bar{g} (s, x- \frac{t-s}{\e} v, v)   \dd s \  \  \text{for} \ 
 (t,x,v) \in 
(-\infty,T] \times \tilde{\O}%(\mathbb{T}^2 \times \R) 
 \times \R^3.
\Ee
\unhide
%%%%%%%%%%%%%%%%%%%%%
%%%%%%%%%%%%%%%%%%%%%
%%%%%%%%%%%%%%%%%%%%%

\begin{lemma}\label{lemma:average}
%Let $0<C_0$, $1\ll N$, and $0 \leq T <\infty$.  Suppose $\tilde{\varphi} \in C^\infty_c (\R^3)$ ($\tilde{\varphi}(v) \equiv 0$ for $|v|\geq N$) and $\tilde{\varphi}= \tilde{\varphi}(v)$ is non-increasing in $|v|$. 
We define 
\Be\label{def:S}
S(\bar{g})(t,x):=
\frac{1}{\e} \int^t_{-\infty%- 10N \e 
}\int_{\R^3}
%\mathbf{1}_{t-s< \frac{\e}{N^2}}
% e^{- \frac{C_0(t-s)}{\kappa \e^2 } } 
| \bar{g}(s, x- \frac{t-s}{\e} v, v)  | {\varphi} (v) \dd v \dd s \ \  \text{for} \ (t,x) \in(-\infty , T ] \times 
\R^3. 
\Ee
Then, for $p<3$ and $1\ll N$,
\Be\label{bound:S}
\| S(\bar{g}) \|_{L^2_t((0,T);  L^p_x(\mathbb{T}^2 \times \R))} \lesssim_{N} %(\kappa\e)^{\frac{3}{p}-\frac{1}{2}}
 \| \mathbf{1}_{(t,x,v) \in \mathfrak{D}_T}   \bar{g}\|_{L^2  ((0,T)\times (\mathbb{T}^2 \times \R) \times \{|v| \leq N \})},%+  
%\frac{e^{- \frac{1}{\kappa \e}}}{\e} \| g \|_{L^2_t((-\e,T);  L^p_x(\mathbb{T}^2 \times \R))}. 
% \ \ \text{for} \ 2 \leq p<3.
\Ee
 where the bound (\ref{bound:S}) only depends on $N$ but can be independent on $p<3$. 
\end{lemma}
We remark that from (\ref{f_R:duhamel}) and (\ref{def:S}) $ \int_{\R^3}\bar{f}_R(t,x,v) {\varphi}_i (v) \dd v \leq S(\bar{g})(t,x)$. For the proof we refer to Lemma 6 in \cite{JK}. 
\hide
Here we have used $\int^T_0 e^{- \frac{C_0 t}{\kappa \e^2 } }\dd t \leq \kappa \e^2 \{1-  e^{- \frac{C_0 T}{\kappa \e^2 } }\} \leq \kappa \e^2$ and 
\Be\begin{split}
\int_0^{T  } s^{\frac{6}{p}-3} e^{- \frac{C_0 s}{\kappa \e^2 } }\dd s
&= \int_0^{T } \frac{1}{\frac{6}{p}-2}\frac{d}{ds} \big(s^{\frac{6}{p}-2}\big) e^{- \frac{C_0 s}{\kappa \e^2 } }\dd s\\
&=\frac{1}{\frac{6}{p}-2}  T ^{\frac{6}{p}-2} e^{- \frac{C_0}{\kappa\e^2}  T }
+ \frac{1}{\frac{6}{p}-2} \int^{T }_0 (\kappa \e^2)^{\frac{6}{p}-2}
\left(\frac{s}{\kappa \e^2}\right)^{\frac{6}{p}-2} \frac{C_0}{\kappa \e^2} e^{- \frac{C_0 s}{\kappa \e^2 } }\dd s\\
& \lesssim  (\kappa \e^2)^{\frac{6}{p}-2} \ \ \text{for} \ \ p<3. \notag
\end{split}\Ee\unhide

%\

%\textbf{Step 3: Applying Lemma \ref{lemma:average}.}  %\smallskip 
%\textit{Step 2.  } 
Now we apply Lemma \ref{lemma:average} to (\ref{f_R:duhamel}) and derive that 
\begin{align}
& \left\| \int_{\R^3}\bar{f}_R (t,x,v) {\varphi} (v)  \dd v \right\|_{L^2_t( (-1,T]; L^p_x (\tilde{\O}))}\notag\\ 
\lesssim& \  \| \mathbf{1}_{(t,x,v) \in \mathfrak{D}_T}  \bar{g} \|_{L^2  ((-1,T]\times U \times V)}\notag \\
\lesssim & \   \| %v\cdot \nabla_x \chi_2 (x) 
 f_R (t,x,v)\|_{L^2  ((0,T]\times \tilde{\O} \times V)}  + \| %v\cdot \nabla_x \chi_2 (x)  
 f_R (0,x,v) \|_{L^2  (  \tilde{\O} \times V)} + \| \nabla_x f_R (0,x,v) \|_{L^2  (  \tilde{\O} \times V)}
 \notag%\label{g1}
 \\
&+ \| %v\cdot \nabla_x \chi_2 (x)  
 \mathbf{1}_{(t,x,v) \in \mathfrak{D}_T}
 f_I (t+ \e \tilde{t}_B(x,v),  \tilde{x}_B(x,v) ,v)
\|_{L^2  ((-1,T]\times (U \backslash  \tilde{\O})\times V)}
\label{g2}\\
&+  \| [\e\p_t + v\cdot \nabla_x ] f_R
\|_{L^2  ((0,T]\times \tilde{\O} \times V)},\label{g3}
\end{align}
where we have used (\ref{ext:f+}), (\ref{ext:f1}), (\ref{ext:f2}), and the fact that $|v\cdot \nabla_x \chi_2 (x)| \lesssim_N 1$ on $v \in V$. 

First we consider (\ref{g2}). We split the cases of (\ref{g2}) according to (\ref{tB}). For $x \in \p\tilde{\O}$, which has a zero measure in $L^2  ((-1,T]\times (U \backslash  \tilde{\O})\times V)$, we have $\tilde{t}_B(x,v)=0$ from the first line of (\ref{tB}). If $\tilde{B}(x,v) = \emptyset$ and $x \notin \p\tilde{\O}$ then $\tilde{t}_B(x,v)=-\infty$ from the last line of (\ref{tB}) and hence $\bar{f}_R(-\infty)=0$ since $\chi_1(-\infty)=0$ in (\ref{ext:f1}). Therefore we derive that 
\begin{align}
(\ref{g2})\leq
%&   \  \| %v\cdot \nabla_x \chi_2 (x)  
%\mathbf{1}_{x \in \p\tilde{\O}}
% \mathbf{1}_{(t,x,v) \in \mathfrak{D}_T} 
% f_1 (t%+ \e \tilde{t}_B(x,v)
% , x%+\tilde{t}_B(x,v)v
% ,v)
%\|_{L^2  ((-1,T]\times (U \backslash  \tilde{\O})\times V)}
%\label{g2:3}\\
&%$%T+
 \   \| %v\cdot \nabla_x \chi_2 (x)  
\mathbf{1}_{\{s<0\} \subset \tilde{B} (x,v)}
 \mathbf{1}_{(t,x,v) \in \mathfrak{D}_T} 
 f_I (t+ \e \tilde{t}_B(x,v), \tilde{x}_B(x,v)
 ,v)
\|_{L^2  ((-1,T]\times (U \backslash  \tilde{\O})\times V)}
\label{g2:1}
\\
&+  \| %v\cdot \nabla_x \chi_2 (x)  
\mathbf{1}_{\{s>0\} \subset \tilde{B} (x,v)}
 \mathbf{1}_{(t,x,v) \in \mathfrak{D}_T} 
 f_I (t+ \e \tilde{t}_B(x,v), \tilde{x}_B(x,v),v)
\|_{L^2  ((-1,T]\times (U \backslash  \tilde{\O})\times V)}.
\label{g2:2} 
%\\
% &+  \| %v\cdot \nabla_x \chi_2 (x)  
%\mathbf{1}_{ \tilde{B}(x,v) = \emptyset} \mathbf{1}_{x \notin \p\tilde{\O}}
% \mathbf{1}_{(t,x,v) \in \mathfrak{D}_T} 
%0
%\|_{L^2  ((-1,T]\times (U \backslash  \tilde{\O})\times V)}\label{g2:4},
\end{align}
We need a special attention to (\ref{g2:1}). Since $(t,x,v) \in \mathfrak{D}_T$ we know that 
%\Be\label{inf:tau}
$\inf \{\tau \geq t: x+  \frac{\tau - t}{\e} v \in cl(\tilde{\O})\} \leq T.$ 
%\Ee 
If $\{s<0\} \subset \tilde{B} (x,v)$ then, from the third line of (\ref{tB}), $\tilde{t}_B(x,v) = \sup \tilde{B} (x,v)= \sup \{s \in \R: x+ sv \in \R^3 \backslash cl (\tilde{\O})\} \leq (T-t)/\e$. Therefore the argument of $f_I$ in (\ref{g2:1}) is confined as
\Be\label{t+tB<}
(t+ \e \tilde{t}_B(x,v), \tilde{x}_B(x,v),v) \in (- \infty,T] \times \p\tilde{\O} \times V .
\Ee
For (\ref{g2:2}), from the second line of (\ref{tB}), $\tilde{t}_B (x,v)= \inf \tilde{B}(x,v) = \inf 
\{s \in \R: x+ sv \in \R^3 \backslash cl (\tilde{\O})\} \leq 0$. Therefore $t+ \e \tilde{t}_B (x,v) \leq t \leq T$ and hence the argument of $f_I$ in (\ref{g2:2}) is confined as in (\ref{t+tB<}). Now we apply the Minkowski's inequality in time, change of variables $t+ \e \tilde{t}_B (x,v) \mapsto t$, and use (\ref{t+tB<}) to derive that 
\Be\label{est1:g2}
(\ref{g2:1}) + (\ref{g2:2}) \lesssim
\Big\|
 \| f_I(t, \tilde{x}_B(x,v),v)\|_{L^2_t ((-1, T])}
 \Big\|_{L^2_{x,v}((U \backslash \tilde{\O}) \times V)}.
\Ee

Let us define an outward normal $\tilde{n}(x)$ on $\p\tilde{\O}$. More precisely 
\Be\label{tilde_n}
\tilde{n}(x)= \begin{cases}
(0,0,-1) &\ \ \text{if} \ x_3=0 \ \text{and} \ x \in \p\tilde{\O},\\
((-1)^{\frac{x_1}{2\pi}+1},0,0) &\ \ \text{if} \ x_1 \in \{0, 2\pi\} \ \text{and} \ x \in \p\tilde{\O},\\
(0,(-1)^{\frac{x_2}{2\pi}+1},0) &\ \ \text{if} \  x_2 \in \{0, 2\pi\} \ \text{and} \ x \in \p\tilde{\O}.
\end{cases}
\Ee
From (\ref{def:UV}) we have therefore $(x ,v) \in (U \backslash \tilde{\O}) \times V$ then $|\tilde{n}( \tilde{x}_B (x,v) ) \cdot v|\geq 1/N$. We consider maps 
\Be\begin{split}\label{bdry_map}
(x_1,x_3) &\mapsto  \tilde{x}_B (x,v) %=x+  \tilde{t}_B (x,v)v 
\in  (0,2\pi) \times (0, 2\pi) \times \{x_3=0\},\\
& \ \ \text{with} \ \ \Big|\det\Big(\frac{\p (  \tilde{x}_{B,1} (x,v) ,     \tilde{x}_{B,2} (x,v) )}{\p (x_1,x_3)}\Big)\Big|= \Big|\frac{ v_2}{v \cdot \tilde{n}}\Big|,\\
(x_i,x_3) &\mapsto  (\tilde{x}_{B,i} (x,v), \tilde{x}_{B,3} (x,v))  %=x+  \tilde{t}_B (x,v)v 
\in  (0,2\pi)   \times (0,\infty), \\
&\ \ \text{with} \ \ \Big|\det\Big(\frac{\p (  \tilde{x}_{B,i} (x,v) ,     \tilde{x}_{B,3} (x,v) )}{\p (x_1,x_3)}\Big)\Big|= \Big|\frac{ v_i }{v \cdot \tilde{n}}\Big|,  \ \ \text{for} \ i =1,2. 
\end{split}\Ee
 Note that if $v \in V$ of (\ref{def:UV}) then $|v_i| \geq 1/N$ for all $i=1,2,3.$ We define 
\Be\label{def:tilde_gamma}
\tilde{\gamma} := \p\tilde{\O} \times \R^3, \ \ \ \tilde{\gamma}^{N} := \p\tilde{\O} \times (\R^3\backslash V). 
\Ee
 %and $(x_1,x_3) \mapsto x+  \tilde{t}_B (x,v)v \in  [0,2\pi] \times [0, 2\pi] \times \{x_3=0\}$
We apply the change of variables (\ref{bdry_map}) to (\ref{est1:g2}):
\Be\label{est2:g2}
\begin{split}
(\ref{est1:g2})&= 
\bigg\| 
\bigg[
\int_{-2\pi}^{4\pi}
\int_{-2\pi}^\infty \int_{-2\pi}^{4\pi}
 \| f_I (t, \tilde{x}_B(x,v),v)\|_{L^2_t ((-1, T])}^2
 \dd x_1 \dd x_3 \dd x_2
\bigg]^{1/2} 
\bigg\|_{L^2_{v}(V)}\\
&\leq   \bigg\|  \bigg[ 5 \times 6 \pi N   \int_{\p\tilde{\O}}\int_{-1}^T  |f_I (t,  y,v) |^2 |v \cdot \tilde{n}(y)| \dd t  \dd y\bigg]^{1/2} \bigg\|_{L^2_{v}(V)}\\
&\lesssim \|f_R\|_{L^2((0,T) \times \tilde{\gamma} \backslash \tilde{\gamma}^{  N   } )}
+ \|f_R(0)\|_{L^2(  \tilde{\gamma} \backslash \tilde{\gamma}^{   N   } )}.
\end{split}
\Ee
We recall the trace theorem:
\Be\label{trace10}
\begin{split}
\int^T_{0} \int_{\tilde{\gamma}\backslash \tilde{\gamma}^{N}} |h| \dd \gamma \dd s 
\lesssim& \sup_{t \in [0,T]} \| h(t) \|_{L^1(\tilde{\O} \times V)}  + \int^T_{0} \| h(s) \|_{L^1(\tilde{\O} \times V)} \dd s  
+ \int^T_{0} \| [\e \p_t + v\cdot \nabla_x ] h \|_{L^1(\tilde{\O} \times V)} \dd s.
\end{split}
\Ee
%We need a following lemma.
We apply (\ref{trace10}) with $h=f^2$ and derive an estimate 
\Be\begin{split}
\label{est:bdry}
 &\|f_R\|_{L^2((0,T) \times \tilde{\gamma} \backslash \tilde{\gamma}^{  N   } )}^2 \\
 &\lesssim \sup_{t \in [0,T]} \| f_R(t) \|_{L^2(\tilde{\O} \times V)}^2  + \int^T_{0} \| f_R(s) \|_{L^2(\tilde{\O} \times V)}^2 \dd s 
+ \int^T_{0}  \iint_{\tilde{\O}\times V}\big| f_R[\e \p_t + v\cdot \nabla_x ] f_R \big| \dd x \dd v  \dd s\\
&\lesssim _T 
 \| f_R \|_{L^\infty ([0,T];L^2( {\O} \times \R^3))}^2
 + \big\|  [\e \p_t+ v\cdot \nabla_x ]   f_R  \big\|_{L^2( [0,T] \times  {\O} \times \R^3)}.
\end{split}\Ee
%%%%%%%%%%
%%%%%%%%%%
\hide
where the last term can be bounded %similarly as (\ref{est:Energy}) 
from (\ref{eqtn_fR}) as 
\Be
\begin{split}
\int^T_{0}  \iint_{\O \times \R^3}
\Big|&
- \frac{1}{\e \kappa} (\mathbf{I} - \mathbf{P}) f_R  L(\mathbf{I} - \mathbf{P}) f_R 
+ \frac{\e}{\kappa} \Gamma(f_2, f_R) (\mathbf{I} - \mathbf{P}) f_R
+ \frac{\delta}{\kappa} \Gamma(f_R, f_R) (\mathbf{I} - \mathbf{P}) f_R
- \frac{(\e \p_t + v\cdot \nabla_x ) \sqrt{\mu}}{\sqrt{\mu}} |f_R|^2\\
&+ \e (\mathbf{I} - \mathbf{P}) \mathfrak{R}_1  (\mathbf{I} - \mathbf{P}) f_R
+ \e   \mathfrak{R}_2  f_R
\Big|
\dd x \dd v  \dd t
\end{split}
\Ee\unhide
%%%%%%%%%%
%%%%%%%%%%
Finally we conclude a bound of (\ref{g2}) as below via (\ref{g2:1}), (\ref{g2:2}), (\ref{est1:g2}), (\ref{est2:g2}), and (\ref{est:bdry}) 
\Be\label{est3:g2}
(\ref{g2})\lesssim \| f_R (0)\|_{L^2_\gamma} + \| f_R \|_{L^\infty ([0,T];L^2( {\O} \times \R^3))} + \underbrace{\big\|  [\e \p_t+ v\cdot \nabla_x ]   f_R  \big\|_{L^2( [0,T] \times  {\O} \times \R^3)}}_{(\ref{est3:g2})_*}. 
\Ee

Next we estimate (\ref{g3}) (and $(\ref{est3:g2})_*$). %Applying Lemma \ref{lemma:average} and 
Using (\ref{eqtn:bar_f}) and (\ref{eqtn_fR}) we conclude that 
\Be\notag
\begin{split}
&(\ref{g3})+ (\ref{est3:g2})_*% &\lesssim  (\kappa \e)^{\frac{6}{p}-1}\big\|      \mathbf{1}_{t\geq 0}   \mathbf{1}_{\O}(x)  
% (\e \p_t   + v\cdot \nabla_x+ \frac{\nu }{\kappa \e}  )f_R(t,x,v)\big\|_{L^2 ((0,T)\times (\mathbb{T}^2 \times \R) \times
% \{|v| \leq N \})}\\
 %&
 \\
& \lesssim  %(\kappa \e)^{\frac{6}{p}-1}
  \bigg\|  % \frac{1}{\e \kappa} Kf_R +
 - \frac{1}{\e \kappa} Lf_R  +  \frac{2}{  \kappa }  \Gamma( f_1 + \e f_2, f_R)  
+    \frac{   \e^{1/2} }{ \kappa}\Gamma(f_R, f_R)
 \\& \ \ \ \ \ \  -  \e^{1/2}  \left \{ (\mathbf{I-P}) (v \cdot \nabla_x f_2) - \frac{2}{\kappa} \Gamma(f_1,f_2) \right \} + \e^{3/2}   \left( -  \p_t f_2  + \frac{1}{\kappa} \Gamma(f_2,f_2) \right)
    \bigg\|_{L^2 ((0,T]\times \O \times
V)}.
 \end{split}
\Ee
Following the arguments of (\ref{est:EG})-(\ref{est:Energy_v^3}), we derive that 
\Be \label{est:bar_f3_1}
\begin{split}
&(\ref{g3})+(\ref{est3:g2})_*
\\ & \lesssim   \  \Big\{
 \frac{\e}{\kappa} \|  (\ref{est:f2})  \|_{L^\infty_t ((0,T); L_x^{\frac{2p}{p-2}}(\O)) } + \frac{\e^{1/2}}{\kappa} \| Pf_R \|_{L^\infty_t ((0,T); L_x^{\frac{2p}{p-2}}(\O)) } \Big\}  \| P f_R \|_{L^2_t ((0,T); L_x^{p}(\O))}
 \\ &  + \frac{1}{\kappa} \| Pf_1 \|_{L^\infty_{t,x}}  \| Pf_R \|_{L^2_{t,x}}
 \\ &+\Big\{ \frac{1}{\e \kappa }  + \frac{1}{\kappa} \left( \| Pf_1 \|_{L^\infty_{t,x}} + \e \|  (\ref{est:f2})  \|_{L^\infty_{t,x} } \right)   + \frac{\e^{1/2}}{\kappa} \| \mathfrak{w}_{\varrho} f_R \|_{L^\infty_t ((0,T) \times \O \times \R^3) } \Big\} \| (\mathbf{I} - \mathbf{P})f_R \|_{L^2_t ((0,T) \times \O \times \R^3) }
 \\ & + \e^{1/2}  \| (\ref{estR11})  \|_{L^2_{t,x}}  +\frac{\e^{1/2}}{\kappa} \| \kappa \eqref{estR12} \|_{L^2_{t,x}} + \e^{3/2} \| (\ref{estR21})\|_{L^2_{t,x,v}}  + \frac{\e^{3/2}}{\kappa}  \| \kappa (\ref{estR22})\|_{L^2_{t,x,v}}
% \\ &+\e \big\|  (\ref{transp:mu}) \big\|_{L^2_t ((0,T); L^\infty_x(\O))} \|  f_R(t) \|_{L^\infty_t ((0,T);L^2(\O \times \R^3))} 
% \\ &+  \e  \{  \|(\ref{est:R1})\| _{L^2_t ((0,T); L_x^{2}(\O))}
% +  \|(\ref{est:R2})\| _{L^2_t ((0,T); L_x^{2}(\O))}  \} 
 ,
\end{split}
\Ee where we further bound 
\Be\label{est:bar_f3_1:1}
\| Pf_R \|_{  L_x^{\frac{2p}{p-2}}(\O)  }
\leq \| Pf_R \|_{L_x^{6}(\O)}^{\frac{3(p-2)}{p}} 
\|  \mathfrak{w}_{\varrho} f_R \|_{L_x^{\infty}(\O)}^{\frac{6-2p}{p}}.
\Ee
%%%%%%%%%%%%%%%%
\hide
and conclude that 
\Be
\begin{split}(\ref{g3})%(\ref{bar_f3})_1
 \lesssim& \ 
%\e ^{\frac{6}{p}-1}\kappa^{\frac{6}{p}-2}
\Big\{
  \frac{\e}{\kappa}%^{\frac{6}{p}} 
 \| \mathfrak{w}_{\varrho} f_2 \|_{L^\infty_t ((0,T); L_x^{\frac{2p}{p-2}}(\O)) } + \frac{ \delta}{\kappa}  \| Pf_R \|_{L^\infty_t ((0,T);L_x^{6}(\O))}^{\frac{3(p-2)}{p}} 
\|  \mathfrak{w}_{\varrho} f_R \|_{L_ {\infty}((0,T) \times\O\times\R^3)}^{\frac{6-2p}{p}}
\Big\}
\| P f_R \|_{L^2_t ((0,T); L_x^{p}(\O))}\\
  &+ %(\kappa \e)^{\frac{6}{p}-2} \Big\{1+
  %\kappa\e^2
\e \big\| |\nabla_x u| + \e |\p_t u| + \e |u| |\nabla_x u| \big\|_{L^2_t ((0,T); L^\infty_x(\O))}
%\Big\}
 \|  f_R(t) \|_{L^\infty_t ((0,T);L^2(\O \times \R^3))}\\
 &+ \frac{\delta}{\kappa}% \delta \kappa ^{\frac{6}{p}-2} \e^{\frac{6}{p}-1}
  \| \mathfrak{w}_{\varrho} f_R \|_{L^\infty_t ((0,T) \times \O \times \R^3) } 
  \| (\mathbf{I} - \mathbf{P})f_R \|_{L^2_t ((0,T) \times \O \times \R^3) }\\
 &+
 \frac{\e}{\delta}
 %\delta^{-1} \kappa^{\frac{6}{p}-1} \e^{\frac{6}{p}}
 \|\mathfrak{q}(|\nabla_x \tilde{u}|, |\nabla_x^2 u|)\| _{L^2_t ((0,T); L_x^{2}(\O))}
 \\
 &  + \frac{\e^2}{\delta \kappa} %\delta^{-1}  \e^{\frac{6}{p}+1} \kappa^{\frac{6}{p}-2}   
 \| \mathfrak{q}(|u|, |\nabla_{x } u |, |\p_{t}u|, |\nabla_{x}^{2} u|, |\nabla_{x} \p_{t} u|, |p|, |\nabla_{x } p |, |\p_{t} p|, |\tilde{u}|, |\nabla_{x } \tilde{u}|, |\p_{t} \tilde{u}|) \|_{L^2_t ((0,T); L_x^{2}(\O))}.
\end{split}
\Ee
\unhide
%%%%%%%%%%%%%%%%%%

\ 

\textbf{Step 3. Proof of (\ref{average_3D}). } First we use (\ref{est:decomp1:L2L3}) and then (\ref{g2}) and (\ref{g3}). We bound (\ref{g2}) via (\ref{est1:g2}) and (\ref{est2:g2}), which are bounded by (\ref{est:bdry}) and (\ref{est:bar_f3_1}) respectively. These conclude that, for $p<3$, 
\Be\begin{split}\label{final_est:3D}
&  \big\|    {P} f_R
 \big\|_{L^2_t((0,T);L^p_x(\tilde{\O}))}  -C_{T,N}   \| \mathfrak{w}_{\varrho, \ss } f_R (t) \|_{ L^\infty((0,T) \times \tilde{\O} \times \R^3)}^{\frac{p-2}{p}} \| (\mathbf{I} - \mathbf{P}) f_R \|_{L^2((0, T) \times \tilde{\O} \times \R^3)}^{\frac{2}{p}}\\
\leq & \  \left\|\int_{\R^3} \bar{f}_R(t,x,v) {\varphi}_i (v) \sqrt{\mu_0(v)}    \dd v \right\|_{L^2_t ((0,T); L^p_x ( \tilde{\O}))}\\
\leq & \  \left\|\int_{\R^3} \bar{f}_R(t,x,v) {\varphi}  (v)     \dd v \right\|_{L^2_t ((0,T); L^p_x ( \tilde{\O}))}\\
\lesssim & \ 
 \| f_R \|_{L^\infty ([0,T];L^2( {\O} \times \R^3))}
+ \| f_R (0) \|_{L^2_\gamma}+ \text{r.h.s. of }  (\ref{est:bar_f3_1}) \text{ with } (\ref{est:bar_f3_1:1}). 
\end{split}\Ee
Then we move a contribution of $\| P f_R \|_{L^2_t ((0,T); L_x^{p}(\O))}$ to the l.h.s and use (\ref{est:bar_f3_1:1}). This concludes (\ref{average_3D}).

\

\textbf{Step 4: Sketch of proof for (\ref{average_3Dt}). } We follow the same argument for (\ref{average_3D}). Thereby we only pin point the difference of the proof of (\ref{average_3Dt}). Recall $\p_t f_R (0,x,v) = f_{R,t}(0,x,v)$ from \eqref{initial_ft}. \hide From (\ref{eqtn_fR}) we define 
\Be\label{f:initial_t}
\begin{split}
 f_{R,t} (0,x,v):= & - \frac{1}{\e} v\cdot \nabla_x   f_R %f_R
  -   \frac{1}{ \e^2\kappa} L%f_R 
  f_R +  \frac{1}{\kappa} \Gamma({f_2} , f_R )
+    \frac{   \delta }{\e\kappa}\Gamma(f_R , f_R )\\
&
-  \frac{( \p_t + 
\e^{-1} v\cdot \nabla_x) \sqrt{\mu}}{\sqrt{\mu}} f_{R }
  + 
(\mathbf{I}- \mathbf{P})\mathfrak{R}_1  + \mathfrak{R}_2 
\Big|_{t=0},
\end{split}
\Ee
where we read all terms in r.h.s at $t=0$. \unhide We regard $\tilde{\O}$ as an open subset but not a periodic domain as $\O$. Without loss of generality we may assume that $f_{R,t}(0,x,v)$ is defined in $\R^3 \times \R^3$ and $\|f_{R,t}(0)
\|_{L^p(\R^3) \times \R^3} \lesssim \|f_{R,t}(0)
\|_{L^p(\tilde{\O}) \times \R^3}$ for all $1 \leq p \leq \infty$. Then we extend a solution for whole time $t \in \R$ as 
%an initial datum $f_R(0,x,v)$ in $(x,v) \in \R^3 \times \R^3$ as 
\Be\label{ext:f1_t}
  f_{  I,t} (t,x,v) : = \mathbf{1}_{t \geq 0 } \p_t f_R (t,x,v) + \mathbf{1}_{t \leq 0} \chi_1 (t)  f_{R,t} (0, x%- \frac{t}{\e} v
,v ).
\Ee
Using $\tilde{t}_B(x,v) $ in (\ref{tB}) we define 
\Be\label{ext:f2_t}
 f_{E,t} (t,x,v): =  \mathbf{1}_{(x,v) \in (\R^3 \backslash \tilde{\O}) \times \R^3 }   f_{I,t} (t+ \e \tilde{t}_B(x,v), \tilde{x}_B(x,v),v).% \ \ \text{with} \ \ \tilde{x}_B(x,v):= x+\tilde{t}_B(x,v)v
\Ee
We define an extension of cut-offed solutions   
\Be\label{ext:f+_t}
%\begin{split}
  \bar{f}_{R,t}(t,x,v) :=%&  
 \chi_2 (x) \chi_3 (v)  \big\{
 % \chi_1(t) 
 %\Big\{
  \mathbf{1}_{\tilde{\O}}(x)% \mathbf{1}_{t\geq 0} 
 f_{I,t}(t,x,v)
 +  f_{E,t} (t,x,v)\big\}
 \  \ \text{for} \ (t,x,v) \in (-\infty,T] \times \R^3 \times \R^3
 . 
\Ee
 We note that in the sense of distributions $\bar{f}_{R,t}$ solves%is a weak solution in $\R \times (\mathbb{T}^2 \times \R) \times \R^3$ to an equation 
\Be\label{eqtn:bar_f_t}
\begin{split}
& \e \p_t \bar{f}_{R,t} + v\cdot \nabla_x \bar{f}_{R,t} =     \   \bar{g}_t 
 \ \ \text{in} \ (-\infty,T] \times \R^3 \times \R^3, \ \ \text{where}\\
& \bar{g}_t :  =     \frac{v\cdot \nabla_x \chi_2  }{\chi_2  } \bar{f}_{R,t} + \mathbf{1}_{t \geq 0} \mathbf{1}_{\tilde{\O}} (x) \chi_2(x)  \chi_3(v) 
[\e \p_t + v\cdot \nabla_x] \p_t f_R \\
 &\quad \ \ +\mathbf{1}_{t \leq 0} 
\chi_2 (x) \chi_3 (v) \{
\e\p_t \chi_1(t) 
f_{R,t} (0,x,v) + \chi_1 (t) v\cdot \nabla_x f_{R,t} (0,x,v) 
\}.
  %,
%\\
%\bar{g}(t,x,v):=&      \mathbf{1}_{t\geq 0}   \mathbf{1}_{\O}(x)\chi_1 (x) \chi_2 (v) %\big[
%\big[\e \p_t   + v\cdot \nabla_x+ \frac{\nu }{\kappa \e}  \big]f_R(t,x,v)
% + \frac{ v\cdot \nabla_x \chi_1 (x) }{\chi_1 (x)} \bar{f}_R (t,x,v )
% :=     \mathbf{1}_{t\geq 0} \chi_1(t)  \mathbf{1}_{\O}(x) \chi_2(v) %\big[
% (\e \p_t   + v\cdot \nabla_x  )f_R(t,x,v) +  [\e {\chi_1^\prime(t)} / {\chi_1 (t)}] \bar{f}_R (t,x,v)
 %+ \mathbf{1}_{t<0}  (\e \p_t   + v\cdot \nabla_x  )f_R(0,x- \frac{t}{\e} v,v)
 %\big]
%\mathbf{1}_{t \geq 0} \mathbf{1}_{x_3\geq 0}\chi_1(t)\chi_2(v) 
% (\e \p_t   + v\cdot \nabla_x   )f_R(t,x,v)
 % 
  \end{split}
\Ee
Here we have used the fact that $\bar{f}_{R,t}$ in (\ref{eqtn:bar_f_t}) is continuous along the characteristics across $\p \tilde{\O}$ and $\{t=0\}$. %in the support of $\chi_2 (x) \chi_3 (v)$. We also have used the fact that $[\e \p_t   + v\cdot \nabla_x   ]f_R(0,x- \frac{t}{\e} v,v)=0$ and $[\e \p_t   + v\cdot \nabla_x   ]f_1 (t+ \e \tilde{t}_B(x,v), x+\tilde{t}_B(x,v)v,v)=0$ in the sense of distribution.  
We derive that, using (\ref{eqtn:bar_f_t}),
\Be\label{f_R:duhamel_t}
\bar{f}_{R,t}(t,x,v) =
%e^{-   \frac{C_\nu }{\kappa \e^2 }   t} 
%|f_R (0,x- \frac{t}{\e} v , v)|
%+
 \frac{1}{\e} \int^t_{-\infty} 
%e^{-  \frac{C_\nu }{\kappa \e^2 } (t-s) }
 \bar{g}_t (s, x- \frac{t-s}{\e} v, v)  \dd s \ \ \text{for} \ (t,x,v) \in 
(-\infty,T] \times \R^3%(\mathbb{T}^2 \times \R) 
 \times \R^3.
\Ee

Now we apply Lemma \ref{lemma:average} to (\ref{f_R:duhamel_t}) and derive that, for $p<3$,
\Be\label{bound:S_t}
\begin{split}
&\| S(\bar{g_t}) \|_{L^2_t((0,T);  L^p_x(\mathbb{T}^2 \times \R))} \\
\lesssim &%(\kappa\e)^{\frac{3}{p}-\frac{1}{2}}
 \| \mathbf{1}_{(t,x,v) \in \mathfrak{D}_T}   \bar{g}_t\|_{L^2  ((0,T)\times (\mathbb{T}^2 \times \R) \times \{|v| \leq N \})}\\
 \lesssim & \ \| f_{R,t} (0)\|_{L^2(\O \times \R^3)} + \| \e \p_t  f_R + v\cdot \nabla_x f_R \|_{L^2 ((0,T) \times \tilde{\O} \times V)} \\
 &+  \| %v\cdot \nabla_x \chi_2 (x)  
 \mathbf{1}_{(t,x,v) \in \mathfrak{D}_T}
 f_{I,t} (t+ \e \tilde{t}_B(x,v),  \tilde{x}_B(x,v) ,v)
\|_{L^2  ((-1,T]\times (U \backslash  \tilde{\O})\times V)}.
 \end{split}
\Ee
Following the same argument of (\ref{est3:g2})-(\ref{est:bar_f3_1}) we deduce that 
\Be\label{est:bar_f3_1_t}
\begin{split}
(\ref{bound:S_t})% &\lesssim  (\kappa \e)^{\frac{6}{p}-1}\big\|      \mathbf{1}_{t\geq 0}   \mathbf{1}_{\O}(x)  
% (\e \p_t   + v\cdot \nabla_x+ \frac{\nu }{\kappa \e}  )f_R(t,x,v)\big\|_{L^2 ((0,T)\times (\mathbb{T}^2 \times \R) \times
% \{|v| \leq N \})}\\
 %&
 \lesssim  &
 \|\p_t  f_R \|_{L^\infty ([0,T];L^2( {\O} \times \R^3))}  +  \|\p_t  f_R (0)\|_{L^2_\gamma} 
 %(\kappa \e)^{\frac{6}{p}-1}
  % \frac{1}{\e \kappa} Kf_R +
\\
& +\big\|  - \frac{1}{\e \kappa}    L(\mathbf{I} - \mathbf{P}) \p_t f_R  + \e \times \text{r.h.s. of } (\ref{eqtn_fR_t}) 
    \big\|_{L^2 ((0,T]\times \O \times
V)}.
 \end{split}
\Ee
From (\ref{est:EG_t})-(\ref{estfR2last}), the last term of (\ref{est:bar_f3_1_t}) is bounded above by 
\Be
\begin{split}\label{est:bar_f3_1_t:last}
&
\Big\{ \frac{1}{\kappa} \left( \| \p_t \mathcal I \|_{L^\infty_{t,x,v}}  + \e \| \eqref{estptf2} \|_{L^\infty_{t,x,v } }  \right)  \Big\}  \Big\{ \| Pf_R \|_{L^2_{t,x}}+ \| \sqrt{\nu} (\mathbf{I} -\mathbf{P}) f_R \|_{L^2_{t,x,v}} \Big\}
\\ &+ \Big\{ \frac{1}{\kappa \e} + \frac{\e^{1/2}}{\kappa} \|   \mathfrak{w}f_R\|_{L^\infty_{t,x,v}} +  \frac{1}{\kappa} \left( \| Pf_1 \|_{L^\infty_{t,x}} + \e \|(\ref{est:f2})  \|_{L^\infty_{t,x}}  \right)    \Big\} \| \sqrt{\nu} (\mathbf{I} - \mathbf{P}) \p_t f_R \|_{L^2_{t,x,v}}
 \\  &+   \Big\{  \frac{\e^{1/2}}{\kappa}\|    {P} f_R \|_{L^\infty_tL^6_{x }}^{\frac{3(p-2)}{p}} \| \mathfrak{w} f_R\|_{L^\infty_{t,x,v}}^{\frac{6-2p}{p}} + \frac{\e}{\kappa}\|  (\ref{est:f2}) \|_{L^\infty_tL^{\frac{2p}{p-2}}_{x }}  \Big\}  \|    {P} \p_t f_R \|_{L^2_tL^p_{x }}
 \\ & + \frac{1}{\kappa} \| Pf_1\|_{L^\infty_{t,x}} \| P \p_t f_R \|_{L^2_tL^2_{x }}
\\  &+  \e^{1/2}   \| (\ref{estR11t})  \|_{L^2_{t,x,v}}  +\frac{\e^{1/2}}{\kappa} \| \kappa \eqref{estR12t} \|_{L^2_{t,x,v}} + \e^{3/2} \| (\ref{estR21t})\|_{L^2_{t,x,v}}^2 + \frac{\e^{3/2}}{\kappa} \|\kappa (\ref{estR22t})\|_{L^2_{t,x,v}}^2
.
\end{split}
\Ee

On the other hand from (\ref{ext:f+_t}) and the argument of (\ref{decomp1:L2L3}) we derive 
\Be\begin{split}\label{decomp1:L2L3_t}
%&
&\| S(\bar{g_t}) \|_{L^2_t((0,T);  L^p_x(\mathbb{T}^2 \times \R))}
   \gtrsim   \left\|\int_{\R^3}  \bar{f}_{R,t}(t,x,v) {\varphi}_i (v) \sqrt{\mu_0(v)}    \dd v \right\|_{L^2_t ((0,T); L^p_x ( \tilde{\O}))} %\notag\\
%= 
 %\left\|\int_{\R^3}
%\chi_2 (x) \chi_3 (v)
%\p_t  f_R(t,x,v)  \tilde{\varphi}_i (v) \sqrt{\mu_0(v)}    \dd v   \right\|_{L^2_t ((0,T); L^p_x (\tilde{\O}))}
\\  \gtrsim &  \  \big\|   {P} \p_t  f_R  \big\|_{L^2_t( (0,T)  ;L^p_x(\tilde{\O}))}
- (\kappa \e ) ^{\frac{2}{p-2}} \|  \mathfrak{w}^\prime \p_t  f_R \|_{L^2_t ((0,T);L^\infty_{x,v} (  \O \times \R^3))}
\\ & - \frac{1}{\kappa \e }\|  (\mathbf{I} - \mathbf{P}) \p_t  f_R  \|_{L^2 ((0,T) \times   \O \times \R^3)}
 %
% C_N \| (\kappa \e)^{\frac{2}{p-2}} \mathfrak{w} \p_t  f_R \|_{L^2_t ((0,T);L^\infty_{x,v} (  \O \times \R^3))}^{\frac{p-2}{p}} (\kappa \e)^{-\frac{2}{p}} \|  (\mathbf{I} - \mathbf{P}) \p_t  f_R  \|_{L^2 ((0,T) \times   \O \times \R^3)}^{\frac{2}{p}}
.
 \end{split}\Ee
Here we have used 
\Be\label{I-P_infty,2}
\begin{split}
& \left\|\int_{\R^3}
\chi_2 (x) \chi_3 (v)
 (\mathbf{I} - \mathbf{P}) \p_t  f_R(t,x,v)  {\varphi}_i (v) \sqrt{\mu_0(v)}    \dd v   \right\|_{L^2_t ((0,T); L^p_x (\tilde{\O}))} \\
 \leq &  \left\|  
 (\mathbf{I} - \mathbf{P}) \p_t  f_R(t,x,v)      \right\|_{L^2_t ((0,T); L^p_{x,v} (\tilde{\O} \times \R^3))} \\
 \lesssim & \ \Big\| \|  \mathfrak{w}^\prime \p_t  f_R \|_{L^\infty_{x,v} (  \O \times \R^3)}^{\frac{p-2}{p}}
 \|  (\mathbf{I} - \mathbf{P}) \p_t  f_R  \|_{L^2_{x,v} (  \O \times \R^3)}^{\frac{2}{p}}
 \Big\|_{L^2_t((0,T))}\\
 \lesssim& \  \Big\| \|  \mathfrak{w} ^\prime\p_t  f_R \|_{L^\infty_{x,v} (  \O \times \R^3)}^{\frac{p-2}{p}}\Big\|_{L^{\frac{2p}{p-2}}_t((0,T))}
 \Big\| \|  (\mathbf{I} - \mathbf{P}) \p_t  f_R  \|_{L^2_{x,v} (  \O \times \R^3)}^{\frac{2}{p}}\Big\|_{L^p_t((0,T))} \\
 \lesssim& \  (\kappa \e)^{\frac{2}{p}} \|  \mathfrak{w}^\prime \p_t  f_R \|_{L^2_t ((0,T);L^\infty_{x,v} (  \O \times \R^3))}^{\frac{p-2}{p}} (\kappa \e)^{-\frac{2}{p}} \|  (\mathbf{I} - \mathbf{P}) \p_t  f_R  \|_{L^2 ((0,T) \times   \O \times \R^3)}^{\frac{2}{p}}\\
 \lesssim & \ (\kappa \e ) ^{\frac{2}{p-2}} \|  \mathfrak{w}^\prime \p_t  f_R \|_{L^2_t ((0,T);L^\infty_{x,v} (  \O \times \R^3))}+ (\kappa \e)^{-1} \|  (\mathbf{I} - \mathbf{P}) \p_t  f_R  \|_{L^2 ((0,T) \times   \O \times \R^3)}.
\end{split}
\Ee 

Combining (\ref{decomp1:L2L3_t}), (\ref{bound:S_t}), (\ref{est:bar_f3_1_t}), and (\ref{est:bar_f3_1_t:last}) and choosing $N\gg1$ we conclude (\ref{average_3Dt}).

\subsection{$L^6_x$-integrability gain for $\mathbf{P}f_R$}
 
We prove an important $L^6_x$-integrability gain for $\mathbf{P}f_R$ in the next proposition.
\begin{proposition}\label{prop:L6}
Under the same assumptions in Proposition \ref{prop:Hilbert}, we have  for all $t  \in [0,T]$
\Be\label{L6}
\begin{split}
&d_6
\|   {P} f_R(t) \|_{L^6_x} \\
\lesssim & \  
 (\e  \kappa^{-1} \|(\ref{est:f2})\|_{L^\infty_{t,x}}) \| f_R(t) \|_{L^2_{x,v}} + \e \| \p_t f_R(t) \|_{ {L^2_{x,v}}}  +o(1)  (\kappa \e)^{1/2} \| \mathfrak{w} f_R (t) \|_{L^\infty_{x,v}} +  \e^{1/2}| (\ref{est:f2})|_{L^4(\p\O)}
\\
&   + \e^{1/2}   \| (\ref{estR11})  \|_{L^2_{t,x}}  +\frac{\e^{1/2}}{\kappa} \| \kappa \eqref{estR12} \|_{L^2_{t,x}}+  \e^{3/2} \| (\ref{estR21})\|_{L^2_{t,x,v}}  + \frac{\e^{3/2}}{\kappa}  \| \kappa (\ref{estR22})\|_{L^2_{t,x,v}}
% \|_{L^2_{x,v}} 
 \\ & + \Big( \frac{1}{\e \kappa }+ \frac{1}{\kappa}\| \mathcal I \|_\infty + \frac{\e}{\kappa} \|(\ref{est:f2})\|_\infty + \frac{\e^{1/2}}{\kappa} \| \mathfrak{w}_{\varrho} f_R(t) \|_{L^\infty_{x,v}} \Big)
\\ & \ \ \  \times \big\{
\| (\mathbf{I} - \mathbf{P}) f_R  \|_{ {L^2_{t,x,v}}} + \| (\mathbf{I} - \mathbf{P})\p_t  f_R  \|_{ {L^2_{t,x,v}}}  \big\}
\\ & + \left(  \frac{1}{\kappa}\| \mathcal I \|_\infty + \frac{\e}{\kappa} \|(\ref{est:f2})\|_\infty \right) \| \mathbf P f_R \|_{L^2_{x,v}}
\\ &+   \| \mathfrak{w}_{\varrho} f_R(t) \|_{L^\infty _{x,v}}^{1/2}
{%\color{red}
\big\{
|  f_R |_{L^2_tL^2({\gamma_+})}^{1/2}
+ |  \p_t f_R  |_{L^2_tL^2({\gamma_+})}^{1/2}
\big\} },
\end{split}
\Ee
where
\Be\label{d6}
d_6:=1-\Big[ \frac{\e^{1/2}}{\kappa} \| Pf_R(t)\|_{L^6_x}^{1/2} \| Pf_R (t) \|_{L^2_x}^{1/2} \Big]^{1/6}.
\Ee

\end{proposition}

\begin{proof} 
%For the sake of simplicity we use notations (\ref{short_notation}) throughout this subsection.

We view (\ref{eqtn_fR}) as a weak formulation for a test function $\psi$
\Be\label{weak_form}
\begin{split}
 &\underbrace{ \iint_{\O \times\R^3}    f_R   v\cdot \nabla_x \psi}_{(\ref{weak_form})_1}- \underbrace{ \int_\gamma f_R \psi}_{(\ref{weak_form})_2} -\underbrace{ \iint_{\O \times\R^3}      \e \p_t f_R   \psi }_{(\ref{weak_form})_3}
\\  &=  \iint_{\O \times \R^3} \psi \Bigg\{ \frac{1}{ \e \kappa} L  f_R + \e^{1/2} \left( (\mathbf{I-P}) (v \cdot \nabla_x f_2) - \frac{2}{\kappa} \Gamma(f_1,f_2) \right )  - \e^{3/2} \left(  - \p_t f_2 + \frac{1}{\kappa}  \Gamma(f_2,f_2) \right) 
\\ &  \quad \quad \quad \quad \quad \quad -  \frac{2}{ \kappa} \Gamma(f_1 + \e f_2,f_R)  +    \frac{\e^{1/2} }{\kappa}\Gamma(f_R, f_R) \Bigg \}.   
\end{split}
\Ee
The proof of the lemma is based on a recent test function method in the weak formulation (\cite{EGKM, EGKM2}). We define
\Be\label{P0f}
 \mathbf{\tilde{P}}f_R :=  \Big\{ a + b \cdot v + c \frac{|v|^2-3}{\sqrt{6}} \Big\} \sqrt{\mu} \ \text{and} \  \tilde{P}f_R := (a,b,c),
\Ee
where $a:= \langle f_R, \sqrt{\mu} \rangle, b:=  \langle f_R,v \sqrt{\mu} \rangle$, and $c:= \langle f_R,  \frac{|v|^2-3}{\sqrt{6}}\sqrt{\mu} \rangle$. We choose a family of test functions as %
\begin{align}
 \psi _{a }&:= (|v|^{2}-\beta _{a})v\sqrt{
\mu_0}\cdot \nabla _{x}\varphi_{a}%= \sum_{i=1}^{3}(|v|^{2}-\beta _{a})v_{i}\sqrt{\mu_0}\partial _{i}\varphi_{a,k}%
,  \label{phia}\\
 \psi _{b,1 }^{i,j}&:= (v_{i}^{2}-\beta _{b})\sqrt{\mu_0}\partial _{j}%
\varphi _{b }^{j},\quad i,j=1,2,3,  \label{phibj}\\
\psi^{i,j}_{b,2} &:=|v|^{2}v_{i}v_{j}\sqrt{\mu_0}\partial _{j}\varphi _{b }^{i} ,\quad i\neq
j,  \label{phibij} \\
\psi_{c }&:= (|v|^{2}-\beta _{c})v\sqrt{\mu_0}\cdot \nabla _{x}\varphi%
_{c } ,  \label{phic}
\end{align}
where we choose $\beta_a=10, \beta_b=1, \beta_c=5$ such that 
\Be
0 = \int_{\R^3} (|v|^2 - \beta_a) \frac{|v|^2-3}{\sqrt{6}} (v_1)^2 \mu_0 (v) \dd v 
=
\int_{\R} (v_1^2 - \beta_b) \mu_0 (v_1) \dd v_1 = \int_{\R^3} (|v|^2-\beta_c) v_i^2 \mu_0 (v) \dd v .\label{beta_abc}
\Ee
Here, 
\begin{align}
-\Delta _{x}\varphi_{a }  =  a^{5}&   \ \ \text{with} \ \  \frac{\partial \varphi_{a }}{%
\partial n}\Big|_{\partial \Omega }=0,%
 % \ \text{and} \ 
%\int_\O\varphi_a=0,
\label{phi_a}
\\ 
-\Delta _{x}\varphi _{b }^{j}  =b^{5}_{j}&   \ \ \text{with} \ \  \varphi _{b}^{j}|_{\partial
\Omega }=0, \label{phi_b}\\
 -\Delta _{x}\varphi_{c }  =c^{5}&  \ \ \text{with} \ \  \varphi %
_{c }|_{\partial \Omega }=0.\label{phi_c}
\end{align} 
A unique solvability to the above Poisson equations when $(a,b,c) \in L^6(\O)$ and an estimate 
\Be
\| \nabla_x^2 \varphi_{(a,b,c)}\|_{L^{6/5}(\O)} + \| \nabla_x \varphi _{(a,b,c)}\|_{L^2(\O)}
+ \| \varphi_{(a,b,c)}\|_{ L^6(\O) } \lesssim \| |\tilde{P}f_R|^5\|_{L^{6/5}(\O)}
\lesssim \| \tilde{P}f_R\|_{L^6(\O)}^5
. \label{elliptic:varphi}
\Ee
 is a direct consequence of Lax-Milgram and suitable extension  (extend $a^5$ of (\ref{phi_a}) evenly in $x_3 \in \R$, and $b^5$ and $c^5$ of (\ref{phi_b}) and (\ref{phi_c}) oddly in $x_3 \in \R$, then solve the Poisson equation, and then restrict the whole space solutions to the half space $x_3>0$)  and a standard elliptic estimate $(L^{\frac{6}{5}}(\O) \rightarrow \dot{W}^{2,\frac{6}{5}} (\O) \cap \dot{W}^{1,2} (\O) \cap L^6(\O))$. 

%From $M_{1,\e u , 1}(v) = M_{1,0 , 1}(v) + O( \e) |u| | v-\e u | M_{1,\e u , 1}(v) $ we can easily check that 
%\Be\label{diff:P-tP}
%|\mathbf{P} f_R (t,x,v)- \mathbf{\tilde{P}}f_R (t,x,v)|\lesssim \e |u(t,x)| |v-\e u| \sqrt{\mu }
% |
%f_R(t,x,v)|.
%\Ee
%Therefore we have 
%\Be\label{P-tildeP}
%\begin{split}
%\| P f_R(t) \|_{ L^6_x} \lesssim&
%\| \mathbf{P} f_R (t) \|_{L^6_{x,v}}  \lesssim \| \mathbf{\tilde{P}} f_R (t) \|_{L^6_{x,v}} 
%+ \e \| u(t) \|_{\infty} \{
%\|  {P} f_R (t) \|_{L^6_{x }}  + \| ( \mathbf{I}-\mathbf{P}) f_R (t) \|_{L^6_{x,v}} 
%\}\\
%\lesssim & (1+ \e \| u \|_\infty)  \|  {\tilde{P}} f_R (t) \|_{L^6_{x }} 
%+ \e \| u(t) \|_{\infty}   \| ( \mathbf{I}-\mathbf{P}) f_R (t) \|_{L^6_{x,v}} 
%.
%\end{split}
%\Ee
%Note that $ \| ( \mathbf{I}-\mathbf{P}) f_R (t) \|_{L^6_{x,v}} \leq  \| ( \mathbf{I}-\mathbf{P}) f_R (t) \|_{L^\infty_{x,v}} ^{2/3} \| ( \mathbf{I}-\mathbf{P}) f_R (t) \|_{L^2_{x,v}} ^{1/3}
%\lesssim o(1)  (\kappa \e)^{1/2}  \| \mathfrak{w} f_R (t) \|_{L^\infty_{x,v}}  + (\kappa \e)^{-1} \| ( \mathbf{I}-\mathbf{P}) f_R (t) \|_{L^2_{x,v}}$. 
Clearly to prove the lemma and (\ref{L6}) it suffices to prove the same bound for $\|  \tilde{P}f_R \|_{L^6_{x,v}}:=\| ( a,b,c) \|_{L^6_{x }} $.
%From we prive (\ref{L6_0}). 
%
%On the other hand,, it follows that $$\| P f_R \|_{L^\infty_t L^6_x} \lesssim\| \mathbf{P} f_R \|_{L^\infty_t L^6_{x,v}} \lesssim  \| \mathbf{\tilde{P}} f_R \|_{L^\infty_t L^6_{x,v}}
 %+ \e \|u\|_{L^\infty_{t,x}} \{
% \|  {P} f_R \|_{L^\infty_t L^6_{x }} + \|( \mathbf{I}- \mathbf{P}) f_R \|_{L^\infty_t L^6_{x,v}}
% \}. 
% $$Hence we finally conclude (\ref{L6}) from (\ref{L6_0}). 

Following the direct computations in the proof of Lemma 2.12 in \cite{EGKM2} we derive that 
\begin{equation}\label{est:wf1}  (\ref{weak_form})_1 =
\begin{cases}
 5  \| a(t) \|_{6}^6 +o(1)  \|   \mathbf{\tilde{P}} f_R(t) \|_6^6
+O(1) \| (\mathbf{I} - \mathbf{ {P}}) f_R(t) \|_6^6
& \text{if } \psi=\psi_a,\\ 
-2 \int_\O   b_i  \p_i \p_j  \varphi_b^j %\Delta^{-1} b_j^5
+o(1)  \|   \mathbf{\tilde{P}} f_R(t) \|_6^6 +O(1) \| (\mathbf{I} - \mathbf{ {P}}) f_R(t) \|_6^6  & \text{if } \psi=\psi_{b,1}^{i,j},\\ 
\int_\O b_j \p_{i} \p_j  \varphi_b^i
 + \int_\O b_i \p_{j} \p_j  \varphi_b^i +O(1) \| (\mathbf{I} - \mathbf{ {P}}) f_R(t) \|_6^6
 & \text{if } \psi=\psi_{b,2}^{i,j} \ \text{and} \ i \neq j,\\
 5  \| c(t)\|_6^6 +o(1)  \|   \mathbf{\tilde{P}} f_R(t) \|_6^6 + O(1) \| (\mathbf{I} - \mathbf{ {P}}) f_R(t) \|_6^6  & \text{if } \psi=\psi_c.
    \end{cases}
\end{equation}
For $\|  b_i\|_{6}^6$, using the second and third estimate of (\ref{est:wf1}) we deduce that %and
\Be\label{est:wf1_b}
\begin{split}
\|  b_i\|_{L^6(\O)}^6&= - \int_\Omega b_i \Delta_x\varphi_b^i \dd x =  - \int_\Omega b_i \p_i^2\varphi_b^i \dd x  - 
\sum_{j(\neq i)} \int_\Omega b_i \p_j^2\varphi_b^i \dd x
%\|b_i \p_i^2 \varphi_b^i \|_{L^1(\O)} +\sum_{j(\neq i)} \|  b_i \p_j^2\varphi_b^i \|_{L^1(\O)} 
\\
&= \frac{1}{2}\sum_{j}(\ref{weak_form})_1|_{\psi^{j,i}_{b,1}}  - \sum_{j(\neq i)}{(\ref{weak_form})_1|_{\psi^{i,j}_{b,2}}} %- \sum_{j(\neq i)} \int_\Omega b_j\p_i\p_j\varphi_b^i dx  
+ o(1)  \|   \mathbf{\tilde{P}} f_R(t) \|_6^6 +O(1) \| (\mathbf{I} - \mathbf{ {P}}) f_R(t) \|_6^6.
%&\lesssim  -\sum_{j(\neq i)} \|  b_j \p_i \p_j \varphi_b^i \|_{L^1(\O)} + o(1)  \|   \mathbf{\tilde{P}} f_R(t) \|_6^6 +O(1) \| (\mathbf{I} - \mathbf{ {P}}) f_R(t) \|_6^6\\
%&\lesssim  o(1)  \|  {\tilde{P}} f_R(t) \|_6^6 +O(1) \| (\mathbf{I} - \mathbf{ {P}}) f_R(t) \|_6^6.
\end{split}\Ee
\hide From the second and third estimate of (\ref{est:wf1}) and 
\Be\label{est:wf1_b}
\begin{split}
\|  b_i\|_{L^6(\O)}^6&=
\|b_i \p_i^2 \varphi_b^i \|_{L^1(\O)} +\sum_{j(\neq i)} \|  b_i \p_j^2\varphi_b^i \|_{L^1(\O)} 
\lesssim  -\sum_{j(\neq i)} \|  b_j \p_i \p_j \varphi_b^i \|_{L^1(\O)}
+ o(1)  \|   \mathbf{\tilde{P}} f_R(t) \|_6^6 +O(1) \| (\mathbf{I} - \mathbf{ {P}}) f_R(t) \|_6^6\\
&\lesssim  o(1)  \|  {\tilde{P}} f_R(t) \|_6^6 +O(1) \| (\mathbf{I} - \mathbf{ {P}}) f_R(t) \|_6^6.
\end{split}\Ee\unhide

Now we consider the boundary term $(\ref{weak_form})_2$. % Recall that $P_{\gamma_+} f(v) /  \sqrt{\mu_0(v)}$ is constant in $v$. 
From (\ref{phia})-(\ref{phic}) and (\ref{beta_abc})
\begin{equation}   \label{no_Pg}
  \int_{\gamma}  \psi P_{\gamma_+} f_R=
\begin{cases}
 \int_{\p\O}\p_n \varphi_a  \int_{\R^3} (|v|^2 - \beta_a) (v\cdot n)^2 \mu_0 \dd v \dd S_x  =0
& \text{if } \psi=\psi_a,\\
0& \text{if } \psi=\psi_{b,1}^{i,j} \ \text{or} \ \psi_{b,2}^{i,j},\\
 \int_{\p\O}\p_n \varphi_c  \int_{\R^3} (|v|^2 - \beta_c) (v\cdot n)^2 \mu_0 \dd v \dd S_x  =0
& \text{if } \psi=\psi_c. 
\end{cases}
\end{equation}
Here we have used the Neumann boundary condition of (\ref{phi_a}) for $\psi_a$, and the last identity in (\ref{beta_abc}) for $\psi_c$. For $\psi_{b,1}^{i,j}$ or $\psi_{b,2}^{i,j}$ we used the fact that the integrands are odd in $v$. From (\ref{bdry_fR}),  we decompose $f_R |_{\gamma}= P_{\gamma_+} f_R + \mathbf{1}_{\gamma_+} (1- P_{\gamma_+})f_R - \mathbf{1}_{\gamma_-} \e^{1/2} (1- P_{\gamma_+}) f_2$. From (\ref{no_Pg}) we have  
\Be
\begin{split}\label{est:wf2}
| (\ref{weak_form})_2|  &= \Big| \cancel{\int_{\gamma}  \psi P_{\gamma_+} f_R} +    \int_{\gamma}  \psi\{ \mathbf{1}_{\gamma_+} (1- P_{\gamma_+})f_R - \mathbf{1}_{\gamma_-} \e^{1/2} (1- P_{\gamma_+}) f_2\} \Big|\\
&\lesssim
 |\nabla_x \varphi| _{L^{4/3}(\p\O)} \big\{|(1-P_{\gamma_+}) f_R|_{4, \gamma_+} 
+ \e^{1/2}  |(\ref{est:f2})| _{L^{4}(\p\O) }  \big\}
\end{split}
\Ee
where we have used $|\int_{\gamma_+} \psi (1-P_{\gamma_+}) f_R |\lesssim  |\nabla_x \varphi|_{L^{ {4} / {3}}(\p\O) } |(1-P_{\gamma_+}) f_R|_{4, \gamma_+}$ at the last line. Here $\varphi \in \{  \varphi_a, \varphi_b, \varphi_c\}$. For the first term of (\ref{est:wf2}) we interpolate 
\Be\label{inter:bdry}
|(1-P_{\gamma_+}) f_R|_{4, \gamma_+}\lesssim | \e^{-\frac{1}{2}}(1-P_{\gamma_+}) f_R|_{2, \gamma_+}^{1/2} \e^{\frac{1}{4}} \|\mathfrak{w}_{\varrho} f_R\|_\infty^{1/2}.
\Ee
 For the second term of (\ref{est:wf2}), we use (\ref{elliptic:varphi}) %an elliptic estimate $(L^{\frac{6}{5}}(\O) \rightarrow \dot{W}^{2,\frac{6}{5}} (\O) \cap \dot{W}^{1,2} (\O) \cap L^6(\O))$ 
 and a trace theorem $(\dot{W}^{1,\frac{6}{5}} (\mathbb{T}^2\times \R_+) \cap L^2(\mathbb{T}^2\times \R_+) \rightarrow  {W}^{1- \frac{1}{6/5},\frac{6}{5}} (\mathbb{T}^2 )) $, and the Sobolev embedding $(W^{\frac{1}{6},\frac{6}{5}}( \mathbb{T}^2) \rightarrow L^{4/3} (\mathbb
 {T}^2))$ to conclude that
 \Be\label{4/3bdry}
 |\nabla_x \varphi|_{L^{\frac{4}{3}}( \mathbb{T}^2)}\lesssim  |\nabla_x \varphi |_{W^{\frac{1}{6},\frac{6}{5}}(  \mathbb{T}^2)}\lesssim \|  \nabla_x \varphi \|_{\dot{W}^{1, \frac{6}{5}}( \mathbb{T}^2 \times \R_+)
 \cap L^2(\O)
 }%\lesssim \| |Pf_R|^5\|_{L^{\frac{6}{5}}(\mathbb{T}^2 \times \R_+)} 
 \lesssim \| \tilde{P}f_R \|_{L^6(\mathbb{T}^2 \times \R_+)}^5 
 . 
 \Ee

Next we consider $(\ref{weak_form})_3$. For $\psi$ of (\ref{phia})-(\ref{phic}) and $\varphi$ of (\ref{phi_a})-(\ref{phi_c}), using (\ref{elliptic:varphi}), it follows that 
\Be\begin{split}
|(\ref{weak_form})_3|&\lesssim \e \| \p_t f_R \|_{L^{2}_{x,v}} \| \psi \|_{L^{2}_{x,v}}
\lesssim \e \| \p_t f_R \|_{L^{2}_{x,v}}\| \nabla_x \varphi \|_{L^{2}_{x }}
\lesssim  \e \| \p_t f_R \|_{L^{2}_{x,v}} \| \tilde{P}f_R \|_{L^6_x}^{5}\\
&
\leq O(1) [\e \| \p_t f_R \|_{L^{2}_{x,v}}]^6 + o(1) \|  \tilde{P}f_R \|_{L^6_x}^{6}. \label{est:wf3}
\end{split}\Ee

Lastly we consider the right hand side of (\ref{weak_form}). From (\ref{elliptic:varphi}), it follows 
\Be\label{est:wf_L}
\begin{split}
&\Big|\iint_{\O \times \R^3} \psi \frac{1}{\e \kappa } Lf_R\Big|  =\Big| \iint_{\O \times \R^3} \psi \frac{1}{\e \kappa } L (\mathbf{I} - \mathbf{P})f_R\Big|\\
%\frac{1}{\e^2\kappa} |(\mathbf{\tilde{P}}- \mathbf{P}) L (\mathbf{I} - \mathbf{P})f_R| &  
&\lesssim 
\frac{1}{\e \kappa} \int_\O \int_{\R^3}
|\nabla_x \varphi_{(a,b,c)}(x)|
 \mu(v)^{1/4}
\Big[ \nu(v ) |(\mathbf{I} - \mathbf{P})f_R (x,v ) |\\
& \ \ \ \ \ \ \ 
+\int_{\R^3} k_\vartheta (v,v_*)  |(\mathbf{I} - \mathbf{P})f_R (x,v_*) | \dd v_*\Big]
\dd v \dd x \\
&\lesssim \frac{1}{\e  \kappa}  \| \nabla_x \varphi_{(a,b,c)} \|_{L^2_x} \| (\mathbf{I} - \mathbf{P}) f_R\|_{L^2_{x,v}} \lesssim  \frac{1}{\e  \kappa} \|  \tilde{P}f \|_{L^6_x}^5\| (\mathbf{I} - \mathbf{P}) f_R\|_{L^2_{x,v}}\\
&\leq o(1)  \|  \tilde{P}f \|_{L^6_x}^6+ \big[  \e^{-1}  \kappa^{-1}  \| (\mathbf{I} - \mathbf{P}) f_R\|_{L^2_{x,v}}\big]^6
.
\end{split}
\Ee
%Note that, from (\ref{est_Carl:Gamma}), $|\Gamma(\frac{\e}{\kappa} {f_2}, f_R)
%%+  \Gamma(\frac{   \delta }{ \kappa}f_R, f_R)
%|\lesssim %\{ 
%\frac{\e}{\kappa} \| \mathfrak{w}_{\varrho} f_2 \|_{\infty} %+ \frac{   \delta }{ \kappa} \| \mathfrak{w} _{\varrho, \ss}  f_R \|_\infty  \} 
%\mathfrak{w}_{\varrho}(v)^{-1}
%\big[\nu(v) f_R(v)+
%\int_{\R^3} k_{\vartheta} (v,v_*) f_R(v_*) \dd v_*\big].$ Then from (\ref{est:I-Pf2}) and (\ref{est:Pf2})
We have
\Be\label{est:wf_Gamma1}
\begin{split}
& \Big|\iint_{\O \times \R^3} \psi \frac{1}{\kappa} \Gamma( f_1 + \e {f_2}, f_R) \Big| 
\\  &\lesssim   \| \nabla_x \varphi_{(a,b,c)} \|_{L^2_x} \left( \frac{1}{\kappa}\| \mathcal I \|_\infty + \frac{\e}{\kappa} \|(\ref{est:f2})\|_\infty \right) \| f_R \|_{L^2_{x,v}}
\\ & \leq o(1)  \|  \tilde{P}f \|_{L^6_x}^6+ \big[   \left( \frac{1}{\kappa}\| \mathcal I \|_\infty + \frac{\e}{\kappa} \|(\ref{est:f2})\|_\infty \right) (  \|  \mathbf P f_R\|_{L^2_{x,v}} + \| \mathbf{(I-P)} f_R\|_{L^2_{x,v}} ) \big]^6.
\end{split}
\Ee
For the contribution of $\Gamma(f_R,f_R)$ we decompose $f_R= \mathbf{P} f_R + (\mathbf{I} - \mathbf{P}) f_R$, and we have
\Be
\begin{split}
&|\Gamma(f_R,f_R)(v)|\\
& \lesssim |\Gamma(\mathbf{P}f_R,\mathbf{P}f_R)(v)| +   |\Gamma(( \mathbf{I}-\mathbf{P})f_R,( \mathbf{I}-\mathbf{P})f_R)(v)|\\
&\lesssim\nu(v) %\mu(v)^{1/4}
|Pf_R |^2\\
& \ \ \ \  + \| \mathfrak{w}_{\varrho} f_R \|_\infty 
\Big\{
\nu(v) |( \mathbf{I}-\mathbf{P})f_R)(v)|  + \int_{\R^3} k_\vartheta (v,v_*) |( \mathbf{I}-\mathbf{P})f_R)(v_*)|\dd v_*
\Big\}.
\end{split}
\Ee
Then from (\ref{phia})-(\ref{phic}), and the H\"older's inequality ($1=1/2+ 1/3+ 1/6$)
\Be
\begin{split}\label{est:wf_Gamma2}
&\Big|\iint_{\O \times \R^3} \psi  
%\frac{\e}{\kappa} \Gamma({f_2}, f_R)
   \frac{   \e^{1/2} }{  \kappa}\Gamma(f_R, f_R) 
\Big| \\ &\lesssim \frac{\e^{1/2}}{ \kappa} \| \nabla_x \varphi_{(a,b,c)}\|_{L^2_x}\Big\{ \| Pf_R \|_{L^3_x}  \| Pf_R \|_{L^6_x} 
+
\| \mathfrak{w}_{\varrho} f_R \|_{L^\infty_{x,v}} \|( \mathbf{I}-\mathbf{P})f_R  \|_{L^2_{x,v}}
 \Big\}\\
 & \lesssim \frac{ \e^{1/2}}{  \kappa} 
 \| \tilde{P}f_R \|_{L^6_x}^{5}  \| Pf_R\|_{L_x^6}^{3/2}
  \| Pf_R \|_{L_x^2}^{1/2} +  \frac{\e^{3/2}}{  \kappa^{1/2}} \| \tilde{P}f_R \|_{L^6_x}^{5} \| \mathfrak{w}_{\varrho} f_R \|_{L^\infty_{x,v}} \|\e^{-1} \kappa^{-1/2}( \mathbf{I}-\mathbf{P})f_R  \|_{L^2_{x,v}}, 
\end{split}
\Ee
where we have used an interpolation $\| Pf_R\|_{L^3}\leq \| Pf_R\|_{L^6}^{1/2} \| Pf_R \|_{L^2}^{1/2}$ and (\ref{elliptic:varphi}) at the last step. A contribution of the rest of terms in the r.h.s of (\ref{weak_form}) can be easily bounded as
\Be
\begin{split}\label{est:wf_Gamma3}
&  \iint_{\O \times \R^3} |\psi| \left|  \e^{1/2} \left( (\mathbf{I-P}) (v \cdot \nabla_x f_2) - \frac{2}{\kappa} \Gamma(f_1,f_2) \right )  - \e^{3/2} \left(  - \p_t f_2 + \frac{1}{\kappa}  \Gamma(f_2,f_2) \right)\right|
\\ \lesssim & \  \| Pf_R \|_{L^6_x}^{5}\Big\{ \e^{1/2}   \| (\ref{estR11})  \|_{L^2_{t,x}}  +\frac{\e^{1/2}}{\kappa} \| \kappa \eqref{estR12} \|_{L^2_{t,x}}+  \e^{3/2} \| (\ref{estR21})\|_{L^2_{t,x,v}}  + \frac{\e^{3/2}}{\kappa}  \| \kappa (\ref{estR22})\|_{L^2_{t,x,v}} \Big\}.
\end{split}
\Ee
Finally, from a standard $1$D embedding (see appendix (A.1) in \cite{JK} for the proof): for $T>0$,
\Be\label{Sob_1D}
|g(t)|^2 \lesssim_T \int_0^T |g(s)|^2 \dd s + \int_0^T |g^\prime(s)|^2 \dd s \ \ \text{for } t \in [0,T],
\Ee
we collect the terms from (\ref{est:wf1}) with (\ref{est:wf1_b}), (\ref{est:wf2}) with (\ref{inter:bdry}) and (\ref{4/3bdry}), (\ref{est:wf3}), (\ref{est:wf_L}), (\ref{est:wf_Gamma1}), (\ref{est:wf_Gamma2}), (\ref{est:wf_Gamma3}) to conclude
{%\color{red} 
 \Be \label{I-P:expansion}
 \begin{split}
 \sup_{0 \leq s \leq t} \|(\mathbf{I} -\mathbf{P}) f_R(s ) \|_{L^2_{x,v}} &\lesssim  \|(\mathbf{I} -\mathbf{P}) f_R  \|_{L^2_{t,x,v}} +  \|(\mathbf{I} -\mathbf{P})\p_t f_R  \|_{L^2_{t,x,v}},
 \\ \sup_{0 \leq s \leq t} \| \mathbf P f_R(s ) \|_{L^2_{x,v}} &\lesssim  \| \mathbf P f_R  \|_{L^2_{t,x,v}} +  \| \mathbf P\p_t f_R  \|_{L^2_{t,x,v}},
 \\  \sup_{0 \leq s \leq t}  |(1- P_{\gamma_+}) f_R(s)|_{L^2({\gamma_+})} &\lesssim \sup_{0 \leq s \leq t}   |  f_R(s)|_{L^2({\gamma_+})}  \lesssim    |  f_R |_{L^2_tL^2({\gamma_+})} +  |\p_t  f_R |_{L^2_tL^2({\gamma_+})}.
     \end{split}
 \Ee
} 
This proves (\ref{L6}).
\end{proof}

\subsection{$L^\infty$-estimate}

In this section we develop a unified $L^\infty$-estimate. 

Recall the weight $\mathfrak w$ in \eqref{weight}. We consider 
\Be\label{h}
h(t,x,v) = \mathfrak{w} (v) f_R(t,x,v).
\Ee

An equation for $h$ can be written from (\ref{eqtn_fR}) and (\ref{bdry_fR}) as 
\Be \label{eqtn:h}
\p_t  h+ \frac{1}{\e} v\cdot \nabla_x h+  \frac{\nu  }{\e^2 \kappa}  
 %- \frac{1}{\e} \frac{v \cdot \nabla_x  \mathfrak{w}_{\varrho}}{ \mathfrak{w}_{\varrho}} + \frac{(\p_t + \e^{-1} v\cdot \nabla_x) \sqrt{\mu}}{\sqrt{\mu}}
h =  \frac{1}{\e^2 \kappa}K_{\mathfrak{w}} h  + \mathcal{S}_h
,
\Ee 
 \Be\label{bdry:h}
 h|_{\gamma_-} = \mathfrak{w} P_{\gamma_+} \Big(\frac{h}{\mathfrak{w}}\Big)+ r.
 \Ee
 For (\ref{h}) and \eqref{eqtn_fR}, we have $r=- \e^{1/2} \mathfrak{w} (1- P_{\gamma_+}) f_2$ and
 \[ \begin{split}
 \mathcal{S}_h: = &  \frac{1}{\kappa \e^{1/2}}  \Gamma_{ {\mathfrak{w} }}( {h} ,  {h} )
+\frac{2}{\kappa \e } \Gamma_{\mathfrak{w}} ( \mathfrak w f_1 + \e \mathfrak{w} f_2,h) 
\\ & +   \frac{1}{\e^{1/2}} \left \{ - \mathfrak{w} (\mathbf{I-P}) (v \cdot \nabla_x f_2) + \frac{2}{\kappa} \Gamma_{\mathfrak{w} } (\mathfrak{w} f_1, \mathfrak{w} f_2) \right \}  +  \e^{1/2} \left(  - \mathfrak{w}  \p_t f_2 + \frac{1}{\kappa}  \Gamma_{\mathfrak{w}} (\mathfrak{w} f_2, \mathfrak{w} f_2) \right),
\end{split} \]
where we denote $\Gamma_{\mathfrak{w}}(\cdot, \cdot )(v) := \mathfrak{w} (v)\Gamma(\frac{\cdot}{\mathfrak{w}}, \frac{\cdot}{\mathfrak{w}})(v)$ and $K_{\mathfrak{w}}(\cdot ):= \mathfrak{w} K( \frac{\cdot}{\mathfrak{w}})$. 
 
%If we have
%\Be\label{decay:u}
%\e^{5/2} \kappa |\p_t u| +  
%\e^{1/2}\sup_{x \in \O}(1+ x_3)  |\nabla_x u(t,x)|  <\infty,
%\Ee
%then for sufficiently small $\e, \kappa>0$, from (\ref{est:vw_x}), 
%\Be
%\nu  \geq \nu(v) +\frac{ \e \kappa}{2}   \mathfrak{z}_{\ss} (x_3) |v|^2 
% - \e^2\kappa \{
% \e |\p_t u|  + |\nabla_x u| |v| 
% \}|v-\e u| \geq \frac{\nu(v) }{2} + \frac{ \e \kappa}{4}   \mathfrak{z}_{\ss} (x_3) |v|^2 .
% \label{lower:nu_ss}
% \Ee

 From (\ref{Gamma})
\Be\label{est:Gamma_w}
\begin{split}
 & |\mathfrak{w}(v) \Gamma( \frac{h}{\mathfrak{w}}, \frac{h}{\mathfrak{w}}) (v)| \\
&\leq   \iint_{\R^3 \times \S^2} 
|(v-v_*) \cdot \mathfrak{u}| 
\sqrt{ \mu(v_* )} e^{- \varrho |v_*|^2}
\big\{
|h( v ^\prime)|| h( v_*^\prime)|
%+ \sqrt{\mu(v^\prime)\mu(v_*^\prime)} f(t,v ^\prime)
+| h( v  )||h( v_* )|
%-\sqrt{\mu(v )\mu(v_* )} f(t,v  )
\big\}
\dd \mathfrak{u} \dd v_* \\
&\lesssim_\varrho \nu(v) \|h\|_{L^\infty_v}^2.
\end{split}
\Ee
 From (expression of k) clearly we have 
 \Be\label{k_w}
\mathbf{k}(v,v_*) \frac{\mathfrak{w}_{\varrho}(v)}{\mathfrak{w}_{\varrho}(v
 _*)} \leq   \mathbf{k}_{\mathfrak{w}} (v,v_*): =  \frac{2C_{2}}{|v-v_*|}  e^{- \frac{|v-v_*|^2}{8}
- \frac{1}{8} \frac{(|v-\e u|^2 - |v_*- \e u|^2)^2}{|v-v_*|^2}}
 \frac{\mathfrak{w}_{\varrho}(v)}{\mathfrak{w}_{\varrho}(v
 _*)} .
 \Ee
As in (estimate for k) we derive  
 \Be\label{est:K_w}
 \int_{\R^3}
 \mathbf{k}_{\mathfrak{w}} (v,v_*)
  \dd v_* \lesssim \frac{1}{1+ |v|}.
 \Ee

 \begin{proposition}\label{est:Linfty}Recall $\mathfrak{w}_{\varrho}$ in (\ref{weight}). Assume the same assumptions in Proposition \ref{prop:Hilbert}.  Then
% In addition we assume (\ref{decay:u}), and the conditions of $\varrho$ and $\ss$ in (\ref{weight}). Then  %{\color{red}JJ: below $\e \kappa^2 \| (\ref{est:R2})\|_{L^\infty_{t,x}} $ should be $\e^2 \kappa( \| (\ref{est:R1})\|_{L^\infty_{t,x}}+ \| (\ref{est:R2})\|_{L^\infty_{t,x}})$ ? This term is from \eqref{est:S_h} and \eqref{infty1}}
 \Be\label{Linfty_3D}
 \begin{split}
&d_\infty
 \| \mathfrak{w}_{\varrho} f _R  \|_ {L^\infty_{t,x,v}}  
 \\ \lesssim &  \   \|  \mathfrak{w}_{\varrho} f (0)\|_ {L^\infty_{ x,v}}    +\e^{1/2} \| (\ref{est:f2}) \|_{L^\infty_{t,x}} 
 \\ & +  {\e^{3/2} \kappa  }  \left( \| (\ref{estR11})\|_\infty + \frac{1}{\kappa} \| \kappa \eqref{estR12} \|_\infty  \right)+   \e^{5/2} \kappa \left( \| (\ref{estR21})\|_\infty + \frac{1}{\kappa} \|  \kappa \eqref{estR22} \|_\infty \right)
\\ &+ \frac{1}{\e^{1/2} \kappa^{1/2}}   \|  P  f_R  \|_{L^\infty_tL^6_{x}} + \frac{1}{\e^{3/2} \kappa^{3/2}}\Big\{ \|  \sqrt{\nu}(\mathbf{I} -\mathbf{P}) f_R  \|_{L^2_{t,x,v}}+  \| \sqrt{\nu}(\mathbf{I} -\mathbf{P}) \p_t f_R  \|_{L^2_{t,x,v}}
\Big\}
%\\&
%{%\color{red}+
%+\frac{1}{\e^{1/2}\kappa^{3/2}} \| \p_t u \|_{L^\infty_{t,x}} \| Pf_R \|_{L^2_{t,x}}
%}
,
 \end{split}\Ee%
 where 
\Be\label{dinfty}
%\bcb 
d_{\infty}:= 1-  \e^2 \| (\ref{est:f2}) \|_{L^\infty_{t,x}}  - \e   \|\mathcal I \|_{L^\infty_{t,x}}  -  \e^{3/2}  \| \mathfrak{w}_{\varrho} f_R    \|_{L^\infty_{t,x,v}}  %\ec
 .
\Ee
\end{proposition}

\begin{proposition}\label{prop:Linfty_t}
Assume the same assumptions of Proposition \ref{est:Linfty}. We denote 
\Be\label{w_prime}
\mathfrak{w}^\prime(v) : =   \mathfrak{w}_{\varrho^\prime}(v) \ \ \text{for} \ \varrho^\prime<\varrho.
\Ee
Let $p<3$. Then 
\Be\label{Linfty_3D_t}
\begin{split}
&d_{\infty, t} \| \mathfrak{w} ^\prime  \p_t  f_R  \|_{L^2_t ((0,T);L^{\infty}_{x,v} (\O \times \R^3))}
\\ \lesssim & \ \e\kappa^{1/2} \| \mathfrak{w} ^\prime  \p_t  f_R(0 ) \|_{ L^{\infty}_{x,v}}+  \frac{1}{\e^{3/p} \kappa^{3/p}} \| P \p_t f  \|_{L^2_t  L^p_{x} } +  \frac{1}{\e^{3/2} \kappa^{3/2} } \| \sqrt{\nu} ( \mathbf{I} - \mathbf{P}) \p_t f  \|_{L^2 _{t,x,v}}
\\ &+ \e^{3/2}  \left(  \kappa \| \eqref{estR11t} \|_{L^2_t L^\infty_x} +  \| \kappa \eqref{estR12t} \|_{L^2_t L^\infty_x} \right) + \e^{5/2} \left( \kappa \| \eqref{estR21t} \|_{L^2_t L^\infty_x} +  \| \kappa \eqref{estR22t} \|_{L^2_t L^\infty_x} \right)
  \\ & + \big(  \e \| \p_t \mathcal I \|_{L^\infty_{t,x} } +  \e^2 \|(\ref{estptf2})\|_{L^\infty_{t,x}} \big) \| \mathfrak{w} f_R\|_{L^\infty_{t,x,v}}, 
\end{split}
\Ee
with 
\Be\label{dinftyt}
d_{\infty, t}: = 
1- \e \|  \mathcal I \|_{L^\infty_{t,x}}  - 
\e^2 \| (\ref{est:f2})\|_{L^\infty_{t,x}} - 
\e^{3/2} \| \mathfrak{w} f_R\|_{L^\infty_{t,x,v}}.
\Ee
\end{proposition}

 In the proof of propositions, for simplicity, we often use $\| \ \cdot \ \|_\infty$ for $\| \ \cdot \ \|_{L^\infty_{t,x,v}}$, $\| \ \cdot \ \|_{L^\infty_{x,v}}$ or $\| \ \cdot \ \|_{L^\infty_{x}}$ if there would be no confusion. 

\begin{proof}[\textbf{Proof of Proposition \ref{est:Linfty}}]
 We define backward exit time and position as 
 \Be\label{tb}
 \tb(x,v) : = \e \frac{x_3}{v_3}, \ 
  \ \ 
  \xb(x,v) := x-   \frac{x_3}{v_3} v \ \ \text{for} \ \ (x,v) \in  \O  \times \R^3. 
 % \ \ \text{for} \  \ v_3> 0.
 \Ee
 %
%Recall the backward exit time $\tb$ in (\ref{tb}). 
Since the characteristics for (\ref{eqtn:h}) are given by $(x- \frac{t-s}{\e}v, v)$, we have, for $0 \leq t-s <  \tb(x,v)$,   
 \Be\label{Duhamel}
 \frac{d}{ds}\Big\{  e^{-\int^t_s \frac{\nu  }{\e^2 \kappa}  } h(s,x- \frac{t-s}{\e}v, v )
\Big\}=e^{-\int^t_s \frac{\nu  }{\e^2 \kappa}  } 
\Big\{ \frac{1}{\e^2 \kappa}K_{\mathfrak{w}} h 
+\mathcal{S}_h
% + \frac{\delta}{\kappa \e}  \Gamma_{ {\mathfrak{w} }}( {h} ,  {h} )
%+\frac{1}{\kappa} \Gamma_{\mathfrak{w}} ( \mathfrak{w} f_2,h)
\Big\}(s,x- \frac{t-s}{\e}v, v ).
 \Ee
Here $e^{-\int^t_s \frac{\nu_{}  }{\e^2 \kappa} }=e^{-\int^t_s \frac{\nu_{} (\tau , x- \frac{t-\tau}{\e}v, v )}{\e^2 \kappa} \dd \tau}$. We regard $(x_1- \frac{t-s}{\e} v_1, x_2- \frac{t-s}{\e} v_2) \in \mathbb{R}^2$ belongs to $\mathbb{T}^2$ without redefining them in $[- \pi, \pi]^2$. 
 
 Now we represent $h$ using (\ref{Duhamel}) and (\ref{bdry:h}) as 
 \begin{align}
 h(t,x,v) =& \mathbf{1}_{t-\tb(x,v)<0}e^{-\int^t_0 \frac{\nu  }{\e^2 \kappa}  } h(0,x- \frac{t }{\e}v, v )
\notag% \label{h:initial}
 \\
 &+ \int^t_{\max\{0, t-\tb(x,v)\}}
 e^{-\int^t_s \frac{\nu  }{\e^2 \kappa}  } 
 \frac{1}{\e^2 \kappa}K_{\mathfrak{w}} h   (s,x- \frac{t-s}{\e}v, v )
 \dd s
 \label{h:K}
  \\
  &+ \int^t_{\max\{0, t-\tb(x,v)\}}
 e^{-\int^t_s \frac{\nu  }{\e^2 \kappa}  } 
\mathcal{S}_h(s,x- \frac{t-s}{\e}v, v )
 \dd s
\notag% \label{h:Gamma}
  \\
& + \mathbf{1}_{t-\tb(x,v)\geq 0}e^{-\int^t_{t-\tb(x,v)} \frac{\nu  }{\e^2 \kappa}  } 
 h(t-\tb(x,v),\xb(x,v), v )
%\mathfrak{w} (\xb(x,v), v )
%c_\mu   \sqrt{\mu(v)}
 %\int_{\mathfrak{v}_3<0} h(t-\tb(x,v),\xb(x,v), \mathfrak{v})  \frac{ \sqrt{\mu(\mathfrak{v})} |\mathfrak{v}_3| }{\mathfrak{w} (\mathfrak v) }\dd \mathfrak{v}
.\label{h:bdry}
 \end{align}
 Since the integrand of (\ref{h:bdry}) reads on the boundary, using the boundary condition (\ref{bdry:h}) and (\ref{Duhamel}) again, we represent it as 
 \begin{align}
&% \mathbf{1}_{t-\tb(x,v)\geq 0}e^{-\int^t_{t-\tb(x,v)} \frac{\nu  }{\e^2 \kappa}  } 
h(t-\tb(x,v),\xb(x,v), v )\notag
\\
=&
\mathfrak{w} (v)
c_\mu   \sqrt{\mu(v)}
 \int_{\mathfrak{v}_3<0} h(t-\tb(x,v),\xb(x,v), \mathfrak{v})  \frac{ \sqrt{\mu(\mathfrak{v})} |\mathfrak{v}_3| }{\mathfrak{w} (\mathfrak v) }\dd \mathfrak{v}
% - \frac{\e}{\delta} \mathfrak{w} (1- P_{\gamma_+}) f_2 
+ r(t-\tb(x,v),\xb(x,v),  v ) 
 \notag
\\
=&%\mathbf{1}_{t-\tb(x,v)\geq 0}e^{-\int^t_{t-\tb(x,v)} \frac{\nu  }{\e^2 \kappa}  }
\mathfrak{w} (v)
c_\mu   \sqrt{\mu(v)}
 \int_{\mathfrak{v}_3<0} 
 e^{
- \int^{t-\tb(x,v)}_0 \frac{\nu  }{\e^2 \kappa}  
} 
 h(0,\xb(x,v)- \frac{t-\tb(x,v)}{\e} \mathfrak{v}, \mathfrak{v})  \frac{ \sqrt{\mu(\mathfrak{v})} |\mathfrak{v}_3| }{\mathfrak{w} (\mathfrak v) }\dd \mathfrak{v}
\notag%\label{bdry:initial}
\\
&+\mathfrak{w} (v)
c_\mu   \sqrt{\mu(v)}
 \int_{\mathfrak{v}_3<0}\notag \\
 & \ \ \ \times 
 \int^{t-\tb(x,v)}_0 
 e^{- \int^{t-\tb(x,v)}_s
 \frac{\nu}{\e^2 \kappa} 
 } \frac{1}{\e^2 \kappa} K_{\mathfrak{w}} h (s, \xb(x,v) -
 \frac{t-\tb(x,v)-s}{\e} \mathfrak{v}, \mathfrak{v}
  )
% e^{
%- \int^{t-\tb(x,v)}_0 \frac{\nu  }{\e^2 \kappa}  
%} 
% h(0,\xb(x,v)- (t-\tb(x,v)) \mathfrak{v}, \mathfrak{v})
   \frac{ \sqrt{\mu(\mathfrak{v})} |\mathfrak{v}_3| }{\mathfrak{w} (\mathfrak v) }
   \dd s
   \dd \mathfrak{v}\label{bdry:K}\\
   &+\mathfrak{w} (v)
c_\mu   \sqrt{\mu(v)}
 \int_{\mathfrak{v}_3<0}\notag \\
 &  \ \ \ \ \ \ \ \ \ \ \ \times
 \int^{t-\tb(x,v)}_0 
 e^{- \int^{t-\tb(x,v)}_s
 \frac{\nu}{\e^2 \kappa} 
 }  \mathcal{S}_h (s, \xb(x,v) -
 \frac{t-\tb(x,v)-s}{\e} \mathfrak{v}, \mathfrak{v}
  )
% e^{
%- \int^{t-\tb(x,v)}_0 \frac{\nu  }{\e^2 \kappa}  
%} 
% h(0,\xb(x,v)- (t-\tb(x,v)) \mathfrak{v}, \mathfrak{v})
   \frac{ \sqrt{\mu(\mathfrak{v})} |\mathfrak{v}_3| }{\mathfrak{w} (\mathfrak v) }
   \dd s 
   \dd \mathfrak{v}\notag% \label{bdry:S}
    %\\
 %   &  - %\frac{\e}{\delta} \mathfrak{w} (1- P_{\gamma_+}) f_2
  \\&  +  r (t-\tb(x,v),\xb(x,v),  v )  ,\notag% \label{bdry:f_2}
 \end{align}
 where $r= - \e^{1/2} \mathfrak{w} (1- P_{\gamma_+}) f_2$ and $e^{- \int^{t-\tb(x,v)}_0 \frac{\nu  }{\e^2 \kappa}  
}:=  e^{-\int^{t-\tb(x,v)}_0 \frac{1}{\e^2 \kappa} 
{\nu (\tau , x- \frac{ \tb(x,v) }{\e}v
- \frac{t-\tb(x,v) -s }{\e} \mathfrak{v}
, \mathfrak{v} )}\dd \tau}$. 

Note that, from (\ref{estR11}), (\ref{estR12}), (\ref{estR21}), (\ref{estR22}), and (\ref{est:f2}), 
\Be\label{est:S_h}
\begin{split}
&  |\mathcal{S}_h  (s, x- \frac{t-s}{\e} v,v)|
\\ &  \lesssim   \nu(v)\frac{1}{\kappa \e^{1/2}}\| h  \|_\infty^2 + \frac{\nu(v)}{\kappa \e}
\left(  \| \mathcal I  \|_\infty + \e \|(\ref{est:f2})\|_\infty \right)
\| h \|_\infty  
\\ & \ \ \ + \frac{1}{\e^{1/2}} \left( \| (\ref{estR11})\|_\infty + \frac{1}{\kappa} \| \kappa \eqref{estR12} \|_\infty  \right)+   \e^{1/2} \left( \| (\ref{estR21})\|_\infty + \frac{1}{\kappa} \|  \kappa \eqref{estR22} \|_\infty \right) ,
 \\ &  | \mathfrak{w} (1- P_{\gamma_+}) f_2|  \lesssim   
\|(\ref{est:f2})\|_\infty.
  \end{split}
\Ee

We derive a preliminary estimate as 
\begin{align}
&|h(t,x,v) | \lesssim e^{- \frac{\nu }{2\e^2 \kappa}t} \| h(0)\|_\infty \notag
\\ &+  \e^{3/2}  \sup_{0 \leq s \leq t}\| h(s)  \|_\infty^2  + \e^2  \sup_{0 \leq s \leq t} \|(\ref{est:f2})\|_\infty \| h(s) \|_\infty + \e \sup_{0 \le s \le t}  \|\mathcal I \|_\infty \| h(s) \|_\infty \notag
 \\ & +\e^{1/2}  \sup_{0 \leq s \leq t} \|(\ref{est:f2})\|_\infty +  {\e^{3/2} \kappa  }  \left( \| (\ref{estR11})\|_\infty + \frac{1}{\kappa} \| \kappa \eqref{estR12} \|_\infty  \right)+   \e^{5/2} \kappa \left( \| (\ref{estR21})\|_\infty + \frac{1}{\kappa} \|  \kappa \eqref{estR22} \|_\infty \right)  \label{infty1} 
  \\ &+  \int^t_0  \frac{e^{- \frac{\nu }{2\e^2 \kappa}(t-s)}}{\e^2 \kappa}
\int_{\R^3} \mathbf{k}_{\mathfrak{w}} (v,v_*) |h(s,x- \frac{t-s}{\e}, v_*)| \dd v_* \dd s \label{K1}
\\ &  +\mathfrak{w} (v)
c_\mu   \sqrt{\mu(v)}  \int_{\mathfrak{v}_3<0}  \int^{t-\tb(x,v)}_0  \frac{ e^{-   \frac{\nu }{ 2\e^2 \kappa} (t-s) } }{\e^2 \kappa}\notag
\\ &  \ \ \ \ \ \times  \int_{\R^3}  \mathbf{k}_{\mathfrak{w}}(\mathfrak{v} ,v_*) |h (s, \xb(x,v) - \frac{t-\tb(x,v)-s}{\e} \mathfrak{v}, v_*  )| \dd v_* \dd s   \frac{ \sqrt{\mu(\mathfrak{v})} |\mathfrak{v}_3| }{\mathfrak{w} (\mathfrak v) }\dd \mathfrak{v}.\label{K2} 
\end{align}
We note that $|h(s,x- \frac{t-s}{\e}, v_*)|$ has the same upper bound. Then we bound (\ref{K1}) by a summation of  $(\ref{infty1})$ and
\Be 
\begin{split}\label{K2_1}
\sup_{
\substack{(\xb, v) \in \p\O \times \R^3  \\
t-\tb\geq 0
}}
\mathfrak{w} (\xb , v )
&c_\mu   \sqrt{\mu(v)}
 \int_{\mathfrak{v}_3<0} 
 \int^{t-\tb}_0 
 \frac{ e^{-  
 \frac{\nu }{ 2\e^2 \kappa} (t-s)
 } }{\e^2 \kappa}\\
 & \times 
 \int_{\R^3}
  \mathbf{k}_{\mathfrak{w}}(\mathfrak{v} ,v_*) |h (s, \xb  -
 \frac{t-\tb -s}{\e} \mathfrak{v}, v_*
  )| \dd v_* \dd s 
% e^{
%- \int^{t-\tb(x,v)}_0 \frac{\nu  }{\e^2 \kappa}  
%} 
% h(0,\xb(x,v)- (t-\tb(x,v)) \mathfrak{v}, \mathfrak{v})
   \frac{ \sqrt{\mu(\mathfrak{v})} |\mathfrak{v}_3| }{\mathfrak{w} (\xb , \mathfrak{v}) }\dd \mathfrak{v},
\end{split}
\Ee
and importantly 
\Be
\begin{split}
\int^t_0 & \frac{e^{- \frac{\nu(v)}{2\e^2 \kappa}(t-s)}}{\e^2 \kappa}
\int_{\R^3} \mathbf{k}_{\mathfrak{w}} (v,v_*) 
\int^s_0
 \frac{e^{- \frac{\nu(v_*)}{2\e^2 \kappa}(s-\tau)}}{\e^2 \kappa}\\ &\times 
 \int_{\R^3}\mathbf{k}_{\mathfrak{w}} (v_*,v_{**}) 
 |h(s,x- \frac{t-s}{\e}v - \frac{s-\tau}{\e}v_* , v_{**})| 
 \dd v_{**}
\dd \tau
\dd v_*
 \dd s. \label{double_K}
\end{split} 
\Ee
We consider (\ref{double_K}). We decompose the integration of $\tau \in [0,s] = [0, s- o(1)\e^2 \kappa] \cup [s- o(1)\e^2 \kappa, s]$. The contribution of $\int^s_{s-o(1) \e^2 \kappa} \cdots \dd \tau$ is bounded as 
\Be\label{bound:small_t}
\frac{2}{\nu(v)}\big(1- e^{- \frac{\nu(v)}{2 \e^2 \kappa}}\big)
 \| \mathbf{k}_{\mathfrak{w}}(v , \cdot)\|_{L^1}
\frac{o(1) \e^2 \kappa}{\e^2 \kappa}
 \| \mathbf{k}_{\mathfrak{w}}(v_*, \cdot)\|_{L^1}
 \sup_{0 \leq s \leq t}\| h(s) \|_\infty\leq o(1) \sup_{0 \leq s \leq t}\| h(s) \|_\infty.
\Ee
For the rest of term we decompose $\mathbf{k}_{\mathfrak{w} }(v_*,v_{**})= \mathbf{k}_{\mathfrak{w},N}(v_*,v_{**}) + \{ \mathbf{k}_{\mathfrak{w} }(v_*,v_{**})-  \mathbf{k}_{\mathfrak{w},N}(v_*,v_{**})\}$ where $\mathbf{k}_{\mathfrak{w},N}(v_*,v_{**}):=\mathbf{k}_{\mathfrak{w}}(v_*,v_{**})$ $\times  
\mathbf{1}_{ \frac{1}{N} <|v_*- v_{**}|< N 
 \  \& 
  \ 
|v_*|< N  
}.
$ From (\ref{est:K_w}), \\ $\int_{\R^3} \mathbf{k}_{\mathfrak{w}}(v_*,v_{**}) \mathbf{1}_{|v_*|\geq N} \dd v_{**} \lesssim 1/N$. Also from the fact $\mathbf{k}_{\mathfrak{w}}(v_*,v_{**}) \leq \frac{e^{-C |v_*- v_{**}|^2}}{|v_*- v_{**}|} \in L^1(\{ v_*- v_{**} \in \R^3 \})$, $\sup_{v_*}\int_{\R^3} \mathbf{k}_{\mathfrak{w}}(v_*,v_{**}) \{\mathbf{1}_{\frac{1}{N}\geq |v_*- v_{**}| } + \mathbf{1}_{  |v_*- v_{**}|\geq N }  \}\dd v_{**} \downarrow 0$ as $N \rightarrow \infty$. Hence for $N\gg1$
\Be\begin{split}
(\ref{double_K})& \leq  \int^t_0  \frac{e^{- \frac{\nu(v)}{2\e^2 \kappa}(t-s)}}{\e^2 \kappa}
\int_{\R^3} \mathbf{k}_{\mathfrak{w} ,N} (v,v_*) 
\int^{s- o(1) \e^2 \kappa}_0
 \frac{e^{- \frac{\nu(v_*)}{2\e^2 \kappa}(s-\tau)}}{\e^2 \kappa}\\
 & \ \ \ \  \times 
% \frac{1}{\e^2 \kappa}%e^{- \frac{\nu(v_*)}{2\e^2 \kappa}(s-\tau)}
 \int_{\R^3}\mathbf{k}_{\mathfrak{w},N } (v_*,v_{**}) 
 |h(s,x- \frac{t-s}{\e}v - \frac{s-\tau}{\e}v_* , v_{**})| 
 \dd v_{**}
\dd \tau
\dd v_*
 \dd s  \\
 &\leq C_N  \int^t_0  \frac{e^{- \frac{\nu(v)}{2\e^2 \kappa}(t-s)}}{\e^2 \kappa}
\int_{|v_*| \leq 2N}  
\int^{s- o(1) \e^2 \kappa}_0
  \frac{e^{- \frac{\nu(v_*)}{2\e^2 \kappa}(s-\tau)}}{\e^2 \kappa}
 %\frac{1}{\e^2 \kappa}%e^{- \frac{\nu(v_*)}{2\e^2 \kappa}(s-\tau)}
\\
 & \ \ \ \  \times  \int_{|v_{**}| < 2N} 
 |f_R(s,x- \frac{t-s}{\e}v - \frac{s-\tau}{\e}v_* , v_{**})| 
 \dd v_{**}
\dd \tau
\dd v_*
 \dd s
 \label{int:Kf}
 \\
 & \ \ \ + o(1) \sup_{0 \leq s \leq t} \| h(s)\|_{L^\infty_{x,v}}, 
\end{split}\Ee 
where we have used the fact $\sup_{x}\mathbf{k}_{\mathfrak{w} } (v_*, v_{**}) \mathfrak{w}_{\varrho}(v_{**}) \leq C_N<\infty$ when $\frac{1}{N} <|v_*- v_{**}|< N$ and $|v_*|< N$ (then $|v_{**}|< 2N$).%, and the fact $\int_{\R^3} |\mathbf{k}_{\mathfrak{w},N}(v_*,v_{**}) - \mathbf{k}_{\mathfrak{w} }(v_*,v_{**})  | \dd v_{**} \downarrow 0$ as $N \rightarrow \infty$. 

Now we decompose $f_R=\mathbf{P}f_R+ (\mathbf{I} -\mathbf{P})f_R$. We first take integrations (\ref{int:Kf}) over $v_{*}$ and $v_{**}$ and use Holder's inequality with $p=6, p=2$ in $1/p+ 1/p^\prime=1$ for $\mathbf{P}f_R , (\mathbf{I} -\mathbf{P})f_R$ respectively to derive 
\Be\label{int:Kf1}
\begin{split}
&(\ref{int:Kf})\\
 \leq& \  (4N)^3C_N
%\int^t_0 \frac{1}{\e^2 \kappa} e^{- \frac{\nu(v)}{2 \e^2 \kappa}  (t-s)}
\frac{1}{\nu(v)}
\sup_{\substack{0 \leq s \leq t\\
0 \leq \tau \leq s- o(1) \e^2 \kappa
}}
% \| \mathbf{k}_{\mathfrak{w}}(v,\cdot )\|_{L^2}
%\int^{s-o(1) \e^2 \kappa}_0
\left(\iint_{|v_*|\leq N,|v_{**}| \leq 2 N}  |\mathbf{P}f_R(s, x- \frac{t-s}{\e} v - \frac{s-\tau}{\e} v_*, v_{**})|^6  \dd v_{**}\dd v_*\right)^{1/6}   \\
& + (4N)^3C_N
%\int^t_0 \frac{1}{\e^2 \kappa} e^{- \frac{\nu(v)}{2 \e^2 \kappa}  (t-s)}
\frac{1}{\nu(v)}\\
& \  \times 
\sup_{\substack{0 \leq s \leq t\\
0 \leq \tau \leq s- o(1) \e^2 \kappa
}}
% \| \mathbf{k}_{\mathfrak{w}}(v,\cdot )\|_{L^2}
%\int^{s-o(1) \e^2 \kappa}_0
\left(\iint_{|v_*|\leq N,|v_{**}| \leq 2 N}  | ( \mathbf{I}-\mathbf{P})f_R(s, x- \frac{t-s}{\e} v - \frac{s-\tau}{\e} v_*, v_{**})|^2  \dd v_{**}\dd v_*\right)^{1/2}  .
\end{split}
\Ee
Now we consider a map
\Be\label{COV}
v_* \in \{\R^3: |v_*| \leq N\} \mapsto y:=x- \frac{t-s}{\e} v - \frac{s-\tau}{\e} v_* \in \O, \ \ \text{where} \ \ 
\Big|\frac{\p y}{\p v_*}\Big|= \Big|\frac{s-\tau}{\e}\Big|^3 \gtrsim \e^3 \kappa^3.
\Ee
We note that this mapping is not one-to-one and the image can cover $\O$ at most $N$ times. Therefore we have 
\begin{align*}
&\left(\iint_{|v_*|\leq N,|v_{**}| \leq N}  |\mathbf{P}f_R(s, x- \frac{t-s}{\e} v - \frac{s-\tau}{\e} v_*, v_{**})|^6  \dd v_{**}\dd v_*\right)^{1/6}\\
&\leq N^{1/6} \left(\iint_{|v_*|\leq N,|v_{**}| \leq N}  |\mathbf{P}f_R(s, y, v_{**})|^6  \dd v_{**}
\frac{\dd y}{\e^3 \kappa^3}\right)^{1/6} \leq \frac{N^{1/6}}{\e^{1/2} \kappa^{1/2}} \| \mathbf{P} f_R(s) \|_{L^6_{x,v}},
\end{align*}
\begin{align*}
&\left(\iint_{|v_*|\leq N,|v_{**}| \leq N}  |( \mathbf{I}-\mathbf{P})f_R(s, x- \frac{t-s}{\e} v - \frac{s-\tau}{\e} v_*, v_{**})|^6  \dd v_{**}\dd v_*\right)^{1/6}\\
&\leq \frac{N^{1/2}}{\e^{3/2} \kappa^{3/2}} \| ( \mathbf{I}-\mathbf{P})f_R(s) \|_{L^2_{x,v}}.
\end{align*}
Therefore we conclude that 
\Be\label{est:double_K}
\begin{split}
&(\ref{double_K})\\ &\leq (4N)^3C_N(\ref{int:Kf1}) + o(1)\sup_{0 \leq s \leq t} \| h(s)\|_{L^\infty_{x,v}} \\
&\leq (4N)^4C_N\left\{
\frac{1}{\e^{1/2}\kappa^{1/2}} \sup_{0 \leq s \leq t}\| \mathbf{P}  f_R(s ) \|_{L^6_{x,v}} 
+
  \frac{1}{\e^{3/2}\kappa^{3/2}}  \sup_{0 \leq s \leq t} \|(\mathbf{I} -\mathbf{P}) f_R(s ) \|_{L^2_{x,v}} \right\}\\
  &  \ \ \ \ + o(1)\sup_{0 \leq s \leq t} \| h(s)\|_{L^\infty_{x,v}} \\
  &\lesssim_N  
\frac{1}{\e^{1/2}\kappa^{1/2}} \sup_{0 \leq s \leq t}\| \mathbf{P}  f_R(s ) \|_{L^6_{x,v}} 
+
  \frac{1}{\e^{3/2}\kappa^{3/2}} 
\Big\{ \|(\mathbf{I} -\mathbf{P}) f_R  \|_{L^2_{t,x,v}}
+  \|(\mathbf{I} -\mathbf{P}) \p_t f_R  \|_{L^2_{t,x,v}}
\Big\} \\
& \ \ \ \  + o(1)\sup_{0 \leq s \leq t} \| h(s)\|_{L^\infty_{x,v}},
\end{split}
\Ee
where we have used (\ref{Sob_1D}) the Sobolev embedding in 1D at the last line.%.  $|g(s )|^2=\big| \int_0^s \frac{d}{d \tau} | g(\tau )|^2 \dd \tau\big|\lesssim \int_0^s   |g(\tau)|^2 \dd \tau +  \int_0^s   |\p_t g(\tau )|^2 \dd \tau$ 

Now we consider (\ref{K2}) and (\ref{K2_1}). We decompose $s \in [0, t-\tb] = [0, t-\tb - o(1) \e^2 \kappa] \cup [t-\tb - o(1) \e^2 \kappa, t-\tb]$. The contribution of $\int^{t-\tb}_{t-\tb- o(1) \e^2 \kappa} \cdots $ is bounded as 
\Be
\frac{o(1) \e^2 \kappa}{\e^2 \kappa } \| \mathbf{k}_{\mathfrak{w}}(\mathfrak{v}, \cdot ) \|_{L^1} \sup_{0 \leq s \leq t} \| h(s) \|_\infty
\leq o(1)\sup_{0 \leq s \leq t} \| h(s) \|_\infty. \label{bound:small_t1}
\Ee
For $s \in  [0, t-\tb - o(1) \e^2 \kappa]$ we consider a map as (\ref{COV})
\Be\label{COV_1}
\mathfrak{v}  \in \{\mathfrak{v}  \in \R^3: \mathfrak{v} _3<0 \} \mapsto y:=\xb - \frac{t-\tb -s}{\e} \mathfrak{v} \in \O, \ \  \text{where} \ \ 
\left|\frac{\p y}{\p \mathfrak{v}}\right| = \left|\frac{t-\tb -s}{\e}\right|^3\gtrsim \e^3 \kappa^3. 
\Ee
Following the argument to have (\ref{int:Kf1}) we bound 
\Be\begin{split}
&\text{the contribution of} \ \int^{t-\tb- o(1) \e^2 \kappa}_0 \cdots \ \text{of } \ (\ref{K2_1})\\
& 
\lesssim_N \frac{1}{\e^{1/2} \kappa^{1/2}} \| \mathbf{P} f_R(s) \|_{L^6_
{x,v}} + \frac{1}{\e^{3/2} \kappa^{3/2}} \Big\{ \|(\mathbf{I} -\mathbf{P}) f_R  \|_{L^2_{t,x,v}}
%\\ & \ \ \ \ \ \ \ 
+  \|(\mathbf{I} -\mathbf{P}) \p_t f_R  \|_{L^2_{t,x,v}}
\Big\}   . \label{est:K_bdry}
\end{split}\Ee

In conclusion, we bound $|h(t,x,v)|$ by $(\ref{infty1})$, (\ref{est:double_K}), (\ref{bound:small_t1}), (\ref{est:K_bdry}) and conclude (\ref{Linfty_3D}) by choosing small enough $o(1)$ in (\ref{est:double_K}) and (\ref{bound:small_t1}). 
\end{proof}
\hide

 From (\ref{est:k}) and (\ref{est_infty:L_t}) it follows easily that  
 \Be\label{eqtn:h}
|  \p_t h +  \e^{-1} v\cdot \nabla_x h +\frac{\nu}{2\e^2 \kappa }  h | \leq \frac{1}{\e^2 \kappa }   \int_{\R^3}
k_{\vartheta, \varrho} (v,v_*)
|h(v_*)| \dd v_* + g, \ \ \ |h||_{\gamma_-} \leq  \mathfrak{w}P_{\gamma_+} | \mathfrak{w}^{-1}h| + r. 
 \Ee
 where $k_{\vartheta, \varrho}(v,v_*)$ has been defined in (\ref{k_the_rho}).
 
  \Be\notag
 K_\varsigma h:= \int_{\R^3} k_\varsigma (v,v_*)h(v_*) \dd v_*, \ \ \ where \ \   k_\varsigma (v,v_*):= C \frac{e^{- \varsigma |v-v_*|^2}}{|v-v_*|}.
 \Ee

 For some later purpose we set a notation 
\Be\label{k_the_rho}
k_{\vartheta, \varrho } (v,v_*) : = e^{- \vartheta{|v-v_*|^2} 
-  \vartheta\frac{(|v|^2 - |v_*|^2)^2}{|v-v_*|^2}}
\frac{e^{\varrho|v|^2}}{e^{\varrho |v_*|^2}} . 
\Ee

  The (scaled) characteristics $[Y,W]$ are given by 
 \Be\label{trajectory_e}\Bs
Y(s;t,x,v) &= X(t- \frac{t-s}{\e};t,x,v),\\
W(s;t,x,v) &= V(t- \frac{t-s}{\e};t,x,v).
\end{split} \Ee
Along the characteristics,
\Be\Bs\label{Duhamel_linear}
&\frac{d}{ds}\Big\{ h(s , Y(s;t,x,v), W(s;t,x,v))
e^{- \frac{\nu(v)}{\e^2 \kappa} (t-s) }
\Big\}\\
=& \  \frac{1}{\e^2 \kappa}K h (s , Y(s;t,x,v), W(s;t,x,v))e^{-\frac{\nu(v)}{\e^2 \kappa}  (t-s) }.
\end{split}\Ee
Then 
\Be\Bs\notag
h(t,x,v) =& \ h_0 (Y(0;t,x,v), W(0;t,x,v)) e^{-  
 \frac{\nu(v)}{\e^2 \kappa}  t  }
 \\
&+ \int^t_0 
\frac{ e^{
-  \frac{\nu(v)}{\e^2 \kappa} (t-s) }}{\e^2 \kappa }
 \int_{\R^3} k_w(W(s;t,x,v),u )
  h (s , Y(s;t,x,v), u) \dd u
 \dd s\\
 & +\cdots 
\end{split}\Ee
Without the boundary we would have 
\Bes
&&h(t,x,v) \\
&\sim&  
\int^t_0 
\frac{e^{-  \frac{\nu_0 (t-s)}{\e^{2 } \kappa} }}{\e^{2 } \kappa}
 \int_{|u| \leq N } 
 \int^s_0 \frac{e^{-  \frac{\nu_0 (s-s^\prime)}{\e^{2 } \kappa} }}{\e^{2 } \kappa} \int_{\R^3}
k_w ( \cdots , u^\prime) h(s^\prime,x- \frac{t-s}{\e}v - \frac{s-s^\prime}{\e}u,u^\prime)
\dd u^\prime \dd  s^\prime
  \dd u
 \dd s\\
 &\sim&\int \int \int^s_{s- o(1) \e^{2 } \kappa} \int + \int \int \int^{s- o(1) \e^{2 } \kappa }_0 \int\\
 &\sim& O(1)\| h\|_\infty+ 
 \int^t_0 
\frac{e^{-  \frac{\nu_0 (t-s)}{\e^{2 } \kappa} }}{\e^{2 } \kappa}\int^{s-o(1) \e^{2 } \kappa}_0 \frac{e^{-  \frac{\nu_0 (s-s^\prime)}{\e^{2 } \kappa} }}{\e^{2 } \kappa} 
\left\| h(s^\prime,x- \frac{t-s}{\e}v - \frac{s-s^\prime}{\e}u,u^\prime)\right\|_{L^p_{u,u^\prime}
(\{|u| \leq N \} \times \R^3)
}
 \dd s^\prime \dd s .
\Ees

Here the dimension enters the business. Assume that the function only depends on $(t, {x},v) \in \R \times \R^3 \times \R^3$. 
Now we consider a change of variables, for $s^\prime \in [0, s- \kappa \e^{2 } \kappa]$,
\Be
 {u} 
\mapsto  {X}: =  {x}- \frac{t-s}{\e} {v} - \frac{s-s^\prime}{\e} {u}.
\Ee
We compute 
\Be
\dd  {u} = \frac{1}{\left|\frac{s-s^\prime}{\e}\right|^3}\dd {X} \lesssim
\frac{1}{\e^3 \kappa^3 }\dd \munderbar{X}.
\Ee
Then 
\Bes
&&\left\| h(s^\prime, {x}- \frac{t-s}{\e} {v} - \frac{s-s^\prime}{\e} {u},u^\prime)\right\|_{L^p_{u,u^\prime}}\\
&\lesssim& \frac{1}{ \e^{3/p} \kappa^{3/p}} \| h(s^\prime, \cdot, \cdot) \|_{L^p}.
\Ees
Therefore for $(x,v) \in \{ |\nabla \xi(x) \cdot v|> o(1) \} \cup \{|v| \leq \frac{1}{o(1)}\}$ we have $C=C(\O, o(1))>0$ such that 
\Be
|h(t,x,v)| \leq
e^{-\f{\nu(v) t}{\e^{2 }\kappa }} |h_0 (Y(0;t,x,v), W(0;t,x,v))|
+
\frac{ C}{ \e^{3/p} \kappa^{3/p}} 
 \int^t_0 
\frac{e^{-  \frac{\nu_0 (t-s)}{\e^{2 }\kappa } }}{\e^{2 }\kappa}
\| h(s) \|_p 
 \dd s 
 .
\Ee
Roughly we will have 
\Be
\| h \|_\infty \lesssim \frac{\e \kappa^{1/2}}{(\e \kappa)^{3/2}} \| \e^{-1} \kappa^{1/2}(\mathbf{I} - \mathbf{P}) f_R \|_2 + \frac{\kappa^{-1/2}}{(\e \kappa)^{3/6}} \| \kappa^{1/2} \mathbf{P}f_R \|_{L^6} \lesssim \frac{1}{\e^{1/2} \kappa}. 
\Ee

\hide

Then, for each time interval $[N, N+1]$, we obtain that 
\Be
\| h(N+1) \|_\infty \leq e^{- \frac{\nu_0}{\e^{2+\kappa}}} \| h(N) \|_\infty + \frac{C}{ \e^{2(1+ \kappa)/p}}
\int^{N+1}_N \frac{e^{- \frac{\nu_0}{\e^{2+ \kappa}}(N+1-s) }}{\e^{2+ \kappa}} \| h(s) \|_p \dd s + C,
\Ee
where $C>0$ does not depend on $N$.

Then inductively we derive 
\Be
\Bs
&\| h(N+1) \|_\infty \\
\leq & \ e^{- \frac{\nu_0 2}{\e^{2+\kappa}}} \| h(N-1) \|_\infty+ \frac{C}{ \e^{2(1+ \kappa)/p}} 
\int^{N+1}_{N-1} \frac{e^{- \frac{\nu_0}{\e^{2+ \kappa}}(N+1-s) }}{\e^{2+ \kappa}} \| h(s) \|_p \dd s 
+ C \sum_{j=0}^1 e^{- \frac{\nu_0}{\e^{2+\kappa}}j}\\
&\cdots
\\
\leq &  \ e^{- \frac{\nu_0 (N+1)}{\e^{2+ \kappa}}} \| h(0) \|_\infty + \f{C}{ \e^{2(1+ \kappa)/p}}
\int^{N+1}_0  \frac{e^{- \frac{\nu_0}{\e^{2+ \kappa}}(N+1-s) }}{\e^{2+ \kappa}} \| h(s) \|_p \dd s 
+ \frac{C}{1- e^{- \frac{\nu_0}{\e^{2+\kappa}}}}.
\end{split}
\Ee

\unhide

\unhide

\begin{proof}[\textbf{Proof of Proposition \ref{prop:Linfty_t}}]
Since many parts of the proof are overlapped with the proof of Proposition \ref{est:Linfty} we only pin point the differences. An equation for $\mathfrak{w}^\prime \p_t f_R$ takes the similar form of (\ref{eqtn:h}) and (\ref{bdry:h}). We can read (\ref{eqtn_fR_t}) for 
\Be\label{h_t}
h(t,x,v) = \mathfrak{w} ^\prime(x,v) \p_t  f_R(t,x,v), \ \ \text{for} \ \varrho^\prime<\varrho,
\Ee
 as (\ref{eqtn:h}) and (\ref{bdry:h}) replacing  
\Be\label{S_h_t}
\begin{split}
\mathcal{S}_h & =  \frac{2}{\kappa \e} \Gamma_{\mathfrak{w}^\prime}( \mathfrak{w'} f_1 +  \e {\mathfrak{w}^\prime}{ f_2},h) 
+    \frac{ 2  }{\e^{1/2} \kappa}\Gamma_{\mathfrak{w}^\prime}( {\mathfrak{w}^\prime} f_R,h)+  \frac{2}{\kappa \e } \Gamma_{\mathfrak{w}^\prime}( \mathfrak{w'} \p_t f_1 +  \e {\mathfrak{w}^\prime}\p_t {f_2}, {\mathfrak{w}^\prime}f_R)  
\\ & - \frac{1}{\e^{1/2}} \left \{ \mathfrak{w'} (\mathbf{I-P}) (v \cdot \nabla_x \p_t f_2) - \frac{2}{\kappa} \Gamma_{\mathfrak{w'}} (\mathfrak{w'} \p_t f_1, \mathfrak{w'} f_2) - \frac{2}{\kappa} \Gamma_{\mathfrak{w'}} ( \mathfrak{w'} f_1, \mathfrak{w'} \p_t f_2) \right \} 
 \\ & + \e^{1/2}  \left( -\mathfrak{w'}  \p_t^2 f_2  + \frac{2 }{\kappa} \Gamma_{\mathfrak{w'}} (\mathfrak{w'} \p_t f_2, \mathfrak{w'} f_2)  \right),
\\
 r&= - \e^{1/2} {\mathfrak{w}^\prime}  (1-P_{\gamma_+})  \p_tf_2.
\end{split}
\Ee

We have the same equality of (\ref{h:K}), (\ref{h:bdry}) with (\ref{bdry:K}) for $h$ of (\ref{h_t}) but replacing $S_h$ and $r$ of (\ref{S_h_t}). From \eqref{estptf2}, \eqref{estR11t}, \eqref{estR12t}, \eqref{estR21t}, \eqref{estR22t}, and (\ref{est:Gamma_w}), we bound terms of (\ref{S_h_t}) 
  \Be
\begin{split}
|S_h|& \lesssim     \nu(v)
\big\{\frac{1}{\kappa \e } \left( \|\mathcal I \|_\infty  + \e \| (\ref{est:f2}) \|_\infty \right) +\frac{1}{\kappa^{1/2} \e } \| \mathfrak{w} f_R\|_\infty  \big\} \| h\|_\infty   
\\ & \ \  + \frac{\nu(v)}{\kappa \e} \big\{ \|  \p_t \mathcal I \|_\infty + \e \| \eqref{estptf2} \|_\infty \big \} \| \mathfrak w f_R \|_\infty
 \\& \ \ + \frac{1}{\e^{1/2}} \left( \| \eqref{estR11t} \|_\infty + \frac{1}{\kappa} \| \kappa \eqref{estR12t} \|_\infty \right) + \e^{1/2} \left( \| \eqref{estR21t} \|_\infty + \frac{1}{\kappa} \| \kappa \eqref{estR22t} \|_\infty \right)
 ,
  \label{est:S_h_t1}
  \end{split}
  \Ee%  +  \frac{1}{\delta}

\Be \label{est:S_h_t2}
\begin{split}
|r| \lesssim & \  \e^{1/2} \| \eqref{estptf2} \|_\infty.
\end{split}\Ee 
  
  \hide
\Be
\begin{split}
|S_h| \lesssim& \   \nu(v)
\big\{\frac{1}{\kappa} (\| p \|_\infty + \| \tilde{u} \|_\infty + \kappa \| \nabla_x u \|_\infty) +\frac{\delta}{\kappa \e } \| \mathfrak{w} f_R\|_\infty 
 \big\}
 \| h\|_\infty\\
 & 
+ \frac{\nu(v)}{\kappa} \big\{
|\p_t p|  + |\p_t \tilde{u}| + \e |\p_t u| (|p| + |\tilde{u}|) + \kappa (|\nabla_x \p_t u| + \e |\p_t u| |\nabla_x u|)
\big\}\| \mathfrak{w} f_R\|_\infty\\
&+  \e (|\p_t^2 u| + |u||\nabla_x \p_t u| + |\p_t u| |\nabla_x u| )+ \e^2 |\p_t u| (|\p_t u| + |u||\nabla_x u|) \| \mathfrak{w} f_R\|_\infty +  \frac{1}{\delta}
\mathfrak{q}(
|\nabla_{x} \p_t \tilde{u}|, 
|\nabla_x^2 \p_t u | 
)\\
&
+  \frac{\e}{\delta \kappa}\{1+ 
|\p_t^2 p| + |\nabla_x \p_t^2 u| 
\} 
\\
& \ \ \   \ \ \  \times 
 \mathfrak{q} (
|p|, |\nabla_x p|, |\p_t p|,|\nabla_x \p_t p | , |u| , |\nabla_x u|, |\p_t u|, |\nabla_x \p_t u|, |\nabla_x^2 u|, |\nabla_x^2 \p_t u|, 
|\tilde{u}|, |\nabla_x \tilde{u}|,|\p_t \tilde{u}| ,|\nabla_x \p_t \tilde{u}|, |\p_t ^2 \tilde{u}|
) ,\\
|r| \lesssim & \ \frac{\e}{\delta}\{ |\p_t p| +|\p_t \tilde{u}| + \kappa |\nabla_x \p_t u |\}
+\e  |\p_t u | \{ \| \mathfrak{w} f_R\|_\infty+ 
 |  p| +|  \tilde{u}| + \kappa |\nabla_x   u |
\}.\label{est:S_h_t}
\end{split}
 \Ee
 \unhide
 
 Then as in  (\ref{infty1})-(\ref{double_K}) we derive a preliminary estimate as 
 \begin{align}
& |h(t,x,v) |
\notag
\\
& \lesssim    \ e^{- \frac{\nu }{2\e^2 \kappa}t} \| h(0)\|_\infty+ 
 \frac{\e^2 \kappa }{\nu(v)}
 (\ref{est:S_h_t1})
 + (\ref{est:S_h_t2})
 \label{infty1_t}
  \\
&+  \int^t_0  \frac{e^{- \frac{\nu }{2\e^2 \kappa}(t-s)}}{\e^2 \kappa}
\int_{\R^3} \mathbf{k}_{\mathfrak{w}^\prime} (v,v_*) |h(s,x- \frac{t-s}{\e}, v_*)| \dd v_*
 \dd s
 \label{K1_t}
 \\
 &  +\mathfrak{w}^\prime (\xb(x,v), v )
c_\mu   \sqrt{\mu(v)}
 \int_{\mathfrak{v}_3<0} 
 \int^{t-\tb(x,v)}_0 
 \frac{ e^{-  
 \frac{\nu }{ 2\e^2 \kappa} (t-s)
 } }{\e^2 \kappa}\notag\\
 & \ \ \ \ \     \times 
 \int_{\R^3}
  \mathbf{k}_{\mathfrak{w}^\prime}(\mathfrak{v} ,v_*) |h (s, \xb(x,v) -
 \frac{t-\tb(x,v)-s}{\e} \mathfrak{v}, v_*
  )| \dd v_* \dd s 
% e^{
%- \int^{t-\tb(x,v)}_0 \frac{\nu  }{\e^2 \kappa}  
%} 
% h(0,\xb(x,v)- (t-\tb(x,v)) \mathfrak{v}, \mathfrak{v})
   \frac{ \sqrt{\mu(\mathfrak{v})} |\mathfrak{v}_3| }{\mathfrak{w}^\prime (\xb(x,v), \mathfrak{v}) }\dd \mathfrak{v}.\label{K2_t} 
\end{align} 
\hide\begin{align}
&|h(t,x,v) |\notag
\\
 \lesssim &  \ e^{- \frac{\nu }{2\e^2 \kappa}t} \| h(0)\|_\infty
 + \frac{\e}{\delta}\{ |\p_t p| +|\p_t \tilde{u}| + \kappa |\nabla_x \p_t u |\}
+\e  |\p_t u | \{ \| \mathfrak{w} f_R\|_\infty+ 
 |  p| +|  \tilde{u}| + \kappa |\nabla_x   u |
\}
\label{infty1_t}
\\
&
 + \e^2  \kappa \big\{ \frac{1}{\kappa}
  (\| p \|_\infty + \| \tilde{u} \|_\infty + \kappa \| \nabla_x u \|_\infty) +\frac{\delta}{\kappa \e } \| \mathfrak{w} f_R\|_\infty 
 \big\}
 \sup_{0 \leq s \leq t}\| h(s)  \|_\infty
 \label{infty2_t}
  \\
 & +  \e^2
  \big\{
|\p_t p|  + |\p_t \tilde{u}| + \e |\p_t u| (|p| + |\tilde{u}|) + \kappa (|\nabla_x \p_t u| + \e |\p_t u| |\nabla_x u|)
\big\}\| \mathfrak{w} f_R\|_\infty
\label{infty3_t}\\
&+  \e^3 \kappa  (|\p_t^2 u| + |u||\nabla_x \p_t u| + |\p_t u| |\nabla_x u| )+ \e^4 \kappa  |\p_t u| (|\p_t u| + |u||\nabla_x u|) \| \mathfrak{w} f_R\|_\infty +  \frac{\e^2 \kappa}{\delta}
\mathfrak{q}(
|\nabla_{x} \p_t \tilde{u}|, 
|\nabla_x^2 \p_t u | 
)
\label{infty4_t} \\
 &+
  \frac{\e^3}{\delta }\{1+ 
|\p_t^2 p| + |\nabla_x \p_t^2 u| 
\} \notag
\\
& \ \ \   \ \ \  \times 
 \mathfrak{q} (
|p|, |\nabla_x p|, |\p_t p|,|\nabla_x \p_t p | , |u| , |\nabla_x u|, |\p_t u|, |\nabla_x \p_t u|, |\nabla_x^2 u|, |\nabla_x^2 \p_t u|, 
|\tilde{u}|, |\nabla_x \tilde{u}|,|\p_t \tilde{u}| ,|\nabla_x \p_t \tilde{u}|, |\p_t ^2 \tilde{u}|
)\label{infty5_t}
 \\
&+  \int^t_0  \frac{e^{- \frac{\nu }{2\e^2 \kappa}(t-s)}}{\e^2 \kappa}
\int_{\R^3} \mathbf{k}_{\mathfrak{w}^\prime} (v,v_*) |h(s,x- \frac{t-s}{\e}, v_*)| \dd v_*
 \dd s
 \label{K1_t}
 \\
 &  +\mathfrak{w}^\prime (\xb(x,v), v )
c_\mu   \sqrt{\mu(v)}
 \int_{\mathfrak{v}_3<0} 
 \int^{t-\tb(x,v)}_0 
 \frac{ e^{-  
 \frac{\nu }{ 2\e^2 \kappa} (t-s)
 } }{\e^2 \kappa}
 \int_{\R^3}
  \mathbf{k}_{\mathfrak{w}^\prime}(\mathfrak{v} ,v_*) |h (s, \xb(x,v) -
 \frac{t-\tb(x,v)-s}{\e} \mathfrak{v}, v_*
  )| \dd v_* \dd s 
% e^{
%- \int^{t-\tb(x,v)}_0 \frac{\nu  }{\e^2 \kappa}  
%} 
% h(0,\xb(x,v)- (t-\tb(x,v)) \mathfrak{v}, \mathfrak{v})
   \frac{ \sqrt{\mu(\mathfrak{v})} |\mathfrak{v}_3| }{\mathfrak{w}^\prime (\xb(x,v), \mathfrak{v}) }\dd \mathfrak{v}.\label{K2_t} 
\end{align} \unhide
As (\ref{K2_1}) and (\ref{double_K}), 
%We note that $|h(s,x- \frac{t-s}{\e}, v_*)|$ has the same upper bound. Then 
we bound (\ref{K1_t}) by a summation of  $(\ref{infty1_t})$ and 
\begin{align}
&\int^t_0  \frac{e^{- \frac{C_\nu%\nu(v)
}{2\e^2 \kappa}(t-s)}}{\e^2 \kappa}
\int^{s-o(1) \e^2 \kappa}_0
 \frac{e^{- \frac{C_\nu%\nu(v_*)
 }{2\e^2 \kappa}(s-\tau)}}{\e^2 \kappa}
\int_{|v_*| \leq 2N} %\mathbf{k}_{\mathfrak{w}^\prime} (v,v_*) 
\notag\\
& \ \ \ \ \ \times \int_{|v_{**} | \leq 2N}%\mathbf{k}_{\mathfrak{w}^\prime} (v_*,v_{**}) 
 |h(s,x- \frac{t-s}{\e}v - \frac{s-\tau}{\e}v_* , v_{**})| 
 \dd v_{**}
\dd v_*
\dd \tau
 \dd s , \label{double_K_t}\\
&+\sup_{
\substack{(\xb, v) \in \p\O \times \R^3  \\
t-\tb\geq 0
}}
\mathfrak{w}^\prime (\xb , v )
c_\mu   \sqrt{\mu(v)}
 \int_{\mathfrak{v}_3<0} 
 \int^{t-\tb -o(1) \e^2 \kappa}_0 
 \frac{ e^{-  
 \frac{\nu }{ 2\e^2 \kappa} (t-s)
 } }{\e^2 \kappa}\notag\\
 &\ \ \ \ \ \  \ \ \ \ \ \  \ \ \ \ \ \ \ \ \ \  \times 
 \int_{|v_*| \leq 2N}
 % \mathbf{k}_{\mathfrak{w}^\prime}(\mathfrak{v} ,v_*)
   |h (s, \xb  -
 \frac{t-\tb -s}{\e} \mathfrak{v}, v_*
  )| \dd v_* \dd s 
% e^{
%- \int^{t-\tb(x,v)}_0 \frac{\nu  }{\e^2 \kappa}  
%} 
% h(0,\xb(x,v)- (t-\tb(x,v)) \mathfrak{v}, \mathfrak{v})
   \frac{ \sqrt{\mu(\mathfrak{v})} |\mathfrak{v}_3| }{\mathfrak{w}^\prime (\xb , \mathfrak{v}) }\dd \mathfrak{v}
   \label{K2_1_t} 
   \\
   &+ o(1) \sup_{0 \leq s \leq t} \| h(s) \|_{L^{\infty}_{x,v}}
   .\label{bound:small_t1_t}
\end{align} 
Then we follow the argument of (\ref{int:Kf1})-(\ref{est:double_K}) to derive that, for $p<3$,
\begin{align}
|(\ref{double_K_t})| \lesssim&\  \int^t_0  \frac{e^{- \frac{C_\nu%\nu(v)
}{2\e^2 \kappa}(t-s)}}{\e^2 \kappa}
\int^{s-o(1) \e^2 \kappa}_0
 \frac{e^{- \frac{C_\nu%\nu(v_*)
 }{2\e^2 \kappa}(s-\tau)}}{\e^2 \kappa}
 \frac{N^{1/3}}{\e^{3/p} \kappa^{3/p}}
 \| \mathbf{P} \p_t f (\tau) \|_{L^p_{x,v}}
 \dd \tau \dd s\label{est:Kf_t:1}
 \\
 &+ \int^t_0  \frac{e^{- \frac{C_\nu%\nu(v)
}{2\e^2 \kappa}(t-s)}}{\e^2 \kappa}
\int^{s-o(1) \e^2 \kappa}_0
 \frac{e^{- \frac{C_\nu%\nu(v_*)
 }{2\e^2 \kappa}(s-\tau)}}{\e^2 \kappa}
 \frac{N^{1/2}}{\e^{3/2} \kappa^{3/2}}
 \| \mathbf{P} \p_t f (\tau) \|_{L^2_{x,v}}
 \dd \tau \dd s.\label{est:Kf_t:2}
\end{align}
Now we use the Young's inequality for temporal convolution twice to derive that, for $p<3$,  
\Be
\begin{split}\label{est:double_K_t}
&\|(\ref{double_K_t})\|_{L^2_t(0,T)}\\
\lesssim& \ 
\bigg\|\frac{e^{- \frac{C_\nu%\nu(v)
}{2\e^2 \kappa}|s |}}{\e^2 \kappa}  \bigg\|_{L^1_s(\R)}
\bigg\|\\
& \times 
\int^s_0
 \frac{e^{- \frac{C_\nu%\nu(v_*)
 }{2\e^2 \kappa}(s-\tau)}}{\e^2 \kappa}
 \bigg(
 \frac{N^{1/3}}{\e^{3/p} \kappa^{3/p}}
 \| \mathbf{P} \p_t f (\tau) \|_{L^p_{x,v}}
 +  \frac{N^{1/2}}{\e^{3/2} \kappa^{3/2}}
 \| (\mathbf{I}- \mathbf{P}) \p_t f (\tau) \|_{L^2_{x,v}}
 \bigg)
 \dd \tau
\bigg\|_{L^2_s (\R)} \\
\lesssim & \ 
\bigg\|\frac{e^{- \frac{C_\nu%\nu(v)
}{2\e^2 \kappa}|s |}}{\e^2 \kappa}  \bigg\|_{L^1_s(\R)}
\bigg\|
\frac{e^{- \frac{C_\nu%\nu(v)
}{2\e^2 \kappa}|\tau |}}{\e^2 \kappa} \bigg\|_{L^1_\tau (\R)} \\
& \times
\bigg(
\frac{N^{1/3}}{\e^{3/p} \kappa^{3/p} }
 \| P \p_t f  \|_{L^2_t ((0,T);L^p_{x}(\O))}
 + \frac{N^{1/2}}{\e^{3/2} \kappa^{3/2} }
 \|  (\mathbf{I}- \mathbf{P}) \p_t f  \|_{L^2 ((0,T) \times \O\times \R^3)}
 \bigg)\\
 \lesssim_N  & \ 
 \frac{1}{\e^{3/p} \kappa^{3/p} }
 \| P \p_t f  \|_{L^2_t ((0,T);L^p_{x}(\O))}
 +  \frac{1}{\e^{3/2} \kappa^{3/2} }
 \| ( \mathbf{I} - \mathbf{P}) \p_t f  \|_{L^2 ((0,T) \times \O\times \R^3)}
.
\end{split}
\Ee 

As in (\ref{est:K_bdry}), for (\ref{K2_1_t}) we use (\ref{COV_1}) to derive that, for $p<3$,  
\Be\label{est:K_bdry_t}
\begin{split}
&\|(\ref{K2_1_t})\|_{L^2_t (0,T)}\\
 \lesssim & \ 
\bigg\|\int^t_0 \frac{e^{- \frac{C_\nu}{2\e^2\kappa} (t-s)}}{\e^2 \kappa}
\bigg(
\frac{1}{\e^{3/p} \kappa^{3/p} }
 \| \mathbf{P} \p_t f (s) \|_{L^p_{x,v}}
 +  \frac{1}{\e^{3/2} \kappa^{3/2}}
 \| (\mathbf{I}- \mathbf{P}) \p_t f (s) \|_{L^2_{x,v}}
\bigg)
\dd s\bigg\|_{L^2_t (0,T)}
\\
 \lesssim & \ 
 \bigg\|\frac{e^{- \frac{C_\nu%\nu(v)
}{2\e^2 \kappa}|s |}}{\e^2 \kappa}  \bigg\|_{L^1_s(\R)}
\Big\{\frac{1}{\e^{3/p} \kappa^{3/p} }
 \| P \p_t f  \|_{L^2_t ((0,T);L^p_{x}(\O))}
 +  \frac{1}{\e^{3/2} \kappa^{3/2} }
 \| ( \mathbf{I} - \mathbf{P}) \p_t f  \|_{L^2 ((0,T) \times \O\times \R^3)}
\Big\}\\
\lesssim & \ 
\frac{1}{\e^{3/p} \kappa^{3/p} }
 \| P \p_t f  \|_{L^2_t ((0,T);L^p_{x}(\O))}
 +  \frac{1}{\e^{3/2} \kappa^{3/2} }
 \| ( \mathbf{I} - \mathbf{P}) \p_t f  \|_{L^2 ((0,T) \times \O\times \R^3)},
\end{split}
\Ee
where we have used the Young's inequality for temporal convolution. 

In conclusion, we bound $\|h\|_{L^2_t L^\infty_{x,v}}$ by $\|(\ref{infty1_t})\|_{L^2_t L^\infty_{x,v}},$ (\ref{est:double_K_t}), (\ref{bound:small_t1_t}), (\ref{est:K_bdry_t}) and conclude (\ref{Linfty_3D_t}) by choosing small enough $o(1)$ in (\ref{bound:small_t1_t}). 
\end{proof}

\subsection{Uniform controls of the Boltzmann remainder $f_R$ (Theorem \ref{main_theorem:conditional})}\label{sec:B1}

Inspired by the energy structure of the PDE and the coercivity of the linear operator $L$ in (\ref{s_gap}), we define an energy and a dissipation as 
\Be\label{ED}
\begin{split}
\mathcal{E} (t):= & \ \| f_R (t) \|_{L^2 (\O \times \R^3)}^2 + \| \p_t  f_R (t) \|_{L^2 (\O \times \R^3)}^2 ,\\
 \mathcal{D} (t) 
:=&  \ \int^t_0 \| \kappa^{-\frac{1}{2}} \e^{-1}
\sqrt{\nu}
 (\mathbf{I} - \mathbf{P}) f_R (s) \|_{L^2 (\O \times \R^3)}^2\dd s \\
 &  +\int^t_0  \|
 \kappa^{-\frac{1}{2}} \e^{-1} \sqrt{\nu}  (\mathbf{I} - \mathbf{P})\p_t  f_R (s) \|_{L^2 (\O \times \R^3)}^2
  \dd s \\
 &
 + \int^t_0 \Big(   | \e^{-\frac{1}{2}}
%(1- P_{\gamma_+}) 
f_R(s)
|_{L^2_\gamma  }^2 +  | \e^{-\frac{1}{2}}
%(1- P_{\gamma_+}) 
\p_t f_R(s)
|_{   L^2_\gamma}^2\Big) \dd s
.
\end{split}
\Ee
As explained in Section 1.2, the temporal derivative gets involved mainly in order to access the $L^6$-bound of the hydrodynamic part $\mathbf{P}f_R$, while we will control the following auxiliary norm to be used in order to handle the nonlinearity: for $p<3$ and $t>0$ 
\Be
\begin{split}
\mathcal{F}_p(t):= \sup_{0 \leq s \leq t} \Big\{&  \| \kappa^{1/2} Pf_R(s)\|_{L^6 (\O)} ^2
+ \| \kappa^{1/2} P   f_R\|_{L^2((0,s); L^p(\O))}^2\\
&+ \|  \kappa^{ \mathfrak{P}+1/2}   P \p_t f_R\|_{L^2((0,s); L^p(\O))}^2
+ \| \e^{1/2}\kappa \mathfrak{w}_{\varrho} f_R(s) \|_{L^{\infty}(\O \times \R^3) }^2
\\
&
+ \|   (\e  \kappa)^{3/p}  \kappa^{ \frac{1}{2}+\mathfrak{P}} \mathfrak{w}_{\varrho^\prime} f_R(s) \|_{L^2((0,s);L^{\infty}(\O \times \R^3)) }^2
\Big\}.
\end{split}
\Ee

We will use the norms of the initial data:
\Be\label{E} 
\mathcal{E}(0):=\mathcal{E} (f_{R,0}):=   \ \| f_{R,0}\|_{L^2 (\O \times \R^3)}^2 + \| \p_t  f_{R,0}  \|_{L^2(\O \times \R^3)}^2 ,\\
\Ee
\Be \label{initial_F}
\begin{split}
\mathcal{F}_p(0):=% \mathcal{F}_p(f_{R,in}):=
 &\big\{\kappa^{\frac{1}{2}} |  f_{R,0}  |_{L^2_\gamma}
+ \kappa^{\mathfrak{P}+ \frac{1}{2}}   | \p_t f_{R,0}   |_{L^2_\gamma}\\
&+
\e^{\frac{1}{2}} \kappa \| \mathfrak{w} f_
{R,0}\|_{L^\infty (\bar{\O} \times \R^3)}
+ (\e \kappa)^{1+ \frac{3}{p}}  \kappa^{\mathfrak{P}} \| \mathfrak{w}^\prime \p_t f_{R,0}\|_{L^\infty(\bar{\O} \times \R^3)}\big\}^2.
\end{split}\Ee

%\subsection{Proof of Theorem \ref{main_theorem:conditional}: a conditional statement of Theorem \ref{main_theorem}}

\ 

\begin{theorem}[Uniform controls of the Boltzmann remainder $f_R$]
\label{main_theorem:conditional}
Recall $\mathcal I$ in \eqref{HH2}. Suppose for $T>0$ and $\mathfrak{P}\geq 1/2$ 
\Be \label{mathfrak_C'}
  \sum_{\ell=0,1} \|\nabla_x\p_t^\ell  \mathcal I\|_{L^\infty ([0,T]\times \bar{\O})}+
 \frac{1}{\kappa^{1/2}} \sum_{\ell=0,1,2}\|  \p_t^\ell \mathcal I\|_{L^\infty ([0,T]\times \bar{\O})}
%+ \|  p\|_{L^\infty ([0,T]\times \bar{\O})} 
+ \frac{1}{\kappa^{1/2}}  \| p \|_{L^\infty ([0,T]\times \bar{\O})}
\lesssim 
\frac{1}{ \kappa^{  \mathfrak{P}}}.
\Ee  
We further assume that, for $0 \leq \mathfrak{P}^\prime < \mathfrak{P}$, 
\Be\label{condition:theorem}
\begin{split}
& 
\sum_{\ell=1,2}\| \p_t ^\ell \mathcal I \|_{L^\infty ([0,T]; L^\infty(\bar{\O})\cap L^2(  {\O})  )}
+\sum_{
 \substack{0 \leq \ell \leq 1 \\
  1 \leq  |\beta| \leq 2
 } } \|\nabla_x^\beta \p_t^\ell  \mathcal I\|_{L^\infty ([0,T]; L^\infty( \bar{\O})
 \cap L^2( {\O}) )}\\
 &+ \sum_{|\beta|=1} \|\nabla_x^\beta \p_t^2 \mathcal I \|_{L^2([0,T]; L^\infty( \bar{\O})
 \cap L^2( {\O}) )}
 \\
 &+  \| \p_t^2 p \|_{L^2 ([0,T]; L^\infty(\bar{\O}) \cap L^2(\O)   )}
% +  \| \p_t p \|_{L^\infty ([0,T]; L^\infty(\bar{\O}) \cap L^2(\O))}
  +
  \sum_{|\beta|=0,1} \| \nabla_x ^\beta \p_t  p \|_{L^\infty ([0,T]; L^\infty(\bar{\O}) \cap L^2(\O))}
 \lesssim  \exp \Big({\frac{1}{\kappa^{\mathfrak{P}^\prime }}}\Big)
 . 
\end{split}\Ee
For given such $T>0$, let us choose $\e$ and $\kappa$ as, for some $\mathfrak{C}\gg 1$, 
\Be\label{choice:delta}
\e =\exp\Big(  \frac{  - \mathfrak{C}  (T+1)}{ \kappa^{ 2 \mathfrak{P}} }\Big). 
% \Big\{1+    %\mathfrak{C}^\prime \kappa^{-\mathfrak{P}} T 
% \Big(  \frac{  \mathfrak{C}_2 T}{ \kappa^{ \mathfrak{P}} }\Big)^2   \exp\Big(  \frac{  \mathfrak{C}_2 T}{ \kappa^{ \mathfrak{P}} }\Big) 
   % \Big\}^{-1}    
\Ee
Assume that an initial datum for the remainder $f_{R,in }$ satisfies,  for some $p<3$ and $|p-3|\ll 1$,  
\Be\label{initial_EF}
 \sqrt{\mathcal{E}(0) }+ \sqrt{\mathcal{F} _p (0) } \lesssim1.
% \exp\Big( \frac{1}{\kappa^{\mathfrak{P}}}\Big)
\Ee
  Then we construct a unique solution $f_R(t,x,v)$ of the form of %:
 \Be\notag
 F = \mu+ \e f_1\sqrt{\mu}+ \e^2 f_2 \sqrt{\mu }  + \e^{3/2} f_R\sqrt{\mu}  \ \ \text{in} \ \ [0,T] \times \O \times \R^3,
 \Ee
which solves the Boltzmann equation (\ref{Boltzmann}) and the diffuse reflection boundary condition (\ref{diffuse_BC}) with the scale of (\ref{StMae}), (\ref{Knke}) and (\ref{choice:delta}), and satisfies the initial condition $F|_{t=0}=  \mu+ \e f_1|_{t=0} \sqrt{\mu}+ \e^2 f_2 |_{t=0}  \sqrt{\mu }  + \e^{3/2} f_R |_{t=0} \sqrt{\mu} 
  +  \delta \e f_{R,in}$, in a time interval $t \in [0,T]$. Moreover, we have 
\Be\label{est:E}
\e^{ \frac{1}{4}- \frac{3}{2p} (1- \frac{p}{3})} \sup_{0 \leq t \leq T}   \big\{ \sqrt{\mathcal{E}(t)} +  \sqrt{\mathcal{D}(t)}+  \sqrt{\mathcal{F}_{ p } (t)}\big\} 
 \lesssim 1.
\Ee
 
 \end{theorem}

\begin{proof}[\textbf{Proof of Theorem \ref{main_theorem:conditional}}]
An existence of a unique global solution $F$ for each $\e>0$ can be found in \cite{EGKM,EGKM2,EGKM3,EGM}. Thereby we only focus on the (a priori) estimates (\ref{est:E}).  
 
 \

\textbf{Step 1. }  
Fix $ \delta_0 >0 , \mathfrak C_1 >0,  \mathfrak C_2 >0$ such that $0< \frac{1}{ \sqrt{ \mathfrak C_1 }}  \ll \delta_0 \ll 1 $. And we choose $\e$ with
\Be\label{choice:delta1}
\e  \leq\left[  \frac{
\kappa^{  \frac{5}{2} + 3 \mathfrak{P} + \frac{3}{  p} (1- \frac{p}{3}) }
}{\mathfrak{C}_1  \big(
\mathcal{E}(0) + \mathcal{F}_p(0)+ 1  \big) }  
\exp\Big(  \frac{  -2\mathfrak{C}_2 (T+1)}{ \kappa^{ 2 \mathfrak{P}} }\Big) 
    \right]^{\frac{1}{\frac{1}{2}- \frac{3}{p} (1- \frac{p}{3})}}. 
\Ee
Then we define $T_*>0$ as  
\Be\label{condition:close}
\begin{split}
T_*=\sup   \Big\{ & t\geq 0:  \  
\min \{d_2, d_{2,t}, d_6, d_3, d_{3,t}, d_\infty, d_{\infty, t}\}
\geq \frac{\sigma_0}{4}, 
\\ &  \frac{\e}{\kappa} \sqrt{\mathcal{D}(s)} +  \e^{3/2} \| \mathfrak{w}_{\varrho} f(s) \|_{L^\infty_{x,v}} 
+ \frac{\e}{\kappa^{1+ \mathfrak{P} }}  \| Pf_R(s) \|_{L^2_x} < \delta_0, \text{ and }
\\ &   \e^{\frac{1}{4}- \frac{3}{2p} \big(1- \frac{p}{3}\big) } \kappa^{- \frac{3}{4} - \frac{\mathfrak{P}}{2} - \frac{3}{2 p} (1- \frac{p}{3}) } \left(  \sqrt{\mathcal{E}(s)} +  \sqrt{\mathcal{D}(s)} +   \frac{1}{\kappa^{1/2+ \mathfrak P}} \| f_R (s) \|_{L^2_{x,v}} \right) < \delta_0,
\\ & \text{ for all}  \ 0 \leq s \leq t 
 \Big\},
\end{split}
\Ee
where $d_2, d_{2,t}, d_6, d_3, d_{3,t}, d_\infty, d_{\infty, t}$ are defined in (\ref{d2}), (\ref{d2t}), (\ref{d6}), (\ref{d3}), (\ref{d3t}), (\ref{dinfty}) and (\ref{dinftyt}). Note that from \eqref{choice:delta1} we have $ \e^{\frac{1}{4}- \frac{3}{2p} \big(1- \frac{p}{3}\big) } \kappa^{- \frac{3}{4} - \frac{\mathfrak{P}}{2} - \frac{3}{2 p} (1- \frac{p}{3}) } < \delta_0$, so $T_* >0 $ is well defined.

First from \eqref{mathfrak_C'} and \eqref{condition:theorem}, 
\[ \begin{split}
& \| \eqref{est:f2} \|_{L^\infty_{t,x} } + \| \eqref{estR11} \|_{L^\infty_{t,x} } + \| \kappa  \eqref{estR12} \|_{L^\infty_{t,x} } + \| \eqref{estR21} \|_{L^\infty_{t,x} } + \| \kappa \eqref{estR22} \|_{L^\infty_{t,x} } 
\\ & + \| \eqref{estR11t} \|_{L^\infty_{t,x} } + \| \kappa \eqref{estR12t} \|_{L^\infty_{t,x} } + \| \eqref{estR21t} \|_{L^\infty_{t,x} } + \| \kappa \eqref{estR22t} \|_{L^\infty_{t,x} } + \| \eqref{estptf2} \|_{L^\infty_{t,x} }  \lesssim \exp \left( \frac{2}{ \kappa^{\mathfrak{P}^\prime} } \right).
\end{split} \]

Now from (\ref{condition:theorem}) and (\ref{condition:close}) we read all the estimates of Proposition \ref{prop:energy}, Proposition \ref{prop:L6}, Proposition \ref{est:Linfty}, Proposition \ref{prop:average}, and Proposition \ref{prop:Linfty_t} in terms of $\mathcal{E} (t)$ and $\mathcal{D} (t)$ as follows.

From (\ref{Linfty_3D}), (\ref{condition:close}), and (\ref{condition:theorem}) 
\Be 
 \begin{split}
 \label{Linfty_3D_ED}
 %(1- (o(1) + \frac{\e^{1/2} \delta}{\kappa } \mathcal{D}^{1/2} ))
 \sup_{0 \leq s \leq t} \| \mathfrak{w}_{\varrho} f _R(s)  \|_{L^\infty_{x,v}} 
 \lesssim  \  & %e^{- \frac{\nu }{2\e^2 \kappa}t}
%(o(1) + \frac{\e^{1/2} \delta}{\kappa } \mathcal{D}^{1/2} ) \sup_{0 \leq s \leq t} \| \mathfrak{w}_{\varrho} f _R(s) \|_\infty
% \frac{1}{\e^{1/2} \kappa^{1/2}
 %}
  %\sup_{0 \leq s  \leq t} \| \mathbf{P} f_R(s) \|_{L^6_
%{x,v}} 
 \frac{1}{\e^{1/2} \kappa^{1/2}
 } \sup_{0 \leq s  \leq t} \| \mathbf{P} f_R(s) \|_{L^6_
{x,v}}
+
 \frac{1}{\e^{1/2} \kappa } \sqrt{ \mathcal{D}(t) } 
 {+ \frac{1}{\e^{1/2}\kappa^{1+ \mathfrak{P}}} %\| \p_t u \|_{L^\infty_{t,x}} 
 \| P f _R\|_{L^2_{t,x}}
 }
 \\
&  
 + \|  \mathfrak{w}_{\varrho} f (0)\|_\infty
 + \e^{1/2} \exp \Big(\frac{3}{\kappa^{\mathfrak{P}^\prime}}\Big).
%
%
% +\frac{\e}{\delta}\| (\ref{est:f2}) \|_{L^\infty_{t,x}} + \e^2 \kappa\{  \| (\ref{est:R1})\|_{L^\infty_{t,x}}+\| (\ref{est:R2})\|_{L^\infty_{t,x}}\} .
%
%
%
\end{split}
 \Ee
 Now applying (\ref{Linfty_3D_ED}) to (\ref{L6}) we derive that  
\Be\label{L6_ED}
\begin{split}
\sup_{0 \leq s \leq t}\|   {P} f_R(s) \|_{L^6_x}\lesssim & \ 
\frac{\e}{\kappa} \exp \Big(\frac{1}{\kappa^{\mathfrak{P}^\prime}}  \Big)\sup_{0 \leq s \leq t} \sqrt{\mathcal{E}(s)}+
\frac{1}{\kappa^{1/2}} \sqrt{\mathcal{D}(t)} {+ \frac{1}{ \kappa^{1/2 + \mathfrak{P}} }% \| \p_t u \|_{L^\infty_{t,x}} 
\| P f _R\|_{L^2_{t,x}}} 
\\
&+ ( \e   \kappa)^{\frac{1}{2}}    \|  \mathfrak{w}_{\varrho} f (0)\|_{L^\infty_{x,v}} +\e^{1/2} \exp \Big(\frac{3}{\kappa^{\mathfrak{P}^\prime}}\Big) + \frac{1}{\kappa^{1+ \mathfrak P}} \sup_{0 \le s \le t} \| f_R \|_{L^2_{x,v}} .
 \end{split}
\Ee

   From (\ref{Linfty_3D_ED}), (\ref{L6_ED}), and (\ref{condition:close}) and (\ref{condition:theorem}) we conclude that
 \Be\label{L6Linfty_ED}
\begin{split}
& \sup_{0 \leq s \leq t} \big\{ \kappa^{\frac{1}{2}} \|   {P} f_R(s) \|_{L^6_x}  + 
 \e ^{\frac{1}{2}}\kappa \| \mathfrak{w}_{\varrho} f _R(s) \|_{L^\infty_{x,v}}\big\}  \\
 &
 \lesssim   \underbrace{  
\e^{1/2} \kappa^{1/2} \exp \Big(\frac{3}{\kappa^{\mathfrak{P}^\prime}}\Big)
 +  \sqrt{\mathcal{F}_{   p  }  (0)}   +  
\sup_{0 \leq s \leq t} \big\{\sqrt{\mathcal{E}(s)}
 +  \sqrt{\mathcal{D}(s)}
\big\}
{+
\frac{1}{\kappa^{\mathfrak{P}}}  \| P f_R \|_{L^2_{t,x}}  + \frac{1}{\kappa^{1/2+ \mathfrak P}} \sup_{0 \le s \le t} \| f_R \|_{L^2_{x,v}}
} 
 }_{(\ref{L6Linfty_ED})_*}.
\end{split}
\Ee

 From (\ref{average_3D}), (\ref{L6Linfty_ED}), (\ref{condition:close}) and (\ref{condition:theorem})
 \Be\begin{split}\label{average_3D_ED}
\kappa^{\frac{1}{2}} \|    {P} f_R
 \|_{L^2_t((0,t);L^p_x )}
 %-C_{T,N}   \| \mathfrak{w}_{\varrho, \ss } f_R (t) \|_{ L^\infty((0,T) \times \tilde{\O} \times \R^3)}^{\frac{p-2}{p}} \| (\mathbf{I} - \mathbf{P}) f_R \|_{L^2((0, T) \times \tilde{\O} \times \R^3)}^{\frac{2}{p}} 
\lesssim  
\underbrace{
(\ref{L6Linfty_ED})_* \left( 1 +  \frac{\e}{\kappa}(\ref{L6Linfty_ED})_*\right)
% + \Big( \e^{\frac{p+2}{2(p-2)} }\kappa^{\frac{2}{p-2}}  \Big)^{\frac{p-2}{p}} (\ref{L6Linfty_ED})_*^{2-\frac{2}{p}}
}_{(\ref{average_3D_ED})_*}.
\end{split}\Ee

Using (\ref{L6Linfty_ED}) and (\ref{mathfrak_C'}), from (\ref{average_3Dt}) and (\ref{Linfty_3D_t}), we deduce that, for $p<3$ and $\varrho^\prime<\varrho$,
\Be\begin{split}\label{average_3Dt_ED}
&  \kappa^{ \frac{1}{2}+\ss} \big\|    {P}  \p_t f_R
 \big\|_{L^2_t((0,t); L^p_x) }
 + (\e  \kappa)^{3/p}  \kappa^{ \frac{1}{2}+\ss} \| \mathfrak{w} _{\varrho^\prime}  \p_t  f_R  \|_{L^2_t((0,t);  L^{\infty}  _{x,v})} 
 %-C_{T,N}   \| \mathfrak{w}_{\varrho, \ss } f_R (t) \|_{ L^\infty((0,T) \times \tilde{\O} \times \R^3)}^{\frac{p-2}{p}} \| (\mathbf{I} - \mathbf{P}) f_R \|_{L^2((0, T) \times \tilde{\O} \times \R^3)}^{\frac{2}{p}} 
 \lesssim    (\ref{average_3D_ED})_*
%\underbrace{ (\ref{L6Linfty_ED})_*}_{(\ref{average_3Dt_ED})_*}.
\end{split}\Ee

%%%%%%%%%%%
 %%%%%%%%%%%
 %%%%%%%%%%%
 %%%%%%%%%%%
%%%%%%%%%%%
%%%%%%%%%%%
%%%%%%%%%%%
%%%%%%%%%%%
%%%%%%%%%%%

%%%%%%%%%%%%%
%%%%%%%%%%%%
%%%%%%%%%%%%
%%%%%%%%%%%%
%%%%%%%%%%%%
%%%%%%%%%%%%
%%%%%%%%%%%%
%%%%%%%%%%%%
%%%%%%%%%%%%%%%%%%%%%%%%

\

%\smallskip

\textbf{Step 2. } Using the estimates of the previous step we will close the estimate ultimately in the basic energy estimates (\ref{est:Energy}) and (\ref{est:Energy_t}) via the Gronwall's inequality. We note that from (\ref{mathfrak_C'}) the multipliers of $\int^t_0 \| Pf_R(s) \|_{L^2_x}^2 \dd s$ in (\ref{est:Energy}) and $\int^t_0 \| P \p_t f_R(s) \|_{L^2_x}^2 \dd s$ in (\ref{est:Energy_t}) are bounded above by 
\Be\label{multiplier}
O(1)  \kappa^{-2 \mathfrak{P}} \big(
1+ \e \kappa^{\frac{1}{2}-\mathfrak{P}} + ( \e \kappa^{\frac{1}{2}-\mathfrak{P}})^2
\big) \lesssim  \kappa^{-2 \mathfrak{P}} ,
\Ee
 where we have used (\ref{choice:delta}). 

In (\ref{est:Energy}) and (\ref{est:Energy_t}) we bound
\Be\label{interp_3}
\begin{split}
\| \kappa^{1/2} Pf_R \|_{L^2_t L^3_x}&\lesssim 
\kappa^{\frac{1}{2}(1- \frac{p}{3})}
\|   P f_R \|_{L^2_t L^\infty_x}^{1- \frac{p}{3}}
\| \kappa^{1/2} P  f_R \|_{L^2_t L^p_x}^{\frac{p}{3}}
\lesssim _T (\e\kappa)^{-\frac{1}{2} (1- \frac{p}{3})} 
|(\ref{L6Linfty_ED})_*|^{1- \frac{p}{3}} |(\ref{average_3D_ED})_*|^{\frac{p}{3}}
,\\
\|   P\p_t f_R \|_{L^2_t L^3_x}&\lesssim
\|   P\p_t f_R \|_{L^2_t L^\infty_x}^{1- \frac{p}{3}}
\|  P\p_t f_R \|_{L^2_t L^p_x}^{\frac{p}{3}}
\lesssim 
\e^{-\frac{3}{p}(1-\frac{p}{3}) } \kappa^{
 %  \bcb
   - \frac{1}{2} - \mathfrak{P} - \frac{3}{p} (1- \frac{p}{3})
  % \ec
   }
|(\ref{average_3D_ED})_*|.
\end{split}
\Ee

%We can check that the multiplier of $ \| \e^{-1}\kappa^{-1/2} \sqrt{\nu} (\mathbf{I} -\mathbf{P}) f_R \|_{L^2_{t,x,v}}^2$ in (\ref{est:Energy_t}) is bounded as,  from (\ref{condition:theorem}) and (\ref{L6Linfty_ED}), 
%\Be\begin{split}\notag
%& \Big\{
% \e (1+ \e \| (\ref{est:f2})\|_{L^\infty_{t,x}}) \| \p_t u \|_{L^\infty_{t,x}}
% + \e \kappa \| \nabla_x \p_t u \|_{L^\infty_{t,x}}
% + (\e \kappa^{1/2} \| (\ref{transp:mu_t})_*\|_{L^\infty_{t,x}})^2
% +( \e \delta \| \mathfrak{w} f_R \|_{L^\infty_{t,x,v}}) ^2 \Big\}\\
%&\lesssim  \e\kappa^{1/2-\mathfrak{P}}+ \e \delta^2 \kappa^{-2} |(\ref{L6Linfty_ED})_*|^2.   \end{split}\Ee

Applying (\ref{L6Linfty_ED}), (\ref{average_3D_ED}), (\ref{average_3Dt_ED}), (\ref{interp_3}) to $(\ref{est:Energy})+o(1)(\ref{est:Energy_t})$,  using the above bound and (\ref{choice:delta}),  and collecting the terms, we derive that    
\Be \begin{split}\label{energy_final1}
&\sup_{0 \leq s \leq t}\mathcal{E} (s) + (d_2 - o(1)) \mathcal{D}(t) 
\\ \lesssim & \      \mathcal{E}(0) + \mathcal{F} (0) + \exp \Big( \frac{6}{\kappa^{\mathfrak{P}^\prime}}\Big) + (\ref{multiplier})
\int^T_0 \mathcal{E} (s) \dd s
\\ &+ \e^{ 1 - (1- \frac{p}{3})}\kappa^{- 4 + \frac{p}{3}}
|(\ref{L6Linfty_ED})_*|^{4 - \frac{2p}{3}}  | \eqref{average_3D_ED}_* |^{\frac{2p}{3}}
%|(\ref{average_3Dt_ED})_*|^{\frac{2p}{3} }
  +  \e^{ 1 - \frac{6}{p} (1- \frac{p}{3})  }    \kappa^{-3 - 2\mathfrak{P} - \frac{6}{p}  (1- \frac{p}{3})}   |(\ref{L6Linfty_ED})_*|^{2 }  | \eqref{average_3D_ED}_* |^2.
%|(\ref{average_3Dt_ED})_*|^{2}.
  \end{split}\Ee  
Now from the last inequality in \eqref{condition:close},
\Be
\begin{split}\label{assump_Gronwall}
% &  \e^{\frac{p+2}{2(p-2)} }\kappa^{\frac{2}{p-2}} (\ref{L6Linfty_ED})_*
%\ll 1,  
 [ \e^{ 1 - (1- \frac{p}{3})}\kappa^{- 4 + \frac{p}{3}}]^{1/4}
 (\ref{L6Linfty_ED})_* \ll1,  \ \ 
  [ \e^{1 - \frac{6}{p} (1- \frac{p}{3})  }   
    \kappa^{-3 - 2\mathfrak{P} - \frac{6}{p}  (1- \frac{p}{3})}
 %   \ec
% \kappa^{-1 - \frac{6}{p}- 2 \mathfrak{P} }
]^{1/4}
    (\ref{L6Linfty_ED})_*  \ll1, 
\end{split}
\Ee
thus  we derive that, for $\mathfrak{C}_1>0$ and $\mathfrak{C}_2>0$ large enough,  
\Be\label{gronwall}
\sup_{0 \leq s\leq t}\mathcal{E}(s)  + \mathcal{D}(t) 
\leq \mathfrak{C}_1
\Big(\mathcal{E}(0) + \mathcal{F}_p(0) +\exp\Big(\frac{6}{\kappa^{\mathfrak{P}^\prime}}\Big) \Big)+  \mathfrak{C}_2 \kappa^{- 2 \mathfrak{P}} \int^t_0\mathcal{E}(s)  \dd s  .
\Ee  

%From (\ref{relat:kappa_delta}) and (??) we bound (\ref{final_E0}) by $o(1) \times \big[\sup_{0 \leq t \leq T}\mathcal{E} (t)+ \mathcal{D}(T) \big]$ which can be absorbed in the left hand side. From (??) there exists a constant $\mathfrak{C}>0$ independent of $\e,\kappa, \delta$, satisfying 
%\Be
%(\ref{final_E1})+(\ref{final_E2}) \ll \mathfrak{C}.\label{mathfrak_C}
%\Ee
% Note that the second condition condition in (\ref{assump_Gronwall}) in stronger, which is
%\Be\label{assump_Gronwall_S}
%\e^{\frac{1}{4}- \frac{3}{2p} \big(1- \frac{p}{3}\big) }
%\kappa^{- \frac{3}{4} - \frac{\mathfrak{P}}{2} - \frac{3}{2 p} (1- \frac{p}{3}) }
% (\ref{L6Linfty_ED})_* \ll 1.
%\Ee

Applying the Gronwall's inequality to (\ref{gronwall}) (we may redefine $\mathcal{E}(t)$ as $\sup_{0 \leq s \leq t}\mathcal{E}(s)$ if necessary), 
%we derive that 
%\Be\notag
%\begin{split}
%\sup_{0 \leq s \leq t}\mathcal{E} (s)%\lesssim (\mathcal{E}(0) + C)
%%+  (\mathcal{E}(0) + C) (\| \nabla_x u\|_\infty + \| \p_t u \|_{\infty}) 
%%t e^{(\| \nabla_x u\|_\infty + \| \p_t u \|_{\infty})t}
%%\leq  ( \mathfrak{C}_1\mathcal{E}(0) + \mathfrak{C}_2)
%%\Big\{1 +
%%\frac{\mathfrak{C}_3t}{\kappa^p}e^{\frac{\mathfrak{C}_3 t}{\kappa^p}}\Big\}
% &\leq 
%  \mathfrak{C} _1\Big(\mathcal{E}(0) + \mathcal{F}_p(0)+ \exp\Big(\frac{6}{\kappa^{\mathfrak{P}^\prime}}\Big)\Big) 
%  \Big\{
% 1+ \frac{\mathfrak{C}_2 t }{\kappa^{\mathfrak{P}}}
% \exp\Big(  \frac{\mathfrak{C}_2  t }{\kappa^{\mathfrak{P}}} \Big)\Big\} 
% .
% \end{split}
% \Ee 
% Applying this estimate to the last term of (\ref{gronwall}) 
and using the fact $\mathfrak{P}^\prime< \mathfrak{P}$ we derive that, after redefining $\mathfrak{C}_1$ if necessary,  
\Be 
\sup_{0 \leq s \leq t}\mathcal{E} (s) +   \mathcal{D}(t) + \mathcal{F}_p (t)  \leq      \mathfrak{C}_1  \big(
\mathcal{E}(0) + \mathcal{F}_p(0)+ 1 \big) 
     %\mathfrak{C}^\prime \kappa^{-\mathfrak{P}} T
%\bcb
  \exp\Big(  \frac{  2\mathfrak{C}_2 (t+1)}{ \kappa^{2  \mathfrak{P}} }\Big)
% \ec
    \ \ \text{for all } \ t \leq T_*,
  \label{est:E_1}
\Ee
under the assumptions of (\ref{condition:theorem}), (\ref{condition:close}), and (\ref{assump_Gronwall}).

%\smallskip

%\
%
%\textbf{Step 3. }   Now we find out the ranges of $ \kappa, \e$ satisfying the assumptions of (\ref{condition:close}) and (\ref{assump_Gronwall}). 
%From (\ref{initial_EF}) and (\ref{est:E_1}), if we choose $\e$ as
%\Be\label{choice:delta1}
%\e  \leq\left[  \frac{
%\kappa^{  \frac{5}{2} + 3 \mathfrak{P} + \frac{3}{  p} (1- \frac{p}{3}) }
%}{\mathfrak{C}_1  \big(
%\mathcal{E}(0) + \mathcal{F}_p(0)+ 1  \big) }  
%\exp\Big(  \frac{  -2\mathfrak{C}_2 (T+1)}{ \kappa^{ 2 \mathfrak{P}} }\Big) 
%% \Big\{1+    %\mathfrak{C}^\prime \kappa^{-\mathfrak{P}} T 
%% \Big(  \frac{  \mathfrak{C}_2 T}{ \kappa^{ \mathfrak{P}} }\Big)^2   \exp\Big(  \frac{  \mathfrak{C}_2 T}{ \kappa^{ \mathfrak{P}} }\Big) 
%   % \Big\}^{-1}    
%    \right]^{\frac{1}{\frac{1}{2}- \frac{3}{p} (1- \frac{p}{3})}}. 
%\Ee
%then we can achieve (\ref{assump_Gronwall_S}) and hence all conditions of (\ref{assump_Gronwall}). Clearly (\ref{choice:delta}) and (\ref{initial_EF}) ensure (\ref{choice:delta1}).

Now from (\ref{est:E_1}) and (\ref{choice:delta}) we derive (\ref{est:E}), 
%\Be
% \sup_{0 \leq s \leq t} \sqrt{\mathcal{E} (s)} +   \sqrt{\mathcal{D}(t)}   + \sqrt{\mathcal{F}_p (t)}
% \lesssim \delta^{-\frac{1}{2}+ \frac{3}{p} (1- \frac{p}{3})} ,\notag
%\Ee
which implies 
\Be
\begin{split}
&\sup_{0 \leq s \leq t} \Big\{ 
   \| \kappa^{1/2} Pf_R(s)\|_{L^6_x}  
 + \| \e^{1/2}\kappa \mathfrak{w}_{\varrho,\ss} f_R(s) \|_{L_{x,v}^{\infty} }
+ \|   (\e  \kappa)^{3/p}  \kappa^{ \frac{1}{2}+\mathfrak{P}} \mathfrak{w}_{\varrho^\prime,\ss} f_R(s) \|_{L^2((0,s);L_{x,v}^{\infty}) } 
\Big\}\\
&\lesssim \e^{-\frac{1}{4}+ \frac{3}{2p} (1- \frac{p}{3})}.\notag
\end{split}\Ee
These imply $\min \{d_2, d_{2,t}, d_6, d_3, d_{3,t}, d_\infty, d_{\infty, t}\}
\geq \frac{1}{4}$ and $\frac{ \e^{}}{\kappa} \sqrt{\mathcal{D}(t)}\ll \delta_0$ from (\ref{d2}), (\ref{d2t}), (\ref{d6}), (\ref{d3}), (\ref{d3t}), (\ref{dinfty}) and (\ref{dinftyt}). Moreover, we have 
\[ \e^{\frac{1}{4}- \frac{3}{2p} \big(1- \frac{p}{3}\big) } \kappa^{- \frac{3}{4} - \frac{\mathfrak{P}}{2} - \frac{3}{2 p} (1- \frac{p}{3}) } \left(  \sqrt{\mathcal{E}(s)} +  \sqrt{\mathcal{D}(s)} +   \frac{1}{\kappa^{1/2+ \mathfrak P}} \| f_R (s) \|_{L^2_{x,v}} \right) < \frac{1}{\sqrt \mathfrak C} \ll \delta_0.
\] 

Then by the standard continuation argument we can verify all assumptions (\ref{condition:close}) up to $t\leq T$ and $T=T_*$. The estimate (\ref{est:E}) follows easily.  
\end{proof}

\section{Fluid estimate} \label{fluidestimate}

We denote the vorticity by 
\Be
\o=\nabla \times   u, \ \ \ u =\nabla \times (- \Delta)^{-1} {\o},\label{vorticity}
\Ee
while the second identity  is the famous Biot-Savart law. Here $(- \Delta)^{-1}$ denotes the inverse of $-\Delta$ with the zero Dirichlet boundary condition on $\p\O$.

Our analysis of the Navier-Stokes-Fourier system is based on the vorticity formulation of the velocity field in 3D (\cite{Mae13,Mae14}):
  \begin{align}
\p_t \o - \kappa \eta_0 \Delta\o = - u \cdot \nabla \o
+ \o \cdot \nabla u \ \ &\text{in} \ \ \O
,  \quad \label{NS_o} \\
\o \,|_{t=0}   = \o_{in}  \ \ &\text{in} \ \ \O, \label{NSI_o} \\
\kappa \eta_0 (\p_{x_3} + \sqrt{- \Delta_h})\o_h \,    = [\p_{x_3}(- \Delta)^{-1} (-u \cdot \nabla \o_h 
+ \o \cdot \nabla  u_h
) ] \,  , \ \ \o_3 =0 \ \ &\text{on} \ \ \p\O, \label{NSB_o}
\\ \p_t \theta + u \cdot \nabla_x \theta  - \kappa \eta_c \Delta \theta   =0 \ \ & \text{in} \  \ \O,
\label{heat_NS}
\\  \ \ \theta = 0 \ \ & \text{on} \ \ \p\O \label{thetaB}
\end{align} 
 % at $\{x_3=0\}$, 
where $\sqrt{- \Delta_h}=|\nabla_h|$ is defined as 
\Be\label{sqrt_D}
 \sqrt{- \Delta_h}g(x_h,x_3)= \sum_{\xi \in \mathbb{Z}^2} |\xi| g_\xi (x_3)  e^{i x_h \cdot \xi}. 
\Ee Here, $g_{\xi } (x_3)= \frac{1}{(2\pi)^2} \iint_{\mathbb{T}^2} e^{-i x_h \cdot \xi} g(x_h,x_3) \dd x_h \in \mathbb C \text{ with } \xi= (\xi_1, \xi_2) \in \mathbb{Z}^2$ denotes the Fourier transform in the horizontal variables, which 
%in our domain $\O=\mathbb T^2\times \mathbb R_+$.   % of (\ref{domain}), 
%the Fourier transform in the horizontal variables 
%and $x_h=(x_1,x_2)$, 
 satisfies $g(x_1,x_2,x_3)%=g(x_h,x_3) 
 = \sum_{\xi \in \mathbb{Z}^2} g_{\xi} (x_3) e^{i x_h \cdot \xi }.$ The Fourier transform can be regarded as a function $g_{\xi} (z)$ where $z$ is sitting in a pencil-like complex domain: for any $\lambda> 0$, 
\Be\label{complex_domain}
\mathcal{H}_\lambda:=\Big\{ z \in \mathbb C : \text{Re}\,z\geq 0, \; | \text{Im}\, z| < \lambda \min \{ \text{Re}\,z, 1\} \Big\}.
\Ee

%Write
%
%\begin{align} 
%\p_t \theta - \kappa \eta_c \Delta \theta = - u \cdot \nabla_x \theta \ \ &\text{in} \ \ \O, \label{theta}
%\\ \theta(x,0 ) = \theta_0(x) \ \ &\text{in} \ \ \O, \label{theta0}
%\\ \theta(x) = 0 \ \ &\text{on} \  \p \O. \label{thetabdry}
%\end{align}
%Follow the argument to treat $\o_3$ and get the estimate without initial-boundary layer (something similar to $\o_3$).

\subsection{Higher regularity of Navier-Stokes-Fourier system in the inviscid limit (Theorem \ref{thm_bound})} 
In this section, we will use the following notations: $x=(x_h,x_3)=(x_1,x_2,x_3)\in \mathbb T^2\times \mathbb R_+=\Omega$, $\nabla_x=\nabla=(\nabla_h,\p_3)=(\p_{x_1},\p_{x_2},\p_{x_3})$;  for a vector valued function $g\in \mathbb R^3$, $g=(g_h,g_3)=(g_1,g_2,g_3)$. 

We define analytic function spaces without the boundary layer, $\mathfrak{L}^{p,\lambda}$, for holomorphic functions with a finite norm,  for $p\geq 1$,  
\Be\label{norm_L1}
\| g \|_{p,\lambda } := \sum_{\xi\in \mathbb Z^2} e^{\lambda |\xi|} \| g_{\xi } \|_{\mathcal{L}^p_\lambda}   \ \ \text{where} \ \ 
\|g_{\xi }\|_{\mathcal{L}^p_\lambda} := \sup_{0\leq \sigma\leq \lambda} 
\left(
\int_{\p \mathcal{H}_\sigma} | g_{\xi } (z)  |^p  |\dd z|
\right)^{1/p}.% \ \ \text{for} \ 1\leq p \leq \infty.
\Ee

Next we introduce an $L^\infty$-based analytic boundary layer function space, for $\lambda>0$ and $\kappa \geq 0$, that consists of holomorphic functions in $\mathcal{H}_\lambda$ with a finite norm 
\Be\label{norm_BL}
\begin{split}
\|g\|_{\infty,\lambda,\kappa}  = \sum_{\xi\in {\mathbb Z^2}} e^{\lambda |\xi|} \| g_{\xi } \|_{\mathcal{L}^\infty_{\lambda, \kappa}}  ,
%\sup_{z\in \Omega_\lambda}  e^{\bar{\alpha} \text{Re}\,z}  |g_{\xi } (z)| 
%\ \text{and} \  \ 
 % \| g_{\xi} \|_{L^\infty_{\lambda, \kappa}}  = 
 % \bigg\| \frac{e^{\bar{\alpha} \text{Re}\,z}}{ 1+ \phi_\kappa(z)} g_{\xi } (z) \bigg\|_{L^\infty_\lambda}
 % \sup_{z\in \Omega_\lambda} \frac{e^{\bar{\alpha} \text{Re}\,z}}{ 1+ \phi_\kappa(z)}  |g_{\xi} (z)|
%  \ \text{for} \ \kappa>0.
\end{split}
\Ee
where $\| g_{\xi } \|_{\mathcal{L}^\infty_{\lambda,0}} : = 
\| e^{\bar{\alpha} \text{Re}\,z}  g_{\xi } (z) \|_{\mathcal{L}^\infty_\lambda}
:= \sup_{z \in \mathcal{H}_\lambda}e^{\bar{\alpha} \text{Re}\,z}  g_{\xi } (z) 
$ and 
\Be\notag
 \| g_{\xi} \|_{\mathcal{L}^\infty_{\lambda, \kappa}}  := 
  \bigg\| \frac{e^{\bar{\alpha} \text{Re}\,z}}{ 1+ \phi_\kappa(z)} g_{\xi } (z) \bigg\|_{\mathcal{L}^\infty_\lambda}
  := \sup_{z \in \mathcal{H}_\lambda}\frac{e^{\bar{\alpha} \text{Re}\,z}}{ 1+ \phi_\kappa(z)} |g_{\xi } (z)| 
  .
\Ee
Here, a boundary layer weight function is defined  as 
\Be\label{BL}
\phi_\kappa (z):= \frac{1}{\sqrt\kappa} \phi ( \frac{z}{\sqrt\kappa}) 
 \ \ \text{with} \ \ \phi (z) = \frac{1}{ 1+|\text{Re}\,z|^\mathfrak{r}} \ \text{for some  } \mathfrak{r}>1.
\Ee
We define $\mathfrak B^{\lambda, \kappa}$ for holomorphic functions $g= (g_1,g_2,g_3)$ with a finite norm 
 \Be\label{[]}
[[g ]]_{ \infty, \lambda,\kappa}=\sum_{i=1,2} \| g_i \|_{\infty, \lambda, \kappa} + \| g_3 \|_{\infty, \lambda, 0}.
\Ee
And we say a scalar valued function $\tilde g: \O \to \mathbb R$ is in $\mathfrak B^{\lambda, \kappa}$ if $\| \tilde g \|_{\infty, \lambda, \kappa} < 0 $.

We note that $\mathfrak B^{\lambda, \kappa } \subset \mathfrak{L}^{1,\lambda}$, but $\mathfrak{B}^{\lambda, 0} \subsetneqq\mathfrak{L}^{\infty,\lambda}$ if $\bar{\alpha}>0$.

\smallskip

Due to its singular nature of the Navier-Stokes flow in the inviscid limit, we introduce the conormal derivatives
\Be
 D= (D_h, D_3)= (\nabla_h, \zeta(x_3) \p_3) \ \ \text{where}  \ \ \zeta(z) = \frac{z}{1+z}. 
\Ee
%while $\nabla_x = (\nabla_h, \p_3)= (\p_{x_1}, \p_{x_2}, \p_{x_3})$. 
With the multi-indices $\beta=(\beta_h,\beta_3):=(\beta_1,\beta_2,\beta_3)\in \mathbb N_0^3$, the higher derivatives are denoted by
%multi-index $a= (a_1,a_2,a_3) \in \R^3$ and $a_h= (a_1,a_2)$ {\color{red}\texttt{...multi-index $\beta=(\beta_1,\beta_2,\beta_3)\in \mathbb N_0^3$, and $\beta_h=(\beta_1,\beta_2)$?...}}
 $D^\beta = \p_1^{\beta_1} \p_2 ^{\beta_2} D_3^{\beta_3}$ and 
 $D^\beta_\xi = (i \xi_1)^{\beta_1} (i \xi_2) ^{\beta_2} D_3^{\beta_3}$. 
%We recall $[[ \ \cdot  \  ]]_{ \infty, \lambda,\kappa}$ in (\ref{[]}), which is the norm of the space $\mathfrak{B}^{\lambda, \kappa}$; and the norm of $\mathfrak{L}^{p,\lambda}$, $\| \cdot\|_{p,\lambda}$ in (\ref{norm_L1}) for $1 \leq p \leq \infty$. We note that $\mathfrak{B}^{\lambda, 0} \subsetneqq\mathfrak{L}^{\infty,\lambda}$ if $\bar{\alpha}>0$. 

Now we define, for $\lambda_0>0$, $\gamma_0>0$, $\alpha>0$, $\kappa \geq 0$, and $t \in (0, \frac{\lambda_0}{2 \gamma_0})$
 \Be\label{norm_BLT}
 \vertiii{g}_{\infty,\kappa}= \sup_{\lambda<\lambda_0-\gamma_0 t} \bigg\{ \sum_{ 0 \leq |\beta| \leq 1}  %\p_{x_1}^i \p_{x_2}^j(\zeta(x_3)\p_{x_3})^k 
[[
D^\beta g ]]_{%\infty,\lambda, \kappa
\infty, \lambda, \kappa } +  \sum_{ %i+j+k
|\beta| =2}(\lambda_0-\lambda-\gamma_0 t)^\alpha [[D^\beta g%\p_{x_1}^i \p_{x_2}^j(\zeta(x_3)\p_{x_3})^k g 
]]_{
%\infty,\lambda, \kappa
 {\infty, \lambda, \kappa}
}  \bigg\},
\Ee
\Be\label{norm_L1T}
\begin{split}
\vertiii{g}_1 
= \sup_{\lambda<\lambda_0-\gamma_0 t} \bigg\{& \sum_{0\leq |\beta|\leq 1} \| %\p_{x_1}^i \p_{x_2}^j(\zeta(x_3)\p_{x_3})^k
D^\beta (1+|\nabla_h|) g \|_{1,\lambda} \\
& \ \ \ \ \  \ \ \ \ \ +  (\lambda_0-\lambda-\gamma_0 t)^\alpha\sum_{ |\beta|=2} \| D^\beta (1+|\nabla_h|) g \|_{1,\lambda}  \bigg\}.
\end{split}\Ee
And for a scalar valued function $\tilde g$, define
\Be\label{normtheta}
 \vertiii{ \tilde g}_{\infty,0}= \sup_{\lambda<\lambda_0-\gamma_0 t} \bigg\{ \sum_{ 0 \leq |\beta| \leq 1}  \|D^\beta \tilde g \|_{\infty, \lambda, 0 } +  \sum_{ |\beta| =2}(\lambda_0-\lambda-\gamma_0 t)^\alpha \| D^\beta \tilde g \|_{ {\infty, \lambda, 0} }  \bigg\},
\Ee

\Be \label{px3thetanorm}
\vertiii{\tilde g}_z :=  \sup_{\lambda < \lambda_0 - \gamma_0 t }  \left\{ (\lambda_0-\lambda-\gamma_0 t)^\alpha   \sum_{0 \le |\beta| \le 1} \| D_h^{\beta_h} \p_{x_3} g \|_{ {\infty, \lambda, 0} } \right\}.
\Ee

With an initial-boundary layer weight function as in \cite{NN2018}
\Be
\phi_{\kappa t} (z)= \frac{1}{\sqrt{\kappa t}} \phi ( \frac{z}{\sqrt{\kappa t}}) 
%= \sqrt{\kappa t}^{p-1} \frac{1}{\sqrt{\kappa t}^{p } + z^p}
,\label{IB_layer}
 \Ee
we define an initial-boundary layer function space $\mathfrak B^{\lambda, \kappa t}$ for holomorphic functions $g= (g_1,g_2,g_3)$ with a finite norm 
\Be\label{[[]]}
[[g ]]_{ \infty, \lambda,\kappa t}=\sum_{i=1,2} \| g_i \|_{\infty, \lambda, \kappa t} + \| g_3 \|_{\infty, \lambda, 0} ,
\Ee
where an $L^\infty$-based analytic norm with the initial-boundary layer is defined as
\Be\label{norm_IBL}
%{\color{red} \ [[g]] ? \ }
\|g\|_{\infty,\lambda,\kappa t} = \sum_{\xi\in \mathbb Z^2} e^{\lambda |\xi|} \| g_{\xi } \|_{\mathcal{L}^\infty_{\lambda, \kappa t}}, \quad 
\| g_{\xi} \|_{\mathcal{L}^\infty_{\lambda, \kappa t}} = %\sup_{z\in \Omega_\lambda}
\bigg\| \frac{e^{\bar{\alpha} \text{Re}\,z}}{ 1+ \phi_\kappa(z)+ \phi_{\kappa t} (z)}   g_{\xi } (z)\bigg\|_{\mathcal{L}^\infty_\lambda} . 
\Ee
We finally define, for $t \in (0, \frac{\lambda_0}{2\gamma_0})$, 
 \Be\label{norm_IBLT}
 \vertiii{g}_{\infty, \kappa t}= \sup_{\lambda<\lambda_0-\gamma_0 t} \bigg\{
  \sum_{%0\leq i+j+k\leq 1
 0 \leq |\beta| \leq 1
  }% \|   \p_{x_1}^i \p_{x_2}^j(\zeta(x_3)\p_{x_3})^k  
 % g \|_{\infty,\lambda, \kappa t} 
 [[
D^\beta g]]_{\infty,\lambda, \kappa t}+ \sum_{|\beta|=2} 
 (\lambda_0-\lambda-\gamma_0 t)^\alpha
  [[
D^\beta g]]_{\infty,\lambda, \kappa t}
% +  \sum_{ i+j+k =2}(\lambda_0-\lambda-\gamma t)^\alpha \| \p_{x_1}^i \p_{x_2}^j(\zeta(x_3)\p_{x_3})^kg \|_{\infty,\lambda, \kappa t} 
 \bigg\}.
\Ee
In this section, $\alpha$, $\bar{\alpha}$ are  given  positive small constants, $\lambda_0$ is a given positive constant, and $\gamma_0$ is a sufficiently large constant to be determined in Theorem \ref{thm_bound}.

Next we discuss the initial data $u_{in}$, $\theta_{in}$, and the corresponding vorticity $\o_{in}= \nabla_x \times u_{in}$. Inspired by the PDEs, let 
\Be\begin{split}\label{idata}
 \o_0:=\o_{in},\quad \p_t\o_0:= \kappa\eta_0\Delta \o_0 - u_0\cdot\nabla \o_0+ \o_0 \cdot \nabla u_0,\\
 \quad
 u_0 := \nabla \times (-\Delta)^{-1} \o_0 ,
  \quad
  \p_t u_0 := \nabla \times (-\Delta)^{-1} \p_t \o_0 ,
 \\
\p_t^2\o_0:= \kappa\eta_0 \Delta \p_t\o_0 - u_{0} \cdot \nabla \p_t\o_0  - \p_t u_0  \cdot \nabla \o_0 + \o_0 \cdot \nabla \p_t u_0 + \p_t\o_0 \cdot \nabla u_0,
\\ \p_t \theta_0 := \kappa \eta_c \Delta \theta_0 - u_0 \cdot \nabla \theta_0,
\\ \p_t^2 \theta_0:= \kappa \eta_c \Delta \p_t \theta_0 - u_0 \cdot \nabla \p_t \theta_0 - \p_t u_0 \cdot \nabla \theta_0.
\end{split}\Ee

\hide

%Let $t\in (0,\frac{\lambda_0}{2\gamma})$. 

%We will adopt the framework and the methodology 
%of Nguyen-Nguyen \cite{NN2018} to study the solution to \eqref{NS}-\eqref{NSB} in analytic function spaces and use a number of results proved in \cite{NN2018} with slight modifications of notations. We will also use a number of results in Wang \cite{FW} for the 3D counter part. 

%Then $C^1$ norm is bounded above by 
%\Be
%\|  \sum_{\xi \in \mathbb{Z}} \xi    f_\xi (x_2) e^{i x_1 \xi} \| _{L^\infty_{x_1,x_2}}
%+ \|  \sum_{\xi \in \mathbb{Z}}   \p_{x_2} f_\xi (x_2) e^{i x_1 \xi} \| _{L^\infty_{x_1,x_2}}
%\Ee

\begin{definition}[Definition of $\mathfrak B^{\lambda, \kappa}$]
%[Weighted $L^\infty$ based analytic norms]\label{def:B}
Define for

For  we define 

Let $\mathfrak B^{\lambda, \kappa}$ denote the space of analytic functions $g$ with $[[g]]_{\infty,\lambda,\kappa} <\infty$.

We also define, for $t \in (0, \frac{\lambda_0}{2\gamma})$,

\end{definition}

\begin{definition}

We define, for $t \in (0, \frac{\lambda_0}{2\gamma})$, 
\Be\label{norm_L1T}
\vertiii{g}_1= \sup_{\lambda<\lambda_0-\gamma t} \bigg\{ \sum_{0\leq |\beta|\leq 1} \| %\p_{x_1}^i \p_{x_2}^j(\zeta(x_3)\p_{x_3})^k
D^\beta (1+|\nabla_h|) g \|_{1,\lambda} +  (\lambda_0-\lambda-\gamma t)^\alpha\sum_{ |\beta|=2} \| D^\beta (1+|\nabla_h|) g \|_{1,\lambda}  \bigg\}
\Ee
where $ = \sqrt{-\Delta_h}$. 
\end{definition}

 \begin{remark} 
% $L^p$ based analytic function spaces $\mathfrak{L}^{p,\lambda}$ and their norm $\| \cdot \|_{p,\lambda } $ for $1\leq p\leq \infty$ can be defined in the same fashion as in \eqref{norm_L1} by replacing $L^1$ by $L^p$. 
 \end{remark}

%\subsection{Norms and Function spaces} 
Consider analytic functions $g=g(x_1,x_2,x_3)$. %For notational convenience, we often suppress the time dependence unless confusion arises. 
The following symbols and notations will be used throughout the rest of the section. 

%Notation. 
\begin{itemize}
\item 
\item 
\item 
\item 
\end{itemize}

\

\begin{definition}[$L^\infty$ based analytic norms including initial layers] %In the presence of initial layers, we define the following analytic norms
\end{definition}

%For the proof of Lemma \ref{lem_embedding}, we refer to \cite{NN2018}. 

\

\begin{theorem}\label{thm_NS} Let $\lambda_0>0$ and $\o  \in \mathfrak B^{\lambda_0,\kappa}$ with
\Be
 \sum_{0\leq |\beta|\leq 2}  \|D^\beta \p_t^\ell\o_0  \|_{1,\lambda_0}+\sum_{0\leq |\beta|\leq 2} \|D^\beta \p_t^\ell\o_0 \|_{\infty,\lambda_0, \kappa}  <\infty \  \text{ for } \ \ell=0,1,2. \label{initial_norm}
\Ee
Further assume that $\o_0$ satisfies the compatibility conditions 
\Be\label{CC}
\begin{split}
\kappa \eta_0 (\p_{x_3} + \sqrt{- \Delta_h})\o_{0,h} \, |_{x_3=0} &= [\p_{x_3} (-\Delta)^{-1} (-u_0 \cdot \nabla \o_{0, h} + \o_0 \cdot \nabla u_{0,h}) ] \, |_{x_3=0}\\
\o_{0,3} |_{x_3=0} &=0, \quad \p_t\o_{0,3} |_{x_3=0} = 0 . 
\end{split}
\Ee

Then there exists a $\gamma>0$ and a time $T>0$ depending only on $\lambda_0$ and the size of the initial data such that the solution $\o(t)$ to the Navier-Stokes equations \eqref{NS}-\eqref{NSB} exists in $C^1([0,T]; \mathfrak B^{\lambda, \kappa})$ with $\p_t^2\o $ in $C(0,T; \mathfrak B^{\lambda, \kappa t})$ for $0<\lambda<\lambda_0$ satisfying 
\Be\label{norm_bound}
\sup_{t\in [0,T]} \left[\sum_{\ell=0}^2 \vertiii{\p_t^\ell \o(t)}_1 + \sum_{\ell=0}^1 \vertiii{\p_t^\ell \o(t)}_{\infty,\kappa} + \vertiii{\p_t^2\o(t)}_{\infty,\kappa t}   \right]<\infty 
\Ee
\end{theorem}

The following point-wise bounds for $\o$, $u$, and $p$ are the consequences of Theorem \ref{thm_bound}. 
\unhide

\begin{theorem}\label{thm_bound} Let $\lambda_0>0$ and $\o_{in}, \, \theta_{in}  \in \mathfrak B^{\lambda_0,\kappa}$ with (\ref{idata}) satisfy for $\ell = 0,1,2$,
\Be \label{initial_norm}
 \sum_{0\leq |\beta|\leq 2}  \|D^\beta \p_t^\ell\o_0  \|_{1,\lambda_0}+\sum_{0\leq |\beta|\leq 2} \|D^\beta \p_t^\ell\o_0 \|_{\infty,\lambda_0, \kappa}  <\infty,
 % \  \text{ for } \ \ell=0,1,2, 
\Ee
\Be \label{inital_norm_theta}
 \sum_{0\leq |\beta|\leq 2}  \|D^\beta \p_t^\ell\theta_0  \|_{1,\lambda_0}+\sum_{0\leq |\beta|\leq 2} \|D^\beta \p_t^\ell\theta_0 \|_{\infty,\lambda_0, 0} + \sum_{0 \le | \beta_h | \le 1 } \| D_h^{\beta_h} \p_{x_3} \p_t^\ell \theta_0(t,x_3) \|_{\infty, \lambda_0 ,0 }  <\infty.
\Ee
Further assume that $\o_{in}=\o_0$, $\theta_{in} = \theta_0$, and (\ref{idata}) satisfies the compatibility conditions on $\p\O$  
\Be \label{CC}
\begin{split}
\kappa \eta_0 (\p_{x_3} + \sqrt{- \Delta_h})\o_{0,h}  = [\p_{x_3} (-\Delta)^{-1} (-u_0 \cdot \nabla \o_{0, h} + \o_0 \cdot \nabla u_{0,h}) ]  , \\ 
\o_{0,3}  =0,  \  \  \p_t\o_{0,3}   = 0, \ \  \theta_0  =0,  \  \  \p_t\theta_0   = 0, \  \  \p_t^2 \theta_0 = 0.
\end{split}
\Ee
Then there exist a constant $\gamma_0>0$ and a time $T>0$ depending only on $\lambda_0$ and the size of the initial data such that the solution $\o(t)$ to the vorticity formulation of the Navier-Stokes equations \eqref{NS_o}-\eqref{NSB_o} exists in $C^1([0,T]; \mathfrak B^{\lambda, \kappa})$ with $\p_t^2\o $ in $C(0,T; \mathfrak B^{\lambda, \kappa t})$ for $0<\lambda<\lambda_0$ satisfying  
\Be\label{norm_bound}
\sup_{t\in [0,T]} \left[\sum_{\ell=0}^2 \vertiii{\p_t^\ell \o(t)}_1 + \sum_{\ell=0}^1 \vertiii{\p_t^\ell \o(t)}_{\infty,\kappa} + \vertiii{\p_t^2\o(t)}_{\infty,\kappa t}   \right]<\infty .
\Ee
And the solution $\theta(t)$ to \eqref{heat_NS}-\eqref{thetaB} exists in $C^2([0,T]; \mathfrak B^{\lambda, 0})$  satisfying
\Be \label{thetanorm_bound}
\sup_{t\in [0,T]} \left[\sum_{\ell=0}^2 \vertiii{\p_t^\ell \theta(t)}_1 + \sum_{\ell=0}^2 \vertiii{\p_t^\ell \theta(t)}_{\infty,0} \right]<\infty,
\Ee
and
\Be \label{thetapx3_bound}
\sqrt \kappa \left( \sup_{ t \in [0,T]} \sum_{\ell=0}^2 \vertiii{\p_t^\ell \theta(t)}_z \right) < \infty.
\Ee

Furthermore, for each $(t,x)\in [0,T]\times\Omega $, 
\begin{enumerate}
\item (Bounds on the vorticity and its derivatives) $\o(t,x)$ enjoys the following bounds: 
\begin{align}
| \nabla_{h}^i \p_t^\ell \o_h (t,x) |& \lesssim e^{-\bar{\alpha} x_3} \left( 1 + \phi_\kappa (x_3) \right), \ \  | \nabla_{h}^i \p_t^\ell \o_3 (t,x) | \lesssim e^{-\bar{\alpha} x_3} \text{ for } i,\ell =0,1, \label{b1} \\
%&{\color{red} | \nabla_{h}^i \p_t^\ell \o_h (t,x_h,x_3) | \lesssim e^{-\bar{\alpha} x_3} \left( 1 + \phi_\kappa (x_3) \right), \quad | \nabla_{h}^i \p_t^\ell \o_3 (t,x_h,x_2) | \lesssim e^{-\bar{\alpha} x_3}  \ ? } \notag\\
| \p_t^2 \o_h (t,x) |& \lesssim e^{-\bar{\alpha} x_3} \left( 1 + \phi_\kappa (x_3) + \phi_{\kappa t} (x_3) \right), \ \ | \p_t^2 \o_3 (t,x) | \lesssim e^{-\bar{\alpha} x_3}, \label{b2}\\
| \p_{x_3} \p_t^\ell \o_h (t,x) |& \lesssim {\kappa}^{-1} e^{-\bar{\alpha} x_3}, \ \  | \p_{x_3} \p_t^\ell \o_3 (t,x) | \lesssim  e^{-\bar{\alpha} x_3} \left( 1 + \phi_\kappa (x_3) \right) \text{ for } \ell =0,1. \label{b3}
%| \p_{x_2} \p_t \o (t,x_1,x_2) |& \lesssim \label{b4}
%\\ &{\color{red} | \p_{x_3} \p_t^\ell \o_h (t,x_h,x_3) | \lesssim {\kappa}^{-1} e^{-\bar{\alpha} x_3} , \quad | \p_{x_3} \p_t^\ell \o_h (t,x_h,x_3) | \lesssim {\kappa}^{-1/2} e^{-\bar{\alpha} x_3} \ ? } \notag
\end{align}

 \item (Bounds on the velocity and its derivatives) The corresponding velocity field $u(t,x)$ satisfies the following: %for $ t \in [0,T]$ % \texttt{...we present the estimates  of $u$ and its derivatives necessary  for $f_R$ here or have it in a separate theorem...} 
%{\color{red}\texttt{...below, when $|\beta|=1$, $\kappa^{-1/2}$ and when $|\beta|=2$, $\kappa^{-1}$ ? ... Need the pressure estimates...}}
 \begin{align}
  \label{est:u_t}
|\p_t^\ell u (t,x)| &\lesssim 1 \ \ \text{for} \ \ell=0,1,2, \\ %, \ \text{and} \  t \in [0,T].
 \label{est:u1}
\sum_{1 \leq |\beta | \leq 2}  |\nabla ^{\beta } \p_t^\ell u(t,x )|%\lesssim  \sup_{t \in [0,T]}\vertiii{\p_t^\ell \o (t) } 
& \lesssim \big(1+ \phi_\kappa (x_3) + (|\beta|-1) {\kappa}^{-1}  \big)
 e^{-\min (1, \frac{\bar{\alpha}}{2} )x_3}  \ \ \text{for} \ \ell=0,1, \\
 \label{est:u2}
\sum_{ |\beta | =1}  |\nabla ^{\beta } \p_t^2 u(t,x )|%\lesssim  \sup_{t \in [0,T]}\vertiii{\p_t^\ell \o (t) } 
& \lesssim  \big(1+ \phi_\kappa (x_3) + \phi_{\kappa t} (x_3)\big)e^{-\min (1, \frac{\bar{\alpha}}{2} )x_3} . 
\end{align}
Moreover, we have the decay estimate for $\p_t^\ell u$: %at the expense of the loss of $\kappa$: 
\Be\label{ut}
|\p_t^\ell u|\lesssim \kappa^{-\frac{1}{2}} e^{-\min(1,\frac{\bar\alpha}{2}) x_3}  \ \ \text{for} \ \ell=1,2. 
\Ee

\item (Bounds on the temperature and its derivatives)
The temperature $\theta(t,x)$ satisfies the following:
 \begin{align}
  \label{est:theta_t}
\sum_{ 0 \le \beta_h \le 2 } | \nabla_h^{\beta_h} \p_t^\ell \theta (t,x)| &\lesssim e^{-\bar \alpha x_3} \ \ \text{for} \ \ell=0,1,2,
 \\  \label{est:theta1}
\sum_{0 \leq |\beta_h | \leq 1}  |\nabla_h ^{\beta_h } \p_{x_3} \p_t^\ell \theta(t,x )| & \lesssim \kappa^{-\frac{1}{2}} e^{- \bar \alpha x_3 }  \ \ \text{for} \ \ell=0,1, 2,
 \\ \label{est:theta2}
 |\p_{x_3}^2 \p_t^\ell \theta(t,x )| & \lesssim  \kappa^{-1} e^{-\bar \alpha x_3}   \ \ \text{for} \ \ell=0,1.
\end{align}

\item (Bounds on the pressure and its derivatives) 
Choosing the pressure such that $p(t,x) \to 0$ as $x_3 \to \infty$, then $p$ satisfies the following: %for $ t \in [0,T]$
\begin{align}
|\p_t^\ell p (t,x)|& \lesssim 1 \ \ \text{for} \ \ell=0,1,2,  \label{est:p}\\
\label{est:pdecay}
\sum_{0 \leq |\beta | \leq 1}  |\nabla ^{\beta } \p_t^\ell p(t,x )|
& \lesssim \kappa^{-\frac{1}{2}} 
 e^{-\min (1, \frac{\bar{\alpha}}{2} )x_3}   \ \ \text{for} \ \ell=0,1, \\
\label{est:pt2}
|\p_t^2 p|   &\lesssim ( \kappa^{-\frac12}  +\phi_{\kappa t}(x_3) )e^{-\min (1, \frac{\bar{\alpha}}{2} )x_3}  . 
\end{align}
%\Be
%|\nabla  p (t,x)| \lesssim \kappa^{-\frac12} e^{-\min(1,\frac{\bar\alpha}{2}) x_3}.  \label{est:p1}
%\Ee
\end{enumerate}
\end{theorem}

The proof of the theorem relies on the integral representation of the solution to the Navier-Stokes-Fourier system using the Green's function for the Stokes problem in the same spirit of \cite{NN2018}. 
%The velocity and pressure estimates are also crucial for the kinetic part and they are obtained by utilizing elliptic regularity results  and Biot-Savart law in analytic setting. 

 \subsection{Elliptic estimates and Nonlinear estimates
 % in a setting of local Maxwellian
}

\begin{lemma}[\cite{NN2018,FW}, Embeddings and Cauchy estimates] The following holds\label{lem_embedding}
\begin{enumerate}
\item $\mathfrak B^{\lambda, \kappa t} \subset \mathfrak{L}^{1,\lambda}$ and 
$\mathfrak B^{\lambda, \kappa } \subset \mathfrak{L}^{1,\lambda}$. 
\item $\| g_1 g_2\|_{\ast,\lambda}\lesssim \|g_1\|_{\infty,\lambda} \|g_2\|_{\ast,\lambda}$. 
\item $ \sum_{|\beta|=1} \|D^\beta g\|_{\ast,\lambda}   %\|\p_{x_1} g\|_{\ast,\lambda}+\|\p_{x_2} g\|_{\ast,\lambda} + \|D_3 g\|_{\ast,\lambda} 
\lesssim \frac{\|g\|_{\ast,\tilde{\lambda}}}{\tilde\lambda -\lambda}$, for any $0< \lambda <\tilde\lambda$. %{\color{red}\texttt{...replace l.h.s by $\|Dg\|_{\ast,\lambda}$ ? ...}}
\end{enumerate}
For (2) and (3), $\|\cdot \|_{\ast,\lambda}$ can be either $\|\cdot\|_{\infty,\lambda,\kappa}$ or $  \|\cdot \|_{\infty,\lambda,\kappa t}$ or $\| \cdot \|_{\infty, \lambda , 0}$ or $ \|\cdot\|_{1,\lambda}$. 
\end{lemma}

%We also record the elliptic estimates. 

%Let $\|g\|_{\infty,\lambda}=\sum_{\xi\in \mathbb Z} e^{\lambda|\xi|} \| g_\xi \|_{L^\infty_\lambda}$ where $ \|g_\xi\|_{L^\infty_\lambda} = \sup_{0\leq \sigma<\lambda} \| g_\xi  \|_{L^\infty(\p \Omega_{\sigma})}$. Recall that $\zeta(z)= \frac{z}{1+z}$. 

\begin{lemma}[\cite{NN2018,FW}, Elliptic estimates] \label{lem_elliptic} Let $\phi$ be the solution of $-\Delta\phi = \o$ with the zero Dirichlet boundary condition, and let $u=\nabla \times \phi$. Then 
\Be
\begin{split}\label{est:elliptic}
\| u  \|_{\infty,\lambda}
+ \|\nabla u \|_{1,\lambda} 
%+ \| u_2 \|_{\infty,\lambda}
 &\lesssim \|\o\|_{1,\lambda} , \\ 
\| \nabla_h u  \|_{\infty,\lambda} + \| \nabla u_3\|_{\infty,\lambda} %+ \|\zeta^{-1} u_3\|_{\infty,\lambda}
  &\lesssim \sum_{0 \leq |\beta| \leq1}
\| \nabla_h^\beta\o\|_{1,\lambda} , \\
\|   \p_3 u _h\|_{\infty, \lambda }  &\lesssim  \sum_{0 \leq |\beta|\leq 1}
\| \nabla_h^\beta\o\|_{1,\lambda}
+ \|   \o_h \|_{\infty, \lambda },\\
\| \zeta^{-1} \nabla_h^{\beta^\prime} u_3 \|_{\infty, \lambda} &\lesssim \sum_{0 \leq |\beta|\leq 1} \|   \nabla_h ^{\beta+ \beta^\prime} \o_h \|_{1, \lambda}.
\end{split}
\Ee 
\end{lemma}

As a consequence of Lemma \ref{lem_elliptic}, we have the following nonlinear estimates. 

%\begin{lemma}[\cite{NN2018}]\label{NN_product}
%\begin{align}
%\| fg \|_{\infty, \lambda, \delta} \lesssim \| f \|_{ \infty,\lambda} \| g \|_{\infty, \lambda, \delta}
%\end{align}
%\end{lemma}

\begin{lemma}[\cite{NN2018,FW}]\label{lem_bilinear} Let $u$ and $\tilde u$ be the velocity field associated with $\o= \nabla_x \times u$ and $\tilde\o= \nabla_x \times \tilde u$ respectively. Then 
\Be
\begin{split}\label{nonlinear1}
\| u \cdot \nabla \tilde\o \|_{1,\lambda }&\lesssim \|\o\|_{1,\lambda}  \|\nabla_h \tilde\o \|_{1,\lambda} + 
%\Big(\sum_{|\beta|=0,1} \| \nabla_h^\beta\o\|_{1,\lambda}  \Big) 
\| (1+ |\nabla_h|) \o \|_{1, \lambda}
\| \zeta \p_z \tilde \o \|_{1,\lambda}, \\
\|\o \cdot \nabla \tilde u_3\|_{1, \lambda} &\lesssim \| \o_h \|_{1, \lambda} \| \nabla_h \tilde{u}_3 \|_{\infty, \lambda} + \| \o_3 \|_{1,\lambda} \| \p_3 \tilde{u}_3 \|_{\infty, \lambda}
\lesssim \| \o \|_{1,\lambda} 
\| (1+ |\nabla_h|) \tilde \o \|_{1, \lambda}
%\Big(\sum_{|\beta|=0,1} \| \nabla_h^\beta \tilde \o\|_{1,\lambda}  \Big)
 ,  \\
\|\o \cdot \nabla \tilde u_h\|_{1, \lambda} &\lesssim \| \o_h \|_{1, \lambda} \| \nabla_h \tilde{u}_h \|_{\infty, \lambda} + \| \o_3 \|_{\infty,\lambda} \| \p_3 \tilde{u}_h \|_{1, \lambda}
\lesssim \| \o \|_{1,\lambda} \big(
\| \tilde \o_3 \|_{\infty, \lambda} + 
\| (1+ |\nabla_h|) \o \|_{1, \lambda}
%\sum_{|\beta|=0,1} \| \nabla_h^\beta\tilde \o\|_{1,\lambda} 
\big)
 .
\end{split}
\Ee

 Moreover
\Be\label{nonlinear2}
\begin{split}
\| u \cdot \nabla \tilde\o_h \|_{*,\lambda }  &\lesssim \|\o\|_{1,\lambda} \|\nabla_h \tilde \o_h\|_{*,\lambda} + 
\big(
\| (1+ |\nabla_h|) \o \|_{1, \lambda}
%\sum_{|\beta|=0,1} \| \nabla_h^\beta\o\|_{1,\lambda}
+ \| \zeta \p_z  \o_3 \|_{\infty, \lambda}
\big)   \| \zeta \p_z \tilde \o_h\|_{*,\lambda}, \\
\| \o \cdot \nabla \tilde u_h \|_{*,\lambda }
 % \lesssim  \|  \o_3\|_{\infty,\lambda } \| \p_3 \tilde{u}_h \|_{\infty, \lambda}+ \|\o_h \|_{*,\lambda }\| \nabla_h \tilde{u}_h\|_{\infty,\lambda}
&\lesssim 
\| \o_3 \|_{\infty, \lambda, 0}
\big( 
\| (1+ |\nabla_h|)\tilde \o \|_{1, \lambda}
% \sum_{|\beta|=0,1}
%\| \nabla_h^\beta \tilde \o\|_{1,\lambda}
+ \|  \tilde  \o_h \|_{*, \lambda } \big)
+ \|\o_h \|_{*,\lambda }
 \sum_{0\leq |\beta|\leq 1}
\| \nabla_h^\beta\tilde \o\|_{1,\lambda} 
,
\end{split}
\Ee
where $\| \cdot \|_{*, \lambda}$ can be either $\| \cdot \|_{\infty, \lambda , \kappa}$ or $\| \cdot \|_{\infty, \lambda , \kappa t}$.

Furthermore 
\Be
\begin{split}\label{nonlinear2_3}
\| u \cdot \nabla \tilde\o_3 \|_{\infty,\lambda,0} 
%&\lesssim \| u_h \|_{\infty, \lambda} \| \nabla_h \tilde{\o}_3 \|_{\infty, \lambda}
%+  \| \zeta(z) ^{-1} u_3 
% \|_{\infty, \lambda}
%\| \zeta(z)  \p_3 \tilde{\o}_3 \|_{\infty, \lambda}
%\\
%&
 &\lesssim \|\o\|_{1,\lambda} \|\nabla_h\tilde\o_3\|_{\infty,\lambda,0} + 
\| (1+ |\nabla_h|) \o \|_{1, \lambda}
%  \sum_{  |\beta| =0, 1}\| \nabla_h^\beta \o\|_{1,\lambda}  
    \| \zeta \p_3 \tilde \o_3\|_{\infty,\lambda,0 },\\
\| \o  \cdot \nabla  \tilde u_3 \|_{\infty,\lambda,0} &\lesssim 
   \| \o_h \|_{*, \lambda } 
   \| (1+ |\nabla_h|^2) \tilde \o_h \|_{1, \lambda}
  % \sum_{|\beta|=1,2} \| \nabla_h^\beta \tilde \o_h \|_{1,\lambda}
+ \| \o_3 \|_{\infty, \lambda, 0}  
\| (1+ |\nabla_h|) \tilde \o_h \|_{1, \lambda},
%\sum_{|\beta|=0,1} \| \nabla_h^\beta \tilde \o_h \|_{1,\lambda}.
\end{split}
\Ee
where $\| (1+ |\nabla_h|^k) g \|_* = \sum_{\ell=0}^k \| \nabla_h^\ell g \|_*$. 
\end{lemma}
%{\color{red}\texttt{...How about making it Proof? ...}} 

%\begin{proof}
% Again we refer to Proposition 2.3 in \cite{NN2018} for 2D and Section 4 of \cite{FW} for the full justification. The bounds (\ref{nonlinear1}) and (\ref{nonlinear2}) directly follow from Lemma \ref{lem_elliptic}. The proof of the first estimate of (\ref{nonlinear2_3}) is an outcome of applying (\ref{est:elliptic}) to an easy bound 
%\Be
%\| u \cdot \nabla \tilde\o_3 \|_{\infty,\lambda,0} 
% \lesssim \| u_h \|_{\infty, \lambda} \| \nabla_h \tilde{\o}_3 \|_{\infty, \lambda,0}
%+  \| \zeta(z) ^{-1} u_3 
% \|_{\infty, \lambda}
%\| \zeta(z)  \p_3 \tilde{\o}_3 \|_{\infty, \lambda,0}.\notag
%\Ee 
%For the second estimate of (\ref{nonlinear2_3}) it suffices to prove the bound for $\o_h \cdot \nabla_h \tilde u_3$. From $|\zeta(z)(1+ \phi_\kappa (z)
% )|\lesssim 1$ or $|\zeta(z)(1+ \phi_\kappa (z)
%+ \phi_{\kappa t} (z))|\lesssim 1$,
%\begin{align*}
%\| \o_h \cdot  \nabla_h \tilde u_3 \|_{\infty, \lambda, 0} & \lesssim \|  \o_h   \| _{*, \lambda } \Big\|  \zeta(z)(1+ \phi_\kappa (z)
%+ \phi_{\kappa t} (z)
%)  \frac{\nabla_h \tilde u_3}{\zeta(z)} \Big\|_{\infty, \lambda}
%\lesssim \|  \o_h   \| _{*, \lambda }
% \|
%  \zeta^{-1}{\nabla_h \tilde u_3} 
% \|_{\infty, \lambda}.
%\end{align*} 
%Then we use the last bound of (\ref{est:elliptic}) to finish the proof. 
%\end{proof}

 We finally record the crucial estimate of nonlinear forcing terms $N=-u\cdot\nabla \o+ \o \cdot \nabla u$, as an outcome of {Lemma} \ref{lem_bilinear}, that will be also crucially used to control  $B= [\p_{x_3} (-\Delta)^{-1} (-u \cdot \nabla \o + \o \cdot \nabla u ) ] \, |_{x_3=0}$ in the vorticity formulation  \eqref{NS} and \eqref{NSB}. 

 \begin{lemma}[\cite{NN2018,FW}, Nonlinear estimate]\label{lem_est:N} {Let $\lambda\in(0,\lambda_0-\gamma s)$ be given. We have the following: }
  \Be
 \begin{split}\label{est:N_1}
\| (1+ |\nabla_h |) N \|_{1, \lambda }&\lesssim\big( \| (1+ |\nabla_h| ) \o \|_{1, \lambda} 
+ \|(1+ |\nabla_h| ) \o_3\|_{\infty, \lambda, 0}\big)
 \| (1+ |\nabla_h|^2 ) \o \|_{1, \lambda} \\
& \ \  + \sum_{|\beta|=1}
 \| (1+ |\nabla_h|) D^\beta \o\|_{1, \lambda}  \| (1+ |\nabla_h|^2) \o \|_{1, \lambda} ,
 \end{split}
\Ee
\Be
\begin{split}\label{est:DN_1}
&\sum_{|\beta|=1}\|D^\beta (1+ |\nabla_h |) N \|_{1, \lambda }\\
&\lesssim
 \sum_{|\beta  | \leq 1}  \| D^\beta (1+ |\nabla_h| ) \o \|_{1, \lambda} 
 \bigg(
  \sum_{|\beta| \leq 2}\| D^\beta(1+ |\nabla_h| ) \o\|_{1,\lambda})
  + \| (1+ |\nabla_h|) \o \|_{\infty, \lambda, 0}
  \bigg)
  \\
  & \ \ + \sum_{|\beta | \leq 1} \| D^\beta (1+ |\nabla_h| ) \o_3 \|_{\infty, \lambda, 0}
  \| (1+ |\nabla_h|)^2 \o \|_{1,\lambda}. 
\end{split}
\Ee
 
% Recall the norm $[[ \ \cdot \ ]]_{\infty, \lambda, \kappa }$ and  in (\ref{[]}). 
 For $[[ \ \cdot \ ]]_{*,\lambda}$ to be either $[[ \ \cdot \ ]]_{\infty, \lambda, \kappa }$ or $[[ \ \cdot \ ]]_{\infty, \lambda, \kappa t}$,
\Be\label{est:N_infty}
[[N]]_{*,\lambda}
%\| N_h \|_{\infty, \lambda, \kappa t}
%+ \| N_3(s)\|_{\infty, \lambda,0}
\lesssim \| (1+ |\nabla_h|^2) \o \|_{1, \lambda}
[[\o]]_{*,\lambda}
%( \| \o_h\|_{\infty, \lambda, \kappa t} + \| \o_3\|_{\infty, \lambda, 0} ) 
+
 \| (1+ |\nabla_h| )\o \|_{1, \lambda}
 [[
 D \o 
 ]]_{*,\lambda},
 % (\| D \o_h \|_{\infty, \lambda, \kappa t}+ \| D \o_3 \|_{\infty, \lambda, 0})
\Ee
\Be\label{est:DN_infty}
\begin{split}
 \sum_{|\beta|=1}
[[
D^\beta N
]]_{*,\lambda}
%\big( \|D^\beta N_h \|_{\infty, \lambda, \kappa t}
%+ \|D^\beta  N_3(s)\|_{\infty, \lambda,0}\big)
&\lesssim \sum_{|\beta|=1}  \| (1+ |\nabla_h|^{|\beta_h|+2}) \o \|_{1, \lambda}
[[ \o ]]_{*,\lambda}\\
%( \| \o_h\|_{\infty, \lambda, \kappa t} + \| \o_3\|_{\infty, \lambda, 0} 
%) 
%+ \| (1+ |\nabla_h| )\o \|_{1, \lambda} (\| D \o_h \|_{\infty, \lambda, \kappa t}+ \| D \o_3 \|_{\infty, \lambda, 0})
& \ \ + \sum_{|\beta|=1} [[D^\beta \o ]]_{*,\lambda} (\| (1+ |\nabla_h|^2) \o \|_{1, \lambda} + \beta_3[[D_3^{\beta_3} \o ]]_{*,\lambda}) \\
& \ \  
 + \sum_{|\beta|=2} [[ D^\beta \o ]]_{*,\lambda} \| (1+ |\nabla_h| )\o\|_{1,\lambda} .
\end{split}\Ee
\end{lemma}

The proof relies on Lemma \ref{lem_bilinear}. We refer to Lemma 4.2 and Lemma 4.5 in \cite{FW} for the detailed proof.

\subsection{Green's function and integral representation}

By taking the Fourier transform of \eqref{NS}-\eqref{incomp_NS} in $x_h \in \mathbb{T}^2$, we obtain 
 \begin{align}
\p_t \o_\xi - \kappa \eta_0 \Delta_\xi \o_\xi &= N_\xi  \quad \text{in }\mathbb R_+, \label{NS_f} \\
\kappa \eta_0 (\p_{x_3} +|\xi|)\o_{\xi ,h}  &= B_\xi, \ \ \o_{\xi,3} =0 \quad \text{on } x_3=0 ,\label{NSB_f}
\end{align} 
with $\o_\xi|_{t=0} = {\o_0}_\xi$ for $\xi\in {\mathbb Z^2}$. Here 
\Be 
\Delta_\xi = - |\xi|^2 + \p_{x_3}^2,
\Ee
and
\Be
N_\xi = N_\xi(t,x_3):=  \left( -u \cdot \nabla \o + \o \cdot \nabla u  \right)_\xi (t,x_3), \quad B_\xi=B_\xi(t):=  (\p_{x_3} (-\Delta_\xi)^{-1} N_{\xi,h} (t))|_{x_3=0}.
\Ee
Here $(-\Delta_\xi)^{-1}$ denotes the inverse of $-\Delta_\xi$ with the zero Dirichlet boundary condition at $x_3=0$.

We give the integral representation and present key estimates on Green's function for the Stokes problem. As shown in \cite{NN2018,FW}, letting $G_{\xi}(t,x_3,y)$ be the Green's function for \eqref{NS_f}-\eqref{NSB_f}, the solution can be represented by the integral formula via Duhamel's principle: 
\Be
 \begin{split}\label{o_xi}
 \o_\xi (t,x_3) =& {\int^\infty_0 G_{\xi } (t,x_3, y) \o_{0\xi} (y) \dd y} 
 + {\int^t_0\int^\infty_0 G_{\xi } (t-s, x_3, y) N_\xi (s,y ) \dd y \dd s}\\
 &
- {\int^t_0  G_{\xi } (t-s, x_3, 0) (B_\xi (s) ,0)\dd s} ,
 \end{split}
 \Ee
 where \Be\label{G}
 G_{\xi }  = \begin{bmatrix}
 G_{\xi h} & 0 & 0\\
 0 &  G_{\xi h} & 0 \\
 0& 0&  G_{\xi 3}
 \end{bmatrix},
 \Ee
 with $G_{\xi h}$ of (\ref{G_xi}) and $G_{\xi  3}$ of (\ref{G_xi2}): for $G_{\xi *}$ can be either $G_{\xi h}$ or $G_{\xi  3}$
\begin{align}
\p_ t G_{\xi *} (t, x_3, y) - \kappa\eta_0 \Delta_\xi G_{\xi *} (t, x_3, y) =0, \ \ x_3>0,\label{eqtn:G_xi2}\\
\kappa\eta_0 (\p_{x_3} + |\xi|) G_{\xi h} (t, x_3, y)=0, \ \ x_3=0, \label{bdry:G_xi2}\\
 G_{\xi 3} (t, x_3, y)=0, \ \ x_3=0. \label{bdry:G_xi3}
\end{align}
%Clearly  $\p_t \o_{\xi, 3}(t,x_3)$ takes the form of the third component of (\ref{o_xi2t}). 

%\subsection{Inhomogeneous Stokes equations} The following Stokes problem 
 %\begin{align}
%\p_t W_\xi - \kappa \eta_0 \Delta_\xi W_\xi &= P_\xi \label{S_f} \\
%\kappa \eta_0 (\p_{x_2} + |\xi| ) W_\xi  &= Q_\xi   \label{SB_f}
%\end{align} 
%with the initial data $W_\xi(0,x_2) = {W_0}_\xi$ can be solved by using the Green's function and Duhamel's principle \cite{NN2018}: 

Similarly for $\theta$, by taking the Fourier transform of \eqref{heat_NS}, \eqref{thetaB} in $x_h$, we have $\theta_{\xi}(t,x_3)$ solves
\Be \begin{split}
\p_t \theta_\xi - \kappa \eta_c \Delta_\xi \theta_\xi = & M_\xi \text{ in } \mathbb R_+,
\\ \theta_\xi = & 0 \text{ on } x_3 = 0,
\end{split} \Ee
with $\theta_\xi |_{t=0} = \theta_{0\xi} $ for $\xi \in \mathbb Z^2$. Here $M_\xi = M_\xi(t,x_3) : =  (- u \cdot \nabla \theta)_\xi(t,x_3) $. Thus the integral representation is
\Be \label{thetaphi}
\theta_\xi(t,x_3) = \int_0^\infty G_{\xi3} (t,x_3, y) \theta_{0 \xi} (y) dy + \int_0^t \int_0^\infty G_{\xi3}(t-s,x_3,y) M_\xi (s,y) dyds,
\Ee
here by abuse of notation $G_{\xi 3} $ is the same Green's function solving \eqref{eqtn:G_xi2} and \eqref{bdry:G_xi3}, with $\eta_0$ replaced by $\eta_c$. Note that since $u = 0 $ on $\p \O$, we have
\Be \label{ptellMequal0}
\p_t^\ell M_\xi(t, 0 ) = 0, \, \ell = 0,1,2.
\Ee

The following estimates and properties for $G_{\xi}$ will be useful to show the propagation of analytic norms of $\p_t^\ell \o$, $\p_t^\ell \theta$, for $\ell = 0,1,2$. % from \cite{NN2018}. 

 \begin{lemma}[\cite{NN2018,FW}] \label{lem_G}
 \begin{enumerate}
 
 \item (Bounds on $G_{\xi h}$)
% For any $T>0$ and for any $P_\xi \in L^\infty(0,T; L^1(\mathbb R_+))$ and $Q_\xi \in L^\infty(0,T)$, the unique solution to  \eqref{S_f}-\eqref{SB_f} with the initial data $W_\xi(0,x_2) = {W_0}_\xi$ in $L^1(\mathbb R_+)$ satisfies 
% \Be
% \begin{split}\label{W_xi}
% W_\xi (t,x_2) = {\int^\infty_0 G_{\xi,1} (t,x_2, y) W_{0\xi} (y) \dd y} 
% + {\int^t_0\int^\infty_0 G_{\xi,1} (t-s, x_2, y) P_\xi (s,y ) \dd y \dd s}
%- {\int^t_0  G_{\xi,1} (t-s, x_2, 0) Q_\xi (s) \dd s} 
 %\end{split}
 %\Ee
The Green's function $G_{\xi h}$ for the Stokes problem (\ref{eqtn:G_xi2}) and (\ref{bdry:G_xi2}) is given by 
  \Be\label{G_xi}
 G_{\xi h} = \tilde{H}_\xi + R_\xi,
 \Ee
 where $\tilde{H}_\xi$ is the one dimensional Heat kernel in the half-space with  the homogeneous Neumann boundary condition which takes the form of 
\Be
 \tilde{H}_\xi (t,x_3,y)= H_\xi(t,x_3-y) + H_\xi(t,x_3+y) = \frac{1}{\sqrt{\kappa\eta_0 t}} \bigg(
 e^{- \frac{|x_3-y|^2}{4 \kappa \eta_0 t}} +  e^{- \frac{|x_3+y|^2}{4 \kappa\eta_0 t}}
 \bigg) e^{- \kappa\eta_0 |\xi|^2 t},
 \label{H_xi}
 \Ee
and the residual kernel $R_\xi$ due to the boundary condition satisfies 
 \Be
 |\p_{x_3}^k R_\xi(t,x_3, y)| \lesssim b^{k+1} e^{- \theta_0 b (x_3 + y)} + 
 \frac{1}{(\kappa \eta_0 t)^{(k+1)/2}} e^{- \theta_0 \frac{|x_3+y|^2}{\kappa\eta_0 t}} e^{- \frac{\kappa\eta_0 |\xi|^2 t}{8}} 
 ,\label{R_xi}
\Ee
 with $b= |\xi| + \frac{1}{\sqrt{\kappa\eta_0}}$  and $R_\xi (t,x_3,y) = R_\xi (t, x_3 + y)$. 
 
  \item  (Formula of $G_{\xi 3}$) The Green's function $G_{\xi 3}$ for the Stokes problem (\ref{eqtn:G_xi2}) and (\ref{bdry:G_xi3}) is given by one dimensional Heat kernel in the half-space with the homogeneous Dirichlet boundary condition as   \Be
  \label{G_xi2}
G_{\xi 3} (t,x_3,y) = H_\xi(t,x_3-y) -H_\xi(t,x_3+y) = \frac{1}{\sqrt{\kappa\eta_0 t}} \bigg(
 e^{- \frac{|x_3-y|^2}{4 \kappa \eta_0 t}} -  e^{- \frac{|x_3+y|^2}{4 \kappa\eta_0 t}}
 \bigg) e^{- \kappa\eta_0 |\xi|^2 t}.
  \Ee

 \item (Complex extension) The Green's function $G_{\xi}$ has a natural extension to the complex domain $\mathcal{H}_\lambda$ for small $\lambda>0$ with similar bounds in terms of $\text{Re}\, y$ and $\text{Re}\, z$ (cf. (3.16) in \cite{NN2018}). 
 The solution $\o_\xi$ to \eqref{NS_f}-\eqref{NSB_f} in $\mathcal{H}_\lambda$ has a similar representation: for any $z\in\mathcal{H}_\lambda$, let $\sigma$ be the positive constant so that $z\in \p \mathcal{H}_\lambda$, then $\o_\xi$ satisfies
 \Be
 \begin{split}\notag
 \o_\xi (t,z) =& {\int_{\p \mathcal{H}_\sigma} G_{\xi } (t,z, y) \o_{0\xi} (y) \dd y} 
 + {\int^t_0\int_{\p \mathcal{H}_\sigma} G_{\xi } (t-s, z, y) N_\xi (s,y ) \dd y \dd s}\\
 &
- {\int^t_0  G_{\xi } (t-s, z, 0) (B_\xi (s), 0) \dd s} .
 \end{split}
 \Ee
 \end{enumerate}
 \end{lemma}
 
 The proof of Lemma \ref{lem_G} can be found in Proposition 3.3 and Section 3.3 of \cite{NN2018}.  The next lemma concerns the convolution estimates.

\begin{lemma} \label{lem_Gc} Let $T>0$ be given. For any $0\leq s < t\leq T$ and $k\geq 0$, there exists a constant $C_T>0$ so that the following estimates hold: for $G_{\xi *}$ can be either $G_{\xi h}$ or $G_{\xi 3}$
\begin{enumerate}
\item ($\mathcal{L}^1_\lambda$ estimates)  
\begin{align}
\sum_{j=0}^k \left\| (\zeta(z)\p_{z})^j \int_0^\infty G_{\xi*}(t, z, y ) g_\xi (y) \dd y \right\|_{\mathcal{L}^1_\lambda} \leq C_T \sum_{j=0}^k \left\|   (\zeta(z)\p_{z})^j g_\xi \right\|_{\mathcal{L}^1_\lambda},\label{est:L1G1} \\
\sum_{j=0}^k \left\|  (\zeta(z)\p_{z})^j \int_0^\infty G_{\xi *}(t-s,z, y ) g_\xi (y) \dd y \right\|_{\mathcal{L}^1_\lambda} \leq C_T \sum_{j=0}^k \left\|  (\zeta(z)\p_{z})^j g_\xi \right\|_{\mathcal{L}^1_\lambda}.\label{est:L1G2} 
\end{align}

\item ($\mathcal{L}^\infty_{\lambda,\kappa t}$ estimates)
\begin{align}
\sum_{j=0}^k \left\| (\zeta(z)\p_{z})^j \int_0^\infty G_{\xi *} (t, z, y ) g_\xi (y) \dd y \right\|_{\mathcal{L}^\infty_{\lambda,\kappa t}} \leq C_T \sum_{j=0}^k \left\|   (\zeta(z)\p_{z})^j g_\xi \right\|_{\mathcal{L}^\infty_{\lambda,\kappa }}, \label{est:Linfty_tG1}  \\
\sum_{j=0}^k \left\|  (\zeta(z)\p_{z})^j \int_0^\infty G_{\xi *}(t-s, z, y ) g_\xi (y) \dd y \right\|_{\mathcal{L}^\infty_{\lambda,\kappa t}} \leq C_T \sum_{j=0}^k \sqrt{\frac{t}{s}}  \left\|  (\zeta(z)\p_{z})^j g_\xi \right\|_{\mathcal{L}^\infty_{\lambda,\kappa s}}.  \label{est:Linfty_tG2}
\end{align}

\item ($\mathcal{L}^\infty_{\lambda,\kappa}$ estimates) For either $\kappa=0$ or $\kappa>0$
\begin{align}
\sum_{j=0}^k \left\| (\zeta(z)\p_{z})^j \int_0^\infty G_{\xi *}(t,z, y ) g_\xi (y) \dd y \right\|_{\mathcal{L}^\infty_{\lambda,\kappa}} \leq C_T \sum_{j=0}^k \left\|   (\zeta(z)\p_{z})^j g_\xi \right\|_{\mathcal{L}^\infty_{\lambda,\kappa}},  \label{est:LinftyG1} \\
\sum_{j=0}^k \left\|  (\zeta(z)\p_{z})^j \int_0^\infty G_{\xi *}(t-s, z, y ) g_\xi (y) \dd y \right\|_{\mathcal{L}^\infty_{\lambda,\kappa}} \leq C_T \sum_{j=0}^k \left\|  (\zeta(z)\p_{z})^j g_\xi \right\|_{\mathcal{L}^\infty_{\lambda,\kappa}}.  \label{est:LinftyG2}
\end{align}

\end{enumerate}
\end{lemma}

The proof of (1) and (2) can be found in Propositions 3.7 and 3.8 of \cite{NN2018}. For the proof of (3) we refer to Lemma 12 of \cite{JK}.

The next result concerns the estimates for the trace kernel. 

\begin{lemma}\label{lem_Gt} Let $a_\xi(s)= [\p_{z}(- \Delta_\xi)^{-1} g_\xi ] \, |_{z=0}$. Then for any $0\leq s<t\leq T$ and $k\geq 0$, we have the following
\begin{align}
\sum_{j=0}^k \left\|  (\zeta(z)\p_{z})^j  G_{\xi h }(t-s,z, 0 ) a_\xi (s) \right\|_{\mathcal{L}^1_\lambda} &\lesssim  \left\|  g_\xi \right\|_{\mathcal{L}^1_\lambda}, \label{trace1} \\
\sum_{j=0}^k \left\|  (\zeta(z)\p_{z})^j  G_{\xi  h}(t-s,z, 0 ) a_\xi (s)  \right\|_{\mathcal{L}^\infty_{\lambda,\kappa}}  &\lesssim \frac{1}{\sqrt{t-s}} \left\| g_\xi \right\|_{\mathcal{L}^1_\lambda}. %+  \sum_{j=0}^k \left\|  (\zeta(z)\p_{z})^j g_\xi \right\|_{L^\infty_{\lambda,\kappa}}
\label{trace2}
\end{align}
\end{lemma}
We refer the proof to Lemma 13 of \cite{JK}.

%\begin{proof} Note that from \eqref{G_xi}, \eqref{H_xi} and \eqref{R_xi}, the conormal derivatives $(\zeta(z)\p_z)^j$ of  $G_{\xi h}(t-s,z,0)$ enjoy the same bounds as $G_{\xi  h}(t-s,z,0)$: for some small constant $c_0$,
%\Be\label{est:Con_G}
%| (\zeta(z)\p_{z})^j  G_{\xi  }(t-s,z, 0 ) | \lesssim b e^{-c_0 b z} + \frac{1}{\sqrt{\kappa(t-s)}} e^{-c_0 \frac{|z|^2}{\kappa (t-s)}}.
%\Ee
%Therefore, it suffices to show the bounds for $k=0$. 
%We first recall the representation formula for $a_\xi$ 
%(cf. (4.29) of \cite{KVW} or (4.2) of \cite{FW}): 
%\Be\notag
%a_\xi(s) = \int_0^\infty e^{-|\xi| y} g_\xi (y) \dd y ,
%\Ee
%from which we have $\|a_\xi\|_{\mathcal{L}^\infty_\lambda} \lesssim  \left\|  g_\xi \right\|_{\mathcal{L}^1_\lambda}$. Since the above upper bound of $G(t-s,z,0)$ is integrable in $z$,  \eqref{trace1} follows. To show \eqref{trace2}, we compute $\| G_{\xi h}(t-s,z, 0 ) \|_{\mathcal{L}^\infty_{\lambda,\kappa}}$: 
%%{\color{red}\texttt{...below $\beta$ to $\bar\alpha$ ?...}}
%\begin{align*}
%\| G_{\xi h}(t-s,z, 0 ) \|_{\mathcal{L}^\infty_{\lambda,\kappa}} \lesssim  \sup_z \left[\frac{ be^{(\bar\alpha-c_0 b) z}}{1+\phi_\kappa(z)}\right]   +  \frac{1}{\sqrt{t-s}} \sup_z\left[ \frac{e^{ \bar\alpha z-c_0 \frac{|z|^2}{\kappa (t-s)}}}{\sqrt\kappa+\sqrt\kappa\phi_\kappa(z)}  \right].
%\end{align*}
%It is a routine to check that both supremum norms are uniformly bounded in $\kappa$ and $|\xi|$. Therefore \eqref{trace2} is obtained. 
%\end{proof}

%\bigskip 

 \subsection{Proof of Theorem \ref{thm_bound}} 
Our goal is to show that $\o(t)$ indeed belongs to $C^1([0,T];\mathfrak B^{\lambda, \kappa})$ without the initial layer under the compatibility condition \eqref{CC}, and that $\p_t^2\o$ in $\mathfrak B^{\lambda, \kappa t}$ with the initial layer.
And $\theta(t)$ belongs to $C^2([0,T]; \mathfrak B^{\lambda, 0})$ under the compatibility condition \eqref{CC}. The existence of $\o(t)$ and $\theta(t)$ in their corresponding spaces under the assumption of Theorem \ref{thm_bound} can be proved by following the argument of \cite{NN2018} and \cite{FW}. For the 2D case, Theorem 1.1 of \cite{NN2018} indeed ensures the existence of $\o(t)$ in $C^1([0,T];\mathfrak B^{\lambda, \kappa t})$ under the assumption of Theorem \ref{thm_bound}.  Such a result 
follows from Lemma \ref{lem_embedding}, Lemma \ref{lem_G},  Lemma \ref{lem_Gc},  Lemma \ref{lem_elliptic},  Lemma \ref{lem_bilinear}.  A 3D result can be obtained analogously. Hence, it suffices to show the propagation of the analytic norms in \eqref{norm_bound}. 

The propagation estimate for the vorticity $\o$ is analogous to %is similar to
 \cite{JK}; we provide steps for reader's convenience and refer to \cite{JK} for the detailed proof of some estimates when the same proof holds. The detailed proof will be given for the new estimate for the temperature fluctuation $\theta(t)$.  

\

{\bf Step 1: Propagation of analytic norms for $\o$.} It is convenient to define  
\Be\label{|||}
 \vertiii{\o(t)}_t:=  \vertiii{\o(t)}_{\infty, \kappa } + \vertiii{\o(t)}_{1}.
\Ee 
From the nonlinear iteration using the representation formula \eqref{o_xi}, by using Lemma \ref{lem_embedding} and Lemma \ref{lem_est:N},  we obtain
for sufficiently large $\gamma_0$
\Be\label{est:|||_t}
\sup_{0 \leq t <\frac{\lambda_0}{2\gamma_0}}\vertiii{ \o(t)}_{t}\lesssim  \sum_{0 \leq  |\beta|  \leq 2} \| D^\beta\o_{0} \|_{\infty, \lambda_0, \kappa}+  \sum_{0 \leq |\beta|\leq 2} \| D^\beta (1+ |\nabla_h|) 
  \o_{0} \|_{1, \lambda_0 }.
\Ee
We refer to \cite{JK} for the detailed proof.

 \

{\bf Step 2: Propagation of analytic norms for $\p_t\o$.}   The continuity of $\o(t)$ in $t$ follows from the mild solution form \eqref{o_xi} of $\o_\xi(t)$.  We claim that $\o(t) \in C^1([0,T]; \mathfrak B^{\lambda,\kappa})$ and moreover 
$\vertiii{ \p_t\o(t)}_{t}% + \vertiii{ \p_t\o(t)}_{\infty,\kappa} 
$ is bounded. To this end, we first derive the mild form of $\p_t\o_\xi$ from \eqref{o_xi}: 
\Be
 \begin{split}\label{o_xi2t}
 \p_t\o_\xi (t,x_3) =& {\int^\infty_0 G_\xi (t,x_3, y) \p_t \o_{0 \xi} (y) \dd y} 
 + {\int^t_0\int^\infty_0 G_\xi (t-s, x_3, y) \p_s N_\xi (s,y ) \dd y \dd s}\\
 &
- {\int^t_0  G_\xi (t-s, x_3, 0) (\p_s B_\xi (s), 0) \dd s} ,
 \end{split}
 \Ee
where we recall $\p_t \o_0$ in \eqref{idata}. To justify this formula, we first recall (\ref{eqtn:G_xi2})-(\ref{bdry:G_xi3}). We start with the horizontal part of the formula (\ref{o_xi2t}) for %the first two components for 
$\p_t\o_{\xi,h}$.  
From Lemma \ref{lem_G}, $G_{\xi h} (t,x_3,y) = H_\xi (t, x_3-y) +H_\xi (t,x_3+y)+ R_\xi (t,x_3+y)$. Then by using the fact that $H^\prime_\xi (t,  \cdot)$ is an odd function, we see 
$
\p_{x_3} G_{\xi h} (t,x_3,y)|_{x_3=0} %&= ( H^\prime_\xi (t,x_2-y_2) +H^\prime_\xi (t,y_2+x_2)+ R^\prime_\xi (t,x_2+y_2))|_{x_2=0}\\
%&=  H^\prime_\xi (t,  -y_2) +H^\prime_\xi (t,y_2 )+ R^\prime_\xi (t, y_2) \\
= R^\prime_\xi (t, y)$. 
Now we read (\ref{bdry:G_xi2}) as 
\Be\begin{split}
\label{bdry:G_xi12}
\kappa \eta_0 R^\prime_\xi (t,y ) + \kappa\eta_0 |\xi| G_{\xi h} (t,0,y)=0, \quad \kappa\eta_0 R^\prime_\xi (t,x_3) + \kappa\eta_0 |\xi| G_{\xi h} (t,x_3,0)=0 ,
\end{split}\Ee
where we have used that $H_\xi (t, \cdot)$ is an even function for the second relation. On the other hand, since we also have 
$
\p_{y_3} G_{\xi h} (t,x_3,y)|_{y=0} %&= (- H^\prime_\xi (t,x_2-y_2) +H^\prime_\xi (t,y_2+x_2)+ R^\prime_\xi (t,x_2+y_2))|_{y_2=0}\\
%&=  -H^\prime_\xi (t,  x_2) +H^\prime_\xi (t,x_2 )+ R^\prime_\xi (t, x_2) \\
= R^\prime_\xi (t, x_3)$, we deduce that 
\Be\label{bdry:G_xiy2}
\kappa\eta_0 (\p_{y_3}  +|\xi|) G_{\xi h} (t,x_3,y_3) =0, \ \ y_3=0.
\Ee
It is straightforward to see $\Delta_\xi G_{\xi h} = \p_{x_3}^2G_{\xi h} - |\xi|^2G_{\xi h} = \p_y^2 G_{\xi h} - |\xi|^2 G_{\xi h}.$  

We now take $\p_t$ of \eqref{o_xi} to obtain%Note that the formal computation below is justified from the regularity and decay property of $G_\xi(t,\cdot,\cdot)$ for $t>0$. 
\begin{align*}
& \p_t \int^\infty_0 G_{\xi h} (t,x_3,y)\o_{0\xi,h}(y) \dd y = \int^\infty_0 \p_t G_{\xi, h} (t,x_3,y)\o_{0\xi, h}(y) \dd y \\
% &\quad= \int^\infty_0 \kappa\eta_0 (\p_{y}^2 - |\xi|^2) G_{\xi h} (t,x_3,y)\o_{0\xi,h}(y) \dd y \\
% &\quad= \big[ \kappa\eta_0 \p_y G_{\xi h} (t,x_3,y)\o_{0\xi, h}(y) \big]_{y=0}^{y=\infty}
% -  \int^\infty_0 \kappa\eta_0  |\xi|^2  G_{\xi h} (t,x_3,y)\o_{0\xi, h }(y) \dd y\\
% &\quad \ \ \ 
% -  \int^\infty_0 \kappa \eta_0  \p_y G_{\xi h} (t,x_3,y)\p_y\o_{0\xi, h}(y) \dd y
% \\
%  &\quad= \big[ \kappa\eta_0 \p_y G_{\xi h} (t,x_3,y)\o_{0\xi, h }(y) \big]_{y=0}^{y=\infty}  -  
% \big[
%  \kappa\eta_0    G_{\xi h} (t,x_3,y)\p_y\o_{0\xi, h}(y) 
% \big]_{y=0}^{y=\infty} \\
% &\quad \ \ \ 
%+ \int^\infty_0   G_{\xi h} (t,x_3,y)\kappa \eta_0 \Delta_\xi \o_{0\xi, h}(y) \dd y\\
& \quad =- \kappa\eta_0 \p_y G_{\xi h} (t,x_3,0)\o_{0\xi,h}(0)  +  
  \kappa\eta_0    G_{\xi h} (t,x_3,0)\p_y\o_{0\xi,h}(0)  \\
  & \quad \ \ \ \ 
+ \int^\infty_0   G_{\xi h} (t,x_3,y)\kappa\eta_0  \Delta_{\xi h} \o_{0\xi,h}(y) \dd y,
\end{align*}
and 
\begin{align*}
\p_t  {\int^t_0\int^\infty_0 G_{\xi h} (t-s, x_3, y) N_{\xi, h} (s,y ) \dd y \dd s}
  &=\int_0^\infty G_{\xi h} (t,x_3,y)  N_{\xi,h} (0,y)\dd y\\
  & \ \ \ 
+ \int^t_0 \int^\infty_0
G_{\xi h} (s,x_3, y) \p_t N_{\xi, h} (t-s,y)
 \dd y \dd s,\\
 \p_t  {\int^t_0  G_{\xi h} (t-s, x_3, 0) B_\xi (s) \dd s} &=
 G_{\xi h} (t,x_3,0) B_\xi (0)
+\int^t_0 G_{\xi h} (t-s,x_3,0) \p_s B_\xi( s) \dd s.
\end{align*}
Therefore we obtain 
\Be\begin{split}
\p_t \o_{\xi, h} (t,x_3) =&{- \kappa\eta_0 \p_y G_{\xi h} (t,x_3,0)\o_{0\xi,h}(0)  +  
  \kappa\eta_0    G_{\xi h} (t,x_3,0)\p_y\o_{0\xi, h}(0)  - G_{\xi h} (t,x_3,0) B_\xi (0)} \\
  &+ \int^\infty_0   G_{\xi h} (t,x_3,y) \{ \kappa \eta_0\Delta_\xi \o_{0\xi,h}(y)
+ N_{\xi, h} (0,y) \}
 \dd y\\
 &+ \int^t_0 \int^\infty_0
G_{\xi h} (t-s,x_3, y) \p_s N_{\xi, h} (s,y) 
 \dd y \dd s - \int^t_0 G_{\xi h} (t-s,x_3,0) \p_s B_\xi(s) \dd s\label{pt_o2}.
 \end{split}
\Ee
Next we show that the first line in the right-hand side is 0. From (\ref{bdry:G_xiy2})
\begin{align*} 
- \kappa\eta_0 \p_y G_{\xi h} (t,x_3,0)\o_{0\xi, h}(0)  +  
  \kappa\eta_0    G_{\xi h} (t,x_3,0)\p_y\o_{0\xi,h}(0)=G_{\xi h} (t,x_3,0) \kappa\eta_0(|\xi | +  \p_y) \o_{0\xi, h}(0),%=G_\xi (t,x_2,0) B_\xi(0).
\end{align*}
and hence the first line of \eqref{pt_o2} reads 
\Be\label{initial_layer0}
G_{\xi h} (t,x_3,0) \left[ \kappa\eta_0(|\xi | +  \p_{x_3}) \o_{0\xi, h}(0) - B_\xi (0) \right]  , 
\Ee
which is zero due to the first compatibility condition of \eqref{CC}. Recalling $\p_t \o_0$ in \eqref{idata}, the formula \eqref{o_xi2t} for $\p_t\o_{\xi,h}$ has been established. We may follow the same procedure to verify the vertical part of the formula \eqref{o_xi2t} for $\p_t\o_{\xi,3}$ by noting  that the second compatibility condition of \eqref{CC} removes the term $- \kappa\eta_0 \p_y G_{\xi 3} (t,x_3,0)\o_{0\xi,3}(0)$ which would create the initial layer otherwise because $ \p_y G_{\xi 3} (t,x_3,0)$ does not vanish.  

We may now repeat Step 1 for $\p_t\o$ using the representation formula \eqref{o_xi2t}. The estimates are obtained in the same fashion. For the nonlinear terms, since $\p_t N = - u \cdot \nabla \p_t \o - \p_t u\cdot \nabla \o
+ \o \cdot \nabla \p_t  u+\p_t \o  \cdot \nabla  u
$, the structure of $\p_t N$ with respect to $\p_t\o$ is consistent  with the one of $N$ with respect to $\o$ and we can use the bilinear estimates \eqref{nonlinear1} and \eqref{nonlinear2}. In summary, one can derive that for $t < \frac{\lambda_0}{2 \gamma_0}$
\begin{align}
\vertiii{ \p_t\o(t)}_{1} &\lesssim \sum_{0\leq |\beta|\leq 2} \| D^\beta (1+|\nabla_h|) \p_t \o_0  \|_{1, \lambda_0}  
 +( t+  \frac{1}{\gamma_0}) \sup_{0 \leq s \leq t}\vertiii{ \o(s)}_{s} \sup_{0 \leq s \leq t} \vertiii{\p_t \o(s)}_s,
 \label{est:t_1}
 \\
\vertiii{ \p_t\o(t)}_{\infty,\kappa} &\lesssim \sum_{0\leq|\beta|\leq 2} \| D^\beta \p_t \o_{0} \|_{\infty, \lambda_0, \kappa} +( \sqrt{t}+  \frac{1}{\gamma_0}) \sup_{0 \leq s \leq t}\vertiii{ \o(s)}_{s} \sup_{0 \leq s \leq t} \vertiii{\p_t \o(s)}_s, \label{est:t_infty}
\end{align}
which lead to the desired bounds for $\p_t\o(t)$ by choosing sufficiently large $\gamma_0$.  

\hide
\begin{align}
\vertiii{ \p_t\o(t)}_{1} &\lesssim \sum_{0\leq i+j\leq 2} \| \p_{x_1}^i (\zeta(z)\p_{z})^j \o_{1\xi} \|_{1, \lambda_0}  
 + \frac{1}{\gamma_0}\vertiii{ \o(t)}_{1} \vertiii{\p_t \o(t)}_1\\
\vertiii{ \p_t\o(t)}_{\infty,\kappa} &\lesssim \sum_{0\leq i+j\leq 2} \| \p_{x_1}^i (\zeta(z)\p_{z})^j \o_{1\xi} \|_{\infty, \lambda_0, \kappa}  + \vertiii{ \o(t)}_{1} \vertiii{\p_t \o(t)}_1 
 + \frac{1}{\gamma_0}\vertiii{ \o(t)}_{\infty,\kappa}\vertiii{ \p_t\o(t)}_{\infty,\kappa}
\end{align}\unhide
 
 \

 {\bf Step 3: Propagation of analytic norms for $\p_t^2\o_\xi$.}  As a consequence of Step 2, $\p_t\o_\xi(t,x_3)$ solves the following system
  \begin{align}
\p_t^2 \o_\xi - \kappa \eta_0 \Delta_\xi \p_t \o_\xi =\p_t N_\xi  \quad &\text{in }\mathbb R_+, \label{NS_f1} \\
\kappa \eta_0 (\p_{x_3} + |\xi| )\p_t\o_{\xi,h}   = \p_tB_\xi \quad &\text{on } x_3=0, \label{NSB_f1}\\
\p_t \o_{\xi, 3}  = 0 \quad &\text{on } x_3=0, \label{NSB_f3}
\end{align} 
with $\p_t\o_\xi|_{t=0} = \p_t \o_{0 \xi}$ for $\xi\in \mathbb Z^2$ where $\p_t \o_0$ is defined in (\ref{idata}). Then as done in Step 2, by using the properties of $G_\xi$ and  integration by parts and by the last compatibility condition of \eqref{CC}, we can derive the representation formula for $\p_t^2\o$: 
 \Be
 \begin{split}\label{o_xi2tt}
& \p_t^2\o_\xi (t,x_3) =( G_{\xi h} (t,x_3,0) \left[ \kappa\eta_0(|\xi | +  \p_{x_3})  \p_t \o_{ 0\xi, h}(0) - \p_tB_\xi (0) \right] , 0 )  \\
&\quad+  {\int^\infty_0 G_\xi (t,x_3, y) \p_t^2 \o_{0\xi} (y) \dd y} 
 + {\int^t_0\int^\infty_0 G_\xi (t-s, x_3, y) \p_s^2 N_\xi (s,y ) \dd y \dd s}\\
 &\quad 
- {\int^t_0  G_\xi (t-s, x_3, 0)( \p_s^2  B_\xi (s),0) \dd s} ,
 \end{split}
 \Ee
where we recall $\p_t^2\o_0$ in \eqref{idata}. 

Similar to the estimate in Step 2, the ${L}^1$-based analytic norm is easily obtained as 
\Be\label{est:tt_1}
\begin{split}
&\vertiii{\p_t^2 \o (t)} _1\\&\lesssim
 \kappa \eta_0
\| (1+ |\nabla_h|^3) \nabla \p_t \o_{0 } \|_{1, \lambda} + 
\| (1+|\nabla_h|^4) \p_t \o_{0} \|_{1, \lambda}
\sum_{0 \leq |\beta| \leq1} \| D^\beta (1+ |\nabla_h|^3) \p_t \o_{0}\|_{1, \lambda}
 \\
& \ \ 
 +  \sum_{0\leq |\beta|\leq 2} \| D^\beta (1+|\nabla_h|) \p^2_t \o_0  \|_{1, \lambda_0}  
 +( t+  \frac{1}{\gamma_0})
  \sup_{0 \leq s \leq t}\vertiii{ \o(s)}_{s} \sup_{0 \leq s \leq t} \vertiii{\p_t^2 \o(s)}_s
 \\
 &\ \  +( t+  \frac{1}{\gamma_0})
  \sup_{0 \leq s \leq t}\vertiii{ \p_t \o(s)}_{s}  ^2.
\end{split}
\Ee

For the $L^\infty$-based analytic norm bound, as we do not require higher order compatibility condition for the horizontal vorticity, a new term representing the initial-boundary layer emerges and we obtain
\Be\label{est:tt_infty}
\begin{split}
&\vertiii{\p_t^2 \o (t)} _{\infty, \kappa t}\\
&\lesssim
\kappa \eta_0 \sum_{0 \leq |\beta| \leq 2} \| \nabla_h^\beta \nabla \p_t \o_{0} \|_{\infty, \lambda } + \| (1+|\nabla_h|^3) \p_t \o_{0} \|_{1,\lambda}
\sum_{0 \leq |\beta| \leq1} \| D^\beta (1+ |\nabla_h|^2) \p_t \o_{0}\|_{1,\lambda}
 \\
& \ \ 
 +  \sum_{0\leq |\beta|\leq 2} \| D^\beta  \p^2_t \o_0  \|_{\infty, \lambda_0, \kappa t}  
 +( \sqrt{t}+  \frac{1}{\gamma_0})
  \sup_{0 \leq s \leq t}\vertiii{ \o(s)}_{s} \sup_{0 \leq s \leq t} \vertiii{\p_t^2 \o(s)}_s
 \\
 & \ \  +(\sqrt{ t}+  \frac{1}{\gamma_0})
  \sup_{0 \leq s \leq t}\vertiii{ \p_t \o(s)}_{s}  ^2.
\end{split}
\Ee
We refer the detailed proof to \cite{JK}.

Finally combining (\ref{est:tt_1}) and (\ref{est:tt_infty}) and then choosing sufficiently large $\gamma_0$ we derive a desired estimate for $\vertiii{\p_t^2 \o(t)}_t$ for $t \in (0, \frac{\lambda_0}{2\gamma_0})$.

Altogether from (\ref{est:|||_t}), (\ref{est:t_1}), (\ref{est:t_infty}), (\ref{est:tt_1}), and (\ref{est:tt_infty}), we finish the proof of the estimate (\ref{norm_bound}).

  \
  
   {\bf Step 4: Propagation of analytic ${\vertiii \cdot }_t$ norms for $\p_t^\ell \theta$, $\ell = 0,1,2$. }

Recall the norm in \eqref{normtheta}. We define
\Be
\vertiii{ \theta(t)}_t := \vertiii{\theta(t) }_{\infty,0} + \vertiii{ \theta(t)}_1.
\Ee

Recall \eqref{thetaphi}. We now follow the argument in Step 1, from (\ref{est:L1G1}), (\ref{est:L1G2}), and Lemma \ref{lem_est:N} we have
 \Be\label{theta1_1}
 \begin{split}
& \sum_{0 \leq |\beta| \leq 2} \| D^\beta (1+ |\nabla_h|) \theta (s)\|_{1,\lambda}
\\ & \lesssim 
\sum_{0 \leq |\beta| \leq 2} \| D^\beta (1+ |\nabla_h|) \theta_0\|_{1,\lambda} + \int^t_0 \sum_{0 \leq |\beta| \leq 2} \| D^\beta (1+ |\nabla_h|) M(s) \|_{1,\lambda} \dd s .
\end{split}
 \Ee
Similar to the computation in Lemma \ref{lem_est:N}, and using Lemma \ref{lem_bilinear}  we bound the nonlinear $M$ term as
\Be \begin{split}
&  \sum_{0 \le | \beta | \le 1 } \| D^\beta ( 1 + | \nabla_h | ) M (s) \|_{1, \lambda} 
 \\ & \lesssim \sum_{0 \le | \beta | \le 1 } \left( \| \nabla_h ( 1 + | \nabla_h | ) \theta \|_{1,\lambda} \| D^\beta ( 1 + | \nabla_h | ) u_h \|_{\infty, \lambda} + \| \nabla_h D^\beta ( 1 + | \nabla_h | ) \theta \|_{1,\lambda} \| ( 1 + | \nabla_h | ) u_h \|_{\infty, \lambda}  \right.
\\   &  \left. \quad \quad \quad   +   \| \zeta \p_z ( 1 + | \nabla_h | ) \theta \|_{1, \lambda} \| D^\beta ( 1 + | \nabla_h | ) \frac{u_3}{\zeta} \|_{\infty,\lambda}  + \| D^\beta \zeta \p_z ( 1 + | \nabla_h | ) \theta \|_{1,\lambda} \| ( 1 + | \nabla_h | ) \frac{u_3}{\zeta} \|_{\infty, \lambda}  \right).
\end{split} \Ee
Using Lemma \ref{lem_elliptic} and the definition of $\vertiii{ \cdot }$ we have
\Be \label{theta1_1M}
 \begin{split}
\int_0^t  \sum_{0 \le | \beta | \le 1 } \| D^\beta ( 1 + | \nabla_h | ) M (s) \|_{1, \lambda}  ds \lesssim &  \int_0^t \vertiii{ \theta(s) }_s \vertiii{ \o (s) }_s ( 1 + (\lambda_0 -\lambda - \gamma_0 s )^{-\alpha} ) ds 
\\ \lesssim & \left( t + \frac{1}{\gamma_0} \right) \sup_{0 \le s \le t }  \vertiii{ \theta(s) }_s   \sup_{0 \le s \le t } \vertiii{ \o (s) }_s.
\end{split} \Ee

Next, by applying the analyticity recovery estimate $(3)$ of Lemma \ref{lem_embedding}, and choosing $\tilde \lambda = \frac{ \lambda + \lambda_0 - \gamma_0 s}{2} $ we get for $|\beta| =2$,
 \Be \label{thetarecovery_1}
 \begin{split}
\sum_{|\beta|=2}\| D^\beta (1+ |\nabla_h|) M (s) \|_{1,\lambda} \lesssim & \frac{1}{\tilde\lambda-\lambda} \sum_{0 \leq |\beta|\leq1}\| D^\beta (1+ |\nabla_h|)  M (s) \|_{1,\tilde\lambda} 
\\ \lesssim &\left( 1 + ( \lambda_0 - \lambda - \gamma_0 s )^{-(\alpha + 1 ) }  \right)  \vertiii{ \theta(s) }_s \vertiii{ \o (s) }_s ,
\end{split} 
\Ee

Therefore we derive that for $t< \frac{\lambda_0}{2\gamma_0}$ and $\lambda< \lambda_0 - \gamma_0 t$
\Be\label{theta1_2}
 \begin{split}
\int_0^t \sum_{|\beta|=2}\| D^\beta (1+ |\nabla_h|) M (s) \|_{1,\lambda} ds \lesssim &  \int_0^t  \left( 1 + ( \lambda_0 - \lambda - \gamma_0 s )^{-(\alpha + 1 ) }  \right)  \vertiii{ \theta(s) }_s \vertiii{ \o (s) }_s  ds
\\ \lesssim &   \left(  ( \lambda_0 - \lambda - \gamma_0 t )^{- \alpha }  \frac{1}{\gamma_0} + t  \right) \sup_{0 \le s \le t}  \vertiii{ \theta(s) }_s \sup_{0 \le s \le t } \vertiii{ \o (s) }_s  ds .
\end{split}
\Ee

Therefore, we conclude that, from (\ref{theta1_1}) with (\ref{theta1_1M}), and (\ref{theta1_2}), for $ t< \frac{\lambda_0}{2 \gamma_0}$,
\Be\label{thetaest:1}
 \vertiii{\theta (t)}_1\lesssim \sum_{0 \leq |\beta|\leq 2} \| D^\beta (1+ |\nabla_h|) 
  \theta_{0} \|_{1, \lambda_0 } + (t+ \frac{1}{\gamma_0}) \sup_{0 \leq s \leq t} \vertiii{\theta(s)}_s  \sup_{0 \leq s \leq t}\vertiii{ \o(s)}_s . 
\Ee

The propagation of the $L^\infty$-based norm $\vertiii{\theta(t) }_\infty $ can be shown analogously. Again follow the argument in Step 1, from (\ref{est:LinftyG1}), (\ref{est:LinftyG2}), and Lemma \ref{lem_est:N} we have
 \Be\label{theta1_12}
 \begin{split}
  \sum_{0 \leq |\beta| \leq 2} \| D^\beta \theta (s)\|_{\infty,\lambda,0}
  \lesssim 
\sum_{0 \leq |\beta| \leq 2} \| D^\beta  \theta_0\|_{\infty,\lambda,0} + \int^t_0 \sum_{0 \leq |\beta| \leq 2} \| D^\beta  M(s) \|_{\infty,\lambda,0} \dd s .
\end{split}
 \Ee
Similar to the computation in Lemma \ref{lem_est:N}, and using Lemma \ref{lem_bilinear}  we again bound
\Be \begin{split}
   \sum_{0 \le | \beta | \le 1 } \| D^\beta  M (s) \|_{\infty,\lambda,0} 
  & \lesssim \sum_{0 \le | \beta | \le 1 } \left( \| \nabla_h \theta \|_{\infty,\lambda,0} \| D^\beta u_h \|_{\infty, \lambda} + \| \nabla_h D^\beta  \theta \|_{\infty,\lambda,0} \|  u_h \|_{\infty, \lambda}  \right.
\\   &  \left. \quad \quad \quad   +   \| \zeta \p_z  \theta \|_{\infty,\lambda,0} \| D^\beta \frac{u_3}{\zeta} \|_{\infty,\lambda}  + \| D^\beta \zeta \p_z \theta \|_{\infty,\lambda,0} \|  \frac{u_3}{\zeta} \|_{\infty, \lambda}  \right).
\end{split} \Ee
Using Lemma \ref{lem_elliptic} and the definition of $\vertiii{ \cdot }$ we have
\Be \label{theta1_1M2}
 \begin{split}
\int_0^t  \sum_{0 \le | \beta | \le 1 } \| D^\beta M (s) \|_{\infty,\lambda,0}  ds \lesssim &  \int_0^t \vertiii{ \theta(s) }_s \vertiii{ \o (s) }_s ( 1 + (\lambda_0 -\lambda - \gamma_0 s )^{-\alpha} ) ds 
\\ \lesssim & \left( t + \frac{1}{\gamma_0} \right) \sup_{0 \le s \le t }  \vertiii{ \theta(s) }_s   \sup_{0 \le s \le t } \vertiii{ \o (s) }_s.
\end{split} \Ee
 
Next, again by applying the analyticity recovery estimate $(3)$ of Lemma \ref{lem_embedding}, and choosing $\tilde \lambda = \frac{ \lambda + \lambda_0 - \gamma_0 s}{2} $ we get for $|\beta| =2$,
 \Be \label{thetarecovery_1}
 \begin{split}
\sum_{|\beta|=2}\| D^\beta  M (s) \|_{\infty,\lambda,0} \lesssim & \frac{1}{\tilde\lambda-\lambda} \sum_{0 \leq |\beta|\leq1}\| D^\beta  M (s) \|_{\infty,\tilde \lambda,0} 
\\ \lesssim &\left( 1 + ( \lambda_0 - \lambda - \gamma_0 s )^{-(\alpha + 1 ) }  \right)  \vertiii{ \theta(s) }_s \vertiii{ \o (s) }_s ,
\end{split} 
\Ee

Therefore we derive that for $t< \frac{\lambda_0}{2\gamma_0}$ and $\lambda< \lambda_0 - \gamma_0 t$
\Be\label{theta1_22}
 \begin{split}
\int_0^t \sum_{|\beta|=2}\| D^\beta  M (s) \|_{\infty,\lambda,0} ds \lesssim &  \int_0^t  \left( 1 + ( \lambda_0 - \lambda - \gamma_0 s )^{-(\alpha + 1 ) }  \right)  \vertiii{ \theta(s) }_s \vertiii{ \o (s) }_s  ds
\\ \lesssim &   \left(  ( \lambda_0 - \lambda - \gamma_0 t )^{- \alpha }  \frac{1}{\gamma_0} + t  \right) \sup_{0 \le s \le t}  \vertiii{ \theta(s) }_s \sup_{0 \le s \le t } \vertiii{ \o (s) }_s  ds .
\end{split}
\Ee

Therefore, we conclude that, from (\ref{theta1_12}) with (\ref{theta1_1M2}), and (\ref{theta1_22}), for $ t< \frac{\lambda_0}{2 \gamma_0}$,
\Be\label{thetaest:2}
 \vertiii{\theta (t)}_{\infty, 0} \lesssim \sum_{0 \leq |\beta|\leq 2} \| D^\beta  \theta_{0} \|_{\infty,\lambda_0,0 } + (t+ \frac{1}{\gamma_0}) \sup_{0 \leq s \leq t} \vertiii{\theta(s)}_s  \sup_{0 \leq s \leq t}\vertiii{ \o(s)}_s . 
\Ee
From \eqref{thetaest:1} and \eqref{thetaest:2} and a standard continuity argument we conclude for sufficiently large $\gamma_0$,
\Be\label{thetaest:|||_t}
\sup_{0 \leq t <\frac{\lambda_0}{2\gamma_0}}\vertiii{ \theta (t)}_{t}\lesssim  \sum_{0 \leq  |\beta|  \leq 2} \| D^\beta\theta_{0} \|_{\infty, \lambda_0, 0}+  \sum_{0 \leq |\beta|\leq 2} \| D^\beta (1+ |\nabla_h|) 
  \theta_{0} \|_{1, \lambda_0 }.
\Ee

\

  {\bf Step 5: Propagation of analytic ${\vertiii \cdot }_z$ norms for $\p_t^\ell \theta$, $\ell = 0,1,2$. }

Next let's investigate $\p_{x_3} \theta$. 
%Let's define a norm:
%\Be \label{px3thetanorm}
%\vertiii{\theta(t)}_z :=  \sup_{\lambda < \lambda_0 - \gamma_0 t }  \left\{ (\lambda_0-\lambda-\gamma_0 t)^\alpha   \sum_{0 \le |\beta| \le 1} \| D_h^{\beta_h} \p_{x_3} \theta \|_{ {\infty, \lambda, 0} } \right\}.
%\Ee
Taking $\p_{x_3}$ to \eqref{thetaphi} gives
\Be \label{px3thetaphi}
\p_{x_3} \theta_\xi(t,x_3) = \int_0^\infty \p_{x_3} G_{\xi3} (t,x_3, y) \theta_{0 \xi} (y) dy + \int_0^t \int_0^\infty \p_{x_3} G_{\xi3}(t-s,x_3,y) M_\xi (s,y) dyds.
\Ee
Now using $G_{\xi 3} (t,x_3,y) = H_\xi(t,x_3-y) -H_\xi(t,x_3+y)$ as in \eqref{G_xi2} (with $\eta_0$ replaced by $\eta_c$), we have $\p_{x_3} G_{\xi 3} (t,x_3,y) = - \p_y H_\xi(t,x_3-y) - \p_y H_\xi(t,x_3+y)$. From the compatibility condition \eqref{CC}, $\theta_{0\xi} (0 ) = 0$. And from \eqref{ptellMequal0}, $M_{\xi} (s,0 ) = 0 $. Thus from integration by parts we get
\Be \label{px3thetaphi2}
\begin{split}
\p_{x_3} \theta_\xi(t,x_3) = & \int_0^\infty  \left(  H_\xi(t,x_3-y) +  H_\xi(t,x_3+y) \right)  \p_y  \theta_{0 \xi} (y) dy 
\\ & + \int_0^t \int_0^\infty \left(  H_\xi(t-s,x_3-y) +  H_\xi(t-s,x_3+y) \right) \p_y M_\xi (s,y) dyds.
\end{split}
\Ee
From the inequality
\Be \label{alphaweight}
e^{-\frac{|y-z|^2}{2M\kappa(t-s)}} e^{-\bar\alpha y} = e^{-\frac{1}{2}|\frac{y-z}{\sqrt{M\kappa(t-s)}} + \bar\alpha \sqrt{M\kappa(t-s)} |^2}  e^{\frac{M}{2}\bar\alpha^2\kappa(t-s)} e^{-\bar\alpha z} \leq e^{\frac{M}{2}\bar\alpha^2\kappa(t-s)} e^{-\bar\alpha z} \lesssim e^{-\bar\alpha z},
\Ee
we bound
\Be \label{Hpytheta0xi}
\begin{split}
& \left|   \int_0^\infty  \left(  H_\xi(t,x_3-y) +  H_\xi(t,x_3+y) \right)  \p_y  \theta_{0 \xi} (y) dy  \right|
\\  \lesssim &  \int_0^\infty \frac{1} { \sqrt{\kappa t }} e^{ - \frac{ |y-z|^2}{ M \kappa  t } } e^{- \bar \alpha y } e^{\bar \alpha y}  \p_y \theta_{0 \xi} (y) dy 
\\ \lesssim & e^{-\bar \alpha x_3 }  \| \p_{x_3} \theta_{ 0 \xi} \|_{L^\infty_{\lambda, 0 }} \int_0^\infty \frac{1} { \sqrt{\kappa t }} e^{ - \frac{ |y-z|^2}{ 2 M \kappa  t } } dy \lesssim  e^{-\bar \alpha x_3 }  \| \p_{x_3} \theta_{ 0 \xi} \|_{L^\infty_{\lambda, 0 }}.
\end{split}
\Ee
And similarly
\Be \label{HpyMxi}
 \left|   \int_0^\infty  \left(  H_\xi(t-s,x_3-y) +  H_\xi(t-s,x_3+y) \right)  \p_y  M_\xi (s,y) dy  \right| \lesssim   e^{-\bar \alpha x_3 } \| \p_{x_3} M_{\xi}(s) \|_{L^\infty_{\lambda,0} } 
\Ee
Therefore from \eqref{px3thetaphi2}, \eqref{Hpytheta0xi}, \eqref{HpyMxi}, and taking summation in $\xi \in \mathbb Z^2$, we get $\| \p_{x_3} \theta(t,x_3) \|_{\infty, \lambda ,0 } \lesssim  \| \p_{x_3} \theta_{0 } \|_{\infty, \lambda, 0 } + \int_0^t  \| \p_{x_3} M_{}(s) \|_{\infty,\lambda,0} ds$. And similarly
\Be \label{px3thetaest}
\sum_{0 \le | \beta_h | \le 1 } \| D_h^{\beta_h} \p_{x_3} \theta(t,x_3) \|_{\infty, \lambda ,0 } \lesssim   \sum_{0 \le | \beta_h | \le 1 } \| D_h^{\beta_h} \p_{x_3} \theta_{0 } \|_{\infty, \lambda, 0 } + \int_0^t   \sum_{0 \le | \beta_h | \le 1 } \| D_h^{\beta_h} \p_{x_3} M_{}(s) \|_{\infty,\lambda,0} ds.
\Ee
We have
\Be \label{px3Msplit}
 \| D_h \p_{x_3} M(s) \|_{\infty, \lambda, 0 }   \le \underbrace{  \| D_h \p_{x_3} u \cdot \nabla \theta +  \p_{x_3} u \cdot D_h \nabla \theta \|_{\infty, \lambda, 0 } }_{\eqref{px3Msplit}_1} + \underbrace{ \| D_h u \cdot \nabla \p_{x_3} \theta + u \cdot D_h \nabla \p_{x_3} \theta \|_{\infty, \lambda, 0 } }_{\eqref{px3Msplit}_2},
\Ee
and we bound the two parts separately. 

From \eqref{est:elliptic}, and that $\| \nabla_h \o_h \|_{\infty, \lambda} \lesssim (1 + \frac{1}{\sqrt \kappa} ) \| \nabla_h \o_h \|_{\infty, \lambda, \kappa}  \lesssim \frac{1}{\sqrt \kappa}$, we have
\Be \label{px3M1}
 \begin{split}
%\| &  ( D_h \p_{x_3} u \cdot \nabla \theta +  \p_{x_3} u_h \cdot D_h \nabla \theta  )(s)\|_{\infty, \lambda, 0 }
\eqref{px3Msplit}_1  \lesssim 
%& \left( \sum_{0 \le | \beta_h | \le 1 } \| D_h^{\beta_h} \p_{x_3} u_h(s) \|_{\infty, \lambda} \right) \left( \sum_{0 \le |\beta_h | \le 2 }  \| D_h^{\beta_h} \theta(s) \|_{\infty, \lambda, 0 } \right) 
%\\ & +  \left( \sum_{0 \le | \beta_h | \le 1 } \| D_h^{\beta_h} \p_{x_3} u_3(s) \|_{\infty, \lambda} \right) \left(  \sum_{0 \le |\beta_h | \le 1 } \| D_h^{\beta_h} \p_{x_3} \theta(s) \|_{\infty, \lambda, 0 } \right)
%\\  \lesssim &\left( \sum_{0 \le | \beta_h | \le 1 } \| D_h^{\beta_h} \p_{x_3} u(s) \|_{\infty, \lambda} \right) \left( \sum_{0 \le |\beta_h | \le 2 }  \| D_h^{\beta_h} \theta(s) \|_{\infty, \lambda, 0 } + \sum_{0 \le |\beta_h | \le 1 } \| D_h^{\beta_h} \p_{x_3} \theta(s) \|_{\infty, \lambda, 0 } \right)
%\\ 
%\lesssim 
& \left( \sum_{ 0 \le | \beta | \le 2 } \| \nabla_h^\beta \o(s) \|_{1,\lambda}  + \sum_{0 \le |\beta_h | \le 1 } \| D_h^{\beta_h} \o_h(s) \|_{\infty,\lambda} \right)   \sup_{0 \le s \le t } \vertiii{ \theta (s) }_s 
\\ & +  \left( \sum_{ 0 \le | \beta | \le 2 } \| \nabla_h^\beta \o(s) \|_{1,\lambda}   \right)   (\lambda_0 - \lambda - \gamma_0 s )^{-\alpha}   \sup_{0 \le s \le t } \vertiii{ \theta (s) }_z
\\ \lesssim &   \frac{1}{\sqrt \kappa} \sup_{0 \le s \le t } \vertiii{ \o (s) }_s    \sup_{0 \le s \le t } \vertiii{ \theta (s) }_s  +  (\lambda_0 - \lambda - \gamma_0 s )^{-\alpha}  \sup_{0 \le s \le t } \vertiii{ \o (s) }_s    \sup_{0 \le s \le t } \vertiii{ \theta (s) }_z
\\ \lesssim &  \frac{1}{\sqrt \kappa}  +  (\lambda_0 - \lambda - \gamma_0 s )^{-\alpha}     \sup_{0 \le s \le t } \vertiii{ \theta (s) }_z,
\end{split} \Ee
where we have used \eqref{norm_bound}, \eqref{thetanorm_bound}, and from previous steps $\sup_{0 \le s \le t } \vertiii{ \o (s) }_s < \infty$, $\sup_{0 \le s \le t } \vertiii{ \theta (s) }_z < \infty$. Next for $\eqref{px3Msplit}_2$, we use \eqref{est:elliptic}, and the analytic recovery lemma (Lemma \ref{lem_embedding} $(3)$) to get
\Be \label{px3M2}
\begin{split}
%\| &  (D_h u \cdot \nabla \p_{x_3} \theta + u \cdot D_h \nabla \p_{x_3} \theta)(s) \|_{\infty, \lambda, 0 } 
\eqref{px3Msplit}_2 \lesssim &  \sum_{i=1,2} \left(  \|  (D_h u_i \p_{x_i} \p_{x_3} \theta )(s) \|_{\infty, \lambda, 0 } +  \|  ( u_i D_h \p_{x_i} \p_{x_3} \theta )(s) \|_{\infty, \lambda, 0 } \right) 
\\ &  + \| ( \zeta^{-1} D_h u_3 \zeta  \p_{x_3} \p_{x_3} \theta )(s) \|_{\infty, \lambda, 0 } + \| ( \zeta^{-1}  u_3 \zeta  \p_{x_3}  D_h \p_{x_3} \theta )(s) \|_{\infty, \lambda, 0 }
\\ \lesssim & \left(  \sum_{0 \le |\beta_h | \le 1}  \|  D_h^{\beta_h} u_h (s) \|_{\infty, \lambda }  \right) \left( \sum_{ 0 \le |\beta_h | \le 1 } \|  \nabla_h D_h^{\beta_h }   \p_{x_3} \theta (s) \|_{\infty, \lambda, 0 } \right)  
\\ & + \left( \sum_{ 0 \le |\beta_h | \le 1 }   \|  \zeta^{-1}  D_h^{\beta_h } u_3 (s) \|_{\infty, \lambda } \right)  \left( \sum_{ 0 \le |\beta_h | \le 1 } \| \zeta \p_{x_3}  D_h^{\beta_h } \p_{x_3} \theta (s) \|_{\infty, \lambda, 0 } \right)
\\ \lesssim &   \left( \sum_{ 0 \le | \beta | \le 2 } \| \nabla_h^\beta \o_h(s) \|_{1,\lambda} + \sum_{ 0 \le |\beta_h | \le 1 }  \| D_h^{\beta_h} \o(s) \|_{1,\lambda} \right) \frac{1} { \tilde \lambda - \lambda }  \sum_{ 0 \le |\beta_h | \le 1 } \| D_h^{\beta_h}  \p_{x_3} \theta (s) \|_{\infty, \tilde \lambda, 0 } 
\\ \lesssim &  ( \lambda_0 - \lambda - \gamma_0 s)^{-(\alpha +1) }  \sup_{0 \le s \le t } \vertiii{ \o (s) }_s  \sup_{0 \le s \le t } \vertiii{ \theta (s) }_z 
\\  \lesssim &  ( \lambda_0 - \lambda - \gamma_0 s)^{-(\alpha +1) }  \sup_{0 \le s \le t } \vertiii{ \theta (s) }_z,
\end{split}
\Ee
where as before, we have chosen $\tilde \lambda = \frac{ \lambda + \lambda_0 -\gamma_0 s }{2}$, so $\tilde \lambda - \lambda = \frac{ \lambda_0 -\gamma_0 s - \lambda}{2} = \lambda_0 - \gamma_0s - \tilde \lambda$. 
%Thus \eqref{px3M2} gives
%\Be
%\|  (u \cdot \nabla \p_{x_3} \theta)(s) \|_{\infty, \lambda, 0 } \lesssim C ( \lambda_0 - \lambda - \gamma_0 s)^{-(\alpha +1) } \sup_{0 \le s \le t} \sup_{ \tilde \lambda < \lambda_0 - \gamma_0 s } ( \lambda_0 - \gamma_0s - \tilde \lambda)^\alpha \|  \p_{x_3} \theta (s) \|_{\infty, \tilde \lambda, 0 }.
%\Ee
%Now let $ \vertiii{ \p_{x_3} \theta (t) }_{r} : = \sup_{  \lambda < \lambda_0 - \gamma_0 t } ( \lambda_0 - \gamma_0t -  \lambda)^\alpha \|  \p_{x_3} \theta (t) \|_{\infty,  \lambda, 0 }$, 
Thus combining \eqref{px3thetaest}, \eqref{px3M1}, and \eqref{px3M2} we get
\Be \label{thetap3bd}
\begin{split}
&\sum_{0 \le | \beta_h | \le 1 } \| D_h^{\beta_h} \p_{x_3} \theta(t,x_3) \|_{\infty, \lambda ,0 }
\\  \lesssim &\sum_{0 \le | \beta_h | \le 1 } \| D_h^{\beta_h} \p_{x_3} \theta_{0 } \|_{\infty, \lambda, 0 }  + \int_0^t  \frac{1}{\sqrt \kappa} ds + \left( \int_0^t  ( \lambda_0 - \lambda - \gamma_0 s)^{-\alpha }  ds  \right)      \sup_{0 \le s \le t } \vertiii{ \theta (s) }_z
\\ &
 + \left( \int_0^t  ( \lambda_0 - \lambda - \gamma_0 s)^{-(\alpha +1) }  ds \right)    \sup_{0 \le s \le t } \vertiii{ \theta (s) }_z
\\ \lesssim &\sum_{0 \le | \beta_h | \le 1 } \| D_h^{\beta_h} \p_{x_3} \theta_{0 } \|_{\infty, \lambda, 0 } +  \frac{t}{\sqrt \kappa}   + \left( \frac{1}{\gamma_0} + ( \lambda_0 - \lambda - \gamma_0 t )^{-\alpha} \frac{1}{\gamma_0 } \right)   \sup_{0 \le s \le t } \vertiii{ \theta (s) }_z.
\end{split}
\Ee
Therefore we conclude that from \eqref{px3thetanorm}, \eqref{thetap3bd}, for sufficiently large $\gamma_0$,
%for sufficiently large $\gamma_0$,
%\Be
% \sup_{0 \le t \le \frac{\lambda_0}{2 \gamma_0}}   \vertiii{ \theta (t) }_{t} \lesssim  \| \p_{x_3} \theta_{0 } \|_{\infty, \lambda_0, 0 } < \infty.
%\Ee
%
%Therefore, we conclude that, from (\ref{theta1_12}) with (\ref{theta1_1M2}), and (\ref{theta1_22}), for 
\Be\label{thetaest:2}
\sup_{ 0\le t \le \frac{\lambda_0}{2 \gamma_0} } \vertiii{\theta (t)}_{z} \lesssim \sum_{0 \le | \beta_h | \le 1 } \| D_h^{\beta_h} \p_{x_3} \theta_{0 } \|_{\infty, \lambda_0, 0 }  +    \frac{1}{\sqrt \kappa}.
\Ee
%From \eqref{thetaest:1} and \eqref{thetaest:2} and a standard continuity argument we conclude for sufficiently large $\gamma_0$,
%\Be\label{thetaest:|||_t}
%\sup_{0 \leq t <\frac{\lambda_0}{2\gamma_0}}\vertiii{ \theta (t)}_{t}\lesssim   \sum_{0 \leq |\beta_h|\leq 2} \| D_h^{\beta_h}  \theta_{0} \|_{\infty,\lambda_0,0 } + \sum_{0 \le | \beta_h | \le 1 } \| D_h^{\beta_h} \p_{x_3} \theta_{0 } \|_{\infty, \lambda, 0 }+  \sum_{0 \leq |\beta|\leq 2} \| D^\beta (1+ |\nabla_h|) 
%  \theta_{0} \|_{1, \lambda_0 }.
%\Ee
This proves the bound for $\p_{x_3} \theta$ in \eqref{thetapx3_bound}.

Next, we look at $\p_t \theta$. For the propagation of analytic norms $\vertiii{\p_t \theta}_t$, we follow the argument in Step $2$. Taking $\p_t$ of \eqref{thetaphi} and using integration by parts, from \eqref{eqtn:G_xi2}, \eqref{bdry:G_xi3}, we achieve
\Be
\begin{split}
\p_t \theta_\xi(t,x_3) = & \kappa \eta_c \p_y G_{\xi 3 } (t,x_3, 0 ) \theta_{0 \xi } (0) +  \int_0^\infty G_{\xi 3 }  (t,x_3,y) \left(  \kappa \eta_c ( - |\xi|^2 + \p_y^2) \theta_{0\xi} + M_\xi (0,y) \right) dy
\\ & + \int_0^t \int_0^\infty G_{\xi 3 } (t-s, x_3 , y) \p_s M_\xi(s,y) dy ds.
\end{split}
\Ee
Recalling $\p_t \theta_{0}$ in \eqref{idata}, and the compatibility condition \eqref{CC}. we get
\Be \label{pttheta}
\p_t \theta_\xi(t,x_3) =   \int_0^\infty G_{\xi 3 }  (t,x_3,y)  \p_t \theta_{0 \xi} dy  + \int_0^t \int_0^\infty G_{\xi 3 } (t-s, x_3 , y) \p_s M_\xi(s,y) dy ds.
\Ee
With this representation formula we can repeat the estimate earlier in this step in the same fashion. For the nonlinear terms, since $\p_t M = - u \cdot \nabla \p_t \theta - \p_t u  \cdot \nabla \theta$, we can use the bilinear estimates \eqref{nonlinear1}, \eqref{px3M1}, and \eqref{px3M2} in the same way to derive that for $t < \frac{\lambda_0}{2 \gamma_0}$
\begin{align}
 \notag \vertiii{ \p_t\theta(t)}_{1} \lesssim & \sum_{0\leq |\beta|\leq 2} \| D^\beta (1+|\nabla_h|) \p_t \theta_0  \|_{1, \lambda_0}  
\\ &  +( t+  \frac{1}{\gamma_0}) \left(  \sup_{0 \leq s \leq t}\vertiii{ \o(s)}_{s} \sup_{0 \leq s \leq t} \vertiii{\p_t \theta(s)}_s  +  \sup_{0 \leq s \leq t}\vertiii{ \p_t \o(s)}_{s} \sup_{0 \leq s \leq t} \vertiii{ \theta(s)}_s  \right) ,
 \label{thetaest:t_1}
 \\
\notag \vertiii{ \p_t\theta(t)}_{\infty,0} \lesssim 
%&  \sum_{0 \leq |\beta_h|\leq 2} \| D_h^{\beta_h}  \theta_{0} \|_{\infty,\lambda_0,0 } + \sum_{0 \le | \beta_h | \le 1 } \| D_h^{\beta_h} \p_{x_3} \theta_{0 } \|_{\infty, \lambda, 0 }
&  \sum_{0\leq|\beta|\leq 2} \| D^\beta \p_t \theta_{0} \|_{\infty, \lambda_0, 0} 
 \\ & +( t+  \frac{1}{\gamma_0}) \left(\sum_{\ell = 0,1}  \sup_{0 \leq s \leq t} \vertiii{ \p_t^\ell \o(s)}_{s} \right) \left(\sum_{\ell = 0,1}  \sup_{0 \leq s \leq t} \vertiii{\p_t^\ell \theta(s)}_s  \right),
 \label{thetaest:t_infty}
\end{align}
which lead to the desired bounds for $\p_t\theta(t)$ in \eqref{thetanorm_bound} by choosing sufficiently large $\gamma_0$.

For the propagation of the norm $\vertiii{\p_t \theta }_z$, we follow the argument in the estimate of $\vertiii{\theta }_z$. Taking $\p_{x_3}$ to \eqref{pttheta} and from \eqref{CC}, \eqref{ptellMequal0}, we have $\p_t \theta_{0\xi} (0) = 0$, $\p_t M_{\xi}(0) = 0$. Thus from integration by parts we can follow the argument from \eqref{px3thetaphi}-\eqref{px3thetaest} to get
\[
\sum_{0 \le | \beta_h | \le 1 } \| D_h^{\beta_h} \p_{x_3} \p_t \theta(t,x_3) \|_{\infty, \lambda ,0 } \lesssim   \sum_{0 \le | \beta_h | \le 1 } \| D_h^{\beta_h} \p_{x_3} \p_t \theta_{0 } \|_{\infty, \lambda, 0 } + \int_0^t   \sum_{0 \le | \beta_h | \le 1 } \| D_h^{\beta_h} \p_{x_3} \p_t M_{}(s) \|_{\infty,\lambda,0} ds.
\]
Then following the argument from \eqref{px3M1}-\eqref{thetap3bd} we get
\[
\begin{split}
&\sum_{0 \le | \beta_h | \le 1 } \| D_h^{\beta_h} \p_{x_3} \p_t \theta(t,x_3) \|_{\infty, \lambda ,0 }
\\  \lesssim &\sum_{0 \le | \beta_h | \le 1 } \| D_h^{\beta_h} \p_{x_3} \p_t \theta_{0 } \|_{\infty, \lambda, 0 }  + \int_0^t  \frac{1}{\sqrt \kappa} ds \left( \sup_{0 \le s\le t }  \sum_{\ell = 0,1} \vertiii{ \p_t^\ell \o(s) }_s \right) \left( \sup_{0 \le s\le t }  \sum_{\ell = 0,1} \vertiii{ \p_t^\ell \theta(s) }_s \right)  
\\ & + \left( \int_0^t  ( \lambda_0 - \lambda - \gamma_0 s)^{-\alpha }  ds  \right)  \left( \sup_{0 \le s\le t }  \sum_{\ell = 0,1} \vertiii{ \p_t^\ell \o(s) }_s \right)  \left(     \sup_{0 \le s \le t } \sum_{\ell = 0,1} \vertiii{\p_t^\ell \theta (s) }_z \right)
\\ &
 + \left( \int_0^t  ( \lambda_0 - \lambda - \gamma_0 s)^{-(\alpha +1) }  ds \right)\left( \sup_{0 \le s\le t }  \sum_{\ell = 0,1} \vertiii{ \p_t^\ell \o(s) }_s \right)  \left(  \sum_{\ell=0,1}   \sup_{0 \le s \le t } \vertiii{ \p_t^\ell \theta (s) }_z \right)
\\ \lesssim &\sum_{0 \le | \beta_h | \le 1 } \| D_h^{\beta_h} \p_{x_3} \p_t \theta_{0 } \|_{\infty, \lambda, 0 } +  \frac{t}{\sqrt \kappa}   + \left( \frac{1}{\gamma_0} + ( \lambda_0 - \lambda - \gamma_0 t )^{-\alpha} \frac{1}{\gamma_0 } \right)   \sup_{0 \le s \le t } \vertiii{  \p_t \theta (s) }_z.
\end{split}
\]
which lead to the desired bounds for $\p_{x_3} \p_t\theta(t)$ in \eqref{thetapx3_bound} by choosing sufficiently large $\gamma_0$.

The propagation of analytic norms for $\p_t^2 \theta$ follows in the same way. From the compatibility condition \eqref{CC} we get
\Be \label{pt2theta}
\p_t^2 \theta(t,x_3) =   \int_0^\infty G_{\xi 3 }  (t,x_3,y)  \p_t^2 \theta_{0 \xi} dy  + \int_0^t \int_0^\infty G_{\xi 3 } (t-s, x_3 , y) \p_s^2 M_\xi(s,y) dy ds.
\Ee
Using the bilinear estimate on $\p_t^2 M$ we derive that for $t < \frac{\lambda_0}{2 \gamma_0}$
\begin{align}
&\vertiii{ \p_t^2 \theta(t)}_{1} \lesssim  \sum_{0\leq |\beta|\leq 2} \| D^\beta (1+|\nabla_h|) \p_t^2 \theta_0  \|_{1, \lambda_0} +( t+  \frac{1}{\gamma_0}) \sup_{0 \le s \le t } \vertiii{ \p_t \o(s) }_s   \sup_{0 \le s \le t } \vertiii{ \p_t \theta (s) }_s \notag
\\ 
&  +( t+  \frac{1}{\gamma_0}) \left(  \sup_{0 \leq s \leq t}\vertiii{ \o(s)}_{s} \sup_{0 \leq s \leq t} \vertiii{\p_t^2 \theta(s)}_s  +  \sup_{0 \leq s \leq t}\vertiii{ \p_t \o(s)}_{s} \sup_{0 \leq s \leq t} \vertiii{ \p_t \theta(s)}_s  \right) ,
 \label{thetaest:tt_1}
 \\
&\vertiii{ \p_t^2 \theta(t)}_{\infty,0} \lesssim  \sum_{0\leq|\beta|\leq 2} \| D^\beta \p_t^2 \theta_{0} \|_{\infty, \lambda_0, 0} +( t+  \frac{1}{\gamma_0}) \sup_{0 \le s \le t } \vertiii{ \p_t \o(s) }_s   \sup_{0 \le s \le t } \vertiii{ \p_t \theta (s) }_s \notag
 \\ & +( t+  \frac{1}{\gamma_0}) \left( \sup_{0 \leq s \leq t}\vertiii{ \o(s)}_{s} \sup_{0 \leq s \leq t} \vertiii{\p_t^2 \theta(s)}_s  +  \sup_{0 \leq s \leq t}\vertiii{ \p_t \o(s)}_{s} \sup_{0 \leq s \leq t} \vertiii{\p_t \theta(s)}_s \right), \label{thetaest:tt_infty}
\end{align}
which lead to the desired bounds for $\p_t^2\theta(t)$ in \eqref{thetanorm_bound} by choosing sufficiently large $\gamma_0$.

Taking $\p_{x_3}$ to \eqref{pt2theta} and using \eqref{CC}, \eqref{ptellMequal0}, we have $\p_t^2 \theta_{0\xi} (0) = 0$, $\p_t^2 M_{\xi}(0) = 0$. Thus from integration by parts we again follow the argument from \eqref{px3thetaphi}-\eqref{px3thetaest} to get
\[
\sum_{0 \le | \beta_h | \le 1 } \| D_h^{\beta_h} \p_{x_3} \p_t^2 \theta(t,x_3) \|_{\infty, \lambda ,0 } \lesssim   \sum_{0 \le | \beta_h | \le 1 } \| D_h^{\beta_h} \p_{x_3} \p_t^2 \theta_{0 } \|_{\infty, \lambda, 0 } + \int_0^t   \sum_{0 \le | \beta_h | \le 1 } \| D_h^{\beta_h} \p_{x_3} \p_t^2 M_{}(s) \|_{\infty,\lambda,0} ds.
\]
Then following the argument from \eqref{px3M1}-\eqref{thetap3bd} we get
\[
\begin{split}
& \sum_{0 \le | \beta_h | \le 1 }  \| D_h^{\beta_h} \p_{x_3} \p_t^2 \theta(t,x_3) \|_{\infty, \lambda ,0 }
\\  \lesssim &\sum_{0 \le | \beta_h | \le 1 } \| D_h^{\beta_h} \p_{x_3} \p_t^2 \theta_{0 } \|_{\infty, \lambda, 0 }  
\\ &+ \int_0^t  \left( \frac{1}{\sqrt \kappa} + \frac{1}{\sqrt{ \kappa s }} \right) ds \left( \sup_{0 \le s\le t }  \sum_{\ell = 0}^2 \vertiii{ \p_t^\ell \o(s) }_s \right) \left( \sup_{0 \le s\le t }  \sum_{\ell = 0}^2 \vertiii{ \p_t^\ell \theta(s) }_s \right)  
\\ & + \left( \int_0^t  ( \lambda_0 - \lambda - \gamma_0 s)^{-\alpha }  ds  \right)  \left( \sup_{0 \le s\le t }  \sum_{\ell = 0}^2 \vertiii{ \p_t^\ell \o(s) }_s \right)  \left(     \sup_{0 \le s \le t } \sum_{\ell = 0}^2 \vertiii{\p_t^\ell \theta (s) }_z \right)
\\ &
 + \left( \int_0^t  ( \lambda_0 - \lambda - \gamma_0 s)^{-(\alpha +1) }  ds \right)\left( \sup_{0 \le s\le t }  \sum_{\ell = 0}^2 \vertiii{ \p_t^\ell \o(s) }_s \right)  \left(  \sum_{\ell = 0}^2   \sup_{0 \le s \le t } \vertiii{ \p_t^\ell \theta (s) }_z \right)
\\ \lesssim &\sum_{0 \le | \beta_h | \le 1 } \| D_h^{\beta_h} \p_{x_3} \p_t^2 \theta_{0 } \|_{\infty, \lambda, 0 } +  \frac{t + \sqrt t}{\sqrt \kappa}   + \left( \frac{1}{\gamma_0} + ( \lambda_0 - \lambda - \gamma_0 t )^{-\alpha} \frac{1}{\gamma_0 } \right)   \sup_{0 \le s \le t } \vertiii{  \p_t^2 \theta (s) }_z.
\end{split}
\]
which lead to the desired bounds for $\p_{x_3} \p_t^2 \theta(t)$ in \eqref{thetapx3_bound} by choosing sufficiently large $\gamma_0$.

   \

 {\bf Step 6: Estimate (1), vorticity estimates. }Both \eqref{b1} and \eqref{b2} are direct consequences of (\ref{norm_bound}). To show \eqref{b3}, we first note that the boundedness of $\o(t)$ norms implies 
 $|\p_{x_3}\o_\xi(t,x_3)|\lesssim e^{-\bar{\alpha} x_3}e^{-\lambda |\xi|}$ for all $|\xi|$ and $x_3 \geq 1$ (away from the boundary). When $x_3\leq 1$, we draw on  the equation \eqref{NS_f} to rewrite $ \p_{x_3}^2\o_{\xi,h} = \frac{1}{\kappa\eta_0}\{ \p_t \o_{\xi,h}  + \kappa\eta_0 |\xi|^2 \o_{\xi,h} - N_{\xi,h}\}$ and the boundary condition \eqref{NSB_f}: 
 \Be\begin{split}
 \p_{x_3} \o_{\xi,h} (t, x_3) &= \p_{x_3} \o_{\xi,h} (t,0) + \int_0^{x_3}  \p_{x_3}^2\o_{\xi,h}(t,y) \dd y\\
 &=- |\xi| \o_{\xi,h} (t,0) + \frac{1}{\kappa\eta_0} B_\xi (t) +  \int_0^{x_3}  \frac{1}{\kappa\eta_0}[ \p_t \o_{\xi,h}  + \kappa\eta_0 |\xi|^2 \o_{\xi,h} - N_{\xi,h}](t,y)\dd y. %\\
% & \sim \frac{1}{\kappa} + \frac{x_2}{\kappa}. 
\end{split}
\Ee
We now appeal to  $|B_\xi(t)| \leq \|N_\xi(t)\|_{\mathcal{L}^1_\lambda}$ and $\sum_{0\leq \ell\leq 1} ( \vertiii{ \p_t^\ell\o(t)}_{\infty,\kappa}+ \vertiii{ \p_t^\ell\o(t)}_{1}) <\infty$ to obtain that for all $x_3\in \mathbb R_+$
\Be 
| \p_{x_3} \o_{\xi, h} (t, x_3)|\lesssim  \frac{1}{\kappa} e^{-\bar{\alpha} x_3}e^{-\lambda |\xi|} \text{ for } 0<\lambda<\lambda_0,
\Ee
which proves \eqref{b3} for $\o_h$ and $\ell=0$. The remaining case can be estimated similarly. Near $O(1)$ boundary, from \eqref{NS_f1} and \eqref{NSB_f1}, we derive 
 \Be\begin{split}
 &\p_{x_3} \p_t\o_{\xi,h} (t, x_3)\\
  &=- |\xi| \p_t\o_{\xi,h} (t,0) + \frac{1}{\kappa\eta_0} \p_tB_\xi (t) +  \int_0^{x_3}  \frac{1}{\kappa\eta_0}[\p_t^2 \o_{\xi,h}  + \kappa\eta_0 |\xi|^2 \p_t\o_{\xi,h} - \p_tN_{\xi,h}](t,y)\dd y. %\\
% & \sim \frac{1}{\kappa} + \frac{x_2}{\kappa}. 
\end{split}
\Ee
Together with  $\sum_{0\leq \ell\leq 1} \vertiii{ \p_t^\ell\o(t)}_{\infty,\kappa}+ \sum_{0\leq \ell\leq 2}  \vertiii{ \p_t^\ell\o(t)}_{1}<\infty$ we deduce \eqref{b3} for $\o_h$ and $\ell=1$.
For $\o_3$ we use $\nabla\cdot \o=0$ to write $\p_3 \o_3 = - \p_1\o_1- \p_2\o_2$. Now \eqref{b3} for $\o_3$ follows from \eqref{b1}.

% {\color{red}\texttt{...Can we use $\p_{x_3}\o_{x_3} = - \p_{x_1}\o_{x_1}- \p_{x_2}\o_{x_2}$ and the estimates of r.h.s to derive the estimate of $\p_{x_3}\o_{x_3}$ ?...}}
%\begin{align*}
%\p_{x_3} \p_t \o_{\xi, 3}(t, x_3) &= \p_{x_3}\p_t \o_{\xi, 3}(t,1) + \int^{x_3}_1 \p^2_{x_3} \p_t\o_{\xi, 3} (t,y) \dd y\\&=\p_{x_3} \p_t\o_{\xi, 3}(t,1) + \int^{x_3} _1 \frac{1}{\kappa \eta_0} 
%[
%\p_t^2\o_{\xi,3} + \kappa |\xi|^2 \eta_0 \p_t\o_{\xi, 3} - \p_tN_{\xi, 3}
%](t,y) \dd y.
%\end{align*}
%From $\sum_{0\leq \ell\leq 1} ( \vertiii{ \p_t^\ell\o(t)}_{\infty,\kappa}+ \vertiii{ \p_t^\ell\o(t)}_{1}) <\infty$ and (\ref{est:N_infty}) we deduce that 
%\Be
%|\p_{x_3}\p_t \o_{\xi, 3}(t, x_3)  |\lesssim   \frac{1}{\kappa}  e^{-\lambda |\xi|} \ \ \text{for} \ x_3 \leq 1.
%\Ee

  \

 {\bf Step 7: Estimate (2), velocity estimates, except (\ref{ut}). } From Lemma \ref{lem_elliptic}, the estimation of the velocity follows from $\phi$ estimate.  From $(|\xi|^2   - \p_z^2 )\phi_\xi = \o_\xi$ and $\phi_\xi (0)=0$ 
 \Be\label{phi_xi}
\begin{split}
\phi_\xi(z)=& \int^z_0 G_- (y,z) \o_\xi (y) \dd y + \int^\infty_z G_+ (y,z) \o_\xi (y) \dd y, \\
&\text{with} \  \ G_{\pm} (y,z) := \frac{-1}{2 |\xi|} \Big( e^{\pm |\xi| (z-y)} -e^{- |\xi | (y+z)}\Big),
\end{split}\Ee 
and we have
\Be\label{est:phi}
|\xi|^{\beta_h} | \p_{z}^{\beta_3}\p_t^\ell\phi_\xi (t,z)|
\lesssim \int_{\p \mathcal{H}_\lambda} |\xi|^{|\beta|-1} e^{-|\xi| |y-z|} | \p_t^\ell \o_\xi (t,y)| |\dd y|
 \ \ \text{for} \ \beta_3\leq 1. 
\Ee
For $|\beta|= |\beta_h| + \beta_3=1$ we bound (\ref{est:phi}) by $e^{-\lambda |\xi|} \| \p_t^\ell \o(t) \|_{1,\lambda}$. Then from (\ref{norm_bound}) we conclude (\ref{est:u_t}).

For $|\beta|\geq 2$ and $\beta_3 \leq 1$, we bound (\ref{est:phi}) by 
\Be\label{est:phi1}
\begin{split}
(\ref{est:phi}) &\lesssim%\| \p_t^\ell \o_\xi (t)\|_{\infty, \lambda, \kappa } 
\int_{\p \mathcal{H}_\lambda} |\xi|^{|\beta|-2} |\xi|e^{-|\xi| |y-z|} 
 e^{- {\bar{\alpha}}  \text{Re}\,y}e^{-\lambda |\xi|} \big(1+ \phi_\kappa (y)+ \phi_{\kappa t} (y)\big) 
|\dd y|  \\
&\lesssim |\xi|^{|\beta|-2} e^{-\lambda |\xi|} e^{- \min (1, \frac{\bar{\alpha}}{2})x_3}
\int_{\p \mathcal{H}_\lambda} 
e^{- \frac{\bar{\alpha}}{2} \text{Re} \, y}\big ( 1+ \phi_\kappa (y)+ \phi_{\kappa t} (y)\big) 
|\dd y|\\
&\lesssim |\xi|^{|\beta|-2} e^{-\lambda |\xi|} e^{- \min (1, \frac{\bar{\alpha}}{2})x_3} \ \ \text{for} \ |\beta| \geq 2, \ \text{and} \ \beta_3 \leq 1, \ \text{and} \  \ell=0,1,2, \ \text{and} \  t \in [0,T],
\end{split}
\Ee
where we have used $|\xi||y-z| + \frac{\bar{\alpha}}{2} \text{Re}\, y\geq \min (1, \frac{\bar{\alpha}}{2} ) x_3$ for $|\xi|\geq 1$ and (\ref{norm_bound}).

For $\beta_3=2,3$ we use $\p_z^2 \p_t^\ell\phi_\xi = |\xi|^2  \p_t^\ell\phi_\xi +   \p_t^\ell\o_\xi$. Then following the same argument of (\ref{est:phi1}), we derive 
\Be\label{est:phi2}
\begin{split}
&|\xi|^{\beta_h} | \p_{z}^{\beta_3}\p_t^\ell\phi_\xi (t,z)|\\
&\lesssim |\xi|^{|\beta_h|+2} | \p_{z}^{\beta_3-2}\p_t^\ell\phi_\xi (t,z)|
+|\xi|^{\beta_h} | \p_{z}^{\beta_3-2}\p_t^\ell\o_\xi (t,z)|\\
&\lesssim 
\begin{cases}
(|\xi|^{|\beta|-2}+|\xi|^{\beta_h}) e^{- \lambda |\xi|}  e^{- \min (1, \frac{\bar{\alpha}}{2}) \text{Re}\, z} (1+ \phi_\kappa (z)) \ \ \text{for} \ \ell=0,1, \ \text{and} \ \beta_3=2,\\
 (|\xi|^{|\beta|-2}+|\xi|^{\beta_h}) e^{- \lambda |\xi|}  e^{- \min (1, \frac{\bar{\alpha}}{2})\text{Re}\, z}\kappa^{-1}  \ \ \text{for} \ \ell=0,1, \ \text{and} \ \beta_3=3, \\
(|\xi|^{|\beta|-2}+|\xi|^{\beta_h}) e^{- \lambda |\xi|}  e^{- \min (1, \frac{\bar{\alpha}}{2})\text{Re}\, z} (1+ \phi_\kappa (z)+ \phi_{\kappa t} (z)) \ \ \text{for} \ \ell=2, \ \text{and} \ \beta_3=2. 
\end{cases}
\end{split}\Ee
Finally from (\ref{est:phi1}) and (\ref{est:phi2}) we conclude (\ref{est:u1}) and (\ref{est:u2}).

%%%%%%%%%%%%%%
%%%%%%%%%%%%%%
%%%%%%%%%%%%%%
\hide
 From $e^{-|\xi| |y-z|} \leq 1$ or $|\xi|  e^{-|\xi| |\cdot |} \in L^1(\R)$ 
\Be\begin{split}\label{est:phi_t}
|\xi|^k|\p_t^\ell\phi_\xi (t,x_3)|%\lesssim|\xi|^{(k-1)}e^{- \bar{\alpha}    \text{Re}\,y } |\p_t^\ell\o_\xi (t,y) |
&\lesssim |\xi|^{ k-1 } e^{-\lambda |\xi|}
 \| \p_t^\ell\o  (t )\|_{1,\lambda},\\
(1+ \phi_\kappa (x_3)  )^{-1} |\xi|^k |\p_t^\ell \phi_\xi(t,x_3)| &\lesssim |\xi|^{k-2}e^{-\lambda |\xi|}     \|\p_t^\ell \o  (t )\|_{\infty, \lambda,\kappa } 
  \ \ \text{for} \ \ell=0,1,
 \\
(1+ \phi_\kappa (x_3) + \phi_{\kappa t} (x_3))^{-1}  |\xi|^k |\p_t^2 \phi_\xi(t,x_3)| &\lesssim |\xi|^{k-2}   e^{-\lambda |\xi|}
\|\p_t^2\o (t   )\|_{ \infty,\lambda, \kappa t }.
\end{split}\Ee
Once again from (\ref{phi_xi}) we have $\p_z \phi_\xi (z) = \int^z_0 \p_z G_- (y,z) \o_\xi (y) \dd y + \int^\infty_z \p_z G_+ (y,z) \o_\xi (y) \dd y$. Since $|\p_z G_{\pm}(y,z)|  \lesssim e^{-|\xi||y-z|}$
\Be\begin{split}\label{est:phi_tz}
|\xi|^k| \p_{x_3}\p_t^\ell\phi_\xi (t,x_3)| &\lesssim |\xi|^{k} e^{-\lambda |\xi|}
 \| \p_t^\ell\o  (t )\|_{1,\lambda},\\
 (1+ \phi_\kappa (x_3)  )^{-1} |\xi|^k | \p_{x_3} \p_t^\ell \phi_\xi(t,x_3)| &\lesssim |\xi|^{k-1}  e^{-\lambda |\xi|} \| \p_t^\ell \o (t )\|_{ \infty,\lambda,\kappa } 
  \ \ \text{for} \ \ell=0,1,
 \\
(1+ \phi_\kappa (x_3) + \phi_{\kappa t} (x_3))^{-1}  |\xi|^k |\p_{x_3} \p_t^2 \phi_\xi(t,x_3)| &\lesssim |\xi|^{k-1}   e^{-\lambda |\xi|}
\|\p_t^2\o  (t )\|_{ \infty,\lambda, \kappa t }.
\end{split}\Ee
Therefore we derive that, from (\ref{est:phi_t}), (\ref{est:phi_tz}), and $|\xi|^{|\beta_h|}| 
\p_t^\ell u_{\xi} (t,x_3)| \lesssim 
|\xi|^{|\beta_h|+1}| %\p_z^{\beta_3}
\p_t^\ell \phi_{\xi} (t,x_3)| 
+|\xi|^{|\beta_h| }| \p_z%^{\beta_3+1}
\p_t^\ell \phi_{\xi} (t,x_3)|$, 
\begin{align}
%|\xi|^{|\beta_h|}
| 
\p_t^\ell u_{\xi} (t,z)|&\lesssim 
 e^{-\lambda |\xi|} \| \p_t^\ell \o_\xi (t ) \|_{1,\lambda}   ,
 \notag
 \\
|\xi|^{|\beta_h|}| 
\p_t^\ell u_{\xi} (t,z)|&\lesssim |\xi|^{|\beta_h| -1}
e^{-\lambda |\xi|}  \| \p_t^\ell \o (t ) \|_{1,\lambda} 
+
(1+ \phi_\kappa(x_3))  
 |\xi|^{|\beta_h|-1} e^{-\lambda |\xi|}  \| \p_t^\ell \o  (t  ) \|_{ \infty,\lambda,\kappa}
  \ \ \text{for} \ |\beta_h|\geq 1 \ \text{and} \ \ell=0,1,
  \label{est:u_t}
  \\
  |\xi|^{|\beta_h|}| 
\p_t^2 u_{\xi} (t,z)|&\lesssim |\xi|^{|\beta_h| -1}
e^{-\lambda |\xi|}  \| \p_t^2\o (t ) \|_{1,\lambda} 
+
(1+ \phi_\kappa(x_3) +\phi_{\kappa t}(x_3) )  
 |\xi|^{|\beta_h|-1} e^{-\lambda |\xi|}  \| \p_t^2\o  (t ) \|_{ \infty,\lambda,\kappa t}
  \ \ \text{for} \ |\beta_h|\geq 1.\notag
\end{align}

From $\p_z^2 \p_t^\ell\phi_\xi = |\xi|^2  \p_t^\ell\phi_\xi +   \p_t^\ell\o_\xi$, (\ref{est:phi_t}) and (\ref{est:phi_tz}) we further derive that 
\Be\begin{split}\label{est:phi_tzz}
|\xi|^k| \p_{x_3}^2\p_t^\ell\phi_\xi (t,x_3)| &\lesssim |\xi|^{k+1} e^{-\lambda |\xi|}
 \| \p_t^\ell\o  (t )\|_{1,\lambda}
 + e^{-\lambda|\xi|} (1+ \phi_\kappa (x_3)) \sum_{m=0,1} \vertiii{ \p_t^m\o(t)}_{\infty, \kappa}  \ \ \text{for} \ \ell=0,1
 ,\\
 |\xi|^k| \p_{x_3}^2\p_t^2\phi_\xi (t,x_3)| &\lesssim |\xi|^{k+1} e^{-\lambda |\xi|}
 \| \p_t^2\o  (t )\|_{1,\lambda}
 + e^{-\lambda|\xi|}
  (1+ \phi_\kappa (x_3)+ \phi_{\kappa t} (x_3))
  \vertiii{ \p_t^2\o(t)}_{\infty, \kappa} .
  \end{split}\Ee
 Therefore we derive that, from (\ref{est:phi_tzz}) and (\ref{est:phi_tz})
\Be\begin{split} \label{est:u_tz}
%|\xi|^{|\beta_h|}
 |\xi|^{|\beta_h|} | 
 \p_{x_3} \p_t^\ell u_{\xi} (t,z)|&\lesssim 
|\xi| ^{|\beta_h|+1}e^{-\lambda |\xi| }  \| \p_t^\ell \o(t) \|_{1, \lambda }
 + |\xi| ^{|\beta_h| } e^{-\lambda|\xi|} (1+ \phi_\kappa (x_3)) \sum_{m=0,1} \vertiii{ \p_t^m\o(t)}_{\infty, \kappa}  \ \ \text{for} \ \ell=0,1,
 \\
  |\xi|^{|\beta_h|} | 
 \p_{x_3} \p_t^2 u_{\xi} (t,z)|&\lesssim 
|\xi| ^{|\beta_h|+1}e^{-\lambda |\xi| }  \| \p_t^2 \o(t) \|_{1, \lambda }
 + |\xi| ^{|\beta_h| } e^{-\lambda|\xi|} (1+ \phi_\kappa (x_3)+ \phi_{\kappa t} (x_3))  \vertiii{ \p_t^2\o(t)}_{\infty, \kappa t}   .
\end{split}\Ee
From (\ref{est:u_t}), (\ref{est:u_tz}), $ 
\nabla_{x_h}^{\beta_h} \p_{z}^{\beta_3} \p_t ^\ell u(t,x_h,z)
= \sum_{\xi  \in \mathbb Z^2}
(i \xi)^{\beta_h } \p_t^\ell \p_{z}^{\beta_3} u_\xi (t,z) e^{i x_h \cdot \xi }$ and (\ref{norm_bound}) we derive that 
\Be
\begin{split}\label{est:u1}
\sum_{0 \leq |\beta | \leq 1} \sum_{\ell=0,1}  |\nabla ^{\beta } \p_t^\ell u(t,x )|%\lesssim  \sup_{t \in [0,T]}\vertiii{\p_t^\ell \o (t) } 
 \lesssim  1+ \phi_\kappa (x_3)  ,\ \
\sum_{0 \leq |\beta | \leq 1}  |\nabla ^{\beta } \p_t^2 u(t,x )|%\lesssim  \sup_{t \in [0,T]}\vertiii{\p_t^\ell \o (t) } 
 \lesssim  1+ \phi_\kappa (x_3) + \phi_{\kappa t} (x_3) .
\end{split}
\Ee

Now we estimate the second order derivatives of the velocity field. Note that $\p_{x_3}^3 \p_t^\ell\phi_\xi = |\xi|^2 \p_{x_3} \p_t^\ell\phi_\xi +  \p_{x_3} \p_t^\ell\o_\xi$. From (\ref{b3}) and (\ref{est:phi_tz}), 
\Be
|\p_{x_3}^3 \p_t^\ell\phi_\xi| \lesssim
|\xi|^{2} e^{-\lambda |\xi| } \| \p_t^\ell \o(t) \|_{1,\lambda }
 +\kappa^{-1} e^{-\bar \alpha x_3} e^{-\lambda |\xi|}
 \ \ \text{for} \ \ell=0,1.\notag
\Ee
Together this bound with (\ref{est:phi_tzz}) and $
\p_{x_3}^2\p_t^\ell u_{\xi} (t,x_3)| \lesssim 
|\xi| | %\p_z^{\beta_3}
\p_{x_3}^2\p_t^\ell \phi_{\xi} (t,x_3)| 
+  |  \p_{x_3}^3%^{\beta_3+1}
\p_t^\ell \phi_{\xi} (t,x_3)|$, we derive that 
\[
\sum_{|\beta|=2} \sum_{\ell=0,1} |\nabla^\beta \p_t^\ell u(t,x)|\lesssim 1+ \phi_\kappa (x_3). 
\]
\unhide
%%%%%%%%%%%%%%
%%%%%%%%%%%%%%
%%%%%%%%%%%%%%

 \
 
  {\bf Step 8: Estimate (3), temperature estimates.} Both \eqref{est:theta_t} and \eqref{est:theta1} are direct consequences of \eqref{thetanorm_bound}. To show \eqref{est:theta2}, we use the equation \eqref{heat_NS} to get
\Be
\begin{split}
\p_{x_3}^2 \theta = & \frac{1}{\kappa \eta_c} \left( \p_t \theta + u \cdot \nabla_x \theta \right) - \p_{x_1}^2 \theta - \p_{x_2}^2 \theta.
\\ = &  \frac{1}{\kappa \eta_c} \left( \p_t \theta + u_1\p_{x_1} \theta + u_2 \p_{x_2} \theta + ( \frac{u_3}{\zeta } ) ( \zeta \p_{x_3} \theta ) \right) - \p_{x_1}^2 \theta - \p_{x_2}^2 \theta
\end{split}
\Ee
Then from \eqref{est:u_t}, \eqref{est:theta_t}, and \eqref{est:theta1} we get
\[
\begin{split}
\sum_{\ell = 0}^1 | \p_t^\ell \p_{x_3}^2 \theta| \lesssim &  \frac{1}{\kappa \eta_c}  \left[ \sum_{\ell = 1}^2 |\p_t^\ell \theta| + \left(  \sum_{\ell = 0}^1 ( |\p_t^\ell u_1 | +  |\p_t^\ell u_2 | + | \p_t^\ell ( \frac{u_3}{\zeta} ) | )   \right) \left(  \sum_{\ell = 0}^1 ( |\p_t^\ell \p_{x_1} \theta | +  | \p_t^\ell \p_{x_2} \theta | +  | \p_t^\ell (\zeta \p_{x_3} \theta ) |   )\right)  \right] 
\\ & + \sum_{\ell =0}^1 \left( |\p_t^\ell \p_{x_1}^2 \theta |  + |\p_t^\ell \p_{x_2}^2 \theta |  \right)
\\ \lesssim & \frac{1}{\kappa}\left( e^{-\bar \alpha x_3} \right) + e^{-\bar \alpha x_3 } \lesssim \kappa^{-\frac{3}{2}} e^{-\bar \alpha x_3}.
\end{split}
\]
This proves \eqref{est:theta2}.

The pressure estimate \eqref{est:p}-\eqref{est:pt2} can be found in Theorem 3 of \cite{JK}. The last estimate for $\p_t^\ell u$ for $\ell=1,2$ follows from the equation: $\p_t  u = \kappa\eta_0 \Delta u - u \cdot\nabla u - \nabla  p $ and $\p_t^2 u = \kappa\eta_0 \Delta \p_t u - u \cdot\nabla  \p_t u - \p_t u \cdot\nabla u - \nabla \p_t  p $. This finishes the proof.

\section{Main theorem}
In the last section we state and prove the precise statement of the main theorem which was informally stated as Theorem \ref{Informal statement}. Before that we first show the convergence of heat flow $\theta$ to $\theta_E$ in $L^\infty([0,T];L^2(\O) )$ in the following lemma.

\begin{lemma}[Convergence of $\theta$ to $\theta_E$] \label{thetaEtheta}
Let $u_E$ be the solution of the Euler equation \eqref{Euler}, \eqref{incomp_E}, \eqref{no-pen}, and $u$ be the solution of the Navier-Stokes equation \eqref{NS}, \eqref{incomp_NS}, \eqref{NSB}. For the following equations
\Be \label{equuE1} \begin{split}
\p_t \theta_E + u_E \cdot \nabla_x \theta_E   & =0 \ \ \text{in} \ \O,  % \ \ u_E \cdot n =0  \ \ \text{on} \ \p\O, 
 \\ \p_t \theta + u \cdot \nabla_x \theta  - \kappa \eta_c \Delta \theta & =0 \ \ \text{in} \ \O,
% \ \ u =0  \ \ \text{on} \ \p\O, 
 \ \ \theta = 0 \ \ \text{on} \ \p\O.
\end{split} \Ee
Assume $ \theta_E(0,x) =  \theta (0,x) = \theta_0(x)$, and
\Be \label{theta0bdry}
\theta_0(x) = 0 \text{ on } \p \O.
\Ee
Then
\Be
\sup_{0 \le t \le T}  \| \theta - \theta_E \|_{L^2(\O ) } \to 0, \text{ as } \kappa \to 0.
\Ee

\end{lemma}

\begin{proof}
From \eqref{equuE1} we have 
\Be \label{thetaEsol}
\theta_E(t,x) = \theta_E (0, X(0;t,x)),
\Ee
where $X(s;t,x)$ satisfies
\[
\frac{d}{ds} X(s;t,x) = u_E(s, X(s;t,x)), \ \ X(t;t,x)  = x.
\]
From assumption \eqref{theta0bdry} and $u_E \cdot n = 0$ on $\p \O$, we have $\{ X(0;t,x) : x  \in \p \O \} \subset \p \O$. Therefore 
\Be \label{thetaE0bdry}
\theta_E(t,x) = 0 \text{ on } \p \O.
\Ee
For the solution $u_E(t,x)$ to the Euler equation \eqref{Euler}, \eqref{incomp_E}, \eqref{no-pen}, it is shown in \cite{Euler} that for smooth initial data $u_0$ satisfying $\nabla \cdot u_0 = 0 $ and $u_0 \cdot n = 0$ on $\p \O$, there exists a time $T>0$ such that $u_E$ is smooth on $(0, T)$. In particular,
\Be \label{uEC2}
\sup_{0 < s < T} \| u_E(s ) \|_{C^2(\O)} < \infty.
\Ee
Thus taking $\nabla_x$ derivative to 
\Be \label{Xstx}
X(s;t,x) = x - \int_s^t u_E(\tau, X(\tau;t,x) ) d\tau,
\Ee
we have
\Be \notag
| \nabla_x X(s;t,x) | \lesssim 1 + \int_s^t  \| \nabla_x u_E(\tau ) \|_{L^\infty} | \nabla_x X(\tau; t,x ) | d\tau.
\Ee
From Gronwall's inequality
\Be \label{pxX}
\sup_{0 \le s \le  t \le T^* } \sup_{x \in \O } | \nabla_x X(s;t,x) | \lesssim e^{T  \sup_{0 < s < T^*} \| \nabla_x u_E(s ) \|_{L^\infty} } < C.
\Ee
Similarly, taking $\nabla_x^2$ to \eqref{Xstx} and using \eqref{uEC2}, \eqref{pxX} we get
\[ 
\begin{split}
| \nabla_x^2 X(s;t,x) | \lesssim &  \int_s^t \left( \| \nabla_x^2 u_E (\tau) \|_{L^\infty} | \nabla_x X(\tau;t,x) | + \| \nabla_x u_E (\tau) \|_{L^\infty} | \nabla_x^2 X(\tau;t,x) | \right) d\tau
\\ \lesssim & T C + \int_s^t  | \nabla_x^2 X(\tau;t,x) | d\tau
\end{split}
\]
Using Gronwall we get
\Be \label{px2X}
\sup_{0 \le  s \le  t \le T^* } \sup_{x \in \O } | \nabla_x^2 X(s;t,x) | \lesssim T  C e^{T   } < C_1.
\Ee
Combing \eqref{thetaEsol}, \eqref{pxX}, and \eqref{px2X} we get
\Be \label{thetaEC2}
\begin{split}
& \sup_{0 \le t \le T^*} \sup_{x \in \O} \left( | \nabla_x \theta_E(t,x) | + | \nabla_x^2 \theta_E(t,x) | \right)
\\ &  \lesssim \left( \| \nabla_x \theta_0 \|_{L^\infty} + \| \nabla_x^2 \theta_0 \|_{L^\infty} \right)  \sup_{0 \le t \le T^*} \sup_{x \in \O}  \left( | \nabla_x X(s;t,x) | +  | \nabla_x^2 X(s;t,x) | \right) < C_2.
\end{split}
\Ee

Now let $w = \theta - \theta_E $, then from \eqref{equuE1}, the equation for $w$ is
\Be \label{eqnw1}
\p_t w + u \cdot \nabla w + ( u- u_E) \cdot \nabla \theta_E  - \kappa \eta_c \Delta w =  \kappa \eta_c \Delta \theta_E.
\Ee

From \eqref{equuE1} and \eqref{thetaE0bdry}, $w = 0 $ on $\p \O$. Thus from integration by parts one get $\int_{\O } (\Delta w) w \, dx = - \int_{\O} | \nabla w |^2 dx $ without the boundary term. Therefore the standard energy estimate $\int_{\O} \eqref{eqnw1} \cdot w \, dx $ gives
\Be \label{enew1} \begin{split}
& \frac{1}{2} \frac{d}{dt } \| w \|_{L^2}^2 + \kappa \eta_c \| \nabla w \|_{L^2}^2 + \int_{\O} (u \cdot \nabla w) w \,  dx 
\\  = &- \int_{\O}   ( u- u_E) \cdot \nabla \theta_E w \, dx  + \kappa \eta_c \int_{\O} \Delta \theta_E w \, dx.
\end{split} \Ee
From $ \nabla \cdot u = 0$, and $ u = 0 $ on $\p \O$, we have
\[
 \int_{\O} (u \cdot \nabla w) w \,  dx = \frac{1}{2}  \int_{\O} u \cdot  \nabla  (w^2 )  \,  dx =  -\frac{1}{2}  \int_{\O}  ( \nabla \cdot u ) w^2   \,  dx = 0.
\]
Now we bound the terms on the RHS of \eqref{enew1}. Using \eqref{thetaEC2} we have
\Be \label{west11}
\int_{\O}  |  ( u- u_E) \cdot \nabla \theta_E w |  \, dx \lesssim \| \nabla \theta_E \|_{L^\infty} \left( \| u - u_E \|_{L^2}^2 + \| w \|_{L^2}^2 \right) \lesssim  \| u - u_E \|_{L^2}^2 + \| w \|_{L^2}^2,
\Ee
and
\Be \label{west12}
\kappa \eta_c \int_{\O} |  \Delta \theta_E w |  \, dx  \lesssim  (\kappa \eta_c)^2 \| \Delta \theta_E \|_{L^2}^2 + \| w \|_{L^2}^2 \lesssim \kappa^2 + \| w \|_{L^2}^2.
\Ee
Now collect \eqref{west11}, \eqref{west12}, \eqref{enew1} gives
\Be \notag
\frac{d}{dt } \| w \|_{L^2}^2  \lesssim \| w\|_{L^2}^2 + \| u - u_E \|_{L^2}^2 + \kappa^2  ,
\Ee
Therefore by Gronwall's inequality
\[
\sup_{0 \le t \le T} \| w \|_{L^2}^2 \lesssim e^T \left(  \| u - u_E \|_{L^2}^2 + \kappa^2  \right) \to 0, \text{ as } \kappa \to 0.
\]

\end{proof}

Finally we prove the main theorem of the paper through combining separate estimates we established in different sections of the paper.
\begin{theorem}[Main Theorem]\label{main_theorem}
Let $\O$ be a upper half space in 3D:
 \Be\label{domain}
 \begin{split}
\O := \mathbb{T}^2 \times \R_+ \ni (x_1, x_2, x_3),
 \ \ & \text{where } \mathbb R_+ :=\{ x_3 \in \mathbb R: x_3 > 0 \},
 \\ & \text{ and } \mathbb{T} \text{ is the periodic interval } (-\pi, \pi).
 \end{split}
\Ee
Suppose an initial velocity field $u_{in}$ is divergence-free 
\Be \label{divuin}
\nabla_x \cdot u_{in}=0 \ \ \text{ in } \O,
\Ee 
the corresponding initial vorticity $\o_{in}= \nabla_x \times u_{in}$ and the initial heat $\theta_{in}$ satisfy
\Be \label{analyticspacein}
\o_{in}, \theta_{in} \in \mathfrak B^{\lambda_0,\kappa},
\Ee
with the real analytic space $\mathfrak B^{\lambda_0,\kappa}$ defined in (\ref{[]}) for some $\lambda_0>0$, such that (\ref{initial_norm}), (\ref{inital_norm_theta}) hold.
%\Be
% \sum_{0\leq |\beta|\leq 2}  \|D^\beta \p_t^\ell\o_{in}  \|_{1,\lambda_0}+\sum_{0\leq |\beta|\leq 2} \|D^\beta \p_t^\ell\o_{in}  \|_{\infty,\lambda_0, \kappa}  <\infty \  \text{ for } \ \ell=0,1,2. \label{initial_norm}
%\Ee
Further we assume that $\o_{in}$, $\theta_{in}$ satisfies the compatibility conditions (\ref{CC}) on $\p\O$.
%\Be\label{CC}
%\begin{split}
%\kappa \eta_{0} (\p_{x_3} + \sqrt{- \Delta_h})\o_{in,h}    = [\p_{x_3} (-\Delta)^{-1} (-u_{in} \cdot \nabla \o_{in, h} + \o_0 \cdot \nabla u_{in,h}) ]  , \ \ 
%\o_{in,3}    =0,   \  \  \p_t\o_{in,3}  = 0 . 
%\end{split}
%r\Ee
 Then there exists a unique real analytic solution $(u(t,x), \theta(t,x))$ to the Navier-Stokes-Fourier flow (\ref{NS})-(\ref{Bouss}) in $[0,T] \times \O$, while $T>0$ only depends on $\lambda_0$ and the size of the initial data as in (\ref{initial_norm}), (\ref{inital_norm_theta}).

\hide

 for $\lambda_0>0$ and satisfies (\ref{initial_norm}), while %: 
%\Be
% \sum_{0\leq |\beta|\leq 2}  \|D^\beta \p_t^\ell\o_0  \|_{1,\lambda_0}+\sum_{0\leq |\beta|\leq 2} \|D^\beta \p_t^\ell\o_0 \|_{\infty,\lambda_0, \kappa}  <\infty \  \text{ for } \ \ell=0,1,2, \notag
%\Ee
where $(\o_0, \p_t \o_0, \p_t^2 \o_0)$ and $(u_0,\p_t u_0)$ defined through the equation as in (\ref{idata}). Further assume that $\o_0$ satisfies the compatibility conditions (\ref{CC}).   
%\Be\notag
%\begin{split}
%\kappa \eta_0 (\p_{x_3} + \sqrt{- \Delta_h})\o_{0,h} \, |_{x_3=0} &= [\p_{x_3} (-\Delta)^{-1} (-u_0 \cdot \nabla \o_{0, h} + \o_0 \cdot \nabla u_{0,h}) ] \, |_{x_3=0}\\
%\o_{0,3} |_{x_3=0} &=0, \quad \p_t\o_{0,3} |_{x_3=0} = 0 . 
%\end{split}
%\Ee
\unhide
Choosing a pressure $p(t,x)$ such that $p(t,x) \to 0$ as $x_3 \to \infty$, and setting the first order and the second order correction terms $f_1$, $f_2$ as (\ref{f_1})-(\ref{Pf2choice}), we also choose $\e$ and $\kappa$ in the relation of (\ref{choice:delta}).
%\hide
%, for some $\mathfrak{C}\gg 1$, 
% \Be\label{choice:delta:intro}
%%\delta= \sqrt{\e} \ \ \text{and} \ \ 
%%\delta
%\sqrt{\e}=\exp\Big(  \frac{  - \mathfrak{C}  T}{ \kappa^{1/2} }\Big) 
%% \Big\{1+    %\mathfrak{C}^\prime \kappa^{-\mathfrak{P}} T 
%% \Big(  \frac{  \mathfrak{C}_2 T}{ \kappa^{ \mathfrak{P}} }\Big)^2   \exp\Big(  \frac{  \mathfrak{C}_2 T}{ \kappa^{ \mathfrak{P}} }\Big) 
%   % \Big\}^{-1}    
% . 
%\Ee
%\unhide
%
%
%%as in  (\ref{choice:delta}). 
%
%
Assume that an initial datum for the remainder $f_{R,in }$ satisfies (\ref{initial_EF}), then for the given $T>0$, 
%\Be\label{initial_EF}
 %\sqrt{\mathcal{E}(0) }+ \sqrt{\mathcal{F} _p (0) } \lesssim1.
% \exp\Big( \frac{1}{\kappa^{\mathfrak{P}}}\Big)
%\Ee
 we construct a unique solution $f_R(t,x,v)$ of the form of %:
 \Be\notag
 F  = \mu + \e f_1 \sqrt \mu + \e^2 f_{2} \sqrt \mu  + \e ^{3/2} f_{R} \sqrt \mu  \ \ \text{in} \ \ [0,T] \times \O \times \R^3,
 \Ee
which solves the Boltzmann equation (\ref{Boltzmann}) and the diffuse reflection boundary condition (\ref{diffuse_BC}) with the scale of (\ref{StMae}), (\ref{Knke}), (\ref{choice:delta}), and satisfies the initial condition 
\[
F|_{t=0}=  \mu + \e  \sqrt \mu f_1 |_{t =0 } + \e^2 \sqrt \mu  f_{2} |_{t=0} + \e ^{3/2}  \sqrt \mu f_{R,in},
\]
such that, for each $\e$ and $\kappa$ of (\ref{choice:delta}), 
\Be\label{approx}
\sup_{0 \leq t \leq T}\left\|  \frac{F (t,x,v)-  \mu (v)}{\e  \sqrt{ \mu  (v)} } - f_1 \right\|_{L^2(\O \times \R^3)}  \lesssim 
\exp\Big(  \frac{  - \mathfrak{C}  (T + 1 )}{ 4 \kappa }\Big)   \ \ \text{for} \ \ \kappa \ll1. 
\Ee
Moreover. let $u_E(t,x), \theta_E(t,x) $ be the unqiue solution of the incompressible Euler equations under heat transfer (\ref{Euler})-(\ref{heat_E}) with the initial condition $(u_E, \theta_E)|_{t=0}= (u_{in}, \theta_{in} )$ satisfying (\ref{divuin}), (\ref{analyticspacein}). Then 
   \Be\notag
\sup_{0 \leq t \leq T}\left\|
\frac{F (t,x,v)- \mu (v)}{\e  \sqrt {\mu (v) } } - \left( - \theta_E(t,x) + u_E(t,x) \cdot v + \theta_E(t,x) \frac{|v|^2-3}{2} \right) \sqrt {\mu(v)}  \right\|_{L^2(\O \times \R^3 )}
\\    \longrightarrow 0
% \ \ \text{as} \ \ \e \downarrow 0.
\Ee
as $\e \to 0$.

  \end{theorem}

  \hide
%We prove Theorem \ref{main_theorem} as a direct consequence of theorems following in this section. Since some of notions of them are complex 

We give a proof of Theorem \ref{main_theorem}, as a direct consequence of theorems following in this section,
%using Theorem \ref{main_theorem:conditional}, Theorem \ref{thm_bound}, and Theorem \ref{thm_bound}, 
while they are deliberately delayed to be stated since some of notions are complex:
\unhide
\begin{proof}%[\textbf{Proof of Theorem \ref{main_theorem}}] 
%We established each part of the main theorem through separate theorems in different sections of the paper. 
The existence of the Navier-Stokes-Fourier system follows from Theorem \ref{thm_bound}. For the remaining assertions, we note that all  the estimates (\ref{est:u_t})-(\ref{ut}) of Theorem \ref{thm_bound} ensure the conditions of Theorem \ref{main_theorem:conditional} with $\mathfrak{P}=\frac{1}{2}$. Therefore \eqref{approx} follows  directly as a consequence of Theorem \ref{main_theorem:conditional} and Theorem \ref{thm_bound}. 
As for the incompressible Euler limit, note that 
\[
\begin{split}
 &   \frac{F (t,x,v)- \mu (v)}{\e  \sqrt {\mu (v) } }  -  \left( - \theta_E + u_E \cdot v + \theta_E \frac{|v|^2-3}{2} \right) \sqrt {\mu}    
  \\ & =   \big[ \frac{F (t,x,v)- \mu (v)}{\e  \sqrt {\mu (v) } }  - f_1\big] +\big[ f_1 -  \left( - \theta_E + u_E \cdot v + \theta_E \frac{|v|^2-3}{2} \right) \sqrt {\mu}\big].
  \end{split}
\]
The first term can be bounded as in (\ref{approx}). 
%We bound the second term by an expansion: 
%\Be\begin{split}\notag
%%&|M_{1, \e u(t,x), 1} (v) - M_{1, \e u_E(t,x), 1} (v)|\\&\lesssim  
%|u(t,x)-u_E(t,x)|
%\int^\e_0 |(v-\e u_E) + a(u_E-u)| e^{- \frac{|(v-\e u_E) + a(u_E-u)|^2}{2}} \dd a .
%\end{split}\Ee
%Note that $\|\e  u \|_{L^\infty} \ll1 $ and $\|\e u_E\|_{L^\infty}\ll 1$ from Theorem \ref{thm_bound}. Then we conclude that 
%
The second term converges to $0$ as $\kappa \downarrow 0$ from Theorem \ref{thm_bound}, Lemma \ref{thetaEtheta}, and the famous Kato's condition for vanishing viscosity limit in \cite{kato}.
\end{proof}

\section*{Acknowledgements} 

JJ was supported  in part by the NSF DMS-grant 2009458 and WiSE Program at the University of Southern California. CK was partly supported by National Science Foundation under Grant No. 1900923 and the Wisconsin Alumni Research Foundation.

%\section{Possible extension}
%
%1) non-isothermal boundary $\theta_w$ is not constant: check \cite{SR}
%
%2) $u_w$ is not constant
%

%\appendix 
%
%\section{Convergence of $\theta$ to $\theta_E$}
%We 

\hide

The next Lemma is a Kato's condition for $\theta$, which would allow the convergence of $\theta$ to $\theta_E$ without assuming $\theta_0 = 0 $ on $\O$. However, we do need to assume $\theta_0 = 0 $ on $\O$ when proving the regularity estimates \eqref{thetanorm_bound}, \eqref{thetapx3_bound} for $\p_t^\ell \theta$, $\ell = 0,1,2$. So this lemma is not useful for this paper.

\begin{lemma}[Kato's condition for $\theta$]
Let $u_E$ be the solution of the Euler equation \eqref{Euler}, \eqref{incomp_E}, \eqref{no-pen}, and $u$ be the solution of the Navier-Stokes equation \eqref{NS}, \eqref{incomp_NS}, \eqref{NSB}. For the following equations
\Be \label{equuE} \begin{split}
\p_t \theta_E + u_E \cdot \nabla_x \theta_E   & =0 \ \ \text{in} \ \O,  % \ \ u_E \cdot n =0  \ \ \text{on} \ \p\O, 
 \\ \p_t \theta + u \cdot \nabla_x \theta  - \kappa \eta_c \Delta \theta & =0 \ \ \text{in} \ \O,
% \ \ u =0  \ \ \text{on} \ \p\O, 
 \ \ \theta = 0 \ \ \text{on} \ \p\O.
\end{split} \Ee
Assume $ \theta_E(0,x) =  \theta (0,x) = \theta_0(x)$. Suppose
\Be \label{kato1}
\kappa \int_0^T \int_{\O} |  \nabla_x \theta | \, dx \to 0, \text{ as } \kappa \to 0,
\Ee
and
\Be \label{kato}
 \kappa \int_0^T \int_{\O_\kappa} \left( | \nabla_x u |^2 + | \nabla_x \theta|^2\right) dxdt  \to 0, \text{ as } \kappa \to 0,
\Ee
where $\O_\kappa := \{ x = (x_1, x_2,x_3) \in \O: x_3 < \kappa \}$.
Then
\Be
\sup_{0 \le t \le T}  \| \theta - \theta_E \|_{L^2(\O ) } \to 0, \text{ as } \kappa \to 0.
\Ee

\end{lemma}

\begin{proof}
Let $\chi : \mathbb R^3 \to \mathbb R^3$ be a smooth cut-off function with support in $[0,1]$ and $\chi (0 ) = 1 $. Let
\[
\theta_K := - \chi( \frac{x_3}{\kappa}) \theta_E.
\]
Then $\theta_E + \theta_K = 0$ on $\p \O$, and from \eqref{thetaEC2}, we have the following bound in terms of $\kappa$
\Be \label{thetaKest}
\| \theta_K \|_{L^\infty} \lesssim 1, \, \| \p_t \theta_K \|_{L^2} \lesssim \kappa^{1/2}, \, \| \theta_K \|_{L^2} \lesssim \kappa^{1/2}, \, \| \nabla \theta_K \|_{L^\infty} \lesssim \kappa^{-1}, \, \| \nabla \theta_K \|_{L^2} \lesssim \kappa^{-1/2}.
\Ee
Let $w = \theta - \theta_E -\theta_K$, then from \eqref{equuE}, the equation for $w$ is
\Be \label{eqnw}
\p_t w + u \cdot \nabla w + ( u- u_E) \cdot \nabla \theta_E  - \kappa \eta_c \Delta \theta = -\p_t \theta_K - u \cdot \nabla \theta_K.
\Ee
The standard energy estimate $\int_{\O} \eqref{eqnw} \cdot w \, dx $ gives
\Be \label{enew} \begin{split}
& \frac{1}{2} \frac{d}{dt } \| w \|_{L^2}^2 + \kappa \eta_c \| \nabla \theta \|_{L^2}^2 + \int_{\O} (u \cdot \nabla w) w \,  dx 
\\  = &- \underbrace{ \int_{\O}   ( u- u_E) \cdot \nabla \theta_E w \, dx }_{\eqref{enew}_1}  - \underbrace{ \kappa \eta_c \int_{\O} \Delta \theta  ( \theta_E + \theta_K) \, dx}_{\eqref{enew}_2}  - \int_{\O} \p_t \theta_K w \, dx -   \underbrace{ \int_{\O} (u \cdot \nabla \theta_K) w \, dx}_{\eqref{enew}_3}.
\end{split} \Ee
From $ \nabla \cdot u = 0$, and $ u = 0 $ on $\p \O$, we have
\[
 \int_{\O} (u \cdot \nabla w) w \,  dx = \frac{1}{2}  \int_{\O} u \cdot  \nabla  (w^2 )  \,  dx =  -\frac{1}{2}  \int_{\O}  ( \nabla \cdot u ) w^2   \,  dx = 0.
\]
Now we bound the terms on the RHS of \eqref{enew}. From \eqref{thetaEC2} We have
\Be \label{west1}
\left| \eqref{enew}_1 \right|  \lesssim \| \nabla \theta_E \|_{L^\infty} \left( \| u - u_E \|_{L^2}^2 + \| w \|_{L^2}^2 \right) \lesssim \| u - u_E \|_{L^2}^2 + \| w \|_{L^2}^2.
\Ee
Since $\theta_E + \theta_K = 0 $ on $\p \O$, we have $\int_{\O} \Delta \theta( \theta_E + \theta_K ) \, dx = \int_\O - \nabla \theta ( \nabla \theta_E + \nabla \theta_K ) \, dx $, thus from \eqref{thetaKest} and \eqref{thetaEC2}
\Be \label{west2} \begin{split}
\left| \eqref{enew}_2 \right| \lesssim  & \kappa \eta_c \| \nabla \theta_E\|_{L^\infty} \| \nabla \theta \|_{L^1}  + \kappa \eta_c \| \nabla \theta_K \|_{L^2 } \| \nabla \theta \|_{L^2_{\O_\kappa}}
\\ \lesssim & \kappa  \| \nabla \theta \|_{L^1} + \kappa^{1/2}   \| \nabla \theta \|_{L^2_{\O_\kappa}}.
\end{split} \Ee
From \eqref{thetaKest}, $   \int_{\O} |  \p_t \theta_K w  | \, dx  \lesssim   \kappa +  \| w \|_{L^2}^2 $. Finally since $ \int_{\O} ( u \cdot \nabla \theta_K) \theta_K \, dx = 0$, the Hardy's inequality, \eqref{thetaEC2}, and \eqref{thetaKest}, we have
\Be  \label{west3}
\begin{split}
\left|  \eqref{enew}_3 \right|  = & \left| \int_{\O}  (u \cdot \nabla \theta_K) (\theta - \theta_E) \, dx \right|
\\ \lesssim & \kappa^2 \| \nabla \theta_K \|_{L^\infty} \int_{\O_\kappa} | \frac{ | u |}{|x_3|} \frac{ | \theta - \theta_E | }{|x_3|} \, dx 
\\ \lesssim & \kappa \left( \| \nabla u \|_{L^2_{\O_\kappa}}^2 + \| \nabla (\theta - \theta_E) \|_{L^2_{\O_\kappa}}^2 \right) 
 \\ \lesssim &  \kappa \left( \| \nabla u \|_{L^2_{\O_\kappa}}^2 + \| \nabla \theta  \|_{L^2_{\O_\kappa}}^2 \right) 
\end{split}.
\Ee
Now collect \eqref{west1} - \eqref{west3}, \eqref{enew} gives
\Be \label{dtwbd}
\frac{d}{dt } \| w \|_{L^2}^2  \lesssim \| w\|_{L^2}^2 +  \underbrace{ \kappa +  \| u - u_E \|_{L^2}^2 + \kappa  \| \nabla \theta \|_{L^1} + \kappa^{1/2}   \| \nabla \theta \|_{L^2_{\O_\kappa}} +  \kappa \left( \| \nabla u \|_{L^2_{\O_\kappa}}^2 + \| \nabla \theta  \|_{L^2_{\O_\kappa}}^2 \right)}_{\eqref{dtwbd}_1}.
\Ee
Now from assumptions \eqref{kato1}, \eqref{kato}, we have $\eqref{dtwbd}_1 \to 0 $ as $\kappa \to 0$,
%\[
%\delta (\kappa) = \kappa +  \| u - u_E \|_{L^2}^2 + \kappa  \| \nabla \theta \|_{L^1} + \kappa^{1/2}   \| \nabla \theta \|_{L^2_{\O_\kappa}} +  \kappa \left( \| \nabla u \|_{L^2_{\O_\kappa}}^2 + \| \nabla \theta  \|_{L^2_{\O_\kappa}}^2 \right) \to 0, \text{ as } \kappa \to 0.
%\]
Therefore by Gronwall's inequality 
\[
\sup_{0 \le s \le T} \| w \|_{L^2}^2  \lesssim \eqref{dtwbd}_1 e^{T} \to 0, \text{ as } \kappa \to 0.
\] 
Thus from \eqref{thetaKest} we conclude
\[
\sup_{0 \le s \le T} \| \theta - \theta_E \|_{L^2}^2 \le \sup_{0 \le s \le T} \left(  \| w \|_{L^2}^2 + \| \theta_K \|_{L^2}^2 \right) \lesssim \sup_{0 \le s \le T}  \| w \|_{L^2}^2 + \kappa \to 0.
\]

\end{proof}

\unhide

\end{document}